\documentclass[twoside, 12pt]{article}
\pdfoutput=1

\usepackage{amsmath, amssymb, amsthm, bbm, multicol, rotating, tabularx, enumitem, floatflt, epsfig}
\usepackage[all]{xy}
\usepackage[a4paper, left=2.5cm, right=3cm, top=3cm, bottom=3cm]{geometry}
\usepackage[english, ngerman]{babel}

\usepackage{graphicx, color}
\usepackage{ucs}
\usepackage[utf8x]{inputenc}
\usepackage[T1]{fontenc}

\usepackage[pdfpagelabels]{hyperref}

\usepackage{fancyhdr}
\theoremstyle{plain}
\newtheorem{thm}{Theorem}[section]
\newtheorem*{thm*}{Theorem}
\newtheorem{prop}[thm]{Proposition}
\newtheorem*{prop*}{Proposition}
\newtheorem{lemma}[thm]{Lemma}
\newtheorem*{lemma*}{Lemma}
\newtheorem{example}[thm]{Example}
\newtheorem{corollary}[thm]{Corollary}
\newtheorem*{corollary*}{Corollary}
\newtheorem{property}[thm]{}
\newtheorem{conj}[thm]{Conjecture}
\newtheorem{observation}[thm]{Observation}

\newtheorem{defthm}[thm]{Theorem and Definition}

\newtheorem*{ThmIso*}{Theorem \ref{Thm_iso}}
\newtheorem*{ThmRD*}{Theorem \ref{Thm_RDexistence}}
\newtheorem*{ThmConv*}{Theorem \ref{Thm_convexity}}
\newtheorem*{ThmBNpair*}{Theorem \ref{Thm_BNpair}}

\theoremstyle{definition}
\newtheorem{definition}[thm]{Definition}
\newtheorem*{definition*}{Definition}

\theoremstyle{remark}
\newtheorem{remark}[thm]{Remark}
\newtheorem{notation}[thm]{Notation}

\newcommand{\C}{\mathbb{C}} 
\newcommand{\R}{\mathbb{R}} 
\newcommand{\N}{\mathbb{N}}
\newcommand{\Z}{\mathbb{Z}}
\newcommand{\Q}{\mathbb{Q}}
\newcommand{\define}{\mathrel{\mathop:}=}
\newcommand{\ddefine}{\mathrel{=\mathop:}}

\newcommand{\MS}{\mathbb{A}} 
\newcommand{\App}{\mathcal{A}}
\newcommand{\Hype}{\mathcal{H}}
\newcommand{\type}{\mathrm{type}} 
\newcommand{\seg}{\mathrm{seg}} 
\newcommand{\Aut}{\mathrm{Aut}} 
\newcommand{\Gdagger}{\mathrm{G}^\dagger} 
\newcommand{\binfinity}{\partial_\App} 

\newcommand{\RS}{\mathrm{R}}
\newcommand{\CC}{\Sigma}
\newcommand{\rk}{\mathrm{rank}}
\newcommand{\PP}{{\bf\mathrm{P}}} 
\newcommand{\QQ}{{\bf\mathrm{Q}}} 
\newcommand{\fweight}[1][\alpha]{\overline{\omega}_{#1}}
\newcommand{\fcw}[1][\alpha]{{\omega}_{#1}}
\newcommand{\sW}{\overline{W}} 
\newcommand{\aW}{W} 
\newcommand{\WT}{\overline{W}T}
\newcommand{\cf}{c_0} 
\newcommand{\Cf}{\mathcal{{C}}_{f}} 
\newcommand{\Cfm}{-\mathcal{{C}}_{f}} 
\newcommand{\Cp}{\mathcal{{C}}_{p}} 

\newcommand{\dconv}{\mathrm{{conv}}} 
\newcommand{\wgt}{\mathrm{{wgt}}} 
\newcommand{\Hdual}{\overline{H}}
\newcommand{\Ghat}{\hat{\mathcal{G}}} 
\newcommand{\G}{\mathcal{G}} 

\newcommand{\KK}{K_\theta^{(2)}} 
\newcommand{\LL}{L_\theta^{(2)}} 
\newcommand{\TT}{T^{(3)}} 
\newcommand{\const}{\frac{1}{\sqrt{2+\sqrt{2}}}} 
\newcommand{\cconst}{\sqrt{2+\sqrt{2}}} 
\newcommand{\one}{\mathbbm{1}} 
\newcommand{\Qd}{\mathcal{Q}_\mathcal{D}} 
\newcommand{\Sz}{\mathcal{S}z} 
\newcommand{\Ree}{\mathcal{R}ee} 
\newcommand{\Fee}{^2\mathrm{F}_4} 
\newcommand{\Btwo}{\mathcal{B}^\mathcal{D}_2} 
\newcommand{\Ffour}{\mathrm{F}_4} 
\newcommand{\Gtwo}{\mathrm{G}_2} 

\newcommand{\Evier}{E^{(4)}} 
\newcommand{\lb}{\langle} 
\newcommand{\rb}{\rangle} 

\newcommand{\val}{\omega} 

\newcommand{\Ant}{\mathrm{\widetilde{A}_n}}

\newcommand{\tie}{{\bowtie}}
\newcommand{\Tie}{{\mathcal{B}}}
 \newcommand{\TTie}{{\widetilde{\mathcal{B}}}}
\newcommand{\mepi}{{e}} 

\DeclareMathOperator{\Span}{span}

\numberwithin{equation}{thm}

\begin{document}

\hypersetup{pdfauthor={Petra Hitzelberger},pdftitle={Generalized affine buildings}}.


%
%



{\large{Mathematik} } \\
\vspace{15ex}

{\LARGE{Generalized affine buildings:}\\}
{\ }\\
{\Large{Automorphisms, affine Suzuki-Ree Buildings and Convexity}\\}

\vspace{5ex}
{\ }\vfill

{\large{Inaugural-Dissertation}}\\
{\small{zur Erlangung des akademischen Grades eines\\
Doktors der Naturwissenschaften\\
im Fachbereich Mathematik und Informatik\\
der Westf\"alischen Wilhelms Universit\"at M\"unster}}\\

\vspace{7ex}
{\small{vorgelegt von }}\\
{\large{Petra Hitzelberger}}\\
{\small{aus Pirmasens\\2008}}


\thispagestyle{empty}

\newpage

{\ }\vfill
\begin{tabular}{llll}
Dekan & & & Prof. Dr. Dr. h.c. Joachim Cuntz \\
Erster Gutachter & & & Prof. Dr. Linus Kramer \\
Zweiter Gutachter & & & Prof. Dr. Richard M. Weiss \\
Tag der m\"undlichen Pr\"ufung & & & 23. Januar 2009\\
Tag der Promotion & & & 23. Januar 2009
\end{tabular}

\thispagestyle{empty}

%
%
%

\thispagestyle{empty}

\newpage

\begin{flushright}

 \begin{bf}{Parabeln und R\"athsel}\\
 {Nummer 7}\\
 \end{bf}
 \vspace{3ex}
 Ein Gebäude steht da von uralten Zeiten,\\
 Es ist kein Tempel, es ist kein Haus;\\
 Ein Reiter kann hundert Tage reiten,\\
 Er umwandert es nicht, er reitet's nicht aus.\\
 \vspace{2ex}
 Jahrhunderte sind vorüber geflogen, \\
 Es trotzte der Zeit und der Stürme Heer; \\
 Frei steht es unter dem himmlischen Bogen, \\
 Es reicht in die Wolken, es netzt sich im Meer.\\
 \vspace{2ex}
 Nicht eitle Prahlsucht hat es gethürmet, \\
 Es dienet zum Heil, es rettet und schirmet;\\ 
 Seines Gleichen ist nicht auf Erden bekannt, \\
 Und doch ist's ein Werk von Menschenhand.\\
 \vspace{3ex}
 \small{Friedrich Schiller}
  
\end{flushright}

{\ } \vfill
\thispagestyle{empty}
\indent
\newpage 

\newpage
\thispagestyle{empty}
{\ } \vfill
\indent

\newpage
\pagestyle{fancy}
\pagenumbering{roman}
\selectlanguage{english}
\section*{Introduction}
\addcontentsline{toc}{section}{Introduction}
\label{sec_introduction}

Buildings, developed by Jacques Tits beginning in the 1950s and 1960s, have proven to be a useful tool in several areas of mathematics. Their theory is ``a central unifying principle with an amazing range of applications''.\footnote{http://www.abelprisen.no/nedlastning/2008/Artikkel\_7E.pdf, \emph{Why Jacques Tits is awarded the Abel Prize for 2008.}} 

First introduced to provide a geometric framework in order to understand semisimple complex Lie groups,
the theory of buildings quickly developed to an area interesting in its own right.

One essentially distinguishes  three classes of buildings differing in their apartment structure: There are the spherical, affine and hyperbolic (sometimes called Fuchsian) buildings whose apartments are subspaces isomorphic to tiled spheres, affine or hyperbolic spaces, respectively. Affine buildings, which are a subclass of the geometric objects studied in the present thesis, were introduced by Bruhat and Tits in \cite{BruhatTits} as spaces associated to semisimple algebraic groups defined over fields with discrete valuations. They were used to understand the group structure by means of the geometry of the associated building. The role of these affine buildings is similar to the one of symmetric spaces associated to semisimple Lie groups.

Spherical and affine buildings, in the aforementioned sense, were viewed at that time as simplicial complexes with a family of subcomplexes, the \emph{apartments}, satisfying certain axioms. All maximal simplices, the \emph{chambers}, are of the same dimension.
Buildings are extensively studied by numerous authors and several books have been written on this subject. There are for example the recent monograph by Abramenko and Brown \cite{AB}, which is a sequel to the introductory book by Brown \cite{Brown}, Garrett's book \cite{Garrett} and Ronan's \emph{Lectures on buildings} \cite{Ronan}. Great references for their classification, which is due to Tits \cite{TitsComo}, are the books of Weiss \cite{Weiss,AffineW}.

Nowadays several approaches to buildings provide a great variety in the methods used to study buildings as well as in the possibilities for applications. Above, we already mentioned the \emph{simplicial approach} where buildings are viewed as simplicial complexes, but there are several equivalent ways to characterize buildings. 

They can, for example, be described as a set of chambers together with a distance function taking values in a Weyl group. Here one forgets completely about apartments and simplices other than chambers. In this \emph{$W$-metric} or \emph{chamber system approach}, which is explained in \cite{AB}, chambers can be thought of as vertices of an edge colored graph, where two chambers are adjacent of color $i$ if, spoken in the language of the simplicial approach, they share a co-dimension one face of type $i$. This viewpoint is taken in \cite{AffineW}. 

Thinking of a building as the geometric realization of one of these structures just described, it turns out that one obtains a metric space satisfying certain nice properties.
Davis \cite{Davis} proved that each building, be it affine or not, has a metric realization carrying a natural $\mathrm{CAT}(0)$ metric. In the case of affine buildings this was already shown by Bruhat and Tits in \cite{BruhatTits}. In fact, spherical and affine buildings are characterized by metric properties of their geometric realizations, as proven by Charney and Lytchak in \cite{LC}.

This so called \emph{metric approach} is the viewpoint generalizing to non-discrete affine buildings and allows the treatment of geometric realizations of simplicial affine buildings as a subclass of this generalized version.

In \cite{TitsComo} and \cite{BruhatTits, BruhatTits2} affine buildings were generalized allowing fields with non-discrete (non-archimedian) valuations rather than discrete valuations. The arising geometries, which no longer carry a simplicial structure, are nowadays usually called \emph{non-discrete affine buildings} or \emph{$\R$-buildings}. Some readers might be familiar with $\R$-trees which appear in several areas of mathematics. They are the one-dimensional examples. In \cite{TitsComo}, $\R$-buildings were axiomatized and, for sufficiently large rank, classified under the name \emph{syst\`eme d'appartements}. A short history of the development of the axioms can be found in \cite[Appendix 3]{Ronan}. A recent geometric reference for non-discrete affine buildings is the survey article by Rousseau \cite{Rousseau}.

Buildings allowed the classification of semisimple algebraic and Lie groups, but also have many other uses.
Applications are known in various mathematical areas, such as the cohomology theory of groups, number theory, combinatorial group theory or (combinatorial) representation theory, which we make use of in Section~\ref{Sec_KostantConvexity}. Connections to incidence geometry, the theory of Kac-Moody groups (which are used in theoretical physics) and several aspects of group theory are known. For example, specific presentations of groups which act on a building are obtained. 

Furthermore, geometric realizations of buildings provide an interesting class of examples of metric spaces. This leads directly to a connection with differential geometry.  
Studying, for example, asymptotic cones of symmetric spaces, $\R$-buildings arise in a natural way; compare for example the work of Kleiner and Leeb \cite{KleinerLeeb} or Kramer and Tent \cite{KTcones}. Notice that Kleiner and Leeb's  \emph{Euclidean buildings} are, as proven by Parreau \cite{Parreau}, a proper subclass of Tits' syst\`eme d'appartements.

Finally, in \cite{Bennett, BennettDiss} Bennett introduced a class of spaces called \emph{affine $\Lambda$-buildings} giving axioms similar to the ones in \cite{TitsComo}. Examples of these spaces arise from simple algebraic groups defined over fields with valuations now taking their values in an arbitrary ordered abelian group $\Lambda$, instead of $\R$. The biggest difficulty in defining affine $\Lambda$-buildings arose in the definition of an apartment structure and of a metric, which now is $\Lambda$-valued. Bennett was able to prove that affine $\Lambda$-buildings again have simplicial spherical buildings at infinity and made major steps towards their classification.
Throughout this thesis we will refer to affine $\Lambda$-buildings as \emph{generalized affine buildings} to avoid the appearance of the group $\Lambda$ in the name.

The class of generalized affine buildings does not only include all previously known classes of (non-discrete) affine buildings, but also generalizes \emph{$\Lambda$-trees} in a natural way, so that $\Lambda$-trees are just the generalized affine buildings of dimension one. Standard references for $\Lambda$-trees are the work of Morgan and Shalen \cite{MorganShalen}, of Alperin and Bass \cite{AlperinBass} and the book by Chiswell \cite{Chiswell}. Applications of $\Lambda$-trees are explained in Morgan's survey article \cite{Morgan}. 

Studying generalized affine buildings rather than non-discrete affine buildings has an important advantage: the class of generalized affine buildings is closed under ultraproducts. Kramer and Tent made use of this fact in \cite{KramerTent} where they give a geometric construction of generalized affine buildings and a new proof of the topological rigidity of non-discrete affine buildings thereby simplifying the proof of the Margulis conjecture.

Little is known about generalized affine buildings. As far as I know, the only references (besides material on $\Lambda$-trees) are \cite{Bennett, BennettDiss} and \cite{KramerTent}.
The present thesis started out aiming at a classification result for generalized affine buildings (a step in this direction is Theorem~\ref{Thm_iso}) and ended up as a collection of miscellaneous results. We hope, however, that we were able to provide an accessible introduction to the theory of generalized affine buildings and to add some useful results to their structure theory.

In addition to the three subjects explained below, in Sections~\ref{Sec_modelSpace} and \ref{Sec_generalizedAffineBuildings} we study the structure of the model space of apartments, the local and global structure of generalized affine buildings and collect certain facts about the one-dimensional case. The latter is done in Section \ref{Subsec_trees} and Appendix \ref{Sec_prooftrees}. Small simplifications to the work of Bennett \cite{Bennett} are made as well.

Let us mention two of the results of Section \ref{Sec_generalizedAffineBuildings}.
Apartments of generalized affine buildings carry an action of a spherical Coxeter group, called the Weyl group. A closure of a fundamental domain of this action is called a Weyl chamber. We say that two Weyl chambers $S,S'$ in $X$ having the same basepoint $x$ are \emph{equivalent} if they coincide in a non-empty neighborhood of $x$. The equivalence class is called a \emph{germ of $S$ at $x$}. 
In Theorem~\ref{Thm_residue} we show, following Parreau \cite{Parreau}, that the equivalence classes of Weyl chambers based at a given point $x$ form the chambers of a (simplicial) spherical building, the \emph{residue of $X$ at $x$}. These residues are, spoken in the language of chamber systems, precisely the residues of special vertices having co-rank one, which are, by definition, the connected components of the edge colored graph considering adjacency with respect to all but one (special) type.  

In analogy to a classical result it was proven in \cite{Parreau} that $\R$-buildings admit a maximal atlas. Two facts are crucial for her proof of this statement: Given a building $X$ equipped with a system of apartments $\App$, also called \emph{atlas}, one first has to see that an intersection of isometric embeddings of half-apartments is contained in a single apartment with chart in $\App$. Secondly, one has to prove that the germs of Weyl chambers $S$ at $x$ and $S'$ at $x'$ are contained in a common apartment. We were not able to generalize the first, but prove the second fact for generalized affine buildings in Proposition \ref{Prop_A3'}.

This thesis is organized as follows: In Section \ref{Sec_rootSystems} we give a brief overview on root systems. We cover basic material on simplicial buildings, which is needed for Section \ref{Sec_KostantConvexity}, in Section \ref{Sec_classicalCase}. The model apartment of a generalized affine building is defined in Section \ref{Sec_modelSpace} where we also discuss its metric structure.
Generalized affine buildings are then defined in Section \ref{Sec_generalizedAffineBuildings} where the aforementioned results on their local structure are proved. The proof of a theorem on trees, stated in the same section, appears in appendix \ref{Sec_prooftrees}. Section \ref{Sec_automorphisms} contains a generalization of a result by Tits about automorphisms of affine buildings. Convexity results for simplicial as well as for generalized affine buildings are proven in Sections \ref{Sec_KostantConvexity} and \ref{Sec_convexityRevisited}, respectively.  Finally in Section \ref{Sec_ReeSuzuki} we prove a partial result of  \cite{Octagons} using algebraic instead of geometric methods. Note that $\Lambda=\R$ in Section \ref{Sec_ReeSuzuki}. Detailed introductions to the last three topics mentioned are given separately below.
The appendix \ref{Sec_dictionary} of this thesis contains a \emph{dictionary} showing the correspondence of names of objects used in the present thesis and the names of their counterparts in \cite{Bennett, BruhatTits, KapovichMillson, KleinerLeeb, Parreau, TitsComo} and \cite{AffineW}.

Parts of the results contained in Section \ref{Sec_KostantConvexity} are available in \cite{Convexity}. Additional work on the subject covered in Section \ref{Sec_ReeSuzuki} can be found in \cite{Octagons}.
Let me now give a detailed introduction to the three major subjects of the present thesis.

\subsubsection*{Automorphisms}

The main task of Section~\ref{Sec_automorphisms} is to generalize a well known result by Tits \cite{TitsComo} to generalized affine buildings. Theorem~\ref{Thm_iso} gives a necessary and sufficient condition on automorphisms of the building at infinity of a generalized affine building $X$ to extend to an automorphism of $X$.

An affine building $(X,\App)$ gives rise to a collection of panel- and wall-trees which are themselves one-dimensional buildings, i.e. $\Lambda$-trees, encoding the branching of the apartments. In fact, the building can be reconstructed given its building $\binfinity X$ at infinity and these trees. The proof of Theorem~\ref{Thm_iso} goes back to Tits who described points in in an affine building using the spherical building at infinity and the mentioned trees. This idea is precisely reflected in the definition of a \emph{bowtie} which was first given in \cite{Leeb}, where Leeb proves a version of \ref{Thm_iso} for thick simplicial affine buildings. We simplified his definition of a bowtie and generalized it to (not necessarily thick) generalized affine buildings in the sense of Definition~\ref{Def_LambdaBuilding}. The main idea of Leeb's proof does, with suitable changes, carry over to generalized affine buildings.

A certain equivalence relation on the set of bowties of a generalized affine building is defined whose equivalence classes are in one to one correspondence with points in the affine building, see Proposition \ref{Prop_oek5}. Due to the fact that bowties are defined using the tree structure it is not surprising that the crucial condition for a morphism of the buildings at infinity of two different generalized affine buildings $X$ and $Y$ to extend to a morphism from $X$ to $Y$ is, that it has to preserve their tree structure. Maps satisfying this condition are called \emph{ecological}.\footnote{This colorful name, \emph{ecological}, was suggested by Richard Weiss in \cite{AffineW}.}  Note that, in contrast to \cite{TitsComo, Leeb} and \cite{AffineW} we do not assume the buildings to be thick.

Let $\Lambda$ and $\Gamma$ be ordered abelian groups and assume that there exists an epimorphism $\mepi:\Lambda\rightarrow \Gamma$. Let $X_\Lambda$ and $X_\Gamma$ be generalized affine buildings with spherical buildings $\Delta_\Lambda$ and $\Delta_\Gamma$ at infinity. We prove

\begin{ThmIso*}
An ecological isomorphism $\tau:\Delta_\Lambda\rightarrow\Delta_\Gamma$ of the buildings at infinity extends uniquely to a map $\rho:X_\Lambda\to X_\Gamma$ which is compatible with the given apartment structures on $X_\Lambda$ and $X_\Gamma$ and satisfies
\begin{equation*}
d_\Gamma(\rho(x), \rho(y))=\mepi(d_\Lambda(x,y)) \text{ for all points } x,y\in X_\Lambda,
\end{equation*}
where $d_\Lambda$, $d_\Gamma$ are the canonical metrics on $X_\Lambda$ and $X_\Gamma$, respectively.
\end{ThmIso*}
%

We obtain Theorem \ref{Thm_corollary1}, which says that for rank at least two, elements of the root groups of the building at infinity extend to the affine building, as a consequence of Theorem \ref{Thm_iso}. The proof of this fact for $\R$-buildings is a major step in the classification of non-discrete buildings \cite{TitsComo}. Chapter 12 of Weiss' book on the classification of affine buildings is dedicated to the proof in the simplicial case. Compare Theorem 12.3 of \cite{AffineW}.

Viewing a building as the set of equivalence classes of its bowties might be useful in other situations as well, like proving a higher dimensional analog of Proposition~\ref{Prop_baseChangeTrees}.

\subsubsection*{Buildings for the Ree and Suzuki groups}

Motivated by a remark made by Tits in \cite{TitsComo} we classified in \cite{Octagons} the non-discrete buildings having Suzuki-Ree buildings at infinity. In Section \ref{Sec_ReeSuzuki} an algebraic proof of a partial result of \cite{Octagons} is given.
The affine buildings dealt with in this section are all modeled on $\Lambda=\R$.

The automorphism group of the building at infinity of an affine $\R$- building contains certain geometrically defined subgroups, the so called \emph{root groups}, which form a \emph{root datum}. A non-discrete  \emph{valuation of the root datum}, defined in \ref{Def_RDvaluation}, is a collection of maps from the root groups to $\R$ satisfying certain ompatability condtions. As proved in \cite{TitsComo}, non-discrete affine buildings with a Moufang building at infinity are (up to equipollence) classified by non-discrete valuations of the root datum.

\emph{Suzuki-Ree buildings} appear as fixed point sets $\Delta^\rho$ of a (non-type preserving) polarity $\rho$ on a spherical building $\Delta$ of type $B_2, G_2$ or $F_4$. The polarity $\rho$ is defined by means of a Tits-endomorphism $\theta$ of the defining field of the spherical building.
These fixed point sets $\Delta^\rho$ are spherical buildings of type $A_1$ in cases $B_2$ and $G_2$ and of type $I_2(8)$ in case $F_4$. 
Let $\Gdagger$ be the subgroup of the automorphism group of $\Delta$ generated by the root groups. The group of automorphisms of $\Delta^\rho$ induced by the centralizer of $\rho$ in $\Gdagger$ is a \emph{Ree-group} in case $G_2$ and $F_4$ and a \emph{Suzuki-group} in case $B_2$. These groups where studied by Tits in \cite{TitsOctagons} and \cite{TitsSuzukiRee}.

The remark Tits made, as mentioned at the beginning, was that an arbitrary real valued non-archimedean valuation of a field $K$ extends to a valuation of the root datum (and hence defines a non-discrete affine building with boundary $\Delta^\rho$) if and only if  $\nu(k^\theta)=\sqrt{p}\;\nu(k)$ for all $k\in K$ where $p$ is the characteristic of $K$. We fill in all details in \cite{Octagons} and prove with purely geometric arguments the slightly stronger result \ref{Thm_strongerResult} that the non-discrete affine buildings determined by the induced valuations of the root datum of $\Delta^\rho$ are naturally embedded in a non-discrete affine building having $\Delta$ as boundary. For a precise formulation of the results compare Theorem 2.1 of \cite{Octagons}.

The following theorem is the main result of Section \ref{Sec_ReeSuzuki}. Based on the surprising equations established in Proposition \ref{Prop_key} we give a purely algebraic proof of this fact.
Note  that in \cite{Octagons} we obtained Theorem \ref{Thm_RDexistence} as a consequence of \ref{Thm_strongerResult}.

\begin{ThmRD*}
Let $K$ be the defining field of $\Delta$. Assume that $\nu$ is a $\theta$-invariant valuation of $K$ such that $\vert \nu(K)\vert \geq 3 $. 
Then $\nu$ extends uniquely to a valuation of the root datum of the Suzuki-Ree building $\Delta^\rho$.
Furthermore, there exists a non-discrete affine building having $\Delta^\rho$ as its building at infinity.
\end{ThmRD*}

As already mentioned, a valuation of the root datum corresponds to a non-discrete affine building. The Moufang spherical buildings associated to the Ree groups of type $F_4$ are precisely the Moufang generalized octagons. We have therefore classified all affine buildings having a Moufang generalized octagon at infinity.

Recent work by Berenstein and Kapovich \cite{BerensteinKapovich} provides the existence of thick non-discrete affine buildings of rank two having generalized $n$-gons, for arbitrary $n\geq 2$, at infinity. Their proof uses a strategy based on a free construction by Tits which provides thick spherical buildings of rank 2 modeled on arbitrary finite Coxeter groups.

\subsubsection*{Convexity}

Kostant \cite{Kostant} proved a convexity theorem for symmetric spaces, generalizing a well known theorem of Schur \cite{Schur}. Let $G\!/\!K$ be a symmetric space and $T$ a maximal flat of $G\!/\!K$. The spherical Weyl group $\sW$
 of $G\!/\!K$  acts on $T$. Denote by $\pi$ the Iwasawa projection of $G$ onto $T$. Kostant proved
that the image of the orbit $K.x$ under the Iwasawa projection is precisely the convex hull of the Weyl group orbit $\sW.x$.

Since affine buildings are in many ways similar to symmetric spaces it is natural to ask whether a convexity result in the spirit of Kostant's can be formulated for affine buildings. It turns out that the analog statement is true, assuming thickness. The role of flats is played by a fixed apartment $A$ in the building $X$, the Iwasawa projection corresponds to a certain retraction $\rho$ of the building onto $A$ and the orbit $K.x$ is replaced by the preimage of the Weyl-group orbit of a vertex $x$ by a second type of retraction which is denoted by $r$.

The retraction $r:X\to A$ is defined with respect to the germ of a certain Weyl chamber $\Cf$ and the retraction $\rho: X\to A$ with respect to a chamber in the building at infinity containing the opposite Weyl chamber $\Cfm$ of the fundamental Weyl chamber $\Cf$ in the apartment $A$. See Sections~\ref{Sec_retractions} and \ref{Sec_convexityThm} for definitions in the simplicial, respectively the generalized setting.

One can, in analogy to the definition of hyperplanes, define \emph{dual hyperplanes} $\Hdual$ which are co-dimension one subspaces of apartments perpendicular to (fundamental) co-weights. We call a subset of an apartment \emph{convex}\footnote{One could also use the somewhat redundant name \emph{dual convexity}.} if it is the intersection of finitely many dual half-apartments determined by dual-hyperplanes. 

In the simplicial case the result reads as follows

\begin{ThmConv*}
Denote by $\QQ$ the co-weight lattice in the fixed apartment $A$. If $X$ is thick we have for each special vertex $x\in X$ that
$$
\rho(r^{-1}(\sW.x)) =\dconv(\sW.x)\cap (x+\QQ),
$$
where $r:X\to A$ and $\rho: X\to A$ are the retractions mentioned above.
\end{ThmConv*}
Note that $x+\QQ$ is precisely the set of special vertices in $A$ having the same type as $x$. 

Let $G$ be a group with affine $BN$-pair whose spherical $BN$-pair is split, i.e. the spherical $B$ equals $UT$, where $T$ can be interpreted as a group of translations on the fixed apartment $A$. Let $K$ be the stabilizer in $G$ of the origin in $A$. The special vertices in $A$ of the same type as $0$ correspond to cosets $tK$ of $K$, with $t\in T$. In this case the result can be reformulated as follows.
 \begin{ThmBNpair*}
For all $tK\in A$ we have 
$$
\emptyset \neq Ut'K \cap KtK \:\Longleftrightarrow\: t'K \in \dconv(\sW. tK).
$$
\end{ThmBNpair*}
Details can be found in Section \ref{Sec_openProblem}.

Extending $\rho$ and $r$ to galleries the proof of Theorem \ref{Thm_convexity} can be reduced to the following two facts.
\begin{itemize}
  \item[\ref{Prop_existenceLSgalleries}] All vertices in $\dconv(\sW.x)$ of the same type as $x$ are endpoints of positively folded galleries of a certain (fixed) type. All such endpoints are contained in this set.
 \item[\ref{Prop_existencePreimage}] Every positively folded gallery in the statement above has a pre-image under $\rho$ which is a minimal gallery starting at $0$.
\end{itemize}
We obtain Proposition~\ref{Prop_existenceLSgalleries} as a consequence of a character formula for highest weight representations given by Gaussent and Littelmann in \cite{GaussentLittelmann}. 
It is obvious that we have to replace the representation theoretic arguments by something combinatorial in order to be able to prove the analog result for generalized affine buildings. During the preparation of this thesis Parkinson and Ram published the preprint \cite{ParkinsonRam} providing a combinatorial proof of \ref{Prop_existenceLSgalleries} on which the generalization made in Section \ref{Sec_convexityRevisited} is based. 

Unfortunately, generalizing Theorem \ref{Thm_convexity} turns out to be much more complicated than we had hoped for and we did not succeed in proving a direct analog. However, we were able to prove in Theorem \ref{Thm_convexityGeneral} that the statement is true for thick generalized affine buildings satisfying a certain condition (FC) described in \ref{Prop_finiteCovering}. We conjecture that (FC) holds for all generalized affine buildings. \emph{(Conjecture proven in Appendix \ref{Sec_FC}.)}

In Section \ref{Sec_looseEnds} we explain problems that occur generalizing the idea of \cite{ParkinsonRam} and formulate in \ref{Conj_convexityGeneral} the precise analog of Theorem \ref{Thm_convexity}. We are able to prove this fact in dimension one, see \ref{Thm_Conj1}.

\subsubsection*{Acknowledgments}
First of all it is a pleasure to thank my advisors Prof. Dr. Linus Kramer and Prof. Dr. Richard Weiss. I am profoundly grateful for all they taught me, for their support and their motivating interest in my research.
I am indebted to the \emph {Studientiftung des deutschen Volkes}, the \emph{Deutsche Forschungsgemeinschaft} and the  \emph{Sonderforschungsbe\-reich SFB 478 “Geometrische Strukturen in der Mathematik”}  for their financial support. 
Further I would like to thank Prof. Dr. Peter Abramenko  and Prof. Dr. Hendrik Van Maldeghem for their hospitality. Special thanks go to the members of Tufts University Department of Mathematics who made my stay a memorable experience. 
Finally I would like to thank all who helped by some means or other to make this thesis possible.

\newpage
\pagestyle{empty}

\tableofcontents

\newpage
\thispagestyle{empty}
\indent

\newpage
\pagestyle{fancy}
\pagenumbering{arabic}

\section{Root systems, Weyl groups and lattices}\label{Sec_rootSystems}

In this section we collect standard facts about root systems, the affine Weyl group and the root and weight lattice. This is mainly done in order to fix notation. If you are familiar with the presented material you might want to skip this section. For details and proofs we refer to \cite[Chapter IV,9 and Chapter VI]{Bourbaki4-6}.

\subsection{Definition of root systems}

\index{reflection}
Let $V$ be a (finite dimensional) Euclidean vector space with scalar product $(\cdot, \cdot)$.
For every nonzero vector $\alpha\in V$ denote by $s_\alpha$ the associated \emph{reflection} defined by 
\begin{equation}
\label{Equ_reflection}
s_\alpha(\beta) = \beta - 2\frac{(\alpha, \beta)}{(\alpha,\alpha)}\alpha.
\end{equation}

The reflection $s_\alpha$ fixes the hyperplane $H_\alpha=\{x\in V : (\alpha,x)=0\}$.

\begin{definition}\label{Def_rootsystem}
\index{{root system}}
A subset $\RS$ of $V$ is called \emph{root system} if the following two axioms are satisfied
	\begin{description}
	\item[(RS1)] $\RS$ is finite, spans $V$ and does not contain $0$, and 
	\item[(RS2)] $\RS$ is $s_\alpha$ invariant for all $\alpha \in \RS$. 
	\end{description}
The elements of $\RS$ are called \emph{roots}, the dimension of $V$ is called \emph{rank of $\RS$}. 
If in addition 
	\begin{description}
	\item[(RS3)] for all $\alpha, \beta \in \RS$ the value $\frac{2(\alpha, \beta)}{(\alpha, \alpha)}$ is an integer.
	\end{description}
the root system is called \emph{crystallographic}.
\end{definition}

Root systems as defined above are sometimes called \emph{reduced}. In order to include the dihedral groups we do in general not assume root systems to be crystallographic. If so, we mention it separately.

\begin{definition}
\index{{Weyl group}!spherical}
\label{Def_Weylgroup}
Let $\RS$ be a root system. The \emph{spherical Weyl group} $\sW$ is the group generated by the reflections $s_\alpha$, $\alpha\in \RS$. The connected components of $V\setminus \{H_\alpha : \alpha \in \RS \}$ are called \emph{open Weyl chambers} and their closures \emph{Weyl chambers} of $\RS$.
\end{definition}

\begin{remark}\label{Rem_CoxeterComplex}
\index{Weyl simplex}
Note that Weyl chambers are polyhedral cones in $V$. A face of the cone is called \emph{Weyl simplex}. The name is well chosen, since the set of all faces of Weyl chambers forms a simplicial complex called the \emph{Coxeter complex} of $\RS$. We will denote this simplicial complex with $\CC(\RS)$.
A \emph{root} of $\CC(\RS)$ is the set of Weyl simplices contained in a half-apartment determined by the hyperplanes $H_{\alpha,0}$, as defined below.
\end{remark}

\index{{Weyl group}!{length function}}
The group $\sW$ acts simply transitive on the set of Weyl chambers in $V$. Each Weyl chamber is a fundamental domain and has $n$ bounding hyperplanes $H_{\alpha_i}$ called \emph{walls}. The $n$ roots perpendicular to these walls and contained in $H_{\alpha_i}^+=\{x\in V : (\alpha,x)\geq 0\}$ form a \emph{basis} $B$ of $\RS$, which is at the same time a basis of the vector space $V$. For each such basis $B$ the $s_\alpha$, $\alpha\in B$, generate the Weyl group, which is a spherical Coxeter group. Given a basis $B$ one can define a \emph{length function} $l:\sW \rightarrow \N_0$ with $l(w)=\min \{k : w= s_{\alpha_1}\cdots s_{\alpha_k}\}$ with $\alpha_i\in B$ for all $i$. Denote by $w_0$ the longest element in $\sW$.

\index{{root system}!{simple root}}
\index{{root system}!basis}
\index{{root system}!{positive root}}
\index{{fundamental chamber}}

The elements of a basis $B$ are sometimes called \emph{simple roots}. They can canonically be enumerated as explained in \cite[VI,1,5, Remark 7]{Bourbaki4-6}. Hence $B=\{\beta_1, \ldots, \beta_n\}$, where $n=\rk(\RS)$. A root $\alpha\in\RS$ is called \emph{positive}, with respect to $B$, if the half-apartment $H_\alpha^+ =\{x\in V :\; (\alpha,x)\geq 0 \}$ contains the \emph{fundamental Weyl chamber} $C\define \cap_{\alpha_i\in B} H_{\alpha_i}^+$. Every positive root is a linear combination of simple roots with integer coefficients. We follow \cite{Bourbaki4-6} and denote half the sum of the positive roots by $\rho$.\index{$\rho$}

\begin{example}\label{Ex_I2(8)}
We give an example for a non-crystallographic root system. Let $\RS$ be a set of sixteen unit length vectors evenly distributed around the unit circle in $\R^2$. Enumerate them as illustrated in figure~\ref{Fig_figure1}. The associated spherical Weyl group, which is for example generated by $\alpha_1$ and $\alpha_8$, is the Coxeter group of type $I_2(8)$. The corresponding Coxeter complex is isomorphic to an apartment of a generalized octagon, i.e. a $16$-gon.

\begin{figure}[htbp]
\begin{center}
	\resizebox{!}{0.3\textheight}{\input{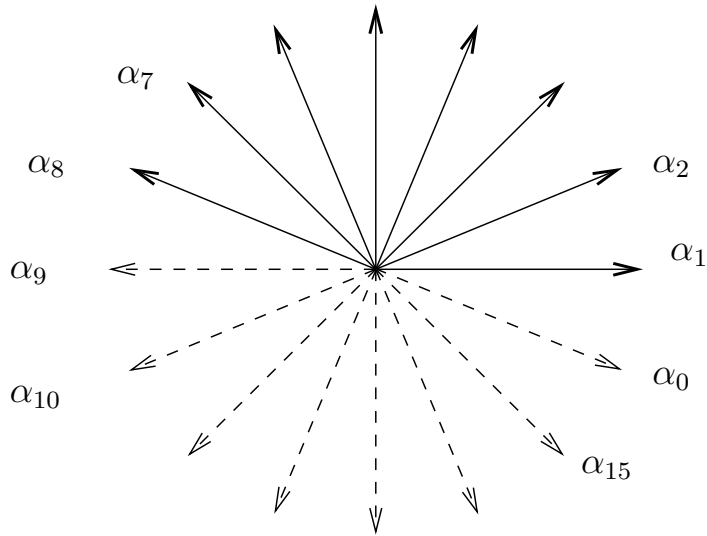}}
	\caption[octagon]{root system of type $I_2(8)$}
	\label{Fig_figure1}
\end{center} 
\end{figure}

\end{example}

\begin{definition}
\index{{root system}!dual}
Let $\RS$ be a root system. For all $\alpha \in \RS$ there exists a unique element $\alpha^\vee$ of the dual space $V^*$ of $V$ such that, 
by the usual identification of $V$ and $V^*$ via the scalar product, $\alpha^\vee$ corresponds to $\frac{2\alpha}{(\alpha, \alpha)}$.
The element $\alpha^\vee$ of $V^*$ is called \emph{dual root of $\alpha$} and the set $\RS^\vee$ the \emph{dual root system of $\RS$}.
\end{definition}

Note that by definition the evaluation $\langle \alpha,\alpha^\vee \rangle$ equals $2$.

The set $\RS^\vee$ is itself a root system, which is crystallographic if and only if $\RS$ is. In particular $(\RS^\vee)^\vee =\RS$. Identifying $V$ and $V^*$, the reflections $s_\alpha$ can be written as 
\begin{equation}
s_\alpha(\beta)=\beta - \alpha^\vee(\beta) \;\alpha.
\end{equation}

\subsection{The geometry of the affine Weyl group}

Let $\RS$ be crystallographic. 

\begin{definition}\label{Def_reflection}
\index{{reflection}!{affine}}
\index{{hyperplane}}
\index{{Weyl group}!affine}
Given $\alpha\in\RS$ and $k\in \Z$. Denote by $s_{\alpha, k}$ the \emph{affine reflection}
\begin{equation}
\label{Equ_affineReflection}
s_{\alpha, k}(x) = s_\alpha(x) + k\frac{2}{(\alpha,\alpha)} \alpha,  \text{ for all } x\in V,
\end{equation}
and by $H_{\alpha, k}$ the \emph{affine hyperplane}
\begin{equation}
\label{Equ_affineHyperplane}
H_{\alpha, k} 
	= \{x\in V : (\alpha, x) = k\} 
	= \{x\in V : \lb x,\alpha^\vee\rb = k\frac{(\alpha, \alpha)}{2} \}.
\end{equation}
The group generated by the set $\{s_{\alpha,k} :\; \alpha\in \RS, k\in\Z \}$ of affine reflections is called \emph{affine Weyl group} and denoted by $\aW$.
\end{definition}

Note that $s_{\alpha, k} = t_{\frac{2k}{(\alpha, \alpha)} \alpha} \circ s_\alpha$, where $t_v$ denotes the translation in $V$ by $v$. For all $\alpha\in\RS$ and all $k\in \Z$ the fixed point set of the affine reflection $s_{\alpha,k}$ equals the hyperplane $H_{\alpha, k}$. 
Denote by $\Hype$ the set $\{H_{\alpha,k} : \alpha\in\RS , k\in\Z\}$ of all affine hyperplanes.

\index{alcove}
Connected components $V\setminus \{\Hype\}$ are colled \emph{open alcoves}, their closures \emph{alcoves}. We have chosen the name alcove to avoid confusion with the chambers of $\sW$. 

\begin{definition}
\index{vertex!special}
\index{vertex}\label{Def_vertex}
An element $v\in V$ is called \emph{vertex} if it is an extremal point of an alcove. A vertex is called \emph{special} if it is the intersection of $n=\rk(\RS)$ hyperplanes.
\end{definition}

One can color the vertices of the fundamental alcove. This coloring extends to a coloring of all vertices in $V$ and we say two \emph{vertices are of the same type} if they have the same color.

\subsection*{The root and weight lattice}

\index{{co-root lattice}}
\index{{co-root lattice}!{radical coweights}}
\index{{weight lattice}}
\index{{weight lattice}!{weights}}
Let $\RS$ be a crystallographic root system.

The additive subgroup of $V$ generated by $\RS$ is a lattice, which is a free abelian subgroup of $V$ generated by a vector space basis. We call it the \emph{root lattice} $\QQ(R)$ of the root system $R$. We have
\begin{equation}\label{Equ_roots}
\QQ(R)=\{ x=\sum_{\alpha\in B} k_\alpha \alpha \in V : k_\alpha\in \Z \text{ for all } \alpha\in B\}.
\end{equation}
The \emph{weight lattice} $\PP(\RS)$ of $\RS$ is defined by the equation below. It is 
by Proposition 26 in chapter VI.9 of \cite{Bourbaki4-6} a discrete subgroup of $V$ containing $\QQ(\RS)$. 
\begin{equation}\label{Equ_weights}
\PP(\RS)=\{x\in V : \lb x,\alpha^\vee\rb\in \Z \text{ for all } \alpha\in\RS \}.
\end{equation}

Assume that $\RS$ is reduced and let $B$ be a basis of $\RS$. The elements of the dual basis of $B^\vee$ in $V$ are called \emph{fundamental weights}. They form a basis of $\PP(Q)$. Denote the dual element of $\alpha^\vee$ by $\fweight[\alpha]$, hence $\{\fweight\}_{\alpha\in\RS}$ is a basis of $\PP(\RS)$.
With this notation
\begin{equation*}
\PP(\RS)=\{x=\sum_{\alpha\in B} k_\alpha\fweight[\alpha] : k_\alpha\in\Z \text{ for all } \alpha\in\RS \}.
\end{equation*}

Obviously one can consider the \emph{co-root-lattice} $\QQ(\RS^\vee)$ of $\RS$ and the \emph{co-weight lattice} $\PP(\RS^\vee)$ of $\RS$. The fundamental weights of $\RS^\vee$ are sometimes called \emph{fundamental co-weights} of $\RS$. The fundamental co-weight corresponding to $\alpha^\vee$ is denoted by $\fcw$. Observe that identifying $V^*$ and $V$ in the usual way we have that
$$\QQ(R^\vee)\subset \QQ(R)\subset \PP(R)\subset \PP(R^\vee).$$

Is $\RS$ for example of type $\Ant$ then $\QQ(\RS)$ equals $\QQ(\RS^\vee)$ which is not the case in general.

\begin{prop}
Let $\RS$ be a reduced root system let $B$ be a basis of $\RS$ and let $\Cf$ be the associated fundamental Weyl chamber.
\begin{itemize}
  \item Let $x$ be an element of $V$. Then $x\in\Cf$ if and only if $\langle x,\alpha^\vee \rangle \geq 0$ for all $\alpha\in B$.
  \item For all $\alpha, \beta\in B$ we have $\langle \fweight, \beta^\vee  \rangle =\delta_{\alpha,\beta}$ and, symmetrically, 
$\lb \beta,\fcw\rb = \delta_{\alpha,\beta}$.
  \item Let $\rho=\frac{1}{2}\sum_{\alpha\in\RS^+} \alpha$ be as defined above. Then $\rho=\sum_{\alpha\in B} \fweight \in\Cf$.
  \item For all $\alpha\in B$ we have $s_\alpha(\rho)=\rho-\alpha$.
  \item Finally $\langle \rho, \alpha^\vee\rangle =1$ for all $\alpha\in B$ and, symmetrically, with $\rho^\vee=\frac{1}{2}\sum_{\alpha^\vee\in (\RS^\vee)^+} \alpha^\vee$ also $\langle \alpha,\rho^\vee \rangle =1$ for all $\alpha\in B$.

\end{itemize}
\end{prop}
\begin{proof}
See page 181 of \cite{Bourbaki4-6}.
\end{proof}

\index{{root system}!{highest root}}
There exists a unique \emph{highest root} $\theta \in \RS$ satisfying $(\fcw[i],\theta)\geq (\fcw[i],\alpha)$ for all $1\leq i\leq n$ and for all $\alpha\in\RS$.

The \emph{fundamental alcove} corresponding to a basis $B$ of $\RS$ is the set 
\begin{equation}
\label{Equ_fundAlcove}
\cf=\{x\in V : (\theta, x ) \leq 1 \text{ and } (\alpha_i, x)\geq 0 \text{ for all } i=1,\ldots, n \}.
\end{equation}
The affine Weyl group acts simply transitively on the set of alcoves. Hence the elements of $\aW$ are in bijection with the alcoves. Identify $\cf$ with the identity of $\aW$. Note that $\cf$ is a fundamental domain of the $\aW$-action on $V$.

\begin{remark}
If it is clear from the context to which root system we refer we will simply write $\QQ$, respectively $\PP$, instead of $\QQ(\RS)$ and $\PP(\RS)$.
For the translation subgroups of $V$ generated by $\{t_v,\, v\in \QQ\}$ and $\{t_v,\, v\in \PP\}$ we will use the same letters $\QQ$, respectively $\PP$. It should not be difficult to figure out from the context if an element of $\PP$ is interpreted as a translation or a vector space element. 
\end{remark}

\begin{prop}
Let $\RS$ be a crystallographic root system and let $\aW$ denote its affine Weyl group. Identify $V$ and $V^*$ via the scalar product. Then 
\begin{enumerate}
 \item $\aW$ is the semidirect product of $\sW$ by $\QQ(\RS^\vee)$, where $\QQ(\RS^\vee)$ is interpreted as the group of translations $t_v$ associated to the elements $v$ of the co-root lattice.
 \item The hyperplanes $H_{\alpha, k}=\{x\in V : \frac{(\alpha,\alpha)}{2}\lb x, \alpha^\vee\rb = k\}$, with $\alpha\in \RS$ and $k\in\Z$, are permuted by $\aW$.
 \item The special vertices of $\aW$ are precisely the elements of the co-weight lattice of $\RS$ which are by definition the weights of $\RS^\vee$.  Two special vertices $x,y\in P$ are of the same type if and only if there exists $v\in  \QQ(\RS^\vee)$ with $x=t_v(y)=y+v$. The co-root lattice $\QQ(\RS^\vee)$ of $\RS$ acts transitive on the set of special vertices with fixed type.
\end{enumerate}
Moreover 
$s_\beta(H_{\alpha, k}) = H_{s_{\beta}(\alpha), k}$
and 
$t_{\beta^\vee}(H_{\alpha, k})=H_{\alpha, k+\beta^\vee(\alpha)},$
where $t_{\beta^\vee}$ is the translation in $V$ by $\frac{2}{(\beta, \beta)}\beta$.
\end{prop}
\begin{proof} See \cite{Bourbaki4-6} VI, \S 2, Propositions 1 and 2.\end{proof}

\newpage
\section{Buildings: The simplicial approach}\label{Sec_classicalCase}

In this section we briefly recall basic definitions of the simplicial approach to buildings, collect necessary preliminaries for Section~\ref{Sec_KostantConvexity} and make some remarks on the metric structure of an affine building which is the starting point for the generalization. The following is by no means a complete, self-contained introduction to affine buildings. For this purpose we refer to \emph{Buildings} by Brown \cite{Brown}, the recent monograph \emph{Buildings: Theory and Applications} by Abramenko and Brown \cite{AB} or Ronan's \emph{Lectures on buildings} \cite{Ronan}.

\begin{definition}\label{Def_building}
\index{{building}!{simplicial}}
\index{system of apartments}
A \emph{building} is a simplicial complex $\Delta$ which is the union of subcomplexes called \emph{apartments}, satisfying the following axioms:
\begin{enumerate}
  \item\label{tec17} Each apartment is a Coxeter complex.
  \item\label{tec18} For any two simplices there exists an apartment containing both. 
  \item\label{tec19} If $A$ and $A'$ are two apartments both containing simplices $a$ and $b$ then there exists an isomorphism $A \rightarrow A'$ fixing $a$ and $b$ pointwise.
\end{enumerate}
Any collection of subcomplexes satisfying these axioms is called \emph{system of apartments}.
The \emph{chambers} of $\Delta$ are the maximal simplices and by a \emph{panel} we mean a co-dimension one face of a chamber.
If each panel is contained in at least three chambers $\Delta$ is \emph{thick}.
\end{definition}

A given wall $H$ in an apartment determines two half-spaces $H^+$ and $H^-$ called \emph{half-apartments}.
If $\Delta$ is thick, then for each apartment $A$ and each wall $H\in A$ there exists a half-apartment $B$ such that $B\cap A=H$ and such that $H^+\cup B$ as well as $H^-\cup B$ is an apartment of $\Delta$ different from $A$.

Note that we do not require that the building is equipped with a specific system of apartments. One can, however, prove that there exists a unique \emph{complete} system of apartments, which is the collection of all subcomplexes isomorphic to a fixed apartment or, equivalently, the union of all possible apartment systems.

\subsection{Spherical and affine buildings}

Let $(W,S)$ be the \emph{Coxeter system} associated to a Coxeter complex $\Sigma$. Then $S\subset W$ is a set of elements of order $2$ that generates $W$, which admits a presentation 
$$W=\langle S : (st)^{m(s,t)} = 1 \rangle$$ 
where $m(s,t)\in \Z$ is the order of the product $st$ of elements $s,t$ of $S$. There is one relation for each pair $s,t$ with $m(s,t)<\infty$. For details on Coxeter complexes and Coxeter systems see Chapter II.4 and III of \cite{Brown}.

\begin{definition}
A Coxeter group $W$ and the associated Coxeter complex $\Sigma$ are called \emph{spherical} if the group $W$ is finite.
We say that $W$ is a \emph{Euclidean reflection group} and the associated Complex $\Sigma$ is a \emph{Euclidean Coxeter complex} if $W$ is isomorphic to an affine Weyl group of a crystallographic root system $\RS$.
\end{definition}

Note that a spherical Coxeter group is the spherical Weyl group of some, not necessarily crystallographic, root system. The geometric realization of a spherical Coxeter complex is a topological sphere, whereas the geometric realization of a Euclidean Coxeter complex is isomorphic to the tiled vector space $V$ underlying a crystallographic root system $\RS$.
The maximal simplices of a Euclidean Coxeter complex $\Sigma$ are, under such an isomorphism, identified with alcoves in $V$.

\begin{definition}\label{Def_sphericalAffine}
Let $\Delta$ be a building with apartments isomorphic to a Coxeter complex $\Sigma$ and denote by $(W,S)$ be the associated Coxeter system. Then $\Delta$ is \emph{spherical} if $W$ is finite and \emph{affine} if $W$ is a Euclidean reflection group. The building $\Delta$ is \emph{irreducible} if the root system associated to $W$ is irreducible.
\end{definition}

In the following we will always denote spherical Weyl groups by $\sW$ and affine Weyl groups by $\aW$. 

Let $\Delta$ be a spherical building with apartments isomorphic to a Coxeter complex $\Sigma$. Denote by $\RS$ the root system associated to $\Sigma$. The half-apartments of $\Delta$ can canonically be be identified with elements of $\RS$. We therefore use the notation $\alpha$ for half-apartments in $\Delta$ and elements of $\RS$. The \emph{rank} of $\Delta$ is defined to be the number of elements in a basis of $\RS$. 

\begin{definition}\label{Def_Moufang}
\index{Moufang property}
\index{root group}
Let $\Delta$ be irreducible, spherical and of rank at least two. For each root $\alpha$ of $\Delta$ denote by $U_\alpha$ the stabilizer in $\Aut(\Delta)$ of all chambers sharing a panel with two chambers which are contained in $\alpha$. The group $U_\alpha$ acts trivially on $\alpha$ and is called the \emph{root group of $\alpha$}. The building $\Delta$ is \emph{Moufang} if it is thick and for each root $\alpha$ the root group acts transitively on the set of apartments containing $\alpha$. 
\end{definition}
Note that the group $U_\alpha$ acts in fact simply transitively on the set of apartments containing $\alpha$, see \cite[9.3 and 11.4]{Weiss}.

The Moufang property is an important tool in buildings theory. Tits showed that all thick irreducible spherical buildings of rank at least three are automatically Moufang. Being Moufang is furthermore a necessary assumption for the classification of spherical and affine buildings, which is illustrated in great detail in \cite{Weiss} and \cite{AffineW}.

\begin{remark}
Let $X=|\Delta|$ be the geometric realization of an affine building as described in \cite{AB}. Hence apartments $A$ of $|\Delta|$ are isomorphic to the tiled vector space $V$ underlying the associated root system $\RS$. 
We can therefore use the Euclidean distance of elements $x$ and $y$ in $V$ to define a distance function on $A$. Using the building axioms it is easy to see that this definition naturally extends to a distance function $d$ on $X$.
Taking the ``geometric'' point of view and thinking of buildings as their geometric realization equipped with some metric $d$, Charney and Lytchak \cite{LC} proved that spherical and affine buildings are characterized by metric properties.
\end{remark}

\begin{property}{\bf The building at infinity}
Associated to an affine building $\Delta$ with apartment system $\App$ there exists a simplicial spherical building called the \emph{building at infinity} $\binfinity \Delta$. It is canonically isomorphic to a certain subset of the Tits boundary of the metric space $(X,d)$, where $X= |\Delta |$ and $d$ is as defined above.
The \emph{chambers} of $\partial \Delta$ are parallel classes $\partial S$ of Weyl chambers $S$ contained in apartments of $\App$ which are \emph{adjacent} if and only if there exist representatives which are adjacent Weyl chambers in $\Delta$.
\end{property}

\subsection{Retractions}\label{Sec_retractions}

In the following let $X$ be an affine building. According to Proposition 2 on page 86 in \cite{Brown} we can define:
\begin{definition}\label{Def_simplChamberRetraction}
For any apartment $A$ and alcove $c$ in $A$ the \emph{canonical retraction} $r_{A,c}: X\rightarrow A$ \emph{onto $A$ centered at $c$} is the unique chamber map that fixes $c$ pointwise and preserves distance to $c$.
\end{definition}

Given an alcove $c$ and a Weyl chamber $S$ in $X$ there exists an apartment containing $c$ and a sub-Weyl chamber of $S$. Therefore $X$ is the union of the apartments which contain sub Weyl chamber of $S$ and we can define

\begin{definition}\label{Def_simplSectorRetraction}
\index{retraction}
Let $A$ be an apartment of $X$ and $\partial S$ a chamber in $\partial A$. The \emph{canonical retraction} $\rho_{A,S}: X\rightarrow A$ \emph{onto $A$ centered at $\partial S$} is the unique map whose restriction onto each apartment $B$ containing a sub-Weyl chamber of $S$ is the isomorphism onto $A$ fixing $A\cap B$ pointwise.
\end{definition}

In contrast to Definition~\ref{Def_simplSectorRetraction}, the definition of $r_{A,c}$ in \ref{Def_simplChamberRetraction} makes perfect sense for spherical buildings. We sometimes say that $\rho_{A,S}$ is a retraction \emph{centered at infinity} and call $r_{A,c}$ an \emph{alcove retraction}.

\begin{figure}[htbp]
\begin{center}
	\input{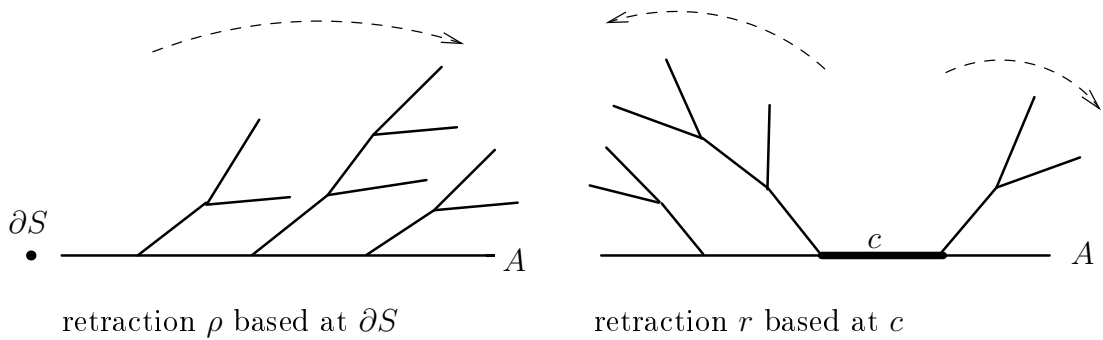}
	\caption[retractions]{On the left you can see the retraction onto $A$ centered at $\partial S$ as it is defined in \ref{Def_simplSectorRetraction} and on the right the alcove retraction $r_{A,c}$ as defined in \ref{Def_simplChamberRetraction}.}
	\label{Fig_retractions}
\end{center} 
\end{figure}

\begin{example}
The canonical retraction centered at infinity flattens out the entire building coming from a fixed chamber ``at infinity''. Under the alcove retraction a building behaves like being ironed onto the image apartment starting at the fixed alcove. We illustrated both retractions for a tree in Figure~\ref{Fig_retractions}.
\end{example}

\begin{thm*}{\cite[p.171]{Brown}}
Let $\rho_{A,S}$ be the sector retraction onto $A$ with respect to $\partial S$ in $\partial A$. Given an alcove $d\in X$ there exists an alcove $c\in A$ such that $$\rho_{A,S}(d)=r_{A,c}(d).$$
\end{thm*}

\begin{lemma}\label{Lem_compatibleRetractions}
Given a Weyl chamber $S$ in $X$ and apartments $A_i, i=1,\ldots, n$ containing a sub-Weyl chamber of $S$. Denote by $\rho_i$ the retraction $\rho_{A_i, S}$. Then
$$(\rho_1\circ\rho_2\circ\ldots\circ\rho_n) = \rho_1 .$$
\end{lemma}
\begin{proof}
Each $\rho_i$ maps apartments containing representatives of $\partial S$ isomorphically onto $A_i$. Every such apartment is isomorphically mapped onto $A_1$ by $\widetilde{\rho}\define \rho_1\circ\rho_2\circ\ldots\circ\rho_n$. As a set $X$ equals the union of all such apartments. The map $\widetilde{\rho}$ is distance non-increasing, since each $\rho_i$ is. Therefore $\widetilde{\rho}:X\mapsto A_1$ is the unique map defined by these two properties and is hence equal to $\rho_1$.
\end{proof}

\newpage
\section{Kostant convexity}\label{Sec_KostantConvexity}

Looking at the retractions defined in \ref{Def_simplSectorRetraction} and ~\ref{Def_simplChamberRetraction} it seems natural to ask how the images of a given alcove or vertex differ under these two retractions. The main Theorem of this section will give (at least a partial) answer. The original motivation to study this question was a different one: Does there exist a convexity theorem comparable to the one by Kostant \cite{Kostant} in the setting of affine buildings?
The answer to this question is ``yes''. After introducing the right notion of convexity, see Definition~\ref{Def_convex}, and collecting other necessary tools, we will prove Theorem \ref{Thm_convexity}.

\begin{notation}\label{Not_cthm}
Let $X$ be a thick affine building and $A$ an apartment of $X$. Fix an origin $0$ and identify the spherical Weyl group $\sW$ with the stabilizer of $0$ in $\aW$. 
Let $\Cf$ denote the fundamental Weyl chamber with respect to a fixed basis $B$ of the underlying root system $\RS$. The Weyl chamber opposite $\Cf$ in $A$ is denoted by $\Cfm$.
To simplify notation, write $r$ instead of $r_{A,\cf}$ and let $\rho$ stand for $\rho_{A,(\Cfm)}$.
\end{notation}

\begin{thm}\label{Thm_convexity}
With notation as in \ref{Not_cthm} let $x$ be a special vertex in $A$. 
Then
$$
\rho(r^{-1}(\sW.x)) =\dconv(\sW.x)\cap (x+\QQ) =: A^\QQ(x).
$$
where $\QQ=\QQ(\RS^\vee)$ is the co-weight lattice of $\RS$.
\end{thm}

\subsection{Dual convexity}

Let notation be as in \ref{Not_cthm}.

\begin{definition}
A \emph{dual hyperplane} in $A$ is a set $\Hdual_{\alpha,k}=\{x\in V : \lb x, \fcw[\alpha]) = k\}$, where $k\in \R$, $\alpha\in B$ and $\fcw[\alpha]$ a fundamental co-weight of $\RS$ as defined in Section~\ref{Sec_rootSystems}.
Dual hyperplanes determine \emph{dual half-apartments} ${\Hdual_{\alpha,k}}^\pm$. 
\end{definition}
As $H_{\alpha,k}$ is perpendicular to $\alpha$, so is $\Hdual_{\alpha,k}$ perpendicular to $\fcw[\alpha]$.
For any special vertex $x$ and root $\alpha\in\RS$ there exists a dual hyperplane containing $x$. Notice that the positive cone is the intersection of all positive half-apartments $\Hdual^+_{\alpha,0}$ with $\alpha\in B$.

\begin{definition}
\label{Def_convex}
\index{convex set}
A \emph{convex set} is an intersection of finitely many dual half-apartments in $A$, where the empty intersection is defined to be $A$. The \emph{convex hull} of a set $C$, denoted by $\dconv(C)$,  is the intersection of all dual half-apartments containing $C$.
\end{definition}

Let $C$ be the orbit $\sW.x$ of a special vertex $x$ in $A$. One can prove that $\dconv(C)$ equals the metric convex hull of $\sW.x$ in $A$.

The following Lemma, which is a direct analog of \cite[Lemma 3.3(2)]{Kostant}, plays a crucial role in the proof of our main result. As in Theorem~\ref{Thm_convexity} let $A^\QQ(x)$ denote the set of all vertices in $\dconv(\sW.x)$ having the same type as $x$. Given a basis $B$ of $\RS$ the \emph{positive cone} $\Cp$ (with respect to $B$) is the set of all positive linear combinations of positive roots in $\RS$.

\begin{lemma}\label{Lem_conv^*}
Let $X$, $A$, $0$ and $\Cf$ be as in \ref{Not_cthm}. Given a special vertex $x$ in $A$ let $x^+$ denote the unique element of $\sW.x$ contained in $\Cf$. Then
$$A^\QQ(x)=\{y\in V : (x^+-y^+)\in (\Cp\cap\QQ) \}=\bigcap_{w\in \sW } w(x^+ - \QQ^+) $$ 
where $\QQ=\QQ(\RS^\vee)$ is the co-root lattice of $\RS$.
\end{lemma}
\begin{proof}
Assume without loss of generality that $x\in \Cf$. Since $A^\QQ(x)$ is $\sW$-invariant it suffices to prove 
$$A^\QQ(x)\cap \Cf = \{y\in \Cf : x-y \in (\Cp\cap\QQ) \}.$$
For all $\alpha\in B$ define $k_\alpha$ implicitly by $x \in \Hdual_{\alpha,k_\alpha}$.
By definition 
\begin{equation}\label{Equ_no1}
A^\QQ(x) \cap\Cf = \bigcap_{\alpha\in B} (\Hdual_{\alpha,k_\alpha})^- \cap \Cf.
\end{equation}

Assume first that $y\in A^\QQ(x)\cap\Cf$.  Using (\ref{Equ_no1}) and the fact that $\Cf \subset \Cp$ conclude 
$$
0\leq ( \fcw , y )\leq k_\alpha \text{ for all } \alpha\in B.
$$
Therefore $( \fcw,x-y ) \frac{1}{(\alpha,\alpha)} \geq 0$ and $x-y\in \QQ$ since $A^\QQ(x)\subset (x+\QQ)$. 
The faces of $\Cp$ are contained in dual hyperplanes of the form $\Hdual_{\alpha,0}$ with ${\alpha\in B}$. Therefore
$$
A^\QQ(x)\cap\Cf \subset \{y\in V : x-y\in(\Cp\cap\QQ)\}\cap\Cf = (x-(\QQ\cap\Cp))\cap\Cf.
$$
Conversely assume that $y\in (x- (\QQ\cap\Cp))\cap\Cf$. Then
$0\leq (\fcw, (x-y))\frac{1}{(\alpha,\alpha)}$ and hence  
$( \fcw, x ) \frac{1}{(\alpha,\alpha)} \geq  ( \fcw, y ) \frac{1}{(\alpha,\alpha)}$. Therefore the assertion holds. 
\end{proof}

A common way to define convexity in affine buildings is the following: A subset $C$ of an apartment $A$ is \emph{$\sW$-convex} (or just \emph{convex}) if it is the intersection of finitely many half-apartments. The \emph{$\sW$-convex hull} of a set $C$ is the intersection of all half-apartments containing $C$. The following example illustrates why in our situation this is not the right notion of convexity.

\begin{figure}[htbp]
\begin{center}
	\resizebox{!}{0.26\textheight}{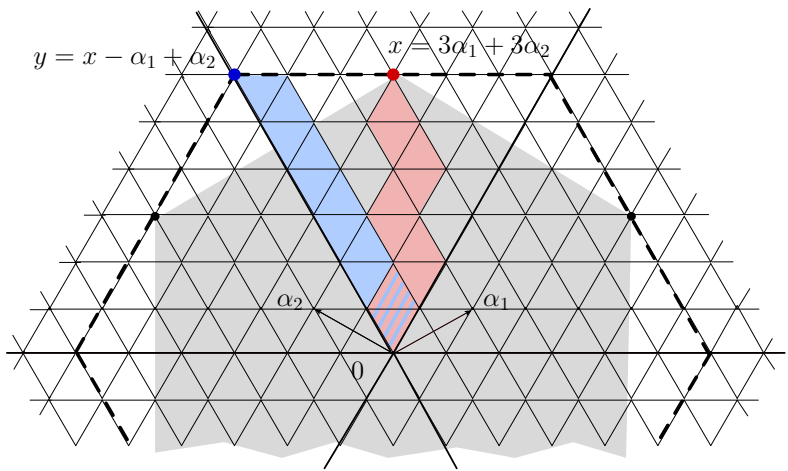}
	\caption[Wconvex]{The notion of  $\sW$-convexity can not be used for Theorem~\ref{Thm_convexity}, since the $\sW$-convex hull of $\sW.x$ contains points, as $y$, which are further away from $0$ than $x$.}
	\label{Fig_Wconvex}
\end{center} 
\end{figure}

\begin{example}
Let $A$ be a Coxeter complex of type $\widetilde{A}_2$. Let $\alpha_1$ and $\alpha_2$ be a basis of the underlying root system $\RS$ as illustrated in Figure~\ref{Fig_Wconvex}. Let $x$ equal $3\alpha_1 + 3\alpha_2$. Observe that the $\sW$-convex hull of the Weyl group orbit $\sW.x$ contains $4\alpha_1+2\alpha_2$ which will be denoted by $y$. For special vertices $a,b$ let $\delta(a,b)$ be the length of the shortest gallery $\gamma$ such that $a$ is contained in the first and $b$ in the last chamber of $\gamma$. Then $\delta(0,x)=10$ and $\delta(0,y)=11$. 

For all elements $z$ in $\rho(r^{-1}(\sW.x))$ the distance $\delta(0,z)$ is smaller or equal to $10$ since $r$ preserves distance to $x$ and $\rho$ diminishes distance. Hence $y$ can not be contained in $\rho(r^{-1}(\sW.x))$ and the theorem can not be true using $\sW$-convexity.
\end{example}

\subsection{Tools and preliminary results}
\subsubsection*{Positively folded galleries}

Let the notation be as in \ref{Not_cthm}. We will extend the retractions $r$ and $\rho$ to galleries and use them to describe how the building is folded onto the fixed apartment. 

\begin{definition}
\label{Def_gallery}
\index{combinatorial gallery}
\index{panel}
A \emph{combinatorial gallery} $\gamma$ of \emph{length} $n$ in $X$ is a sequence 
$$\gamma=(u, c_0, d_1, c_1, \ldots, d_n, c_n, v) $$
of panels $d_i$  and alcoves $c_i$ such that 
\begin{enumerate}
\item \emph{source} $u$ and \emph{target} $v$ are special vertices contained in $c_0$ and $c_n$, respectively, and $u-v\in \QQ(\RS^\vee)$, i.e. $u$ and $v$ have the same type, and 
\item for all $i=1\ldots n-1$ the panel $d_i$ is contained in $c_{i-1}$ and $c_i$.
\end{enumerate}
The \emph{type} of a combinatorial gallery $\gamma$ is the list $t=(t_0,t_1,\ldots, t_n)$ of types $t_i$ of the faces $d_i$.
We say $\gamma$ is \emph{minimal} if it is a gallery of minimal length with source $u$ and target $v$.
\end{definition}

We abbreviate ``$\gamma$ is a combinatorial gallery with source $u$ and target $v$'' by $\gamma:u\rightsquigarrow v$. In the following a \emph{gallery} is a combinatorial gallery in the sense of Definition~\ref{Def_gallery}.

\begin{definition}\label{Def_Ghat}
Denote by $\Ghat_t$ \index{$\Ghat$} the set of all minimal galleries in $X$ of fixed type $t$ with source $0$. Denote by $\G_t$ \index{$\G_t$} the set of targets of galleries $\gamma$ in $\Ghat_t$. Notice that the elements of $\G_t$ are all vertices of the same type.
\end{definition}

\begin{lemma}\label{Lem_preimage}
With notation as in \ref{Not_cthm} let $x$ be a special vertex in $A$ and let $\gamma:0\rightsquigarrow x$ be a minimal gallery of fixed type $t$. Then
$$
r^{-1}(\sW.x) = \G_t.
$$
If $K\define\mathrm{Stab_{\mathrm{Aut(X)}}(0)}$ acts transitively on the set of all apartments containing $0$ 
then 
$$ r^{-1}(\sW.x) = K.x .$$
\end{lemma}
\begin{proof}
The retraction $r$ preserves adjacency and distance to $0$ in $X$. Part one is an easy consequence of these properties. Secondly the transitivity of the $K$-action directly implies $\G_t=K.x$. 
\end{proof}

Note, for example, that $\mathrm{Stab_{\mathrm{Aut(X)}}(0)}$ is transitive on the set of all apartments containing $0$ if $G$ has an affine $BN$-pair.

\begin{definition}
\index{positively folded}
Fix an apartment $A$ in $X$ and let $H$ be a hyperplane, $d$ an alcove and $S$ a Weyl chamber in $A$. We say \emph{$H$ separates $d$ and $S$} if there exists a representative $S'$ of $\partial S$ in $A$ such that $S'$ and $d$ are contained in different half-apartments determined by $H$.
\end{definition}

\begin{definition}
A gallery $\gamma= (u, c_0, d_1, c_1, \ldots , d_n, c_n, v)$  is \emph{positively folded at $i$} if 
$c_i=c_{i-1}$ and 
the hyperplane $H= \Span({d_i})$ separates $c_i$ and $\Cfm$.
A gallery $\gamma$ is \emph{positively folded} if it is positively folded at $i$ whenever $c_{i-1}=c_i$.
\end{definition}

\begin{notation}\label{Not_extendedRetr}
Let $\gamma=(0, c_0, d_1, c_1, \ldots , d_n, c_n, v)$ be a minimal gallery in $X$. The retractions $r$ and $\rho$, see \ref{Not_cthm}, are defined on alcoves and faces of alcoves. Denote by $\hat{r}$ and $\hat{\rho}$ their extensions to galleries, defined as follows. The image $r(\gamma)$ is defined to be the sequence $\hat{r}=(r(u), r(c_0), r(d_1), r(c_1), \ldots , r(d_n), r(c_n), r(v))$, and $\hat\rho(\gamma)$ is defined analogously.
\end{notation}

\begin{figure}[htbp]
\begin{center}
	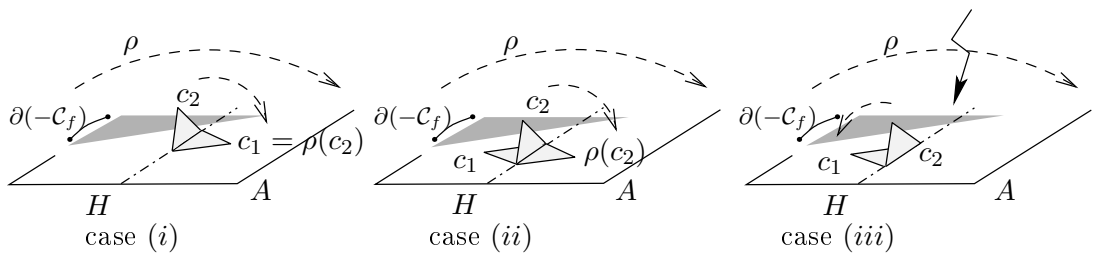
	\caption[positivefold]{Retraction $\rho$ produces positive foldings.}
	\label{Fig_positiveFold}
\end{center} 
\end{figure}

\begin{lemma}\label{Lem_folding}
Let $X$, $A$ and $\rho:X\mapsto A$ be as in \ref{Not_cthm} and let $\hat{\rho}$ be the extension of $\rho$ defined in \ref{Not_extendedRetr}. The image of a gallery $\gamma \in\Ghat_t$ under $\hat{\rho}$ is a positively folded gallery in $A$ of the same type.
\end{lemma}
\begin{proof}
By Lemma~\ref{Lem_compatibleRetractions} and induction it is enough to consider the case where $\gamma$ is of length 2 and the first chamber is contained in $A$. Assume that $\gamma=(0,c_1,p,c_2,v)$ is minimal with $c_1$ an alcove in $A$ and $c_2$ an adjacent alcove not contained in $A$ sharing the $i$-panel $p$. The alcove $\rho(c_2)$ is either the unique alcove in $A$ which is $i$-adjacent to $c_1$ but different from $c_1$ or $\rho(c_2)=c_1$, compare cases $(i)$ and $(ii)$ of Figure~\ref{Fig_positiveFold}. 

If $\rho(c_2)\neq c_1$, then  $\hat{\rho}(\gamma)$ is minimal. Assume $\rho(c_2)= c_1$ and that $\hat{\rho}(\gamma)$ has a negative fold at $p$ (defined analogously to positive fold). The hyperplane $H=\Span_\R(p)$ in $A$ does not separate $c_1$ and $\Cfm$.
Since $X$ is a thick building there exists an apartment $A'$ containing $c_2$ and the half-apartment $H^-$.

By definition of $A'$ the retraction $\rho$ maps $A'$ isometrically onto $A$. Hence $c_2$ is mapped onto a chamber $c$ in $A$ which is $i$-adjacent to $c_2$. Furthermore let $u, u'$ be chosen such that the gallery $(u,c_2,p,c,u')$ is minimal and positively folded. The two choices for $c$ are $c_1$ and the unique $i$-adjacent chamber of $c_1$ in $A$, which we excluded in our assumption. Hence $c=c_1$. But this is a contradiction to the assumption that $H$ does not separate $\rho(c_2)=c_1$ and $\Cfm$. Hence case $(iii)$ of Figure~\ref{Fig_positiveFold} does not occur.
\end{proof}

\subsubsection*{A character formula}

The character formula stated below is used in the proof of Theorem~\ref{Thm_convexity}. For further references on representation theory see for example \cite{Bourbaki7-9} or \cite{Humphreys}.

Let $G$ be a connected semisimple complex algebraic group. Let $T$ be a maximal torus in $G$. Denote by $\mathfrak{X}$ the character group of $T$ and the subset of dominant co-characters in $\mathfrak{X}$ by $\mathfrak{X}_+$. Let $X$ be the affine Bruhat-Tits building associated to $G(\C[\![t]\!])$ 
as defined in \cite{BruhatTits}. Apartments in $X$ can be identified with translates $gA$ of $A\define \mathfrak{X}\otimes_\Z\R$. Dominant co-weights $\lambda\in \mathfrak{X}_+$ correspond to special vertices in $\Cf$.
Denote by $V_\lambda$ the irreducible complex representation of highest weight $\lambda$ for $G$. Let $\chi_\lambda$ be the character of $V_\lambda$. 

\begin{thm}{\cite{GaussentLittelmann}}
\label{Thm_characterFormula}
Fix $\lambda\in \mathfrak{X}_+$. 
Let $p_{\nu}$ denote the dimension of the $\nu$-weightspace of $V_\lambda$ and $\wgt(\gamma)$ the vertex of the final chamber of $\gamma$ having the same type as $\lambda$. Then
\begin{equation}
 \sum_{\gamma\in LS(t)} k^{\wgt(\gamma)} = \chi_\lambda = \sum_{\nu\in X} p_{\nu}k^\nu.
\end{equation}
\end{thm}

The set $LS(t)$ is a certain subset of the set of positively folded galleries of fixed type $t$ in $A$ with source $0$. The term $k^{wgt(\gamma)}$ is a group-like element of the group algebra associated to the targets $\wgt(\gamma)$ of $LS$-galleries. For all $\nu\in\Cf$ the coefficients $p_{\lambda,\nu}$ are strictly positive if and only if $\nu\leq \lambda$, which is equivalent to the fact that $\nu$ is contained in $A^\QQ(x)\cap\Cf$.

\begin{remark}
Actually Gaussent and Littelmann use a slightly more general notion of gallery as the one defined in \ref{Def_gallery}. Nevertheless by observations made in \cite{Schwer} their result is true in our setting as well.
\end{remark}

\subsection{Proof of \texorpdfstring{Theorem~\ref{Thm_convexity}}{theorem}}

Let notation be as in ~\ref{Not_cthm} and let $\hat{r}$ and $\hat{\rho}$ denote the extended retractions according to \ref{Not_extendedRetr}. The proof of Theorem~\ref{Thm_convexity} can be reduced to the following two propositions.

\begin{prop}\label{Prop_existenceLSgalleries}
Let $A$ be a Euclidean Coxeter complex with origin $0$ and fundamental Weyl chamber $\Cf$. Let $x$ be a special vertex in $A$ and let $x^+$ denote the unique element of $\sW.x$ contained in $\Cf$. Let $t$ be the type of a fixed minimal gallery $\gamma: 0 \rightsquigarrow x^+$.
All vertices in $A^\QQ(x)=\{y\in \dconv(\sW.x)\cap (x+\QQ(\RS^\vee))\}$ are targets of positively folded galleries having type $t$. Conversely the target of any positively folded gallery of type $t$ with source $0$ is contained in $A^\QQ(x)$.
\end{prop}
\begin{proof}
Recall that $\Gamma_t^+$ denotes the set of all positively folded galleries of type $t$ in $A$ with source $0$. The so called $LS$-Galleries $LS(t)$ are contained in $\Gamma_t^+$. Let $G$ be a connected semisimple complex algebraic group and denote by $\mathfrak{X}$ the character group of a chosen maximal torus $T\subset G$. Assume that $G$ is such that $A$ can be identified with $\mathfrak{X}\otimes_\Z\R$.
By Theorem~\ref{Thm_characterFormula} the character of the (irreducible) highest weight representation $V_{x^+}$, with highest weight $x^+$, equals
\begin{equation*}
\sum_{\sigma\in LS(t)} k^{\wgt(\sigma)} = \chi_{x^+} = \sum_{\nu\in X} p_{\nu}k^\nu
\end{equation*}
where $\wgt(\sigma)$ is the target of $\sigma$. The coefficient $p_{\nu}$ equals the dimension of the $\nu$-weight space of $V_{x^+}$. Hence $p_\nu$ is positive if and only if $\nu\leq x^+$. By comparison of coefficients $k^\nu\neq 0$ if and only if there exist an $LS$-gallery with target $\nu$. This holds for all $\nu\leq x^+$. By Lemma~\ref{Lem_conv^*} the proposition follows.
\end{proof}

\begin{prop}\label{Prop_existencePreimage}
Let $A$ be a fixed apartment of an affine building $X$. Fix an origin $0$ and fundamental Weyl chamber $\Cf$ in $A$. If $\gamma\subset A$ is a positively folded gallery with source $0$ of type $t$ then there exists a minimal gallery $\widetilde{\gamma}\subset X$ with source $0$ such that $\hat{\rho}(\widetilde{\gamma})=\gamma$.
\end{prop}
\begin{proof}
Let $\gamma=(0, c_0, d_1, c_1, \ldots , d_n, c_n, v)$. We will inductively construct a preimage $\widetilde{\gamma}$ of $\gamma$ under $\rho$ which is minimal in $X$.
Denote by $J=\{i_1, \ldots, i_k\}\subset \{1,2,\ldots,n\}$ the set of folding indices. Hence for each $i_j\in J$ we have  $c_{i_j}=c_{i_j -1}$. Assume $i_1 <i_2 <\ldots < i_k$. Let $H_{i_1}$ be the hyperplane separating $c_{i_1}$ from $\Cfm$. Write $H_{i_1}^-$ (respectively $H_{i_1}^+$) for the half-apartment of $A$ determined by $H_{i_1}$ which contains a subsector of $\Cfm$ (respectively a subsector of $\Cf$). Since $X$ is thick there exists an apartment $A_1$ such that $A_1\cap A = H_{i_1}^-$.

Let $A'= H_{i_1}^+\cup(A_1\setminus A)$. 
By construction $\gamma$ is contained in $A'$. Denote by $r_{i_1}$ the reflection at $H_{i_1}$ in $A'$ and define $\gamma^1$ as follows:
$$\gamma^1=(0, c_0, d_1, \ldots  , d_{i_1}, c_{i_1}^1, d_{i_1+1}^1, \ldots , d_n^1, c_n^1, y^1)$$
where $c_j^1 = r_{i_1}(c_j)$ and $d_j^1=r_{i_1}(d_j)$ for all $j\geq i_1$. Note that $d_{i_1}=d_{i_1}^1$ since $d_{i_1}$ is contained in $H_{i_1}$. Then $\gamma^1$ is a gallery of type $t =\type( \gamma)$. But, in contrast to $\gamma$ it is not folded at $i_1$. Furthermore $\gamma^1$ is minimal up to the next folding index, meaning the shortened gallery $(0, c_0, d_1, \ldots , d_{i_1}, c_{i_1}^1, d_{i_1+1}^1, \ldots , d_{i_2 -1}^1, c_{i_2 -1}^1)$ is minimal. The previous step reduced the number of folding indices by one and is done such that $\hat{\rho}(\gamma^1)=\gamma$.

Shorten the gallery $\gamma^1$ such that it is contained in $A_1$: 
$$\gamma_1 = (x^1, c_{i_1}^1, d_{i_1+1}^1, \ldots , d_n^1, c_n^1, y^1)$$
where $x^1$ is the vertex of $c_{i_1}$ such that $(x^1, c_{i_1}^1, d_{i_1+1}^1, \ldots , d_{i_2}^1, c_{i_2-1}^1)$ is a minimal gallery from $x^1$ to $c_{i_2-1}^1$. By definition $\gamma_1$ is positively folded of a suitably shortened type and contained in $A_1$.
As in the first step, define a gallery 
$$\widetilde{\gamma}^2=(x_1, c_{i_1}^1, d_{i_1+1}^1, \ldots  , d_{i_2}^1, c_{i_2}^2, d_{i_2+1}^2, \ldots , d_n^2, c_n^2, y^2)$$ having the analogous properties. Denote by $\hat{\rho}_1$ the retraction $\hat{\rho}_{A_1, \Cfm}: X\mapsto A_1$. By definition of $\widetilde{\gamma}^2$ one has $\hat{\rho}_1(\widetilde{\gamma}^2)=\gamma_1$. 
Let 
$$
\gamma^2 = (0, c_0, d_1, \ldots  , d_{i_1}, c_{i_1}^1, d_{i_1+1}^1, \ldots  , d_{i_2}^1, c_{i_2}^2, d_{i_2+1}^2, \ldots , d_n^2, c_n^2, y^2).
$$
(We add the part which was removed when $\gamma^1$ was shortened.) Now $\hat{\rho}(\gamma^2)=\gamma^1$.

Iterating the same procedure $k=\sharp J$ times, we get 
\begin{align*}
\gamma^k = 	&(0 , c_0, d_1,  \ldots , d_{i_1}, c_{i_1}^1, d_{i_1+1}^1, \ldots 
		 \ldots, d_{i_k}^{k-1}, c_{i_k}^k, d_{i_k+1}^k, 
		\ldots, d_n^k, c_n^k, y^k)
\end{align*}
and a sequence of apartments $A_1, \ldots A_k$ such that $c_*^j, d_*^j\subset A_j$ for all $j=1,\ldots, k$. By construction $\gamma^k: 0\rightsquigarrow y^k$ is minimal in $X$ and of type $t$. Again denote the retraction $\hat{\rho}_{A_{j}, \Cfm}$ onto $A_{j}$ by $\hat{\rho}_j$. Lemma~\ref{Lem_compatibleRetractions} implies $\hat{\rho}(\gamma^k)=(\hat{\rho}_k\circ\hat{\rho}_{k-1}\circ\ldots\circ\hat{\rho}_1)(\gamma)$. Hence $\gamma^k$ is the desired preimage of $\gamma$.
\end{proof}

The proof of Theorem~\ref{Thm_convexity} reads as follows:

\begin{proof}[{\textit{Proof of Theorem \ref{Thm_convexity}.}}]
By Lemma~\ref{Lem_preimage}, the set $r^{-1}(\sW.x)$ equals $\G_t$. Therefore $\rho(r^{-1}(\sW.x))$ equals $\rho(\G_t)$. Let $\hat{r}$ and $\hat{\rho}$ be as in \ref{Not_extendedRetr}. Lemma~\ref{Lem_folding} implies that $\hat{\rho}(\Ghat_t)$ is a set of positively folded galleries in $A$ and thus $\rho(\G_t)$ is the set of targets of the galleries contained in $\hat{\rho}(\Ghat_t)$. By Propositions~\ref{Prop_existenceLSgalleries} and~\ref{Prop_existencePreimage} the assertion follows.
\end{proof}

\begin{remark}
Proposition~\ref{Prop_existenceLSgalleries} is a purely combinatorial property of $LS$-galleries. It was our aim to find a proof that avoids representation theoretic methods. During the preparation of this thesis Parkinson and Ram published the preprint~\cite{ParkinsonRam}, where they give the desired geometrical proof of Proposition~\ref{Prop_existenceLSgalleries}. The main idea of their proof carries over to thick generalized affine buildings. Details can be found in Section~\ref{Sec_convexityRevisited}.
\end{remark}

\subsection{Application and open problem}\label{Sec_application} \label{Sec_openProblem}

We give an application of Theorem \ref{Thm_convexity}.

Let $G$ be a group with affine $BN$-pair. There exists then an affine building $(X,\App)$ associated to $G$ on which the group acts. Let $B$ and $N$ denote the groups of the sphercial $BN$-pair of $G$ which we assume to be \emph{split}, i.e. $B=UT$. Denote by $\Delta$ the spherical building associated to $G$ by means of the spherical $BN$-pair. Notice that $\Delta=\binfinity X$. The group $N$ stabilizes an apartment $a$ of $\Delta$ and $B$ stabilizes a chamber $c$ of $a$.  The groups $U$ and $T$ can be interpreted as follows: $T$ is a group of translations in the affine apartment $A$ with boundary $\partial A = a$ and $U$ acts simply transitive on the set of all affine apartments containing $c$ at infinity. Fix a chart of $A$ such that $c=\partial(\Cfm)$.

Let $K$ be the stabilizer of $0\in A$ in $G$. Then $G$ has an \emph{Iwasawa decomposition}\index{Iwasawa decomposition}
$$ G=BK=UTK .$$ 
Notice that special vertices of the same type as $0$ in $X$ are in one-to-one correspondence with cosets of $K$ in $G$. Let $x$ be a special vertex in $X$ having the same type as $0$.
Then there exists an element $u\in U$ such that $(u^{-1}).x\in A$. Let $t\in T$ be the translation mapping $0$ to $(u^{-1}).x$. If the origin $0$ is identified with $K$ the vertex $x$ corresponds to the coset $utK$. Denote by $t_0$ the type of the origin. One has :
\begin{align*}
\{ \text{special vertices of type } t_0 \text{ in } X\} &\stackrel{1:1}{\longleftrightarrow} \{ \text{ cosets of } K \text{ in } G \}\\
 0 	& \longmapsto  K \\
X\ni x 	& \longmapsto  utK  \text{ with } u,t \text{ chosen as above.}
\end{align*}
Any special vertex of type $t_0$ of $X$ can hence be identified with a coset $utK$ with suitably chosen $u\in U$ and $t\in T$. Note that vertices of type $t_0$ contained in $A$ correspond precisely to cosets of the form $tK$ with $t\in T$. 

Furthermore it is easy to see that $\rho$ is exactly the projection that maps $utK$ to $tK$, and hence to $A$. 

For all special vertices $tK$ contained in $A$ the set $r^{-1}(\sW.tK)$ is the same as the $K$-orbit of $tK$ which is 
$$r^{-1}(\sW.tK)= KtK.$$ 

In the situation as described above the following theorem is a direct reformulation of Theorem~\ref{Thm_convexity}.

\begin{thm}\label{Thm_BNpair}
For all $tK\in A$ we have 
$$\rho(KtK)=\dconv(tK)$$
or, since $\rho^{-1}(t'K)=Ut'K$, equivalently
$$
\emptyset \neq Ut'K \cap KtK \:\Longleftrightarrow\: t'K \in \dconv(\sW.tK).
$$
\end{thm}

The proof of ``$\Rightarrow$'' in the second statement is well known and can be found in \cite{BruhatTits}. As already mentioned in the introduction partial results on ``$\Leftarrow$'' and related questions can, for example, be found in \cite{GaussentLittelmann, MirkovicVilonen, Rapoport, Schwer} or \cite{Silberger}.

\subsubsection*{An open problem}

It is natural to ask whether or not a result in the spirit of Theorem \ref{Thm_convexity} holds for faces of alcoves having smaller co-dimension or even for alcoves themselves. Assume that we are in the situation as described in \ref{Thm_convexity}. Hence $X$ is a thick simplicial affine building, an apartment $A$ is fixed together with an origin and a fundamental Weyl chamber $\Cf$. The retractions $r$ and $\rho$ are also as defined in \ref{Not_cthm}. 

\begin{figure}[htbp]
\begin{center}
 \begin{minipage}[b]{5 cm}
    \resizebox{!}{0.18\textheight}{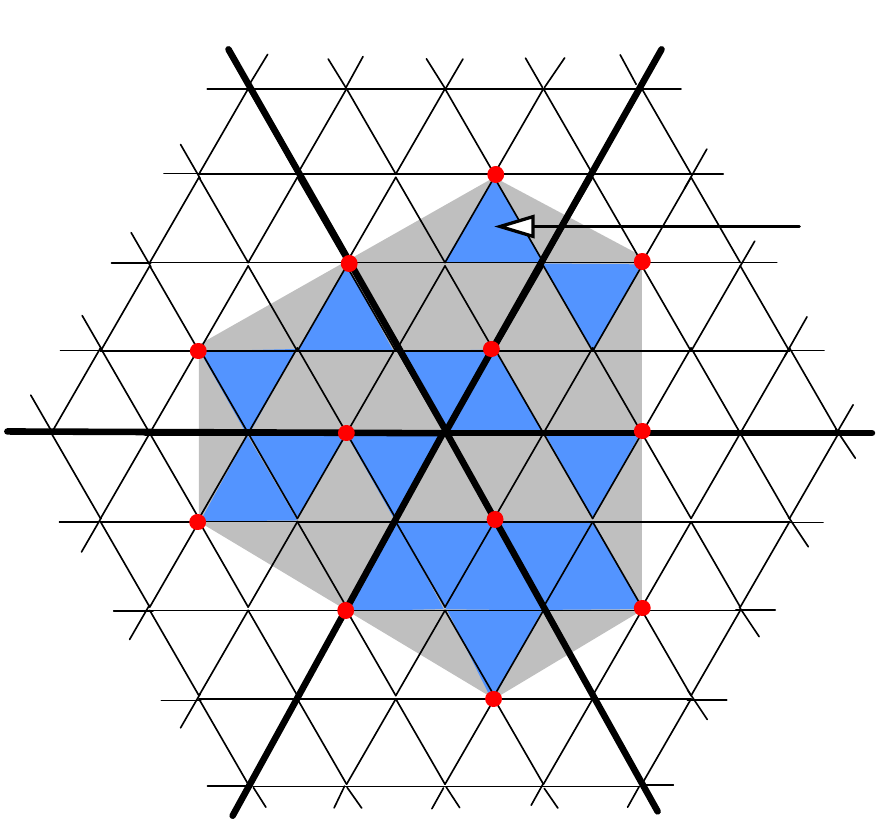}
  \end{minipage}
  \begin{minipage}[b]{5 cm}
    	\resizebox{!}{0.18\textheight}{\input{graphics/hit-length5b.pdftex_t}}
  \end{minipage}
  \begin{minipage}[b]{5 cm}
    	\resizebox{!}{0.18\textheight}{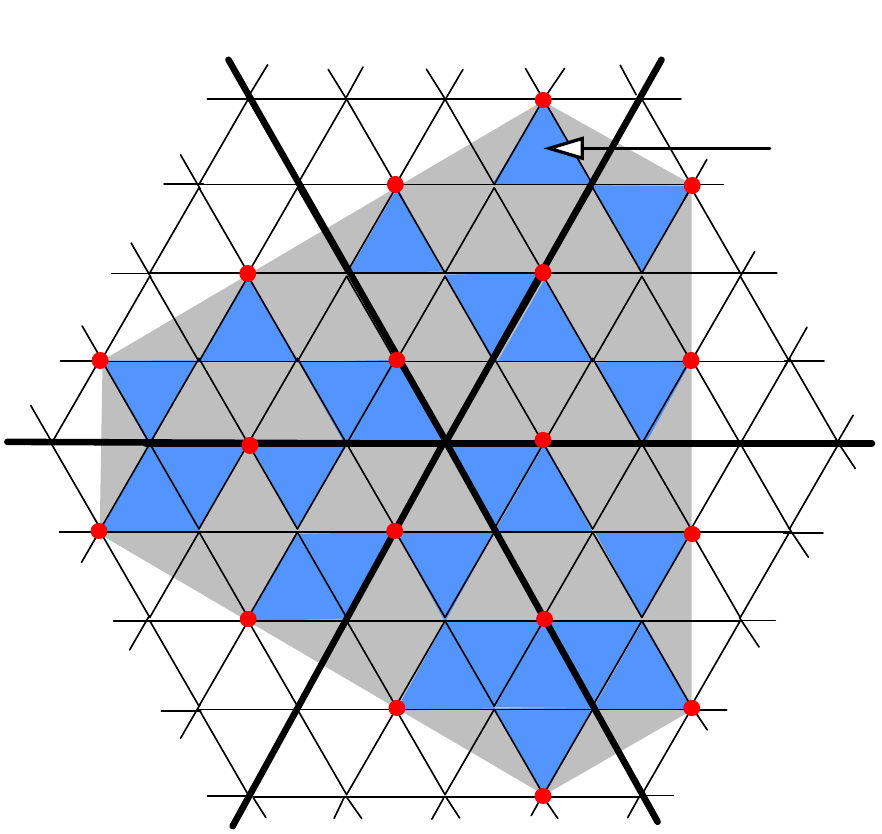}
  \end{minipage}
  \caption[retractions]{Given an alcove $c$ contained in $\Cf$ the dark shaded alcoves are the ones contained in $\rho(r^{-1}(\sW.c))$.}
  \label{Fig_chamber}
\end{center} 
\end{figure}

Easy examples show that if we replace in the statement of Theorem~\ref{Thm_convexity} the vertex $x$ by an alcove $c$ contained in $A$ makes the question seems to be much harder, but solving the question for alcoves might have representation theoretic consequences. 

Notice that the set $\rho(r^{-1}(\sW.c))$ does not cover the convex hull of the orbit $\sW.c$ and has a structure that is hard to describe. Let $v$ be the vertex of $c$ furthest away from $0$. The light shaded region in the pictures of Figure \ref{Fig_chamber} is the convex hull of $\sW.v$. Notice that not even this set is covered by $\rho(r^{-1}(\sW.c))$, quite the contrary: the holes seem to ``grow''.

Three examples showing $\rho(r^{-1}(\sW.c))$ for different choices of the alcove $c$ are given in Figure~\ref{Fig_chamber}. The picture illustrate that the pattern ``evolves'' in a certain regular way when gallery is elongated in each step by one alcove, as done from left to right in Figure \ref{Fig_chamber}.

As the application given above shows the question for vertices is related to the Iwasawa decomposition of groups with affine $BN$-pair. In similar ways the question for alcoves is related to the Iwahori decomposition. We will not make this connection precise here.

 \newpage
\section{The model space}
\label{Sec_modelSpace}

Geometric realizations of simplicial affine buildings are metric spaces ``covered by'' Coxeter complexes which are isomorphic to a tiled $\R^n$. The basic idea of the generalization is to substitute the real numbers by a totally ordered abelian group $\Lambda$.

\subsection{Definitions and basic properties}
\begin{definition}\label{Def_modelSpace}
\index{model space}
Let $\RS$ be a (not necessarily crystallographic) spherical root system and $F$ a subfield of $\R$ containing the set $\lb \beta, \alpha^\vee\rb : \alpha, \beta \in \RS\}$. Assume that $\Lambda$ is a totally ordered abelian group admitting an $F$-module structure. The space 
$$
\MS(\RS,\Lambda) = \mathrm{span}_F(\RS)\otimes_F \Lambda
$$ 
is the \emph{model space} of a generalized affine building of type $\RS$. 
\end{definition}

We omit $F$ in the notation, since $F$ can always be chosen to be the quotient field of $\Q[\lb \beta, \alpha^\vee\rb : \alpha, \beta \in \RS\}]$. If $\RS$ is crystallographic then $F=\Q$ is a valid choice. The space $\MS(\RS, \Lambda)$ is a $\rk(\RS)$-dimensional vector space over $F\otimes_F\Lambda$. If there is no doubt which root system $\RS$ and which $\Lambda$ we are referring to, we will abbreviate $\MS(\RS,\Lambda)$ by $\MS$.

\begin{remark}\label{Rem_coordinates}
\index{coordinates}
A fixed basis $B$ of the root system $\RS$ provides natural coordinates for the model space $\MS$. The vector space of formal sums 
$$
\left\{\sum_{\alpha\in B} \lambda_\alpha\alpha : \lambda_\alpha\in\Lambda\right\}
$$ 
is canonically isomorphic to $\MS$.
The evaluation of co-roots on roots $\lb\cdot, \cdot\rb$ used in section \ref{Sec_rootSystems} are linearly extended to elements of $\MS$.
\end{remark}

\begin{definition}\label{Def_reflections}
\index{hyperplane}
An action of the spherical Weyl group $\sW$ on $\MS$ is defined as follows.
Let $\alpha, \beta \in\RS$, $c\in F$ and $\lambda\in\Lambda'$, let $r_\alpha:\MS\mapsto\MS$ be the linear extension of 
$$r_\alpha(c\beta\otimes \lambda ):= c\, s_\alpha(\beta) \otimes \lambda$$
to $\MS$, where the reflection $s_\alpha$ is defined as in Section~\ref{Sec_rootSystems}. The fixed point set of $r_\alpha$ is called \emph{hyperplane} or \emph{wall} and is denoted by $H_\alpha$.
\end{definition}

As in the classical case $H_\alpha=\{x \in \MS: r_\alpha(x)=x\}=\{x \in\MS: \lb x, \alpha^\vee\rb =0\}$.
 
\begin{definition}\label{Def_WeylChamber1}
\index{Weyl chamber}
\index{panels}
A basis $B$ of $\RS$ determines a set of positive roots $\RS^+\subset \RS$. The set 
$$
\{x\in\MS : \lb x,\alpha^\vee\rb \geq 0 \text{ for all } \alpha\in\RS^+\}
$$
is the \emph{fundamental Weyl chamber} with respect to $B$, denoted by $\Cf$. 
\end{definition}

\begin{definition}\label{Def_affineWG}
\index{{Weyl group}!{affine}}
\index{{Weyl group}!{$\WT$, $\aW$}}
For a non-trivial group of translations $T$ of $\MS$ normalized by $\sW$ the \emph{affine Weyl group} with respect to $T$ is the semi direct product $\WT \define \sW\rtimes T$. In case $T=\MS$ we will call it the \emph{full affine Weyl group} and write $\aW$.
Elements of $T$ can be identified with points in $\MS$ by assigning to $t\in T$ the imgage $t(0)$ of the origin $0$ undert $t$. Given $k\in \MS$ we write $t_k$ for the translation defined by $0\mapsto k$.
\end{definition}

The actions of $\sW$ and $T$ on $\MS$ induce an action of $\WT$, respectively $\aW$, on $\MS$. 

\begin{notation}\label{Not_modelSpace}
In order to emphasize the freedom of choice for the translation part of the affine Weyl group, we will denote the model space $\MS(\RS, \Lambda)$ with affine Weyl group $\WT$ by $\MS(\RS, \Lambda, T)$.
\end{notation}

\begin{definition}\label{Def_specialHyperplane}
\index{{reflection}!{affine}}
\index{hyperplane}
An element of $\WT$ which can be written as $t \circ r_\alpha$ for some $t\in T$ and $\alpha\in \RS$ is called \emph{(affine) reflection}.
A reflection in $\aW$ is an element of the form $t_k\circ r_\alpha$ for arbitrary $k\in \MS$ and $\alpha\in\RS$.
A \emph{hyperplane} $H_r$ in $\MS$ is the fixed point set of an affine reflection $r \in \aW$. It is called \emph{special with respect to $T$} if $r\in \WT$. 
\end{definition}

\begin{remark}
Note that for any affine reflection there exists $\alpha\in\RS$ and $\lambda\in\Lambda$ such that the reflection is given by the following formula
$$
r_{\alpha,\lambda}(x) 
	= r_\alpha(x) + \frac{2\lambda}{(\alpha,\alpha)} \alpha, 
	\text{ for all } x\in \MS.
$$
Further easy calculations imply
$$
r_{\alpha, \lambda}(x) 
	= r_\alpha\left( x-\frac{\lambda}{(\alpha,\alpha)}\alpha \right) + \frac{\lambda}{(\alpha,\alpha)}\alpha ,
	\text{ for all } x\in \MS.
$$
The fixed point set $H_{\alpha, \lambda}$ is given by
$$
H_{\alpha,\lambda}=\left\{x \in\MS: \frac{(\alpha,\alpha)}{2}\lb x,\alpha^\vee\rb =\lambda \right\}.
$$

As in the classical case any hyperplane defines two \emph{half-apartments} 
$$H_{\alpha,k}^+=\{x \in\MS: \frac{(\alpha,\alpha)}{2}\lb x,\alpha^\vee\rb \geq \lambda \}\; \text{ and }\;
H_{\alpha,k}^-=\{x \in\MS: \frac{(\alpha,\alpha)}{2}\lb x,\alpha^\vee\rb \leq \lambda \}.$$
\end{remark}

\begin{definition}
A vertex $x\in\MS$ is called \emph{special} if for each $\alpha\in\RS^+$ there exists a special hyperplane parallel to $H_{\alpha,0}$ containing $x$. Hence $x$ is the intersection of the maximal possible number of special hyperplanes.
\end{definition}

Note that the translates of $0$ by $T$ are a subset of the set of special vertices.

\begin{definition}\label{Def_WeylChamber}\label{Def_specialVertex}
\index{Weyl chamber}
\index{special vertex}
We define a \emph{Weyl chamber} in $\MS$ to be an image of a fundamental Weyl chamber, as defined in~\ref{Def_WeylChamber1},  under the full affine Weyl group $\aW$. If $\Lambda=\R$ a Weyl chamber is a simplicial cone in the usual sense. Therefore Weyl chambers and faces of Weyl chambers are called \emph{Weyl simplices} and Weyl simplices of co-dimension one \emph{panels}.
\end{definition}

\begin{remark}
\index{{Weyl chamber}!{based at}}
Note that a Weyl chamber $S$ contains exactly one vertex $x$ which is the intersection of all bounding hyperplanes of $S$. We call it \emph{base point} of $S$ and say $S$ is \emph{based at $x$}.
\end{remark}

The following proposition is used to introduce a second type of coordinates on $\MS$.

\begin{prop}{ \cite[Prop. 2.1]{Bennett}}\label{Prop_aboveHyperplane}
Given $\alpha\in \RS$ and $x\in\MS$. Then there exist a unique $m_\alpha \in H_{\alpha,0}$ and a unique $x^\alpha\in\Lambda'$ such that 
$$x= m_\alpha + x^\alpha \alpha.$$
The value of $x^\alpha$ is $\frac{1}{2} \lb x, \alpha^\vee\rb$. Furthermore $x\in H_{\alpha,\frac{x^\alpha}{2}}$.
\end{prop}
\begin{proof}
Define $x^\alpha = \frac{1}{2} \lb x,\alpha^\vee\rb$ and consider $m_\alpha=x- x^\alpha\alpha$. Then 
$$\lb  m_\alpha,\alpha^\vee\rb = \lb x, \alpha^\vee\rb - x^\alpha \lb \alpha,\alpha^\vee\rb = 0 $$
and $m_\alpha$ is contained in $H_{\alpha,0}$. 
It remains to prove uniqueness. 
Let $y^\alpha$ and $n_\alpha$ be value such that 
$m_\alpha +x^\alpha\alpha = n_\alpha + y^\alpha\alpha.$
Then $n_\alpha=r_\alpha(n_\alpha) = r_\alpha(x) - y^\alpha r_\alpha(\alpha)$ and
$m_\alpha=r_\alpha(m_\alpha) = r_\alpha(x) - x^\alpha r_\alpha(\alpha)$. Therefore we have
$$
r_\alpha(x)-x^\alpha\alpha = m_\alpha + x^\alpha\alpha = x = n_\alpha + y^\alpha\alpha = r_\alpha(x) - y^\alpha\alpha 
$$ 
and conclude that $x^\alpha=y^\alpha$ and $m_\alpha=y_\alpha$.
\end{proof}

\begin{corollary}\label{Cor_hyperplaneCoordinates}
\index{hyperplane coordinates}
Let $\RS$ be a root system of rank $n$ and let $B$ be a basis of $\RS$. Any $x\in\MS$ is uniquely determined by the $n$ values $\{x^\alpha\}_{\alpha\in B}$, which will be called \emph{hyperplane coordinates} of $\MS$ with respect to $B$. 
\end{corollary}

\subsection{The metric structure of \texorpdfstring{ $\MS(\Lambda, \RS)$ } {the model space}}

The remainder of this section is used to define a $\aW$-invariant $\Lambda$-valued metric on the model space $\MS=\MS(\Lambda, \RS, T)$ of a generalized affine building and to discuss its properties.

\begin{definition}\label{Def_metric}
\index{metric}
Let $\Lambda$ be a totally ordered abelian group and let $X$ be a set. A metric on $X$ with values in $\Lambda$, short a \emph{$\Lambda$-valued metric}, is a map $d:X\times X \mapsto \Lambda$ such that for all $x,y,z\in X$ the following axioms hold
\begin{enumerate}
\item $d(x,y)=0$ if and only if $x=y$
\item $d(x,y)=d(y,x)$ and
\item the triangle inequality $d(x,z)+d(z,y)\geq d(x,y)$ holds.
\end{enumerate}
The pair $(X,d)$ is called \emph{$\Lambda$-metric space}.
\end{definition}

\begin{definition}
An \emph{isometric embedding} of $\Lambda$-metric spaces $(X,d)$, $(X',d')$ is a map $f:X\rightarrow X'$ such that $d(x,y)=d'(f(x),f(y))$ for all $x,y\in X$. Such a map is necessarily injective, but need not be onto. If it is onto we call it \emph{isometry} or isomorphism of $\Lambda$-metric spaces.
\end{definition}

\begin{definition}\label{Def_distance}
\index{metric}
Let $\MS=\MS(\Lambda,\RS, T)$ be as in \ref{Not_modelSpace}. Let the \emph{distance} between two points $x$ and $y$ in $\MS$ be defined by
$$d(x,y)= \sum_{\alpha\in\RS^+} \vert \lb y-x, \alpha^\vee \rb \vert.$$
\end{definition}

Note that $d(x,y)=2\lb y-x, \rho^\vee\rb$ if $y-x\in\Cf$, where $\rho^\vee=\frac{1}{2}\sum_{\alpha^\vee \in(\RS^\vee)^+} \alpha^\vee$, as defined in Section \ref{Sec_rootSystems}.

\begin{remark}
Let $B$ be a basis of $\RS$. Define $\varepsilon_\alpha=\frac{2}{(\alpha,\alpha)}$ and identify the co-roots $\alpha^\vee$ with $\frac{2\alpha}{(\alpha, \alpha)}$. Making these assumptions in the definition of the metric in \cite{BennettDiss} one obtains exactly the metric as defined in \ref{Def_distance}.

The metric $d$ of \ref{Def_distance} generalizes the length of a translation defined in the context of simplicial affine buildings: Let $v$ be a vertex in a Euclidean Coxeter complex. Let the length $l(t_v)$ of a translation $t_v$ be the number of hyperplanes crossed by a minimal gallery from $x$ to $t_v(x)=x+v$. The formula giving this integer is exactly $\frac{1}{2}\sum_{\alpha\in\RS} \vert \lb y-x, \alpha^\vee \rb \vert $.
The fact that $d$ is the direct generalization of a combinatorial length function justified, at least in my opinion, to make a specific choice of the $\varepsilon_i$ appearing in the definition of the metric as written in \cite{BennettDiss}.
\end{remark}

\begin{prop}
The distance $d:\MS\times\MS\mapsto \Lambda$ defined in~\ref{Def_distance} is a $\aW$-invariant $\Lambda$-valued metric on $\MS$.
\end{prop}
\begin{proof}
By definition $d(x,y)=d(y,x)$ and $d(x,y)=0$ if and only if $x=y$, since otherwise, by Corollary 2.2 in \cite{Bennett}, one of the terms 
$\vert \lb z, \alpha^\vee\rb \vert$ would be strictly positive. 
It remains to prove $d(x,y)+d(y,z)\geq d(x,z)$:
\begin{align*}
d(x,z) 	&= \sum_{\alpha\in\RS^+} \vert \lb  z+ (y-y) - x , \alpha^\vee \rb\vert 
	 = \sum_{\alpha\in\RS^+} \vert \lb y-x, \alpha^\vee\rb + \lb z-y, \alpha^\vee\rb \vert \\
	&\leq \sum_{\alpha\in\RS^+} \big( \vert \lb y-x,\alpha^\vee\rb \vert +  \vert \lb z-y, \alpha^\vee\rb \vert \big) 
	 = d(x,y)+d(y,z).
\end{align*}
Hence $d$ is a metric. We prove $\aW$-invariance: 
Let $t_a: x\mapsto x+a$ be a translation in $\aW$. Then
\begin{align*}
d(x,y) 	&= \sum_{\alpha\in\RS^+} \vert \lb y-x, \alpha^\vee \rb \vert
	=  \sum_{\alpha\in\RS^+} \vert \lb y+a - (x+a), \alpha^\vee \rb \vert
	= d(t_a(x), t_a(y)).
\end{align*}
Therefore $d$ is translation invariant. With $w\in \sW$ we have
\begin{align*}
d(w.x, w.y)
	& = \sum_{\alpha\in\RS^+} \vert \lb w.y-w.x,\alpha^\vee \rb \vert 
	    =  \frac{1}{2}\sum_{\alpha\in\RS} \vert \lb w.y-w.x,\alpha^\vee \rb \vert  \\
	& = \frac{1}{2}\sum_{\alpha\in\RS} \vert \lb y-x, (w^{-1}.\alpha)^\vee \rb \vert 
	    = \frac{1}{2}\sum_{\alpha\in\RS} \vert \lb y-x, \alpha^\vee \rb\vert \\
	& =d(x,y).
\end{align*}
The second last equation holds since $\sW$ permutes the roots in $\RS$.
Therefore $d$ is $\sW$ invariant and $\aW$-invariance follows.
\end{proof}

Let $x$ be an element of the model space $\MS=\MS(\RS,\Lambda)$ defined in \ref{Def_modelSpace}. Fix a basis $B$ of $\RS$ and recall the definition of the hyperplane coordinates $\{x^\alpha\}_{\alpha\in B}$ of $x$ with respect to $B$ introduced in Corollary~\ref{Cor_hyperplaneCoordinates}.

\begin{prop}\label{Prop_distance0x}
Given $\MS$ with fixed basis $B$ of $\RS$. Let $x$ be an element of $\MS$. The distance $d(0,x)$ is uniquely determined by the hyperplane-coordinates $\{x^\alpha\}_{\alpha\in B}$ of $x$. 
With  $\alpha = \sum_{\beta\in B} p_\beta^\alpha\beta$ we have
$$
d(x,0)=\frac{1}{2} \sum_{\alpha \in \RS^+} \sum_{\beta\in B}  p_\beta^\alpha \; \vert x^\beta\vert.
$$
\end{prop}
\begin{proof}
Assume first that $x\in \Cf$. Then $x$ has hyperplane coordinates $\{x^\alpha\}_{\alpha\in B}$ defined in Corollary~\ref{Cor_hyperplaneCoordinates}, with $x^\alpha=\frac{1}{2}\lb x, \alpha^\vee\rb \geq 0$ for all $\alpha\in \RS^+$.  Hence, using $\alpha = \sum_{\beta\in B} p_\beta^\alpha\beta$, we have
\begin{align*}
d(x,0) 	&= \sum_{\alpha\in\RS^+} \vert \lb x, \alpha^\vee \rb\vert 
	 = \sum_{\alpha\in\RS^+} \sum_{\beta\in B} p_\beta^\alpha \,\lb x, \beta^\vee \rb
	 = \frac{1}{2} \sum_{\alpha\in\RS^+} \sum_{\beta\in B} p_\beta^\alpha  \,x^\beta.
\end{align*}
If $x$ is not $\Cf$ then $\aW$-invariance of $d$ implies that $d(0,x)=d(0,\overline{x})$, where $\overline{x} = w.x$ is the unique element of $\sW.x$ contained in $\Cf$ and $\overline{\beta}=w.\beta$. Further $\vert x^\beta \vert = \overline{x}^{\overline{\beta}}$ and  the assertion follows.
\end{proof}

\subsection{Convexity and parallelism}

As in the classical case, one can define

\begin{definition}\label{Def_GammaConvex}
\index{$\WT$-convex}
A subset $Y$ of $\MS$ is called \emph{convex}, or \emph{$\WT$-convex}, if it is the intersection of finitely many special half-apartments in the sense of Definition~\ref{Def_WeylChamber}. The \emph{$\WT$-convex hull} $c_{\WT}(Y)$ of a subset $Y\subset X$ is the intersection of all special half-apartments containing $Y$.
\end{definition}

Note that Weyl chambers and hyperplanes are $\aW$-convex, as well as finite intersections of convex sets. Special hyperplanes and Weyl chambers are, analogosly, $\WT$-convex.

\begin{lemma}{\cite[Prop.2.13]{BennettDiss}}\label{Lem_segment}\index{segment}
For any two special vertices $x,y$ in the model space $\MS$ the \emph{segment} $\seg(x,y)=\{z\in\MS : d(x,y)=d(x,z)+d(z,y)\}$ is the same as the convex hull $c_{\WT}(\{x,y\})$.
\end{lemma}

\begin{definition}\label{Def_parallel}
\index{parallel}
Two subsets $\Omega_1,\Omega_2$ of a $\Lambda$-metric space are at \emph{bounded distance}\index{bounded distance} if there exists $N\in\Lambda$ such that for all $x\in \Omega_i$ there exists $y\in \Omega_j$ such that $d(x,y)\leq N$ for $\{i,j\}=\{1,2\}$.
Subsets of a metric space are \emph{parallel} if they are at bounded distance. 
\end{definition}

Note, that parallelism is an equivalence relation. One can prove

\begin{prop}{\cite[Section 2.4]{Bennett}}
Let $\MS=\MS(\RS, \Lambda)$ equipped with the full affine Weyl group $\aW$. Then the following is true
\begin{enumerate}
\item Two hyperplanes, Weyl chambers or faces of Weyl chambers are parallel if and only if they are translates of each other by elements of $\aW$.
\item For any two parallel Weyl chambers $S$ and $S'$ there exists a Weyl chamber $S''$ contained in $S\cap S'$ and parallel to both. 
\end{enumerate}
\end{prop}

Moreover $H_{\alpha, k}$ is parallel to $t_{\beta^\vee}(H_{\alpha, k})=H_{\alpha, k+\beta^\vee(\alpha)},$ for all $\beta\in\RS$
where $t_{\beta^\vee}$ is the translation in $\MS$ by $\beta^\vee $ where we identify $\beta^\vee$ with $\frac{2}{(\beta, \beta)}\beta$. Compare Section~\ref{Sec_rootSystems}.
 \newpage
\section{Generalized affine buildings}\label{Sec_generalizedAffineBuildings}

\subsection{Basic definitions and properties}

Throughout the following let $\MS=\MS(\RS,\Lambda, T)$ be as defined in \ref{Not_modelSpace} and denote the spherical Weyl group associated to $\RS$ by $\sW$.

\begin{definition}\label{Def_LambdaBuilding}
\index{generalized affine building}
\index{{generalized affine building}!apartment}
\index{{generalized affine building}!atlas}
\index{{generalized affine building}!{vertex retraction}}
Let $X$ be a set and $\App$ a collection of injective maps $f:\MS\hookrightarrow X$, called \emph{charts}.
The images $f(\MS)$  of charts $f\in\App$ are called \emph{apartments} of $X$. Define \emph{Weyl chambers, hyperplanes, half-apartments, special vertices, ... of $X$} to be images of such in $\MS$ under any $f\in\App$. The set $X$ is a \emph{(generalized) affine building} with \emph{atlas} (or \emph{apartment system}) $\App$ if the following conditions are satisfied
\begin{enumerate}[label={(A*)}, leftmargin=*]
\item[(A1)] Given $f\in\App$ and $w\in \WT$ then $f\circ w\in\App$. 
\item[(A2)] Given two charts $f,g\in\App$ with $f(\MS)\cap g(\MS)\neq\emptyset$. Then $f^{-1}(g(\MS))$ is a closed convex subset of $\MS$. There exists $w\in \WT$ with $f\vert_{f^{-1}(g(\MS))} = (g\circ w )\vert_{f^{-1}(g(\MS))}$.
\item[(A3)] For any two points in $X$ there is an apartment containing both.
\item[(A4)] Given Weyl chambers $S_1$ and $S_2$ in $X$ there exist sub-Weyl chambers $S_1', S_2'$ in $X$ and $f\in\App$ such that $S_1'\cup S_2' \subset f(\MS)$. 
\item[(A5)] For any apartment $A$ and all $x\in A$ there exists a \emph{retraction} $r_{A,x}:X\to A$ such that $r_{A,x}$ does not increase distances and $r^{-1}_{A,x}(x)=\{x\}$.
\item[(A6)] Given charts $f,g$ and $h$ such that the associated apartments intersect pairwise in half-apartments. Then $f(\MS)\cap g(\MS)\cap h(\MS)\neq \emptyset$. 
\end{enumerate}
The \emph{dimension} of the building $X$ is $n=\rk(\RS)$, where $\MS\cong(\Lambda')^n$. 
\end{definition}

\begin{remark}\label{Def_buildingMetric}
Condition $(A1)-(A3)$ imply the existence of a $\Lambda$-distance on $X$, that is a function $d:X\times X\mapsto \Lambda$ satisfying all conditions of Definition~\ref{Def_metric} but the triangle inequality.  Given $x,y$ in $X$ fix an apartment containing $x$ and $y$ with chart $f\in\App$ and let $x',y'$ in $\MS$ be defined by $f(x')=x, f(y')=y$. 
The \emph{distance} $d(x,y)$ between $x$ and $y$ in $X$ is given by $d(x',y')$. This extends to a well defined distance function on $X$. Therefore it makes sense to talk about a distance non-increasing function in $(A5)$. Note further that, by $(A5)$, the defined  distance function $d$  satisfies the triangle inequality. Hence $d$ is a metric on $X$.
\end{remark}

Tits defined his syst{\`e}me d'appartements in \cite{TitsComo} by giving five axioms. The first four are the same as $(A1)-(A4)$ above. The fifth axiom originally reads different from ours but was later replaced with $(A5)$ as presented in the definition above. One can find a short history of the axioms in \cite{Ronan}. In fact if $\Lambda=\R$ axiom $(A6)$ follows from $(A1)-(A5)$. But in the general case this additional axiom is necessary as illustrated with an example given on p. 563 in \cite{Bennett}.
However in \cite{Bennett}  axiom $(A6)$ is mostly used to avoid pathological cases and to guarantee the existence of the panel and wall trees as constructed in Section~\ref{Subsec_trees}.

\begin{definition}\label{Def_thick}
\index{generalized affine building!thick}
Let $X$ be a generalized affine building with model space $\MS(\RS, \Lambda, T)$ and apartment system $\App$. The building $X$ is called \emph{thick with respect to $\WT$} if for any special hyperplane $H$ of $X$ there exist apartments $A_1=f_1(\MS)$ and $A_2=f_2(\MS)$, with $f_i\in\App, i=1,2$ such that $H\in A_i, i=1,2$ and $A_1\cap A_2$ is one of the two half-apartments of $A_1$ (or $A_2$) determined by $H$.  Furthermore apartments do not branch at non-special hyperplanes.
\end{definition}

\begin{remark}
If in the previous definition $T=\MS$ then $X$ is a building branching everwhere.
\end{remark}

\begin{definition}\label{Def_iso}
\index{{generalized affine building}!{isomorphism}}
Two affine buildings $(X_1,\App_1),(X_2,\App_2)$ of the same type $\MS(\Lambda, \RS)$ are \emph{isomorphic} if there exist maps $\pi_1: X_1 \to X_2$, $\pi_2: X_2 \to X_1$, further maps $\pi_{\App_1}:\App_1 \to \App_2$ , $\pi_{\App_1}:\App_2 \to \App_1$ and an automorphism $\sigma$ of $\MS$ such that 
\begin{align*}
\pi_i\circ\pi_j=\one_{X_i} & \text{ with } \{i,j\}=\{1,2\},\\
\pi_{\App_i}\circ \pi_{\App_i}=\one_{\App_i} & \text{ with } \{i,j\}=\{1,2\},
\end{align*}
and the following diagram commutes for all $f\in\App_i$ with $\{i,j\}=\{1,2\}$
\[ \begin{xy}
 	\xymatrix{
		\MS \ar[d]_\sigma \ar[r]^f  & X_i \ar[d]^{\pi_i} \\
		\MS \ar[r]_{\pi_{\App_{i}}(f)} & X_j
	}.
\end{xy}\]
\end{definition}

Examples of a generalized affine buildings are $\Lambda$-trees without leaves. The definition of a $\Lambda$-tree, \cite[p.560]{Bennett} or \ref{Def_lambdatree}, is equivalent to the definition of an affine building of dimension one.
Simplicial buildings arise from groups defined over fields with discrete valuations. An example of this type is given in \cite[Example 3.2]{Bennett} associating to $SL_n(K)$, with $K$ a field with $\Lambda$-valued valuation, a generalized affine building. It is a generalization of the example given in \cite[Section 9.2]{Ronan}.

Note that the Davis realization of a simplicial affine building is a generalized affine building, as defined in~\ref{Def_LambdaBuilding} with $\Lambda=\R$ and $T$ chosen equal to the co-root lattice $\QQ(\RS^\vee)$ of $\RS$.

\subsection{Local and global structure}

Any simplicial affine building has an associated spherical building, the so called spherical building at infinity. This useful and important result by Bruhat and Tits \cite{BruhatTits} is also true in the generalized case.

\begin{definition}\label{Def_buildingAtInfinity}
\index{building at infinity}
Let $(X,\App)$ be an affine building. Denote by $\partial S$ the parallel class of a Weyl chamber $S$ in $X$. Let
$$\binfinity X =\{\partial S : S \text{ Weyl chamber of } X \text{ contained in an apartment of } \App\}$$
be the set of chambers of the \emph{spherical building at infinity} $\binfinity X$. Two chambers $\partial S_1$ and $\partial S_2$ are \emph{adjacent} if there exist representatives $S'_1, S'_2$ which are contained in a common apartment with chart in $\App$, have the same basepoint and are adjacent in $X$.
\end{definition}

\begin{prop}\label{Prop_buildingAtInfinity}
Let $(X,\App)$ be an affine building modeled on $\MS(\RS, \Lambda, T)$. The set $\binfinity X$ defined above is a spherical building of type $\RS$ with apartments in one to one correspondence with apartments of $X$.
\end{prop}
\begin{proof}
It is obvious that $\binfinity X$ is a simplicial complex with adjacency as defined in \ref{Def_buildingAtInfinity}. 
An apartment in $\binfinity X$ is defined to be a the set of equivalence classes determined by an apartment of $X$. It is obvious that they are Coxeter complexes of type $\RS$ and that hence $\binfinity X$ has to be of type $\RS$.

Given two chambers $c$ and $d$ in $\binfinity X$. Let $S$ and $T$ be representatives of $c$, respectively $d$. By axiom $(A4)$ there exists an apartment $A$ containing sub-Weyl chambers of $S$ and $T$. The set of equivalence classes of Weyl chambers determined by $A$ hence contains $c$ and $d$. Therefore 2. of Definition \ref{Def_building} holds. 

If $\partial A$ and $\partial A'$ are two apartments of $\binfinity X$ both containing the chambers $c$ and $d$. Then there exist charts $f$ and $f'$ such that $f(\MS)$ and $f'(\MS)$ contain representatives $S, S'$ of $c$ and $T, T'$ of $d$. The Weyl chambers $S,S'$ and $T,T'$ intersect in sub-Weyl chambers $S''$ and $T''$, respectively. The map $f'\circ f^{-1}$ fixes $S''$ and $T''$ and induces an isomorphism from $\partial A$ to $\partial A'$. Therefore $\binfinity X$ is indeed a spherical building.
\end{proof}

In contrast to a remark made in \cite{BennettDiss} it is possible to prove Proposition \ref{Prop_buildingAtInfinity} without using axiom $(A5)$.

The local structure of an affine building was not examined in \cite{BennettDiss}. In analogy to the residues of vertices in a simplicial affine building one can associate to a vertex of a generalized affine building a spherical building. Most of the following in based on \cite{Parreau}.

Let in the following $(X,\App)$ be an affine building of type $\MS=\MS(\Lambda, \RS, T)$ and let $\binfinity X$ denote its spherical building at infinity.

\begin{definition}\label{Def_germ}
\index{{generalized affine building}!{germ}}
Two Weyl simplices $S$ and $S'$ \emph{share the same germ} if both are based at the same vertex and if $S\cap S'$ is a neighborhood of $x$ in $S$ and in $S'$.
It is easy to see that this is an equivalence relation on the set of Weyl simplices based at a given vertex. The equivalence class of $S$, based at $x$, is denoted by $\Delta_x S$ and is called \emph{germ of $S$ at $x$}.
\end{definition}

\begin{remark}\label{Rem_partialorder}
The germs of Weyl simplices at a special vertex $x$ are partially ordered by inclusion: $\Delta_x S_1\subset \Delta_xS_2$ if there exist representatives $S'_1, S'_2$ contained in a common apartment such that $S_1'$ is a face of $S_2'$. Let $\Delta_xX$ be the set of all germs of Weyl simplices based at $x$.
\end{remark}

\begin{prop}\label{Prop_tec16}
Let $(X, \App)$ be an affine building and $c$ a chamber in $\binfinity X$. Let $S$ be a Weyl chamber in $X$ based at $x$. Then there exists an apartment $A$ with chart in $\App$ containing a germ of $S$ at $x$ such that $c\in \partial A$.
\end{prop}
The proof of the above proposition is literally the same as of Proposition 1.8 in \cite{Parreau}.

\begin{corollary}\label{Cor_WeylChamber}
Fix a point $x\in X$. For each (face of a) Weyl chamber $F$ exists a unique (face of a) Weyl chamber $F'$ which is based at $x$ and parallel to $F$.
\end{corollary}
\begin{proof}
Apply Proposition~\ref{Prop_tec16} to $x$ and $c=\partial F$ and arbitrary $S$ based at $x$.
\end{proof}

\begin{corollary}\label{Cor_tec17}
For any chamber $c\in\binfinity X$ the affine building $X$ is as a set the union of all apartments containing a representative of $c$.
\end{corollary}
\begin{proof}
Fix a chamber $c$ at infinity. For all  points $x\in X$ and arbitrary Weyl chambers $S$ based at $x$. there exists by \ref{Prop_tec16} an apartment $A$ containing $x$ and a germ of $S$ at $x$ and $c$ is contained in $\partial A$. 
\end{proof}

\begin{corollary}\label{Cor_tec18}
Given the germ $\Delta_xS$ of a Weyl chamber $S$ at $x$. Then $X$ is the union of all apartments containing $\Delta_xS$.
\end{corollary}
\begin{proof}
Given $y\in X$ there exists by axiom $(A3)$ an apartment $A'$ containing $x$ and $y$. Let $T$ be a Weyl chamber in $A'$ based at $x$ containing $y$ and define $c=\partial T$. By Proposition~\ref{Prop_tec16} there exists an apartment $A$ such that a germ of $S$ at $x$ is contained in $A$ and such that the corresponding apartment $\partial A$ of $\binfinity X$ contains $c$. But then the unique representative of $c$ based at $x$ is also contained in $A$. Therefore $A$ contains $y$ and a germ of $S$ at $x$.
\end{proof}

\begin{corollary}\label{Cor_GG}
Any two germs of Weyl chambers based at the same vertex are contained in a common apartment.
\end{corollary}
\begin{proof}
Let $S$ and $T$ be Weyl chambers both based at $x$. By Proposition~\ref{Prop_tec16} there exists an apartment $A$ of $X$ containing $S$ and a germ of $T$ at $x$. Therefore $\Delta_xS$ and $\Delta_xT$ are both contained in the apartment $\Delta_xA$.
\end{proof}

\begin{prop}\label{Prop_A3'}
Let $(X,\App)$ be an affine building. Let $S$ and $T$ be Weyl chambers based at $x$ and $y$, respectively. Then there exists an apartment $A$ of $X$ containing a germ of $S$ at $x$ and a germ of $T$ at $y$.
\end{prop}
\begin{proof}
By axiom $(A3)$ there exists an apartment $A$ containing $x$ and $y$. Denote by $S_{xy}$ a Weyl chamber in $A$ based at $x$ containing $y$ and denote by $S_{yx}$ the Weyl chamber based at $y$ such that $\partial S_{xy}$ and $\partial S_{yx}$ are opposite in $\partial A$. Then $x$ is contained in $S_{yx}$ by definition. If $\Delta_yT$ is not contained in $A$ apply Proposition~\ref{Prop_tec16} to obtain an apartment $A'$ containing a germ of $T$ at $y$ and $\partial S_{yx}$ at infinity. But then $x$ is also contained in $A'$. 

Let $S'_{xy}$ denote the unique Weyl chamber contained in $A'$ having the same germ at $x$ as $S_{xy}$. 
Without loss of generality we can assume that the germ $\Delta_yT$ is contained in $S'_{xy}$. Otherwise $y$ is contained in a wall of $S'_{xy}$ and we can replace $S'_{xy}$ by an adjacent Weyl chamber in $A'$ satisfying this condition.
A second application of Proposition~\ref{Prop_tec16} to $\partial S'_{xy}$ and the germ of $S$ at $x$ yields an apartment $A''$ containing $\Delta_xS$ and $S'_{xy}$ and therefore $\Delta_yT$.
\end{proof}

\begin{remark}
Corollaries \ref{Cor_WeylChamber} and \ref{Cor_GG} are the direct analogs of 1.9 and 1.11 of \cite{Parreau} and Proposition \ref{Prop_A3'} corresponds to \cite[1.16]{Parreau}. Note, however, that the proof is different.
\end{remark}

\begin{defthm}\label{Thm_residue}
\index{{generalized affine building}!{residue}}
Let $(X,\App)$ be an affine building with model space $\MS(\RS,\Lambda,T)$. Then $\Delta_xX$ is a spherical building of type $\RS$ for all $x$ in $X$. If $x$ is special and $X$ thick with respect to $\WT$, then $\Delta_xX$ is thick as well. Furthermore $\Delta_xX$ is independent of $\App$.
\end{defthm}
\begin{proof}
We verify the axioms of Definition~\ref{Def_building}.
It is easy to see that $\Delta_xX$ is a simplicial complex with the partial order defined in \ref{Rem_partialorder}. It is a pure simplicial complex, since each germ of a face is contained in a germ of a Weyl chamber. The set of equivalence classes determined by a given apartment of $X$ containing $x$ is a subcomplex of $\Delta_xX$ which is, obviously, a Coxeter complex of type $\RS$. Hence define them to be the apartments of $\Delta_xX$. Therefore equation \ref{tec17} of Definition~\ref{Def_building} holds.
Two apartments of $\Delta_xX$ are isomorphic via an isomorphism fixing the intersection of the corresponding apartments of $X$ which implies \ref{tec18} of Definition~\ref{Def_building}. 
Finally \ref{tec19} holds by Corollary~\ref{Cor_GG}.

Assume that $x$ is special and $X$ thick with respect to $\WT$. Let $c$ be a chamber in $\Delta_xX$ and $\Delta_xA$ an apartment containing $c$. For each panel $p$ of $c$ there exists a chamber $c'$ contained in $\Delta_xA$ such that $c\cap c'=p$. The panel $p=\Delta_xF$ determines a wall $H\subset A$. Since $X$ is thick there exists an apartment $A'$ whose intersection with $A$ is a half-apartment bounded by $H$. Hence there is a third chamber $c''$ of $\Delta_xX$ determined by a Weyl chamber in $A'$ based at $x$ containing $F$. Therefore $\Delta_xX$ is thick.

Let $\App'$ be a different system of apartments of $X$ and assume w.l.o.g. that $\App\subset \App'$.
Let $\Delta$ denote the spherical building of germs at $x$ with respect to $\App$ and denote by $\Delta'$ the building at $x$ with respect to $\App'$. Since spherical buildings have a unique apartment system $\Delta$ and $\Delta'$ are equal if they contain the same chambers. Assume there exists a chamber $c\in\Delta'$ which is not contained in $\Delta$. Let $d$ be a chamber opposite $c$ in $\Delta'$ and $a'$ the unique apartment containing both. Note that $a'$ corresponds to an apartment $A'$ of $X$ having a chart in $\App'$. There exist $\App'$-Weyl chambers $S_c$, $S_d$ contained in $A$ representing $c$ and $d$, respectively. Let $y$ be a point in the interior of $S_c$ and $z$ an interior point of $S_d$. Using axiom $(A3)$ there exists a chart $f\in \App$ such that $A=f(\MS)$ contains $y$ and $z$. The apartment $A$ contains $x$ since $x\in\seg(y,z)$ and $\seg(y,z)\subset A\cap A'$. The unique Weyl chambers of $A$ based at $x$ containing $y$, respectively $z$, have germs $c$, respectively $d$, which is a contradiction. Hence $\Delta=\Delta'$.
\end{proof}

\begin{remark}
Let $\MS=\MS(\RS,\Lambda,T)$ be the model space of an affine building and let $\partial \MS$ be canonically identified with the associated Coxeter complex. Note that for each $x\in \MS$ and each chamber $c$ of $\partial \MS$ there exists a Weyl chamber $S$ contained in $c$ and based at $x$. Therefore the type of $\Delta_xX$ for an affine building $(X,\App)$ modeled on $\MS$ is always $\RS$. 

If $(X,\App)$  is the geometric realization of a simplicial affine building (in the sense of Definition~\ref{Def_building}) then $\Delta_xX$ is canonically isomorphic to the residue (or link) of $x$ if and only if $x$ is a special vertex. The definition of a spherical building corresponding to the residue (respectively link) of a non-special vertex would be possible defining a second class of Weyl chambers based at a vertex $x$ with respect to the stabilizer $(\WT)_x$ of $x$ in the restricted affine Weyl group $\WT$. Since we will not make use of this fact, we will not give details here.
\end{remark}

\begin{prop}~\label{Prop_epi}
Let $(X,\App)$ be an affine building, $\binfinity X$ its building at infinity. For all vertices $x\in X$ there exists an epimorphism 
$$\pi_x:\; \partial_\App X \rightarrow \Delta_xX .$$
\end{prop}
\begin{proof}
Given $x\in X$ and $c\in \binfinity X$. Let $S$ be the Weyl chamber based at $x$ and contained in $c$, which exists by Corollary~\ref{Cor_WeylChamber}. Define $\pi_x(c)=\Delta_xS$, the germ of $S$ at $x$. Since for all $d \in \Delta_xX$ there exists a Weyl chamber $S'$ in $X$ such that $\Delta_xS' = d$ the map $\pi_x$ is surjective. By definition of $\Delta_xX$ the partial order and adjacency is preserved.
\end{proof}

\begin{prop}\label{Prop_A5}
Let $X$ be an affine building and $A_i$ with $i=1,2,3$ three apartments of $X$ pairwise intersecting in half-apartments. Then $A_1\cap A_2\cap A_3$ is either a half-apartment or a hyperplane.
\end{prop}
\begin{proof}
For $i\neq j$ denote the intersection $A_i\cap A_j$ by $M_{ij}$. The corresponding apartments $\partial A_i$ in the spherical building at infinity pairwise intersect in half-apartments as well. Hence $\partial A_1\cap \partial A_2 \cap \partial A_3$ is either a half-apartment itself or a hyperplane in $\binfinity X$. Assume that $\partial A_1\cap \partial A_2 \cap \partial A_3$ is a half-apartment. Then $A_1\cap A_2\cap A_3$ is a half-apartment contained in each of the $A_i$.

Assume now that we are in the case where $\partial A_1\cap \partial A_2 \cap \partial A_3$ is a hyperplane $m$ in $\binfinity X$.
Walls at infinity correspond to parallel classes of hyperplanes in the affine building. Hence there are three hyperplanes $H_{ij}$ bounding the half-apartments $M_{ij}= A_i\cap A_j$ which are all contained in $m$. Note that the half-apartments $M_{13}$ and $M_{23}$ are opposite in $A_3$ in the sense that their union equals $A_3$. By axiom $(A6)$ the intersection $A_1\cap A_2\cap A_3$ is nonempty and equal to the strip $M_{13}\cap M_{23}$. It is obvious that the hyperplanes $H_{13}$ and $H_{23}$ are contained in $M_{13}\cap M_{23}$. Since this argument is completely symmetric in the indices, $H_{12}$ is contained in $M_{13}\cap M_{23}$ as well. Again by symmetry each of the hyperplanes is between the other two and hence $H_{12}=H_{13}=H_{23}= A_1\cap A_2 \cap A_3$.
\end{proof}

This leads to the following observation.

\begin{property}{\bf The sundial configuration. } \label{Prop_sundial}
Let $A$ be an apartment in $X$ and let $c$ be a chamber in $\binfinity X$ containing a panel of $\partial A$ but not contained in $\partial A$. Then $c$ is opposite to two uniquely determined chambers $d_1$ and $d_2$ in $\partial A$. Hence there exist apartments $A_1$ and $A_2$ of $X$ such that $\partial A_i$ contains $d_i$ and $c$ with $i=1,2$.  The three apartments $\partial A_1,\partial A_2$ and $\partial A$ pairwise intersect in half-apartments. Axiom $(A6)$ together with the proposition above implies that their intersection is a hyperplane.
\end{property}

 \begin{figure}[htbp]
 \begin{center}
 	\resizebox{!}{0.2\textheight}{\input{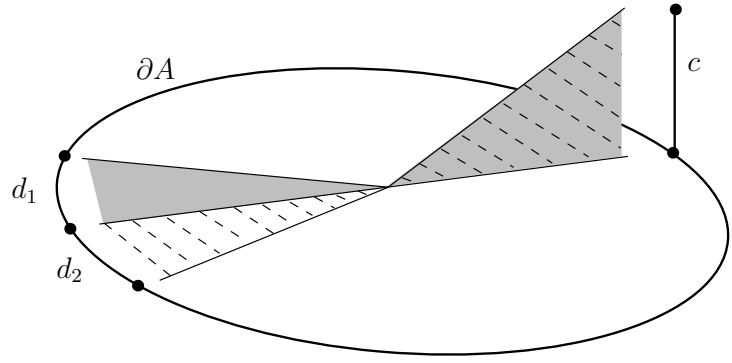}}
 	\caption[adjacent]{The sundial configuration.}
 	\label{Fig_sundial}
 \end{center} 
 \end{figure}

\begin{prop}\label{Prop_liftGallery}
Let $x$ be an element of $X$. Let $(c_0, \ldots, c_k)$ be a minimal gallery in $\binfinity X$. Find $x$-based representatives $S_i$ of $c_i$. If $(\pi_x(c_0), \ldots, \pi_x(c_k))$ is minimal in $\Delta_xX$, then there exists an apartment containing $\cup_{i=0}^k S_i$.
\end{prop}
\begin{proof}
The proof is by induction on $k$. For $k=0$ there is just one Weyl chamber and the result holds. Let $A'$ be an apartment containing $S_1\cup S_2\cup \ldots \cup S_{k-1}$. If $c_k$ is contained in $\partial A'$ we are done. If $c_k$ is not contained in $\partial A'$ we have the sundial configuration, which determines a unique hyperplane $H$ in $A'$. Let $H^+$ be the unique half-apartment of $A'$ determined by this hyperplane which contains a representative of $c_0$. Then $H^+$  also contains representatives of $c_1,\ldots,c_{k-1}$, since this is a minimal gallery and $S_k$ is on the other side of $H$. 

We claim that $x$ is contained in $H^+$.
Let $A''$ be the apartment in the sundial configuration containing $H^- = (A'\setminus H^+)\cup H$ and $c_k$ at infinity. If $x\in A'\setminus H^+ \subset A''$ the Weyl chamber $S_k$ is contained in $A''$. Let $\rho: A'' \rightarrow A'$ be the isometry  fixing $A''\setminus H^+$. The Weyl chamber $S_k$ is mapped onto $S_{k-1}$ and the set $S_k\cap(A'\cap H^+)$ is pointwise fixed. Therefore $\pi_x(c_{k-1})=\pi_x(c_k)$ which is a contradiction. Therefore $x$ is contained in $H^+$ and $\cup_{i=0}^{k-1} S_i \subset H^+$. Let $A$ now be the apartment in the sundial configuration containing $H^+$ and $S_k$. Then $\cup_{i=0}^{k} S_i$ is contained in $A$.
\end{proof}

\begin{corollary}\label{Cor_CO}
Given two Weyl chambers $S,T$ based at the same vertex $x$. If their germs $\Delta_xS$ and $\Delta_xT$ are opposite in $\Delta_xX$ then there exists a unique apartment containing $S$ and $T$.
\end{corollary}
\begin{proof}
Choose a minimal gallery $(c_0,c_1,\ldots, c_n)$ from $c_0=\partial S$ to $c_n=\partial T$ and consider the representatives $S_i$ of $c_i$ based at $x$. Then $S_0=S$ and $S_n=T$. Proposition~\ref{Prop_liftGallery} implies the assertion.
\end{proof}

\begin{remark}
Propositions \ref{Prop_A5} and \ref{Prop_liftGallery} as well as \ref{Prop_sundial} are due to L. Kramer.
\end{remark}

\begin{thm}\label{Thm_projection1}
Let $(X,\App)$ be an affine building, $A$ an apartment of $X$ and $c,d$ two opposite chambers in $\partial A$. Then
$$x\in A \Longleftrightarrow \pi_x(c) \text{ and } \pi_x(d) \text{ are opposite in } \Delta_xX. $$ 
The restriction of $\pi_x$ to the boundary of an apartment $A$ containing $x$ is an isomorphism onto its image.
\end{thm}
\begin{proof}
Assume $x\in A$. Each panel $p_i\in c$ defines an equivalence class $m_i$ of parallel hyperplanes in $A$. Denote by $H_i$ the unique wall in $m_i$ containing $x$ and let $\alpha^i$ be the (unique) half-apartment of $A$ with wall $H_i\in m_i$ such that $\partial \alpha_i$ contains $c$. The intersection of all $\alpha_i$ is a Weyl chamber $S\in c$ based at $x$. Similarly we have a Weyl chamber $T$ based at $x$ contained in $d$. Since the defining walls of $S$ and $T$ are the same, the chambers $\pi_x(c)$ and $\pi_x(d)$ are opposite in $\Delta_xX$.

Given $c,d$ in $\binfinity X$ such that the chambers $\pi_x(c)$ and $\pi_x(d)$ are opposite. Let $S$ and $T$ denote the Weyl chambers based at $x$ contained in $c$, $d$, respectively.
Choose a minimal gallery $\gamma'=(c'_0=\pi_x(c), c'_1, \ldots, c'_{n-1}, c'_n=\pi_x(d))$ in $\Delta_xX$.  Then there exists a minimal gallery $\gamma=(c_0=c, c_1, \ldots, c_{n-1}, c_n=d)$ in $\binfinity X$ such that $\pi_x(c_i)=c'_i$. Denote by $S_i$ the unique Weyl chamber based at $x$ and contained in $c_i$. Proposition~\ref{Prop_liftGallery} implies the existence of an apartment $A$ containing $\cup_{i=0}^k C_i$ and hence $x$. Uniqueness is clear by axiom $(A2)$. 
\end{proof}

%
%

\subsection{Trees: affine buildings of dimension one}\index{tree with sap}
\label{Subsec_trees}

The notion of an affine building generalizes in a natural way the well known $\Lambda$-trees as for example defined in \cite{AlperinBass}. Bennett proves \cite[Example 3.1]{Bennett} that a $\Lambda$-tree with sap, as defined below, is nothing else than a generalized affine building of dimension one in the sense of Definition~\ref{Def_LambdaBuilding}.

\begin{definition}\label{Def_lambdatree}
Let $(T,d)$ be a $\Lambda$-metric space. Denote by $[k,k']_\Lambda$ the interval between $k$ and $k'$ in $\Lambda$. Then $(T,d)$ is a \emph{$\Lambda$-tree} if the following axioms are satisfied:
\begin{enumerate}[label={(T*)}, leftmargin=*]
\item[(T1)] Given $x,y$ in $T$ then there exists a unique isometry $f:[0, d(x,y)]_\Lambda \rightarrow T$ such that $f(0)=x$ and $f(d(x,y))=y$. We define  $[x,y]\define f([0,d(x,y)]_\Lambda)$. 
\item[(T2)] Given $x,y,z$ in $T$ there exists $w\in T$, necessarily unique, such that $[x,y]\cap[x,z]=[x,w]$.
\item[(T3)] Given $x,y,z$ in $T$ such that $[x,y]\cap [y,z]=\{y\}$ then $[x,z]=[x,y]\cup[y,z]$.
\end{enumerate}
A \emph{ray of $T$ at $x$} is a subset $S_x$ in $T$ isometric to $[0,\infty)_\Lambda$. We say $T$ is a tree \emph{without leaves} if for any two points $x,y\in T$ there is a ray based at $x$ containing $y$. 
A \emph{line} of $T$ is a set $l\subset T$ such that $l$ is isometric to $\Lambda$.
\end{definition}

Condition $\mathrm{(T2)}$ is sometimes called the \emph{Y-condition}. It guarantees that rays diverge at points of $T$. For example two lines in a $\Q$-tree cannot diverge at an ``irrational point''. 
 
\begin{definition}\label{Def_end}
Let $(T,d)$ be a metric tree. Two rays $R, R'$ in $T$ are \emph{equivalent} if they are at bounded Hausdorff distance. Define an \emph{end} of $T$ to be an equivalence class of rays.

Fix a set of lines $\App$ in $T$ and define an \emph{$\App$-end} of $T$ to be an equivalence class of rays contained in lines of $\App$. The set of $\App$-ends of $T$ will be denoted by $\partial_\App T$.
\end{definition}

Obviously $\partial_\App T\subseteq\{\text{ends of } T\}$.
 
\begin{definition}\label{Def_atlas}
A \emph{tree with sap},\footnote{Here \emph{sap} stands for {\bf s}ystem of {\bf ap}artments. This abbreviation was first suggested by Mark Ronan.}  denoted by $(T,\App)$, is a metric tree $(T,d)$ together with a set of lines $\App$ such that the following axioms are satisfied:
\begin{enumerate}[label={(TA*)}, leftmargin=*]
\item[(TA1)] Given $x,y\in T$ there exists a line $l$ in $\App$ with $x,y\in l$.
\item[(TA2)] Given $\App$-ends $a,b$ of $(T,\App)$, there exists a line $l$ in $\App$ such that $a$ and $b$ are the ends of $l$.
\end{enumerate}
The set $\App$ is called \emph{atlas} or \emph{system of apartments}. Let the \emph{boundary} or \emph{building at infinity} of a tree $(T,\App)$ with sap be the set $\partial_\App T$ of $\App$-ends, as defined in \ref{Def_end}. 
\end{definition}
By axiom $\mathrm{(TA1)}$ the apartment system $\App$ is uniquely determined by $\partial_\App T$.

\begin{remark}\label{Rem_ends}
Note that the building at infinity $\partial_\App T$ of a tree $T$ with sap might be smaller than the set of all ends. An example is given in figure~\ref{Fig_ends}. Let $\App$ be the set of ``vertical'' lines, i.e. the set of all lines which are unions of two rays starting at the same point, one contained in an equivalence class $a_i$, the other in an equivalence class $b_j$. Then $x$ and $y$ are not contained in $\partial_\App T$ but they are ends of $(T,\App)$ in the sense of Definition~\ref{Def_end}. In fact the building at infinity strongly depends on $\App$.
\end{remark}

\begin{figure}[htbp]
\begin{center}
	\resizebox{!}{0.2\textheight}{\input{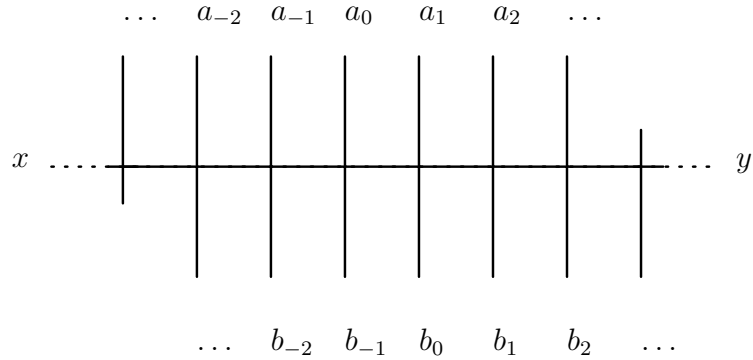}}
	\caption[The boundary of a tree with sap depends on the system of apartments.]{If $\App$ is the set of ``vertical'' lines then $\partial_\App T$ equals $\{a_i, b_i\vert\; i\in\Z\}$ which is a proper subset of the set of ends of $T$ which is the set $\partial_\App T\cup \{x,y\}$.}
	\label{Fig_ends}
\end{center} 
\end{figure}

An important tool is the following \emph{base change functor}. Given a morphism of ordered abelian groups $\Lambda$ and $\Lambda'$ this functor associates to each $\Lambda$-tree a $\Lambda'$-tree. A proof can be found on p. 70ff. of \cite{Chiswell}.

\begin{prop}\label{Prop_baseChangeTrees}
\index{{base change}!{$\Lambda$-trees}}
Let $e:\Lambda\rightarrow \Lambda'$ be a homomorphism of ordered abelian groups and let $(T,d)$ be a $\Lambda$-tree, then there exists a $\Lambda'$-tree $(T',d')$ and a map $\phi:T\rightarrow T'$ satisfying 
$$d'(\phi(x),\phi(y))=e(d(x,y))$$ 
for all $x,y\in T$. 

Furthermore, if $(T'',d'')$ is another $\Lambda'$-tree and $\varphi: T\rightarrow T''$ a map satisfying 
$$d''(\varphi(x),\varphi(y))=e(d(x,y)) $$
for all $x,y\in T$, then there is a unique $\Lambda'$-isometry $\mu:T'\rightarrow T''$ such that $\mu\circ\varphi=\phi$. 
\end{prop}

\begin{remark}
Note that $C=ker(e)$ is a convex subgroup of $\Lambda$ in the sense that for any $k,l\in C$ the sum $l+k$ is also contained in $C$. In a certain sense $T'$ is the quotient of $T$ modulo $C$.
\end{remark}

\begin{notation}\label{Not_kappa}
\index{$\kappa(a,b,c)$}
Let $(T,\App)$ be a tree with sap. Any triple $(a,b,c)$ of pairwise distinct $\App$-ends uniquely determines three  apartments $[ab], [ac]$ and $[bc]$. The intersection of these apartments is a hyperplane in a line, which is the unique point contained in all three apartments. We denote this point by $\kappa(a,b,c)$. 
\end{notation}

The aim of the following is to give data that will allow us to reconstruct a tree from $\partial_\App T$ plus certain additional information. This construction is used in the proof of Theorem~\ref{Thm_trees}. Compare Section~\ref{Sec_prooftrees}.

\begin{definition}
Let $E$ be a set. The pair $(E,\wedge)$ is called \emph{rooted tree datum} if $\wedge$ is a map 
$$ E\times E \longrightarrow \Lambda \cup \{\pm \infty\},\;\; (x,y) \longmapsto x\wedge y$$
such that for all $x,y,z\in E$ the following axioms are satisfied
\begin{description}
\item[(RT0)] $x\wedge y \geq 0$
\item[(RT1)] $x\wedge y = y\wedge x$
\item[(RT2)] $x\wedge z \geq \min\{x\wedge y, y\wedge z\}$.
\end{description}
\end{definition}

\subsubsection*{The Alperin-Bass construction}

The following construction will imply that a rooted tree datum $(E,\wedge)$ gives rise to a unique tree with sap having $E$ as its set of ends. Note, that the Alperin-Bass tree for an arbitrary rooted tree datum datum could have leaves. 
The construction provided is more general than needed in our case. Proofs can be found in \cite{AlperinBass}.

Let $(E, \wedge)$ be a rooted tree datum. We construct a tree using $(E, \wedge)$. It is easy to verify, that if $x\wedge x=\infty$ for all $x\in E$ the construction gives a tree without leaves. 
We define 
$$\Sigma = \Sigma(E,\wedge)= \{(x,t)\in E\times \Lambda \;\;\vert\;\; 0\leq t\leq x\wedge x\}.$$

We define a distance on $\Sigma$ by setting 
$$
d((x,s), (y,t)) = 
\left\{\begin{array}{ll} 
	\vert s-t \vert & \text{ if } s \text{ or } t \leq x\wedge y\\
	\vert s- x\wedge y\vert + \vert t- x\wedge y \vert & \text{ if } s \text{ or } t \geq x\wedge y .
      \end{array}
\right. 
$$

Note, that if $s\leq x\wedge y\leq t$, or $t\leq x\wedge y\leq s$, then $\vert s-t \vert = \vert s- x\wedge y\vert + \vert t- x\wedge y \vert$ and the two cases coincide.

See \cite[p. 302]{AlperinBass} for a proof of the following lemma.
\begin{lemma}
The function $d$ defined above is a pseudo metric on $\Sigma$.
\end{lemma}

We introduce an equivalence relation on $\Sigma$.
\begin{definition}
Two points $(x,s)$ and $(t,y)$ in $\Sigma(E,\wedge)$ are called \emph{equivalent}, denoted by $(x,s)\sim (t,y)$, if $d((x,s),(t,y))=0$. Denote the equivalence class of $(x,s)$ by $\lb x,s \rb$.
\end{definition}

See \cite{AlperinBass} for the proof of the fact that $\sim$ is indeed an equivalence relation.

\begin{definition}
Let $T(E,\wedge)= \Sigma\diagup\sim$ and denote $\lb x,0\rb$ with $o_T$. Define a metric $d_T$ on $T$ by $d_T(\lb x,s\rb, \lb y,t\rb)\define d((x,s),(y,t))$. The space $T(E,\wedge)$ is called \emph{Alperin-Bass tree of $(E,\wedge)$}.
\end{definition}

As it was proven in \cite{AlperinBass} the space $T(E,\wedge)$ is a $\Lambda$-tree with $\Lambda$-valued metric $d_T$ in the sense of Definition \ref{Def_lambdatree}. 


\begin{thm}[Universal property]
Let $(E, \wedge)$ be a rooted tree datum and $(S,o_S)$ a rooted tree. Denote its set of ends by $F$ and let $\psi: E \mapsto F$ be a map satisfying $\psi(x)\wedge_S\psi(y) = x\wedge y$ for all $x,y\in E$. 
Then there exists a unique metric morphism $\hat{\psi}: T(E,\wedge) \rightarrow S$ such that $\hat{\psi}(o_T)=o_S$.
\end{thm}

Theorem~\ref{Thm_trees} is a generalization of a theorem by Tits for $\R$-trees saying that a metric tree with sap is uniquely determined by a projective valuation of its building at infinity.

\begin{definition}\label{Def_projval}
Let $E$ be a set and denote by $\Evier$ the set of ordered quadruples in $E$. A \emph{projective valuation $\val$} on $E$ is a map $\val: \Evier \longrightarrow \Lambda$ such that 
\begin{description}
\item[(PV1)] $\val(a,b;c,d)=\val(c,d;a,b)=-\val(a,b;d,c)$ 
\item[(PV2)] $\val(a,b;c,d)=k>0 \Rightarrow \val(a,d;c,b)=k \text{ and } \val(a,c;b,d)=0$ 
\item[(PV3)] $\val(a,b;d,e)+\val(b,c;d,e)=\val(a,c;d,e)$.
\end{description}
\end{definition}

\begin{definition}\label{Def_canval}
Associated to a tree $(T,\App)$ with sap there is its \emph{canonical valuation} $\val_T$, which is obtained as follows. Denote by $d$ the $\Lambda$-metric on $T$. Choose pairwise distinct $\App$-ends $a,b,c,d$ of $T$ and let $x=\kappa(a,b,c)$ and $y=\kappa(a,b,d)$. Then define 
$$\val_T(a,b;c,d)=
\left\{\begin{array}{ll} 
	d(x,y) & \text{ if } y\in \overrightarrow{x b}\\
	-d(x,y) & \text{ if } y\notin \overrightarrow{x b}.
      \end{array}
\right. 
$$
Bennet proved in \cite{BennettDiss} that $w_T$ is a projective valuation in the sense of Definition~\ref{Def_projval} on $\partial_\App$ of $T$. We call $w_T$ the \emph{canonical valuation of $T$}. 
\end{definition}

\begin{thm}{\cite[ Theorem 4.4]{BennettDiss}}\label{Thm_trees}
Given a $\Lambda$-valued projective valuation $\val$ on a set $E$. Then there exists a $\Lambda$-tree $T=T(\val, E)$ with sap and a one to one correspondence from $E$ to $\partial_\App T$ under which $\val$ corresponds to the canonical valuation $\val_T$ on $E$ arising from $T$.
\end{thm}

We included a detailed proof of the above theorem in Section~\ref{Sec_prooftrees}. The main idea is as follows. One uses $(E, \val)$ to construct a tree $T$. First a rooted tree datum is defined; the Alperin-Bass construction will then give a tree $T$. Finally we have to prove that the canonical projective valuation $\val_T$ arising from $T$ equals $\val$. The outline of the given proof is due to Bennett, \cite[Chapter 4]{BennettDiss}. We correct some minor inaccuracies and try to give a complete proof that is easy to read.  None of the proofs are literally the same as the ones given in \cite{BennettDiss}, but almost all are based on the calculations there. If not, we will mention it separately.
 
\begin{remark}
Theorem \ref{Thm_trees} and Bennett's proof of it in \cite{BennettDiss} was also discussed in Chapter 3 of the PhD thesis by Koudela \cite{Koudela}. She claims, that the theorem as stated in \cite{BennettDiss} is not true and that in general the correspondence between $E$ and the set of ends of $T$ is not one to one, but that the best that can be done is an embedding of $E$ in the ends of $T$. A remark concerning this task is made in \cite{ChiswellKoudela} as well. 

The theorem Koudela proves, making a few remarks on the proof of Bennett, is

\begin{itemize}
  \item[] \cite[Theorem 3.5]{Koudela}\\
Given a projective valuation $\omega$ on a set $E$, we can find a $\Lambda$-tree $T$ such that $E$ can be embedded in the set of ends of $T$ and $\omega$ is the projective valuation arising from $T$.
\end{itemize}

On the first glance it seems that we have mutually contradictory statements, however both theorems are correct. The solution is quite simple and is buried in the definition of an end of a tree. Bennett formulates Theorem \ref{Thm_trees} for trees with sap. The definition of an ``end'' of a tree, as made in \cite{BennettDiss}, corresponds to an $\App$-end in the sense of Definition~\ref{Def_end}. On the contrary ``ends'' in \cite{Koudela} or \cite{ChiswellKoudela} are precisely what we call an end. Therefore the difference between the two assertions is precisely owed to the fact that in general $\partial_\App T \subsetneq \{ \text{ends of }T\}$, compare Remark~\ref{Rem_ends}.
\end{remark}

\subsubsection*{Panel- and wall-trees}

Associated to an affine building there are two classes of trees encoding the panel- and wall structure of the affine building. 

Let $H$ and $H'$ be parallel hyperplanes. By Corollary 3.11 of \cite{Bennett} 
there exists a chart $f\in\App$, values $k,k'\in\Lambda$ and $\alpha\in\RS^+$ such that $f^{-1}(H)=H_{\alpha, k}$ and $f^{-1}(H')=H_{\alpha, k'}$. Use this fact to define a distance between hyperplanes and panels of Weyl chambers as follows:

\begin{definition}\label{Def_wall-panelDistance}
Given two hyperplanes $H$ and $H'$. With the above notation the \emph{distance} from  $H$ to $H'$ is defined to be $\vert k-k'\vert$.
For parallel panels $P$ and $P'$ choose parallel hyperplanes $H$ and $H'$ containing sub-panels of $P$,  respectively $P'$.   
Let the distance between $P$ and $P'$ be the distance from $H$ to $H'$.
\end{definition}

Note that the proof of the following proposition directly carries over to affine buildings modeled on a non-reduced root system $\RS$.

\begin{prop}{ \cite[Prop. 3.14]{Bennett}}\label{Prop_wallTree}
\index{wall tree}
Let $(X,\App)$ be an affine building of type $\MS(\Lambda, \RS)$ and dimension at least two. The set of hyperplanes in $X$ belonging to a given parallel class $m\define \partial H$ make up the points of an affine building $(T_{m},\App_{T_{m}})$ of dimension one (i.e. a $\Lambda$-tree with sap) with apartments in one to one correspondence with the apartments of $X$. Moreover the $\App_{T_{m}}$-ends of this tree are in one-to-one correspondence with the half-apartments of $\partial X$ having $m$ as boundary. 
These trees are called \emph{wall trees} of $\partial X$.
\end{prop}

A direct consequence of \cite{Bennett}[Cor 3.12] is the analog of \ref{Prop_wallTree} for panels.

\begin{corollary}{ \cite[Cor. 3.15]{Bennett}}
Let $(X,\App)$ be an affine building of type $\MS(\Lambda, \RS)$ and dimension at least two. The set of panels of Weyl chambers in $X$ belonging to a given parallel class $p\define \partial P$ make up the points of an affine building $(T_{p},\App_{T_{p}})$ of dimension one (i.e. a $\Lambda$-tree with sap) with apartments in one to one correspondence with the apartments of $X$. Moreover the $\App_{T_{p}}$-ends of this tree tree are in one-to-one correspondence with the chambers of $\partial X$ containing $p$. In analogy to the wall trees, we call them \emph{panel trees} of $\partial X$.
\end{corollary}

\begin{remark}
Note that the panel tree associated to the panel $\partial P$ is naturally isomorphic to all wall trees associated to hyperplanes $\partial H$ of $\partial X$ containing $\partial P$. By standard facts on opposition maps of spherical buildings there are at most two isomorphism types of panel/wall trees associated to a thick spherical building at infinity. There are two different type if ``taking the opposite in an apartment'' is not transitive on all panels of $\partial X$ and just one type otherwise. Compare \cite{Bennett}[Cor 3.17] for a proof of this fact.
\end{remark}

 \newpage
\section{Automorphisms}
\label{Sec_automorphisms}

The main result of this section is Theorem \ref{Thm_iso} which is a generalization of a well known result by Tits in \cite{TitsComo}. 

In the first subsection we explain the concept of bowties and prove the main result in the second. We hope that seeing a building as the collection of its equivalence classes of bowties will have other applications than the given one. It might for example be useful to prove a higher dimensional analog of Proposition~\ref{Prop_baseChangeTrees}.

\begin{remark}
By writing \emph{affine building} in the present section we always mean a generalized affine building in the sense of Definition~\ref{Def_LambdaBuilding}.
\end{remark}

\subsection{The space of bowties}\label{Sec_spaceofbowties}

The term \emph{bowtie} is due to Leeb \cite{Leeb} who refers to B. Kleiner for the main idea. However the basic idea of a bowtie was already used by Tits. In \cite{TitsComo}, Tits described points in an affine building using the building at infinity and points in the wall trees.
We simplified the definition of a bowtie to make it usable in the more general setting of generalized (not necessarily thick) affine buildings.

Let $\RS$ be a root system of rank $n$ and let the associated Coxeter complex be colored by $I=\{1, \ldots, n\}$. 

\begin{definition}\label{Def_bowtie}
\index{bowties}
\index{{bowties}!{space of}}
Let $(X, \mathcal{A})$ be an affine building modeled on $\MS(\RS,\Lambda)$, let $n$ be the rank of $\RS$ and denote by $\Delta=\binfinity X$ the building at infinity of $X$. Assume that the panels of $\Delta$ are consistently colored by $I=\{1,\ldots, n\}$. A \emph{bowtie} in $(X, \mathcal{A})$ is a triple $\tie \;= (c, \hat{c}, \{y_i\}_{i\in I})$ such that
\begin{enumerate}
\item \label{num9} $c$ and $\hat{c}$ are opposite chambers in $\Delta$
\item \label{num10} $y_i$ is a point in the panel tree $T_{p_i}$ of the $i$-panel $p_i$ of $c$.
\end{enumerate}
Denote by $a_\tie$ the unique apartment of $\Delta$ containing $c$ and $\hat{c}$ and by $A_\tie$ the associated affine apartment of $X$. The collection $\Tie$ of all bowties is called the \emph{space of bowties of $X$}.
\end{definition}

\begin{prop}\label{Prop_oek1}
A bowtie $\bowtie$ in an affine building $(X, \mathcal{A})$ determines a unique point $x_\tie$ in the unique apartment $A_\tie$ associated to $\bowtie$.
\end{prop}
\begin{proof}
Assume that $\bowtie = (c,\hat{c}, \{y_i\}_{i\in I})$ and let $(X,\App)$ be modeled on $\MS=\MS(\RS, \Lambda)$. For each $i\in I$ the point $y_i$ in the panel tree $T_{p_i}$ is an asymptote class of co-dimension one Weyl simplices. Let $m_i$ be the unique wall of $a_\tie$ containing $p_i$. The panel tree $T_{p_i}$ is canonically isomorphic to the wall tree  $T_{m_i}$ via a map $\psi_{p_i, m_i}$. Hence each representative of $y_i$ is contained in a unique hyperplane $H_i$ of the parallel class $m_i$.
Now fix a chart $f$ of $A_\tie$ and identify $A_\tie$ with the model space $\MS$ via $f$. The chamber $c$ determines a basis $B=\{\alpha_i, i\in I\}$ of $\RS$. Then there exist $k_i\in \Lambda$, for all $i\in I$, such that for all $i\in I$ the hyperplane $H_i$ equals $H_{\alpha_i,k_i}=\{x=\sum_{j=1}^n x_j\alpha_j \in A_\tie : \frac{(\alpha_i,\alpha_i)}{2} \lb x,\alpha_i^\vee\rb =k_i \}$. It remains to prove that $\bigcap_{i=0}^r H_i \neq \emptyset$, which is equivalent to finding a solution to the following system of equations
$$ 
k_i = \sum_{j=1}^n \frac{(\alpha_i,\alpha_i)}{2}\lb \alpha_j,\alpha_i^\vee\rb x_j \;\;\text{ for all } i\in I.
$$
But the system has a unique solution $x_\tie$ since the Cartan matrix $(\lb\alpha_i, \alpha_j^\vee \rb)_{ij}$ is invertible over $\Q(\{\alpha_i, \alpha_j^\vee\}_{i,j\in I})$ and hence over any subfield of $\R$ containing the set $\{\alpha_i, \alpha_j^\vee\}_{i,j\in I}$. 
\end{proof}

In the following an equivalence relation on the set of bowties is defined whose equivalence classes are, as we will prove in Proposition~\ref{Prop_oek5}, in one to one correspondence with the points of $X$. 

\begin{definition}\label{Def_adjacentbowties}
\index{{bowties}!{adjacent}}
Two bowties $\tie=(c, \hat{c},\{y_i\}_{i\in I})$ and $\tie'=(d, \hat{d},\{z_i\}_{i\in I})$ in an affine building $(X,\App)$ are called \emph{adjacent} if 
\begin{enumerate}
\item $c = d$
\item the intersection $\hat{c} \cap \hat{d}$ is a panel of both, and 
\item $y_i = z_i$ for all $i\in I$.
\end{enumerate}
\end{definition}

Figure~\ref{Fig_adj-bowties} gives an example of two adjacent bowties which might well remind you of the sundial configuration illustrated in Figure~\ref{Fig_sundial}.

 \begin{figure}[htbp]
 \begin{center}
 	\resizebox{!}{0.2\textheight}{\input{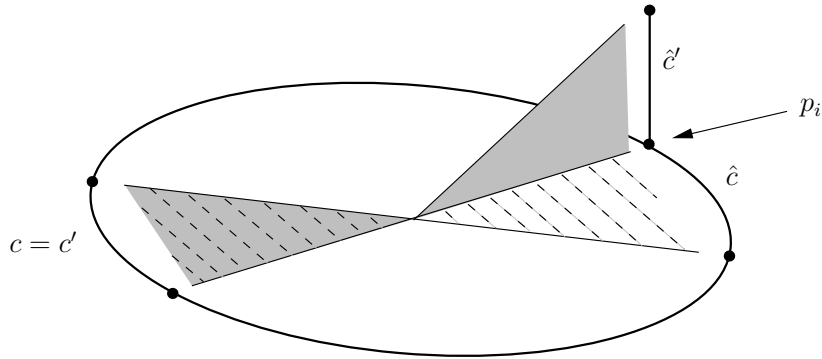}}
 	\caption[adjacent]{The bowties $\tie = (c, \hat{c}, \{y_i\}) $ and $\tie' = (c', \hat{c}', \{y_i'\})$ are adjacent.}
 	\label{Fig_adj-bowties}
 \end{center} 
 \end{figure}

\begin{observation}\label{Obs_W-action}
One can easily observe that there is a natural action of the spherical Weyl group on $\Tie$. Let $\tie$ be a bowtie. Recall, that for any (fixed) apartment containing $x_\tie$ the stabilizer of $x_\tie$ under the full affine Weyl group $\aW$ is isomorphic to the spherical Weyl group $\sW$. Hence given such an apartment $A$, the spherical Weyl group permutes the walls in $A$ containing $x_\tie$. Obviously $\sW$ also acts on the chambers, hyperplanes and panels in $\partial A$.

Let $\tie=(c, \hat{c}, \{y_i\}_{i\in I})$ be a bowtie. The points $y_i\in T_{p_i}$ determine walls $H_i$ in $\App_\tie$ and their images $w.H_i$ of some fixed $w\in \sW$ correspond to points $w.y_i$ in $T_{w.p_i}$. Define
$$
w.\tie = (w.c, w.\hat{c}, \{w.y_i\}_{i\in I}).
$$ 
Let $q_i$ be the unique panel of $\hat{c}$ opposite $p_i$ and denote by  $[p_i, q_i]:T_{p_i}\rightarrow T_{q_i}$ the \emph{perspectivity} map from the panel tree associated to $p_i$ to the tree associated to $q_i$. The obvious \emph{involution} $\iota$ of bowties is given by the multiplication with the longest element $w_0\in\sW$ 
$$
\iota(\tie)= w_0.\tie=(\hat{c}, c, \{[p_i, q_i](y_i)\}).
$$
\end{observation}

\begin{definition}\label{Def_bowtiepath}
\index{{bowties}!{equivalent}}
\index{{bowties}!{$\TTie$ }}
A \emph{path of bowties} is a sequence $(\tie_0, \tie_1, \ldots, \tie_n)$ such that $\tie_i$ and $\tie_{i-1}$ are either adjacent or $\tie_i=w.\tie_{i-1}$ for some $w\in \sW$. 
Two bowties $\tie$ and $\tie'$ are \emph{equivalent}, denoted by $\tie\sim\tie'$, if there exists a path of bowties connecting them. 
\end{definition}

\begin{lemma}\label{tec1}
The relation $\sim$ defined in \ref{Def_bowtiepath} is an equivalence relation.
\end{lemma}
\begin{proof}
Reflexivity is clear since $\sW$ contains the identity. Symmetry follows from reversing a connecting path and transitivity by concatenation of paths.
\end{proof}

\begin{prop}\label{Prop_oek5}
Two bowties $\tie$ and $\tie'$ in an affine building  $(X,\App)$ are equivalent if and only if they determine the same vertex, i.e.
$$\tie \sim \tie'\; \Longleftrightarrow \;x_\tie = x_{\tie'}.$$
\end{prop}

For the proof of Proposition~\ref{Prop_oek5} some technical Lemmas are needed. The first one is an observation for spherical buildings.

\begin{lemma}\label{Lem_tec9}
Given three chambers $c,d,e$ in a spherical building $\Delta$ such that both $d$ and $e$ are opposite $c$. Then there exist minimal galleries $\gamma=(d_0, d_1, \ldots, d_n)$ and $\sigma=(c_0, c_1,\ldots, c_n)$ such that $d_0=d$, $d_n=c=c_0$, $c_n=e$ and such that for all $i$ the chamber $d_i$ is opposite $c_i$.\\
Furthermore, if $f$ is adjacent to $c$ and contained in the apartment which is determined by $c$ and $e$ then $\sigma$ can be chosen such that $c_1=f$.
\end{lemma}
\begin{proof}
Let $A$ denote the unique apartment containing $c$ and $d$ and $B$ the one determined by $c$ and $e$. Let $r_{\partial A,c}$ denote the retraction\footnote{It is the analog in the spherical case of the retraction $r_{A,c}$ defined in \ref{Def_simplChamberRetraction}. } onto $\partial A$ based at the chamber $c$, which is an isomorphism restricted onto an apartment containing $c$. Since $c\in B$, the retraction $r_{\partial A,c}$ maps $B$ isomorphically onto $A$. Hence $r_{\partial A,c}(e)=d$. 
For any chamber $f\in B, f\sim c$ there exists a minimal gallery $\widetilde{\sigma}$ connecting $c$ and $d$, 
$$
\widetilde{\sigma}=(\widetilde{c}_0=c, \widetilde{c}_1=r_{\partial A,c}(f), \ldots, \widetilde{c}_n=d).
$$
Let $d_i$ be the opposite of $\widetilde{c}_i$ in $\partial A$. Then $d_0=d$ and $d_n=c$. Define
$$
\gamma=(d_0=d, d_1,\ldots,  d_n=c).
$$ 
The restriction of $r_{\partial A,c}$ to $B$ is an isomorphism of spherical apartments. Hence we can define
$$
c_i\define (r_{\partial A,c}\vert_b)^{-1}(\widetilde{c}_i) \;\text{ and }\; \sigma=(c_0, \ldots, c_n).
$$
By construction we have $c_0=c$ and $c_n=e$. The retraction $r_{\partial A,c}$ preserves distance to $c$ implying that $\sigma$ is a minimal gallery from $c$ to $e$. 

Assume there exists an index $i$ such that $d_i$ is not opposite $c_i$, meaning that $d(c_i, d_i)$ is strictly smaller than the maximal distance in the given spherical building $\Delta$. But
$$
d(d_i, c_i) \geq d(r_{\partial A,c}(d_i), r_{\partial A,c}(c_i))=d(d_i, \widetilde{c}_i)
$$
and $d_i$ and $\widetilde{c}_i$ are at maximal distance. This contradicts the assumption since $r_{\partial A,c}$ does not increase distances between chambers in $\Delta$. Hence $d_i$ is opposite $c_i$ for all $i$ and the assertion follows.
\end{proof}

\begin{lemma}\label{Lem_tec2}
For any bowtie $\tie=(c, \hat{c}, \{y_i\}_{i\in I})$ in $X$ and any chamber $d$ in $\binfinity X$ there exists a bowtie $\tie'=(d', d, \{z_i\}_{i\in I} )$ equivalent to $\tie$.
\end{lemma}
\begin{proof}
Assume $d$ is contained in $a_\tie$. Then the assertion follows by the action of $\sW$.

First assume $d\cap a_\tie$ is a panel of $d$. Without loss of generality we may further assume that $d$ and $\hat{c}$ are adjacent and that  $\Delta_{x_\tie}d$ and $\Delta_{x_\tie}c$ are opposite in $\Delta_{x_\tie}X$. Then $d$ and $c$ are opposite as well. Otherwise replace $\tie$ by an equivalent bowtie $w.\tie$ in $a_\tie$. Define $\tie'\define (c, d, \{y_i\}).$
According to Definition~\ref{Def_adjacentbowties} the bowties $\tie$ and $\tie'$ are adjacent and therefore equivalent. 

If $d$ and $a_\tie$ do not share a panel choose a gallery $\gamma=(d_0, d_1, \ldots, d_n)$ of minimal length such that $d_0\in a_\tie$ and $d_n=d$. There exists a bowtie $(c_1, \hat{c}_1, \{z_i\})=\tie_1\sim \tie$ with $\hat{c_1}=d_1$ using the same argument as in the first step. Inductively find bowties $\tie_{i+1}\sim\tie_i$ using the fact that $d_{i+1}$ and $a_{\tie_i}$ share a panel. This implies the existence of a path of bowties from $\tie$ to a bowtie $\tie'$ such that $\hat{c}'=d_n=d$. The assertion follows.
\end{proof}

\begin{lemma}\label{Lem_tec3}
Given a bowtie $\tie=(c, \hat{c}, \{y_i\}_{i\in I})$ and an apartment $A$ of $(X,\App)$ such that $x_\tie$ is contained in $A$ and such that $c$ is a chamber of $ \partial A$. Then $\tie$ is equivalent to a bowtie in $A$.
\end{lemma}
\begin{proof}
Let $a\define\partial A$. Without loss of generality we may assume that the chamber $c$ has a panel $p_i$ contained in the boundary of $a\cap a_\tie$, as illustrated in Figure~\ref{Fig_figure12}. Otherwise use the $\sW$-action on $\Tie$, as described in~\ref{Obs_W-action}, to replace $\tie$ by an equivalent bowtie in $A$. Let $c^{op}$ denote the chamber opposite $c$ in $a$ and let $f$ be the unique chamber in $a_\tie$ different from $c$ containing $p_i$. Let $l$ be the length of the shortest gallery from $\hat{c}$ to a chamber in $a$.
Denote by $q_i$ the panel of $c^{op}$ opposite the $i$-panel $p_i$ of $c$.

If $l=2$ then $q_i$, which is opposite $p_i$, is contained in $\hat{c}$, $c^{op}$ and $a$.
The bowties $\tie_0=(c, c^{op}, \{y_j\}_{j\in I})$ and $\tie$ are, by definition, adjacent and hence equivalent.

The induction step is as follows. Again assume without loss of generality that the situation is as described at the beginning of the proof. Abbreviate $x\define x_\tie$. Both apartments $A$ and $A_\tie$ contain $x$. By Proposition~\ref{Prop_epi} the images $\pi_x(a)$ and $\pi_x(a_\tie)$ under the natural epimorphism from $\Delta_\App$ to $\Delta_xX$ are, therefore, apartments in $\Delta_xX$. 
Lemma~\ref{Lem_tec9} implies the existence of galleries 
\begin{align*}
\dot{\gamma} = (\dot{c}_0=\pi_x(c^{op}), \dot{c}_1, \ldots, \dot{c}_n=\pi_x(c)) &\;\text{ in } \pi_x(a) \text{ and }\\
\dot{\sigma} = (\dot{d}_0=\pi_x(c), \dot{d}_1=\pi_x(f), \ldots, \dot{d}_n=\pi_x(\hat{c})) &\;\text{ in } \pi_x(a_\tie) 
\end{align*}
such that $\dot{c}_i$ is opposite $\dot{d}_i$ for all $i$. 
We can lift $\dot{\gamma}$ and $\dot{\sigma}$ to galleries 
\begin{align*}
\gamma =(c_0=c^{op}, c_1, \ldots, c_n=c) &\;\text{ in } a \;\text{ and } \\
\sigma =(d_0=c, d_1=f, \ldots, d_n=\hat{c}) &\;\text{ in } a_\tie  .
\end{align*}
The bowtie $\tie_{l}\define (c, c^{op}, \{y_j\})$ is equivalent to $w.\tie_{l}$ with $w\in \sW$ chosen such that $w.c=c_1$ and $c \cap w.c^{op} = c\cap f$. 

Let $a_1$ be the apartment spanned by $c_1$ and $d_1=f$. Denote the corresponding affine apartment by $A_1$. The affine apartment associated to $a_1$ contains $x$, since $\dot{d}_1=\pi_x(d_1)$ and $\dot{c}_1=\pi_x(c_1)$ are opposite in $\Delta_xX$. Hence there is a bowtie $\tie_{l-1}\define (c_1, d_1, \{ w.y_j\})$ contained in $a_1$ which is by definition adjacent to $\tie_l$. Substitute $a$ by $a_1$ and $\tie$ by an equivalent bowtie in $A_\tie$ such that $c=f=d_1$. We are again in the situation described at the beginning of the proof with the distance to the apartment reduced by one. By reverse induction $\tie$ is equivalent to a bowtie in $A$. 
\end{proof}

\begin{figure}[htbp]
\begin{center}
	\resizebox{!}{0.3\textheight}{\input{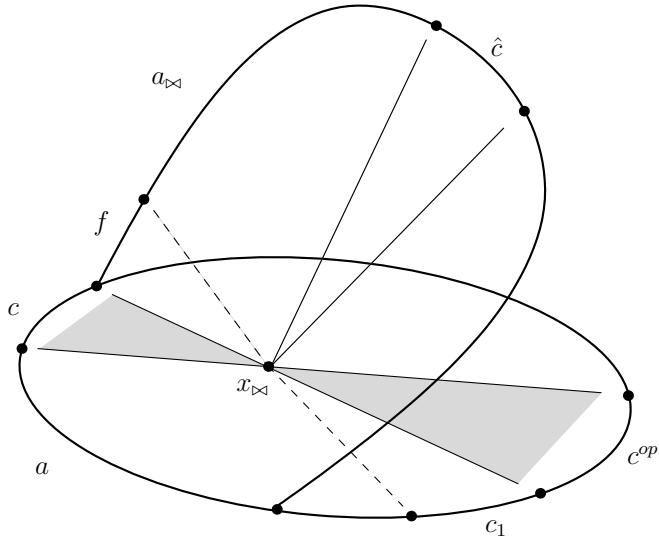}}
	\caption[Finding a path of equivalent bowties]{Let $\tie=(c,\hat{c},\{y_i\})$ be given and assume that $c$ has a panel in the boundary of $a\cap a_\tie$ and let $c^{op}$ be the chamber opposite $c$ in $a$. Shift the gray bowtie $(c, c^{op}, \{z_i\})$ to the (equivalent) one with chambers $f$ and $c_1$ which is ``closer'' to $\tie$. }
	\label{Fig_figure12}
\end{center}
\end{figure}

\begin{remark}
The main idea of the proof of Lemma~\ref{Lem_tec3} is as follows: Assume the situation is as described in Figure~\ref{Fig_figure12}. We are able to construct ``opposite'' galleries $\gamma:c'\rightsquigarrow c$ and $\sigma: c \rightsquigarrow \hat{c}$ such that each of the apartments $A_i$ determined by the opposite chambers $d_i$ and $c_i$ contains $x_\tie$. This enables us to ``shift'' $\tie_n\define (c^{op}, c, \{z_i^{op}\})$ along these galleries to $\tie$ while never leaving the equivalence class.
\end{remark}

\begin{proof}[Proof of Proposition~\ref{Prop_oek5}]
The implication $\tie\sim\tie' \; \Rightarrow \; x_\tie=x_{\tie'}$ is an easy consequence of the definition of equivalence of bowties.
To prove the converse assume that bowties $\tie_1=(c_1, \hat{c}_1, \{y_i\})$ and $\tie_2=(c_1, \hat{c}_1, \{z_i\})$ with the same basepoint $x$ are given. Let $a_1$ denote the apartment at infinity defined by $\tie_1$ and let $A_1$ denote the corresponding affine apartment. By Lemma~\ref{Lem_tec2} there exists a bowtie $\tie=(d,\hat{d}, \{x_i\})$ equivalent to $\tie_2$ such that $d= c_1$. But then $\tie$ is such that $c_\tie \in a_1$ and $x_\tie=x \in A_{1}$. Therefore, by Lemma~\ref{Lem_tec3}, $\tie$ is equivalent to $\tie_1$. Transitivity implies the equivalence of $\tie_1$ and $\tie_2$.
\end{proof}

\begin{remark}
These constructions heavily rely on the fact that for any vertex $x\in X$ and any chamber $c\in \Delta$ there exists a unique Weyl chamber $S\in c$ based at $x$, which we proved in \ref{Cor_WeylChamber}. 
\end{remark}

\subsection{Extending automorphisms}

The main result in this section is Theorem~\ref{Thm_iso} which says that an isomorphism from one building at infinity to another is induced by an isomorphism of the affine buildings if and only if it preserves certain additional data at infinity.
The analog result in the setting of $\R$-buildings was first proven by Tits \cite{TitsComo}. A detailed proof in the case of thick simplicial affine buildings can be found in \cite[Thm 12.3]{AffineW}. The simplicial case is also covered by \cite[Theorem 1.3]{Leeb}.

\begin{notation}\label{Not_eco}
Let $\RS$ be a root system in the sense of Definition~\ref{Def_rootsystem}. Let $n$ be the rank of $\RS$ and let the associated Coxeter complex be colored by the set $I=\{1, \ldots, n\}$. Assume that $F$ is a subfield of $\R$ containing the set of evaluations $\lb \beta,\alpha^\vee\rb$ for all pairs of roots $\alpha, \beta\in \RS$. Fix two ordered abelian groups $\Lambda$ and $\Gamma$ admitting an $F$-module structure, and assume that there exists an epimorphism $\mepi:\Lambda\to \Gamma$  of $F$-modules. Let further $(X_\Lambda,\App_\Lambda)$ and $(X_\Gamma,\App_\Gamma)$ be affine buildings in the sense of Definition \ref{Def_LambdaBuilding} modeled on $\MS(\RS, \Lambda)$, respectively $\MS(\RS, \Gamma)$. Denote by $\Delta_\Lambda$ and $\Delta_\Gamma$ the associated spherical buildings at infinity. 
\end{notation}

\begin{definition}\label{Def_ecological}
Let notation be as in ~\ref{Not_eco} and let $\tau:\Delta_\Lambda\rightarrow \Delta_\Gamma$ be an isomorphism. Then $\tau$ is \emph{ecological}\footnote{The colorful name \emph{ecological} was suggested by Richard M. Weiss in \cite{AffineW}. }
if for each wall $m$ and panel $p$ of $\Delta_\Lambda$ we have
$$ \val_{\tau(m)} \circ \tau = \mepi\circ\val_m 
\;\text{  and  } \;
\val_{\tau(p)} \circ \tau = \mepi\circ\val_p$$
where $\val_m$ and $\val_p$ are as defined in~\ref{Def_projval}.
\end{definition}

\begin{prop}\label{Prop_ind-iso}
Notation is as in \ref{Not_eco}. An ecological isomorphism $\tau: \Delta_\Lambda \longrightarrow \Delta_\Gamma$ induces maps
$\tau_m :T_m \longrightarrow T_{\tau(m)}$ and $\tau_p :T_p \longrightarrow T_{\tau(p)}$ for all walls $m$ and panels $p$ in $\Delta_\Lambda$ such that 
\begin{equation}\label{num42}
d_{T_{\tau(m)}}(\tau_m(x), \tau_m(y))=\mepi(d_{T_m}(x,y))
\end{equation}
for all $x,y\in T_m$, and 
\begin{equation}\label{num43}
d_{T_{\tau(p)}}(\tau_p(x), \tau_p(y))=\mepi(d_{T_p}(x,y))
\end{equation}
for all $x,y\in T_p$. In particular if $\Lambda=\Gamma$ and $e=id_\Lambda$ then $\tau_m$ and $\tau_p$ are isometries for all walls $m$ and panels $p$.
\end{prop}
\begin{proof}
Fix $m$ in $\Delta_\Lambda$. By Proposition~\ref{Prop_baseChangeTrees} there exists an $\Gamma$-tree $(T,d)$ and a map $\phi: T_m \rightarrow T$, unique up to an $\Gamma$-isometry, such that for all $x,y\in T$
$$
d(\phi(x),\phi(y))=e(d_{T_m}(x,y)).
$$
The set $\partial_{\App} T$ is, by definition, in canonical bijection to the set $\partial_{\App_m} T_m$. Identify $\partial_{\App} T$ and $\partial_{\App_m} T_m$ and denote this set with $E$. The canonical valuation $\val_T$ of pairwise different $a,b,c,d$ in $E$ is the signed distance between the two branching points $\kappa(a,b,c)$ and $\kappa(a,b,d)$. Compare Definition~\ref{Def_canval}. Since $\tau$ is an isomorphism, one can also identify $\partial_{\App_m} T_{\tau(m)}$ and $E$. Therefore
\begin{align*}
 \val_T(a,b,c,d) 
	& = \pm d(\kappa(a,b,c),\kappa(a,b,d)) \\
	& = \pm e(d_{T_m}(\kappa(a,b,c),\kappa(a,b,d))) \\
	& = \val_{\tau(m)}(a,b,c,d)
\end{align*}
for pairwise different $a,b,c,d$ in $E$. Theorem~\ref{Thm_trees} implies the assertion.
\end{proof}

It would be possible to construct the maps $\tau_m$, $\tau_p$ directly, yet using the base change functor the proof is much shorter. 

The following is the main result of this section. The two major consequences are stated in \ref{Thm_corollary2} and \ref{Thm_corollary1}.

\begin{thm}\label{Thm_iso}
Let notation be as in \ref{Not_eco} and let $\tau:\Delta_\Lambda\rightarrow\Delta_\Gamma$ be an ecological isomorphism of the buildings at infinity. Then there exists a unique surjective map $\rho:X_\Lambda\mapsto X_\Gamma$ mapping $\App_\Lambda$ to $\App_\Gamma$ such that 
\begin{equation}
\partial\rho(S)=\tau(\partial S)
\end{equation}
for all $\App_\Lambda$-Weyl-chambers $S$ of $X_\Lambda$ and such that 
\begin{equation}\label{num44}
d_\Gamma(\rho(x), \rho(y))=\mepi(d_\Lambda(x,y))
\end{equation}
for all points $x,y\in X_\Lambda$. 
\end{thm}

Note that $\rho$ is an isomorphism if $e$ is the identity. Hence

\begin{corollary}\label{Cor_iso}
Let $(X,\App)$ be an affine building. Assume $\tau$ is an ecological automorphisms of the building at infinity $\partial_\App X$. Then there exists a unique automorphism $\rho:X\mapsto X$ such that $\tau$ is induced by $\rho$.
\end{corollary}

\begin{prop}\label{Prop_bij}
In addition to the notation fixed in \ref{Not_eco} let $\Tie_\Lambda$ and $\Tie_\Gamma$ denote the space of bowties of $X_\Lambda$ and $X_\Gamma$, respectively. Let $\tau$ be an ecological isomorphism and let $\tau_p$ be as in Proposition~\ref{Prop_ind-iso}. Let $\rho :\Tie_\Lambda \rightarrow \Tie_\Gamma$ be defined as follows
$$
\rho: (c, \hat{c}, \{y_i\}) \longmapsto (\tau(c), \tau(\hat{c}), \{\tau_{p_i}(y_i)\}).
$$
Then the following hold
\begin{enumerate}
\item \label{num45} The map $\rho$ is surjective and preserves equivalence of bowties.
\item \label{num46} The preimage $\rho^{-1}(\tie )$ of a bowtie $\tie\in\Tie_\Gamma$ is parametrized by $(\mathrm{ker}(e))^n$.
\end{enumerate}
\end{prop}
\begin{proof}
Since $\tau$ and the $\tau_m$ are surjective $\rho$ is surjective as well.
Let bowties $\tie$ and $\tie'$ be given such that the associated apartments $a_\tie$ and $ a_{\tie'}$ are equal. Then $a_{\rho(\tie)}= a_{\rho(\tie')}$ since $\tau$ is an isomorphism and the apartments $a_{\tie}$, respectively $a_{\tie'}$, are uniquely determined by $c$ and $\hat{c}$, respectively $c'$ and $\hat{c}'$.
To prove \ref{num45} it is therefore enough to show that $\rho$ preserves adjacency.
Assume $\tie$ and $\tie'$ are adjacent bowties in $\Tie_\Lambda$. Then there exists an index $i$ such that $\hat{c} \cap \hat{c}'=\hat{p}_i$.
By assumption $c=c'$ and $y_j=y'_j$ for all $j$. 
The map $\tau$ is an isomorphism hence $\tau(c)=\tau(c')$ and 
$\tau(\hat{c})\cap\tau(\hat{c}')= \tau(\hat{p}_i)$ therefore $\rho(\tie)$ and $\rho(\tie')$ are adjacent.

It remains to prove \ref{num46}. The preimage under $\rho$ of a bowtie $\tie=(c, \hat{c}, \{y_i\})$ in $\Tie_\Gamma$ is the following set
$$
\{ (d,\hat{d}, \{x_i\}) \;:\; c=\tau(d), \hat{c}= \tau(\hat{d}), \tau_{\tau^{-1}(p_i)}(x_i)=y_i \text{ for all } i\in I\}.
$$
Further
$$
\rho^{-1}(\tie)
=\{(d,\hat{d}, \{x_i\}\}) : \tau_{q_i}(x_i)=y_i \text{ for all } i\in I\}.
$$
Therefore $\rho^{-1}(\tie)$ is parametrized by the pre-images of the $y_i$ in the panel tree. Each apartment of a panel tree is isomorphic to $\Lambda$, respectively $\Gamma$, and for each $i$ the restriction of $\tau_{q_i}$ to a fixed apartment equals $e:\Lambda\mapsto \Gamma$. Therefore we see that $\ker{e}$ is exactly $\{x_i : \tau_{q_i}(x_i)= y_i\}$ and the assertion follows.
\end{proof}


\begin{proof}[Proof of Theorem~\ref{Thm_iso}]
It is clear by definition of $\rho$ that $\partial (\rho(S)) = \tau(\partial S)$ for all Weyl chambers $S$ in $X_\Lambda$. Therefore $\tau$ is induced by $\rho$. It remains to prove (\ref{num44}).
Given $x,y\in X_\Lambda$. Let $A$ be an apartment of $X_\Lambda$ containing $x$ and $y$. Fix a chart and identify $A$ with the model space $\MS_\Lambda$ such that $y=0_\Lambda$ and $x\in\Cf^\Lambda$, where by $\Cf^\Lambda$ we mean the fundamental Weyl chamber in $\MS_\Lambda$. 
Let $A'$ be the the image $\rho(A)$ in $X_\Gamma$. Identify $A'$ with $\MS_\Gamma$ such that $0_\Gamma=\rho(0_\Lambda)$ and $\rho(x)\in\Cf^\Gamma$. 
Recall that by Proposition~\ref{Prop_ind-iso} an ecological isomorphism induces maps $\tau_m:T_m\mapsto T_{\tau(m)}$ such that (\ref{num42}) holds. 

Associated to $x$ there exists an $n$-tuple $(k_1, \ldots, k_n)\in \Lambda^n$ such that 
$$
x\in H_{i, k_i}=\{z :  \lb z, \alpha_i^\vee\rb = k_i\}
$$
for all $i=1,\ldots,n$. Let $m_i$ be the parallel class of $H_{i, k_i}$ and let $ T_{m_i}$ be the corresponding wall tree. Then $H_{i,k_i}$ determines a unique point $y_i$ in $T_{m_i}$. By definition of the distance function $d_{T_{m_i}}$ on $T_{m_i}$ and by definition of $x^i$ we have that
$$
d_{T_{m_i}}(y_i, 0_{m_i})=k_i= \lb x, \alpha_i^\vee \rb = 2x^i. 
$$
Hence the $k_i$ determine the coordinates $x^i$ uniquely. 

By Proposition~\ref{Prop_ind-iso} and the definition of $\rho$ in Proposition~\ref{Prop_bij} the following is true for all $i\in I$ 
$$
2(\rho(x))^i=d_{T_{\tau(m_i)}}(\tau_{m_i}(y_i), \tau_{m_i}(0_{m_i})=\mepi(d_{T_{m_i}}(y_i, 0_{m_i})).
$$
Proposition~\ref{Prop_distance0x} implies that for all $z\in\Cf^\Gamma$
$$
d_\Gamma(z,0)=\sum_{\alpha\in\RS^+}\sum_{i=1}^n b_i^\alpha z^i
$$
with the coefficients $b_i^\alpha$ defined by $\alpha=\sum_{i=1}^n b_i^\alpha\alpha_i$. 
Put $z=\rho(x)$. Using the equations above, we have
\begin{align*}
d_\Gamma(\rho(x),0)
	&=\sum_{\alpha\in\RS^+}\sum_{i=1}^n p_i^\alpha (\rho(x))^i \\
	&=\sum_{\alpha\in\RS^+}\sum_{i=1}^n p_i^\alpha \mepi(d_{T_{m_i}}(y_i, 0_{m_i})) \\
	&=\mepi\left(\sum_{\alpha\in\RS^+}\sum_{i=1}^n p_i^\alpha d_{T_{m_i}}(y_i, 0_{m_i})\right) =\mepi\left(d_\Lambda(x,0)\right).	
\end{align*}
In particular $d_\Gamma(\rho(x),0)=d_\Lambda(x,0)$ if $\Lambda=\Gamma$ and $e$ is the identity.
\end{proof}

Theorem~\ref{Thm_iso} can be reformulated as follows.

\begin{thm}\label{Thm_corollary2}
Let $\Aut(X,\App)$ be the subgroup of the automorphism group of an affine building $(X,\App)$ that leaves the apartment system $\App$ invariant. For each $\tau\in\Aut(X,\App)$ let $\rho_\tau$ denote the automorphism of $X$ induced by $\tau$. Then
$\tau\mapsto\rho_\tau$ is an isomorphism from the group of ecological automorphisms of $\partial_\App X$ to $\Aut(X,\App)$.
\end{thm}
\begin{proof}
An element of $\Aut(X,\App)$ preserves the tree structure at infinity and therefore induces an ecological automorphism of $\partial_\App X$. The converse holds by Theorem~\ref{Thm_iso}.
\end{proof}

Let $(X,\App)$ be an affine building and denote by $\Delta$ its building at infinity. Let $G^\dagger$ be the subgroup of the automorphism group $\Aut(\Delta)$ generated by the root groups of $\Delta$. Theorem \ref{Thm_corollary1}  is the analog of Theorem 12.31 in \cite{AffineW}. Due to Corollary~\ref{Cor_iso} the proof is literally the same as in \cite{AffineW}. 

\begin{thm}\label{Thm_corollary1}
Suppose the rank of $\Delta$ is at least two, then all elements of $G^\dagger$ are ecological and are therefore induced by a unique element of $\Aut(X,\App)$.
\end{thm}

 \newpage
\section{Buildings for the Ree and Suzuki groups}\label{Sec_ReeSuzuki}

In \cite{Octagons} we gave a geometric classification of affine buildings associated to the Ree and Suzuki groups. A partial result of that classification is Theorem~\ref{Thm_RDexistence}. The main purpose of this section is to give an algebraic proof of \ref{Thm_RDexistence}. Along the way we prove interesting formulas in \ref{Prop_key} that might be of independent interest.

\subsection{Root data and valuations}

This section is a review of the definitions of root data of spherical buildings and their valuations.

\begin{notation}
In the following let $\sW$ be the spherical Weyl group of an irreducible root system $\RS$ in the sense of Definition~\ref{Def_rootsystem} and let $S$ be the set of generators of $\sW$ determined by a fixed basis $B$ of $\RS$. If $\vert S\vert=2$ and $\sW$ is a dihedral group of order $2n$ for $n=5$ or $n>6$, let $\RS$ consist of $2n$ vectors evenly distributed around the unit circle in $\R^2$ and think of $S$ as the reflections along the hyperplanes determined by two vectors forming an angle of $\frac{(n-1)}{n}180$ degrees, i.e. $\RS$ is a root system of type $I_2(n)$. Denote by $V$ the ambient vector space of $\RS$. 
\end{notation}

In the following the Moufang property, defined in \ref{Def_Moufang}, will play an important role. Let us define the analog in the spherical rank one case.

\begin{definition}\label{Def_MoufangStructure}
\index{Moufang!Moufang structure}
Let $\Delta$ be a spherical building of dimension zero, i.e. of type $\RS=A_1$. Let $\Sigma$ be an apartment of $\Delta$ and identify the two vertices of $\Sigma$ with the two elements of $\RS$. A \emph{Moufang structure} on $\Delta$ is a set of nontrivial groups $(U_\alpha)_{\alpha\in \RS}$ satisfying the following two axioms
\begin{enumerate}
  \item For any $\alpha \in \RS$ the root group $U_\alpha$ fixes $\alpha$ and acts simply transitively on $\Delta\setminus\{\alpha\}$.
  \item For each $\alpha\in \RS$ the stabilizer of $\alpha$ in $G\define \langle U_\alpha, U_{-\alpha} \rangle $ normalizes $U_\alpha$.
\end{enumerate}
The conjugates $U_\alpha^g$, for $g\in G$ are called \emph{root groups}.
\end{definition}

\begin{remark}
Note that a Moufang structure is independent of the choice of the apartment $\Sigma$. The building $\Delta$ has to be thick, i.e. $\vert \Delta \vert \geq 3$, since we assumed the root groups to be non-trivial. Saying that a rank one building $\Delta$ is \emph{Moufang}, we mean that we have a particular (fixed) Moufang structure in mind.
\end{remark}

For the remainder of this section we fix the following notation

\begin{notation}\label{Not_RD}
Let $\Delta$ be an irreducible spherical building of type $\RS$ satisfying the Moufang condition as defined in \ref{Def_Moufang} respectively \ref{Def_MoufangStructure}. Denote by $(\sW,S)$ the Coxeter system associated to $\RS$ and by $V$ its ambient space. Fix an apartment $\Sigma$ of $\Delta$ and identify its roots with the elements of $\RS$. The root group associated to a root $\alpha$ is denoted by $U_\alpha$.
\end{notation}

\begin{prop}\label{Prop_tec1}
With notation as in \ref{Not_RD} let $\alpha$ be an element of $\RS$ and $-\alpha$ its opposite root in an apartment $\Sigma$. Then
\begin{enumerate}
  \item There exist maps $\lambda,\kappa: U^*_\alpha \rightarrow U^*_{-\alpha}$ such that for all $u\in U^*_\alpha$
	$$m_\Sigma(u)\define \kappa(u) u \lambda(u)$$
	fixes $\Sigma$ setwise and induces the unique reflection $s_\alpha$ on $\RS$, as defined in \ref{Def_reflection}.
  \item For all $u\in U_\alpha$ one has $m_\Sigma(u)^{-1}=m_\Sigma(u^{-1})$.
\end{enumerate}
Further let $\Gdagger$ be the subgroup of $\Aut(\Delta)$ generated by the root groups $U_\alpha$, $\alpha\in\RS$. Then $\Gdagger$ is transitive on the set of all pairs $(\Sigma,c)$ of apartments $\Sigma$ in $\Delta$ and chambers $c\in \Sigma$.
\end{prop}
\begin{proof}
 For assertions 1.-3. see \cite[6.1-6.3]{TW} for the second half of the proposition compare \cite[Proposition 11.12]{Weiss}.
\end{proof}

\begin{definition}\label{Def_rootDatum}
\index{root datum}
Let $\Delta$ be an irreducible Moufang spherical building of type $\RS$. A \emph{root datum} of $\Delta$ (based at $\Sigma$) is a pair $(\Sigma, \{U_\alpha\}_{\alpha\in\RS})$, where $\Sigma$ is an apartment of $\Delta$ and $\{U_\alpha\}_{\alpha\in\RS}$ is the set of corresponding root groups.
\end{definition}

One can prove that a root datum is, up to conjugation in $\Gdagger$, independent of the choice of an apartment $\Sigma$ and the identification of the roots of $\Sigma$ with the elements of $\RS$. 

\begin{definition}\label{Def_intervalRS}
Let $\RS$ and $(\sW,S)$ be as in \ref{Not_RD}. For roots $\alpha,\beta\in\RS$ with $\alpha\neq\pm \beta$ let the \emph{interval} $[\alpha,\beta]$ be the sequence $(\gamma_1, \ldots,\gamma_s)$ of roots $\gamma_i\in\RS$ such that
$$
\frac{\gamma_i}{(\gamma_i,\gamma_i)} = p_i \frac{\alpha}{(\alpha,\alpha)} + q_i \frac{\beta}{(\beta,\beta)}
$$
for some positive real numbers $p_i,q_i$ and such that
$$
\angle (\gamma_i,\alpha) < \angle(\gamma_j ,\alpha) \text{ if and only if } i < j.
$$
Define the open interval $(\alpha, \beta)$ to be the set $[\alpha, \beta]\setminus\{\alpha, \beta\}$.
\end{definition}

Note that $s$ depends on $\alpha$ and $\beta$ and that sometimes $s=0$ and $[\alpha, \beta]=\emptyset$.

\begin{definition}\label{Def_RDvaluation}
\index{{root datum}!{valuation of a }}
Let $(\Sigma,\{U_\alpha\}_{\alpha\in\RS})$ be a root datum of a spherical building $\Delta$. A collection $\varphi=(\varphi_\alpha)_{\alpha\in\RS}$ of maps $\varphi_\alpha:U^*_\alpha \rightarrow \R$ is a \emph{valuation of the root datum} if the following axioms are satisfied.
\begin{itemize}
\item[$(V0)$] $\vert\varphi_\alpha(U_\alpha)\vert\geq 3$ for all $\alpha\in\RS$
\item[$(V1)$] The set $U_{\alpha,k}\define\{u\in U_\alpha : \varphi_\alpha(u)\geq k\}$ is a subgroup of $U_\alpha$ for all $\alpha\in\RS$ and all $k\in\R$, where we assign $\varphi_\alpha(1)=\infty$ for all $\alpha$.
\item[$(V2)$] Given $\alpha\neq \pm \beta$, the commutator 
$$
\left[U_{\alpha,k}, U_{\beta,l}\right]\subset \prod_{\gamma\in(\alpha,\beta)}U_{\gamma, pk+ql}
$$
where $p,q$ and $(\alpha,\beta)$ are as in \ref{Def_intervalRS}.
\item[$(V3)$] Given $\alpha,\beta\in\RS$ and $u\in U_\alpha^*$, there exists $t\in\R$, such that for all $x\in U_\beta$
$$
\varphi_{s_\alpha(\beta)}\left(x^{m_\Sigma(u)}\right)=\varphi_\beta(x)+t.
$$
Moreover if $\alpha=\beta$ then $t=-2\varphi_\alpha(u)$.
\end{itemize}
Note that condition $(V2)$ is empty in the rank one case.
\end{definition}

\begin{definition}\label{Def_equipollent}
\index{equipollence}
Let $\varphi$ and $\varphi'$ be valuations of a root datum based at $\Sigma$. Then $\varphi$ and $\varphi'$ are \emph{equipollent} if there exists $v$ in the ambient space $V$ of $\RS$ such that for all $\alpha\in\RS$ and all $u\in U_\alpha$ 
$$
\varphi'_\alpha(u)=\varphi_\alpha(u) + (v, \alpha).
$$
We also write $\varphi'=\varphi + v$.
\end{definition}

\subsection{Ree and Suzuki groups}

In this section we collect well known facts about the three families of Ree and Suzuki groups using \cite{TitsOctagons, TitsSuzukiRee} and \cite{TW} as references.

\begin{notation}\label{Not_cases}
We will consider three cases $B,F$ and $G$. Let in all cases $K$ be a field of positive characteristic $p$ and let $\theta$ be a \emph{Tits endomorphism}, meaning that $\theta^2$ is the Frobenius $x\mapsto x^p$ on $K$. Therefore $F\define K^\theta$ is a subfield of $K$ and containing $K^p$.

In \textbf{case B} let $p=2$ and let $L$ be an additive subgroup of $K$ containing $F$ such that $LF\subset L$.
Hence $(K,L,L^\theta)$ is an \emph{indifferent set} as defined in \cite[10.1]{TW}. The induced building $\Btwo(K,L,L^\theta)$ is a Moufang quadrangle called $\Qd(K,L,L^\theta)$ in \cite[16.4]{TW}. Hence its type is $\mathrm{B}_2$.

In \textbf{case F} let again $p=2$. The pair $(K,F)$ is a composition algebra as defined in \cite[30.17]{TW}. This data determines a building $\Ffour(K,F)$ which is spherical and of type $\mathrm{F}_4$.

In \textbf{case G} let $p=3$. Then $(K/F)^\circ$ is a \emph{hexagonal system} in the sense of definition \cite[15.20]{TW}. Determined by this data there is a Moufang hexagon $\Gtwo((K/F)^\circ)$ (of type $\mathrm{G}_2$) which was called $\mathcal{H}_\mathcal{D}((K/F)^\circ)$ in \cite[16.8]{TW}.
From now on let $\Delta$ denote one of the buildings $\Btwo(K,L,L^\theta), \Ffour(K,F)$ and $\Gtwo((K/F)^\circ)$.
\end{notation}

Let $\RS$ be the type of $\Delta$. Denote by $V$ the ambient vector space of $\RS$ and fix a Weyl chamber $S$ in $V$. 
In all three cases there exists a unique nontrivial element $\tau\in\Aut(\RS)$ fixing $S$. Then $\tau$  induces an automorphism of the associated Coxeter complex which has order two and is non-type-preserving.
One can prove 

\begin{thm}\label{Thm_ambientBdg}
Fix an apartment $\Sigma$ of $\Delta$ and a chamber $c\in\Sigma$. Identify the roots of $\Sigma$ with the elements of $\RS$ such that $S$ corresponds to $c$. Then there exists for each $\alpha\in\RS$ an isomorphism $x_\alpha$ from the additive group $(L,+)$ to $U_\alpha$ in case B and from $(K,+)$ to $U_\alpha$ in cases $F$ and $G$ such that the following is true:
There exists a (non-type preserving) automorphism $\rho$ of $\Delta$ fixing $\Sigma$ and $C$ setwise such that
$$x_\alpha(t)^\rho=x_{\tau(\alpha)}(t)$$
for all $\alpha \in\RS$ and all $t\in L$, respectively in $K$.
\end{thm}
\begin{proof}
Compare \cite[Thm 5.3]{Octagons}
\end{proof}

The map $\rho$ of Theorem~\ref{Thm_ambientBdg}  is an involution of $\Delta$ switching the roots in the basis $B$ of $\RS$ which is determined by $C$ and switching the root groups associated to $B$ as well.

\begin{definition}\label{Def_SuzRee}
\index{Suzuki group}
\index{Ree group}
Denote by $\Delta^\rho$ the fixed point set of $\rho$ in $\Delta$. Let $\Gdagger$ be as in \ref{Prop_tec1} and let $G$ denote the group of automorphisms of $\Delta^\rho$ induced by the centralizer of $\rho$ in $\Gdagger$. In case B the group $\Sz(K,\theta, L)\define G$ is a \emph{Suzuki group}. In case G and F these are the \emph{Ree groups} denoted by $\Ree(K,\theta)$ and $\Fee(K,\theta)$, respectively.\footnote{The groups $\Fee(K,\theta)$ are nowadays more closely associated with Tits and more often called Ree-Tits groups.}
\end{definition}

The spherical buildings for the Ree and Suzuki groups are the fixed point sets of certain involutions of $\Delta$.

Let $\RS$ and $V$ be as in \ref{Thm_ambientBdg}. Denote by $\dot{V}$ the set of fixed points of $\tau$ in $V$. Let $\ddot{\alpha}\define \alpha+\alpha^\tau$ then $\ddot{\alpha}\in \dot{V}$ since $\tau^2$ is the identity. Define $\dot{\alpha}\define \frac{\ddot{\alpha}}{(\ddot{\alpha},\ddot{\alpha})}$ and 
$$
\dot{\RS} = \{\dot{\alpha} : \alpha\in\RS \}.
$$
Then $\dot{\RS}$ is the root system $A_1$ in case B and G, which is easy to see, and equals the root system $I_2(8)$ in case F, which was proved by Tits in \cite{TitsOctagons}.

\begin{prop}\label{Prop_fixedpoints}
The fixed point set $\Delta^\rho$ carries the structure of a spherical building of type $\dot{\RS}$ satisfying the Moufang condition. These are the \emph{Suzuki-Ree-buildings}.
\end{prop}
\begin{proof}
Compare Theorem 6.5 in \cite{Octagons}
\end{proof}

In case F the building $\Delta^\rho$ is a generalized octagon, in case B it is sometimes called Suzuki ovoid and Ree-Tits ovoid in case G.

In the following we will describe the structure of the root datum of the Suzuki-Ree buildings. Let us first fix some notation.

\begin{notation}\label{Not_universal}
Adopt notation for the cases B,F and G from \ref{Not_cases}. Let $\RS$, $V$ and $\Sigma$, $c$ as well as $\rho$ be as in \ref{Thm_ambientBdg}. Denote by $\dot{\Delta}$ the building of type $\dot{\RS}$ appearing in Proposition~\ref{Prop_fixedpoints} and let $\dot{\Sigma}$ be the set of fixed points under $\rho$ in $\Sigma$, which is isomorphic to the Coxeter complex $\Sigma(\dot{\RS})$. Identify the elements in $\dot{\RS}$ with the roots of $\dot{\Sigma }$. Let $\dot{c},\dot{V}$ and $\dot{S}$ be defined analogously. 
In general data belonging to the ambient spherical building $\Delta$ is denoted by letters without a dot and data belonging to the Ree- and Suzuki groups is denoted by the same letter with a dot. 
\end{notation}

\begin{property}{\bf Case F - Octagons}\label{Prop_rgCaseF}
Let $(K,\theta)$ be an octagonal set and let $\KK$ denote the group with underlying set $K\times K$ and multiplication 
$$
(s,t)*(u,v) = (s+u+ t^\theta v, t+v)
$$
for all $(s,t), (u,v)\in \KK$. Note that the inverse of an element $(s,t)\in\KK$ is given by $(s+t^{\theta+1}, t)$. 
The positive roots in $\dot{\RS}$ can be enumerated by $\{i=1,\ldots,8\}$ such that the root group $U_i$ is parametrized by the additive group of $K$ if $i$ is odd and by $\KK$ if $i$ is even. The parametrization maps are $t\mapsto x_i(t)$ if $i$ is odd and 
$(s,t)\mapsto x_i(s,t)$ if $i$ is even.
In the second case write $x_i(t)\define x_i(t,0)$ and $y_i(t)\define x_i(0,t)$  for arbitrary $t\in K$. We refer to these elements as \emph{monomials}. The building $\dot{\Delta}$ is uniquely determined by a list of commutators satisfied by the root groups $\dot{U}_\alpha$. Compare \cite[16.9]{TW}.
A \emph{norm} on $\KK$ is defined by
$$
R(s,t)=t^{\theta+2}+st+s^\theta 
$$ 
for all $(s,t)\in\KK$. Tits proved in \cite{TitsSuzukiRee} that $R$ is anisotropic, i.e. $R(s,t)=0$ only if $(s,t)=(0,0)$.
If $\alpha$ is a root whose root group is parametrized by $\KK$ then let $h_\alpha(s,t)$ denote the product $m_\Sigma(x_\alpha(1,0))m_\Sigma(x_\alpha(s,t))$.
\end{property}

\begin{property}{\bf Case B - Suzuki ovoids}\label{Prop_rgCaseB}
Let $K$ and $L$ be as in \ref{Not_cases}. Note, that as additive groups $L=K^\theta$. Define $\LL$ to be the set $L \times L$ together with the multiplication 
$$
(s,t)*(u,v) = (s+u+ t^\theta v, t+v)
$$
for all $(s,t), (u,v)\in \LL$. 
Note that, as in case F, the inverse of an element $(s,t)\in\LL$ is given by $(s+t^{\theta+1}, t)$. 
The root datum of the building $\Delta^\rho\ddefine \dot{\Delta}$ contains two root groups, namely  $U_+$ and $U_-$, both parametrized by $\LL$. There exist maps $x_\pm:\LL \rightarrow U_\pm$ such that $x_\pm(s,0)x_\pm(0,v)=x_\pm(s,v)$ and 
$$
x_\pm(s,t)x_\pm(u,v)=x_\pm\left(s+u+t^\theta v, t+v\right).
$$
In analogy to case F the \emph{norm} of $\LL$ is defined by
$$
R(s,t)=t^{\theta+2}+st+s^\theta 
$$ 
for all $(s,t)\in\LL$. By \cite{TitsSuzukiRee} the norm $R$ is anisotropic.
Let $h_\alpha(s,t)$ denote the product $m_\Sigma\left(x_\alpha(1,0)\right)m_\Sigma\left(x_\alpha(s,t)\right)$.
\end{property}

\begin{remark}\label{Rem_caseBF}
Let $\dot{\Delta}_F$ denote the building described in \ref{Prop_rgCaseF} and $\dot{\Delta}_B$ the one in \ref{Prop_rgCaseB}. Assume that in case B we have $K=L$. There exist rank two residues of $\dot{\Delta}_F$ which are fixed by $\rho$ and are isomorphic to $\dot{\Delta}_B$. One can identify the residue with $\dot{\Delta}_B$ such that the induced Moufang structure from $\dot{\Delta}_F$ coincides with the one on $\dot{\Delta}_B$.
\end{remark}

\begin{property}{\bf Case G - Ree Tits ovoids}\label{Prop_rgCaseG}
Let $K$ be as described for case G in \ref{Not_cases}. Define $\TT$ to be the set $K \times K \times K$ together with the multiplication 
$$
(r,s,t)*(w,u,v) = (r+w, s+u+r^\theta w, t+v-ru+sw-r^{\theta+1}w)
$$
for all $(r,s,t), (w,u,v)\in \TT$. Note that $(-r, -s + r^{\theta+1}, -t)$ is the inverse of $(r,s,t)$.
The root datum of a Ree Tits ovoid contains two root groups, namely  $U_+$ and $U_-$ which are both parametrized by $\TT$. Hence there exist isomorphisms $x_\pm:\TT \rightarrow U_\pm$ such that
$$
x_\pm(r,s,t)x_\pm(w,u,v)= x_\pm\left(r+w, s+u+r^\theta w, t+v-ru+sw-r^{\theta+1}w\right).
$$
A \emph{norm} of the group $\TT$ is defined by
$$
N(r,s,t)= r^{\theta+1}s^\theta - rt^\theta - r^{\theta+3}s-r^2s^2 + s^{\theta+1}+t^2-r^{2\theta+4} 
$$
for all $(r,s,t)\in\TT$. Note that $N$ is anisotropic by \cite{TitsSuzukiRee}, meaning $N(r,s,t)=0$ only if $(r,s,t)=(0,0,0)$.
Finally let $h_\alpha(r,s,t)$ be equal to $m_\Sigma(x_\alpha(0,0,1))m_\Sigma(x_\alpha(r,s,t))$.
\end{property}

\begin{lemma}\label{Lem_NUM19}
Consider the cases B and F as described in \ref{Prop_rgCaseB} and \ref{Prop_rgCaseF}. Then there exists a root $\alpha\in\RS$ such that $U_\alpha$ is parametrized by $\LL$, respectively $\KK$ and the following formula holds
\begin{equation}\label{Equ_no17}
x_\alpha(u,v)^{h_\alpha(s,t)} = x_\alpha\left(uR(s,t)^\theta, vR(s,t)^{2-\theta}\right)  
\end{equation}
for all $(s,t)\in \LL$ in case B, respectively in $\KK$ in case F. The root $\alpha$ is unique up to the choice of a basis of $\RS$.
Furthermore in case F there exists a second root $\beta$ such that $\alpha$ and $\beta$ form a basis of $\RS$ and such that 
\begin{equation}\label{Equ_no19}
x_\beta(u)^{m_\Sigma\left(x_\beta(t)\right)m_\Sigma\left(x_\beta(1)\right)} = x_\beta(t^{-2}u).
\end{equation}
\end{lemma}

\begin{remark}
In case F let the roots be enumerated as in \ref{Prop_rgCaseF} then $\alpha$ in the lemma above corresponds to $\alpha_8$ and $\beta$ to $\alpha_1$. In case B one can identify $\alpha$ with the positive root and $U_\alpha$ with $U_+$.
\end{remark}

\begin{proof}
We prove (\ref{Equ_no17}) using \cite[16.9]{TW} and \cite[32.13]{TW}. Let us first consider case F.
The commutator relations in \cite[16.9]{TW} are given for monomials $x_8(v)$ or $y_8(v)$ only. Hence let us first calculate $x_8(v)^{h_8(s,t)}$. The formulas in \cite[16.9]{TW} imply
$$
\left[x_2(u)^{h_8(s,t)} , x_8(v)^{h_8(s,t)}\right] = x_4\left(u^\theta v\right)^{h_8(s,t)} x_5(uv)^{h_8(s,t)} x_6\left(uv^\theta\right)^{h_8(s,t)}.
$$
Applying \cite[32.13]{TW} twice we get
\begin{align*}
x_2(u)^{h_8(s,t)} 		&= x_2(uR(s,t)^{-1})\\
x_4(u^\theta v)^{h_8(s,t)}	&= x_4(u^\theta v)\\
x_5(uv)^{h_8(s,t)} 		&= x_5(uv R(s,t)^{\theta-1})\\
x_6(uv^\theta)^{h_8(s,t)}	&= x_6(uv^\theta R(s,t)).
\end{align*}
Let $v'$ be defined implicitly by $x_8(v)^{h_8(s,t)}=x_8(v')$. Hence the commutator of $x_8(v')$ and $x_2(u')$, where $u'=uR(s,t)^{-1}$, determines $v'$ uniquely. By \cite[16.9]{TW} $(u')^\theta v'= u^\theta v$ hence $v'=vR(s,t)^\theta$ and therefore $$x_8(v)^{h_8(s,t)}=x_8(vR(s,t)^\theta).$$

Now calculate $y_8(v)^{h_8(s,t)}$. Using \cite[16.9]{TW} again
$$
\left[x_1(u), y_8(v)^{-1}\right]	= y_2(uv) \cdots x_7\left(uv^{\theta+2}\right).
$$ 
The element $y_8(v)$ is, by \cite[(6.4)(i)]{TW},
uniquely determined by $x_1(u)$ and $y_2(uv)$. To calculate the action of $h_8(s,t)$ it is therefore enough to look at the first term of the commutator.
Assuming that $(y_8(v)^{-1})^{h_8(s,t)} = y_8(v')^{-1}$ for some $v'$ the formulas in \cite[32.13]{TW} imply that 
$$
\left[x_1(u)^{h_8(s,t)} , \left(y_8(v)^{-1}\right)^{h_8(s,t)}\right]
			= y_2(uv)^{h_8(s,t)} \cdots   x_7\left(uv^{\theta+2}\right)^{h_8(s,t)}
$$
which is equivalent to
$$
\left[x_1(uR(s,t)^{-1}), y_8(v')^{-1}\right] = y_2\left(uvR(s,t)^{1-\theta}\right) \cdots x_7\left(uv^{\theta+2}\right)^{h_8(s,t)}.
$$
Since $uR(s,t)^{-1} v' = uvR(s,t)^{1-\theta}$ we have $v'=vR(s,t)^{2-\theta}$ and 
$$
\left(y_8(v)^{-1}\right)^{h_8(s,t)}=y_8\left(vR(s,t)^{2-\theta}\right)^{-1}.
$$
Since $v'$ is unique and since our choice satisfies the appropriate conditions we are done.

By~\ref{Prop_rgCaseF} we have $x_8(u,v)=x_8\left(u+v^{\theta+1}\right) y_8(v)^{-1}$. Using the above calculations we can conclude
\begin{align*}
x_8(u,v)^{h_8(s,t)} &= x_8\left(u+v^{\theta+1}\right)^{h_8(s,t)}\left(y_8(v)^{-1}\right)^{h_8(s,t)} \\
			&= x_8\left((u+v^{\theta+1})R(s,t)^\theta\right) y_8\left(vR(s,t)^{2-\theta}\right)^{-1}\\
			&= x_8\left(uR(s,t)^\theta,vR(s,t)^{2-\theta}\right).
\end{align*}
This finishes the proof in case F. By remark~\ref{Rem_caseBF} the formula holds in case B by restricting the parameters to $\LL$.
 
Consider the case F and let $\beta=\alpha_1$ with the enumeration of the roots as in \ref{Prop_rgCaseF}. The proof of (\ref{Equ_no19}) is as follows. The commutator has to satisfy 
$$
\left[x_1(u)^{m_\Sigma\left(x_1(1)\right)x_\Sigma\left(x_1(t)\right)}, x_6(v)^{m_\Sigma\left(x_1(1)\right)x_\Sigma\left(x_1(t)\right)}\right] = x_4(uv)^{m_\Sigma\left(x_1(1)\right)x_\Sigma\left(x_1(t)\right)}.
$$
This is by \cite[32.13]{TW} equivalent to
$$
\left[x_1(u)^{m_\Sigma(x_1(1))x_\Sigma\left(x_1(t)\right)}, x_6\left(t^{-1}v\right)\right] = x_4(tuv).
$$
Let $v'=t^{-1}v$ and $u'v'=tuv$. The value of $x_1(u)^{m_\Sigma(x_1(1))x_\Sigma(x_1(t))}$ is implicitly defined by the commutator of $x_1(v')$ and $x_6(u')$, which is $x_4(u'v')$ by \cite[32.13]{TW}. 
Therefore $x_1(u)^{m_\Sigma(x_1(1))x_\Sigma(x_1(t))} = x_1(u') = x_1\left(t^2u\right)$. Hence (\ref{Equ_no19}) holds.
\end{proof}

\begin{lemma}\label{Lem_tec21}
In case G there exists a root $\alpha\in\RS$, unique up to a choice of a basis of $\RS$, such that $U_\alpha$ is parametrized by $\TT$ and the following formula holds
$$
x_\alpha(w,u,v)^{h_\alpha(r,s,t)} = x_\alpha\left(w N(r,s,t)^{2-\theta}, u N(r,s,t)^{\theta-1}, v N(r,s,t)\right)
$$
for all $(r,s,t)\in \TT$.
\end{lemma}
\begin{proof}
  This is proved in chapter 6 of \cite{Octagons}.
\end{proof}

The following will be useful in the next subsection.

\begin{property}\label{Prop_oct1}
Assume in case F as described in \ref{Prop_rgCaseF}. Denote by $\dot{\RS}$ the roots of an apartment of $\dot{\Delta}$. Let $m_1$ denote $m_\Sigma(x_1(1))$ and write $m_8$ instead of $m_\Sigma(x_8(1,0))$. The group $N\define\langle m_1, m_8\rangle$ is a dihedral group of order 16 which acts by conjugation on the set of roots $\RS$. This action has two orbits, the roots with even index and the roots having odd index. The stabilizer of each root in $N$ is a group of order two centralizing the corresponding root group.
\end{property}

\begin{lemma}\label{Lem_oct2}
Let $\varphi$ be a valuation of the root datum of the building $\dot{\Delta}$ of case B. Let $\dot{\RS}$ be the set of roots in a fixed apartment $\dot{\Sigma}$ and let $N$ be as in \ref{Prop_oct1}. For all $g\in N$, all roots $\alpha\in\dot{\RS}$ and all $u\in U_\alpha^*$ the following holds:
\begin{enumerate}
\item $\varphi_\alpha(u)=\varphi_{\alpha^g}\left(u^g\right)$\label{NUM4}
\item $m_\Sigma(u)^g=m_\Sigma(u^g).$\label{NUM5}
\end{enumerate}
\end{lemma}
\begin{proof}
The action of $N$ on $\dot{\RS}$ is induced by the action of $N$ on $\dot{\Sigma}$. By definition we have $u^g\in U_\alpha^g$ if and only if $u\in U_\alpha$. The root groups are parametrized by maps $t\mapsto x_i(t)$ where $t\in K$ or $t\in\KK$ if $i$ is even or odd, respectively. By \cite[32.13]{TW} the parameter $t$ is not changed by elements of $N$. Since $\varphi_i(x_i(t))$ only depends on the parity of the index (and the parameter $t$) equation \ref{NUM4} follows.
For part \ref{NUM5} see \cite[(6.2)]{TW}.
\end{proof}

\subsection{Affine buildings for the Ree and Suzuki groups}

The main result of this section is Theorem~\ref{Thm_RDexistence} proving (under certain conditions) the existence of generalized affine buildings $X$ associated to each of the Ree and Suzuki groups. Together with \ref{Thm_RDuniqueness} this gives a classification of all generalized affine buildings having a Suzuki-Ree-building at infinity.

For the rest of this chapter fix the following notation.
\begin{notation}\label{Not_global}
Let $K, \Delta$, $\Sigma$, $\RS$, $V$ and $\rho$ (as well as the same letters tagged with a dot) be as in \ref{Not_universal}. The root $\alpha$ is as in \ref{Lem_NUM19} and \ref{Lem_tec21}. Assume that opposite root groups are parametrized such that $x_\alpha(s,t)^{m_\Sigma(1,0)}=x_{-\alpha}(s,t)$ in case B and F and $x_\alpha(r,s,t)^{m_\Sigma(0,0,1)}=x_{-\alpha}(r,s,t)$ in case G.
\end{notation}

\begin{definition}\label{Def_thetaInvariant}
\index{{valuation}!{$\theta$-invariant}}
Let $K$ be a field of positive characteristic $p$ endowed with the Tits endomorphism $\theta$, as in \ref{Not_cases}. A valuation $\nu$ of $K$ is \emph{$\theta$-invariant} if $\nu\left(x^\theta\right)=\sqrt{p}\;\nu(x)$ for all $x\in K^*$. 
\end{definition}

A proof of the following result can be found in \cite{Octagons}.
\begin{thm}\label{Thm_RDuniqueness}
Let notation be as in \ref{Not_global}. Assume that $\psi$ is a valuation of the root datum of $\dot{\Delta}$ based at $\dot{\Sigma}$, as defined in \ref{Def_RDvaluation}. Let $\varphi=\psi-\psi_\alpha(w)$ with $w=x_\alpha(1,0)$ in cases B and F and $w=x_\alpha(0,0,1)$ in case G. Then there exists a unique $\theta$-invariant valuation $\nu$ of $K$ such that $\varphi$ is uniquely determined by $\nu$, i.e. 
  \begin{equation}\label{Equ_no3}
	\varphi_\alpha\left(x_\alpha(s,t)\right)=\nu\left(R(s,t)\right)
  \end{equation}
	for all $(s,t)$ in $(\LL)^*$ in case B, respectively for all $(s,t)$ in $(\KK)^*$ in case F and 
  \begin{equation}\label{Equ_no4}  
	\varphi_\alpha\left(x_\alpha(r,s,t)\right)=\nu\left(N(r,s,t)\right)
  \end{equation}
	for all $(r,s,t)\in (\TT)^*$ in case G.
\end{thm}

Conversely each such valuation $\nu$ extends to a valuation of the root datum. 
\begin{thm}\label{Thm_RDexistence}
\index{{valuation}!{$\theta$-invariant}}
With notation as in \ref{Not_global} assume that $\nu$ is a $\theta$-invariant valuation of $K$ such that $\vert \nu(K)\vert \geq 3 $. Let $\varphi_\alpha: \dot{U}_\alpha \rightarrow \R$ be defined by equation~(\ref{Equ_no3}), in cases B and F, and by equation~(\ref{Equ_no4}) in case G. Then $\varphi_\alpha$ extends to a valuation $\varphi$ of the root datum of $\dot{\Delta}$ based at $\dot{\Sigma}$ and there exists a non-discrete affine building having $\dot{\Delta}$ as building at infinity.
\end{thm}

In \cite{Octagons} we obtained \ref{Thm_RDexistence} as a consequence of the following stronger result. Below we will give a direct (algebraic) proof of this theorem.

\begin{thm}\label{Thm_strongerResult}
Let the building $\Delta$ with polarity $\rho$ and $\dot{\Delta} = \Delta^\rho$ be as in \ref{Not_global}. Let $K$ be the defining field of $\Delta$.
Then for each $\theta$-invariant valuation $\nu$ of $K$ there exists a (nondiscrete) affine building $(X,\App)$ with $\Delta$ as boundary and there exists a unique automorphism $\widetilde{\rho}$ of $X$ inducing $\rho$ on $\Delta$. Furthermore there is a unique building $(\dot{X},\dot{\App})$ contained in the fixed point set of $\widetilde{\rho}$ in $X$ such that $\partial_{\dot{\App}} \dot{X} = \dot{\Delta}$. In addition $\Aut(X)$ contains a subgroup inducing a Suzuki or Ree group (depending on the case) on $\dot{\Delta}$.
\end{thm}

The main ingredient in the proof of \ref{Thm_RDexistence} is the following proposition.
In case G the proof is based on a suggestion by Theo Grundh\"ofer.

\begin{prop}\label{Prop_key}
Let notation be as in \ref{Not_global}. Then (in cases B and F)
\begin{equation}\label{Equ_no5}
  \nu\left(R(s,t)\right)= min\left\{\sqrt{2}\;\nu(s), (\sqrt{2}+2)\nu(t)\right\} \text{ and}
\end{equation}
\begin{equation}\label{Equ_no6}
  \nu\left(R\left(x_\alpha(s,t)\cdot x_\alpha(u,v)\right)\right)\geq \min\left\{\nu\left(R(s,t)\right), \nu\left(R(u,v)\right)\right\}.  
\end{equation}
And in case G
\begin{equation}\label{Equ_no7}
 \nu\left(N(r,s,t)\right)= \min\left\{\nu(r)(2\sqrt{3}+4), \nu(s)(\sqrt{3}+1),2\nu(t)\right\} \text{ and} 
\end{equation}
\begin{equation}\label{Equ_no8}
 \nu\left(N\left(x_\alpha(r,s,t)\cdot x_\alpha(w,u,v)\right)\right) \geq \min\left\{\nu\left(N(r,s,t)\right), \nu\left(N(w,u,v)\right)\right\} .
\end{equation}
\end{prop}

\begin{proof}
In case G the proof of the assertion can be found in appendix 9 of \cite{Octagons}. For the proof of cases B and F suppose first that 
$$
\nu(s) > (1+\sqrt{2})\nu(t).
$$
Then $\min\left\{\nu\left(s^\theta\right), \nu\left(t^{\theta+2}\right)\right\} = \nu\left(t^{\theta+2}\right)$ and hence
\begin{align}
\nu(st) 	&> (\sqrt{2}+2)\nu(t)=\nu\left(t^{\theta +2}\right) \label{NUM21} \text{ and}\\
\nu\left(s^\theta\right) 	&= \sqrt{2}\nu(s) > (\sqrt{2}+2)\nu(t) = \nu\left(t^{\theta +2}\right). \label{NUM22}
\end{align}
Since for any $a,b$ and any valuation $\nu(a+b)\geq \min\left\{\nu(a),\nu(b)\right\}$ is satisfied, we can, together with inequalities (\ref{NUM21}) and (\ref{NUM22}), conclude that
$$
 \nu\left(s^\theta+st\right) > \nu\left(t^{\theta +2}\right).
$$
And hence 
$$
\nu\left(R(s,t)\right)=\nu\left(t^{\theta+2} + st + s^\theta\right) \geq  \min\left\{\nu\left(t^{\theta+2}\right),\nu\left(st + s^\theta\right)\right \}= \nu\left(t^{\theta +2}\right).
$$

Secondly suppose $\nu(s)\leq (1+\sqrt{2})\nu(t)$. Therefore $\min\left\{\nu\left(s^\theta\right), \nu\left(t^{\theta+2}\right)\right\}= \nu\left(s^\theta\right)$. In case that $\nu\left(R(s,t)\right)$ is strictly bigger than $\nu\left(s^\theta\right)$ we have
\begin{align*}
\nu\left(tR^\theta(s,t)\right) 	&= \nu(t) +\sqrt{2}\;\nu\left(R(s,t)\right) \geq \frac{1}{1+\sqrt{2}}\nu(s) + \sqrt{2}\;\nu\left(R(s,t)\right)\\
		&= (\sqrt{2}-1)\nu(s) + \sqrt{2}\;\nu\left(R(s,t)\right) \\
		&> (\sqrt{2}-1)\nu(s) + \sqrt{2}\;\nu\left(s^\theta\right) = (\sqrt{2} +1)\nu(s).
\end{align*}
On the other hand, using the assumption $\nu(s)\leq (1+\sqrt{2})\nu(t)$, we can conclude
\begin{align*}
\nu\left((t^{\theta +1}+s)R(s,t)\right) 
	& = \nu\left(R(s,t)\right)+\nu\left(t^{\theta+1}+s\right) \\
	& \geq \nu\left(R(s,t)\right)+\nu(s) 
	 > (\sqrt{2}+1)\nu(s).
\end{align*}
And hence 
\begin{align}
\nu\left(tR^\theta(s,t)\right) 	&>(\sqrt{2} +1)\nu(s) \label{NUM23}\\
\nu\left(\left(t^{\theta +1}+s\right)R(s,t)\right) &> (\sqrt{2}+1)\nu(s). \label{NUM24}
\end{align}
Using the fact that $s^{\theta +1}= tR^\theta(s,t) + (t^{\theta +1} + s)R(s,t)$ and the inequalities (\ref{NUM23}) and (\ref{NUM24}) we have
\begin{align*}
\nu\left(s^{\theta+1}\right) &\geq \min \left\{ \nu\left(tR^\theta(s,t)\right), \nu\left((t^{\theta+1} + s)R(s,t)\right) \right\}\\
		& > (\sqrt{2} +1)\nu(s)
\end{align*}
which is a contradiction. Therefore 
$\nu\left(R(s,t)\right) = \nu\left(s^\theta\right)$ and (\ref{Equ_no5}) holds.

To prove (\ref{Equ_no6}) let $y\define x_\alpha(s,t)$ and $z=x_\alpha(u,v)$. Then $y\cdot z=x_\alpha\left(s+u+t^\theta v, t+v\right)$ and, using (\ref{Equ_no5}), therefore
\begin{align*}
\nu\left(R(y\cdot z)\right) &= \min \left\{ \nu\left((s+u+t^\theta v)^\theta\right), \nu\left((t+v)^{\theta+2}\right) \right\} \\
	&= \min \left\{ \nu\left (s^\theta+u^\theta +t^2v^\theta \right), \nu\left( t^{\theta+2}+v^{\theta+2}+t^\theta v^2+t^2v^\theta \right) \right\} \\
	&\geq \min \left\{ \nu\left(s^\theta\right), \nu\left(u^\theta\right), \nu\left(t^2v^\theta\right), \nu\left(v^2t^\theta\right), \nu\left(t^{\theta+2}\right), \nu\left(v^{\theta+2}\right) \right\} 
\end{align*}
Simple calculations imply that both expressions $\nu\left(t^2 v^\theta\right)$ and $\nu\left(v^2 t^\theta\right)$ are greater or equal than $\min \left\{ \nu\left(t^{\theta+2}\right), \nu\left(v^{\theta+2}\right) \right\} $.
Hence 
$$
\nu\left(R(y\cdot z)\right)\geq \min \left\{\nu(s^\theta), \nu(t^{\theta+2}), \nu(u^\theta), \nu(v^{\theta+2}) \right\}=\min \left\{\nu\left(R(s,t)\right), \nu\left(R(u,v)\right)\right\}
$$
and equation (\ref{Equ_no6}) follows.
\end{proof}

In case F we need the following additional Lemma.
\begin{lemma}\label{Lem_opRg}
In case F let $\varphi$ be a valuation of a root datum of $\dot{\Delta}$. Let the roots be enumerated as in \ref{Prop_rgCaseF} and let the root $\alpha$ of \ref{Not_global}, be identified with $\alpha_8$. Assume that $\varphi_\alpha(x_\alpha(s,t))=\nu(R(s,t))$. Then 
$$
\varphi_{\alpha_i}\left(x_{\alpha_i}(k)\right) = \sqrt{2+\sqrt{2}}\; \nu(k)
$$
for odd $i$ and all $k\in K$.
\end{lemma}
\begin{proof}
Denote $x_{\alpha_i}$ by $x_i$. Using $s_1(\alpha_2)=\alpha_8$ and the formula in the proof of (10.21) in \cite{Ronan} we have
$$
\varphi_8 \left( x_2(s,t)^{m_\Sigma\left(x_{\alpha_1}(k)\right)}\right) 
	= \varphi_2\left(x_{\alpha_2}(s,t)\right)+2\varphi_1\left(x_{\alpha_1}(k)\right)\cdot(\alpha_1,\alpha_8).
$$
By \ref{Ex_I2(8)} the scalar product $(\alpha_1, \alpha_8)$ equals $-\frac{1}{2}\sqrt{2+\sqrt{2}}$. Equation 32.13 of \cite{TW} implies
$$
\varphi_8\left(x_{\alpha_8}\left(k^{-1-\theta}s, k^{-1}t\right)\right)
	= \varphi_2\left(x_{\alpha_2}(s,t)\right)  - \sqrt{2+\sqrt{2}}\cdot \varphi_1\left(x_{\alpha_1}(k)\right).
$$
With $s=0$ and $t=1$ the previous equation reads
$$
\varphi_8\left(x_{\alpha_8}(0, k^{-1})\right)
	= \varphi_2\left(x_{\alpha_2}(0,1)\right) - \sqrt{2+\sqrt{2}}\;\varphi_1\left(x_{\alpha_1}(k)\right). 
$$ 
Defining $c\define \varphi_2\left(x_{\alpha_2}(0,1)\right)$ we can conclude
$$
 \varphi_1\left(x_{\alpha_1}(k)\right) 
	= \frac{-1}{\sqrt{2+\sqrt{2}}}\; \varphi_8\left(x_{\alpha_8}(0,k^{-1})\right) +c 
	= \sqrt{2+\sqrt{2}}\; \nu(k) .
$$
Using \ref{Prop_oct1} the assertion follows.
\end{proof}

\subsection{Proof of the main result}

We verify the axioms of Definition~\ref{Def_RDvaluation}. 

\begin{proof}[Proof of Theorem \ref{Thm_RDexistence}]
Axiom $(V0)$ is a direct consequence of the assumption that $\vert\nu(K)\vert\geq 3$. 
We prove $(V1)$.
In case F let $\alpha$ be a root whose root group is parametrized by the additive group of $K$. 
The valuation of the product of two elements $x_\alpha(s)$ and $ x_\alpha(t)$ of $U_{\alpha,k}$ is calculated using \ref{Lem_opRg}:
\begin{align*}
\varphi_\alpha\left(x_\alpha(s)\cdot x_\alpha(t)\right) 
	&= \varphi_\alpha\left(x_\alpha(s+t)\right) 
	 = \sqrt{2+\sqrt{2}} \;\nu(s+t) \\
	&= \sqrt{2+\sqrt{2}} \;\min\left\{\nu(s), \nu(t)\right\} \geq k\\
\end{align*}
Hence in this case $(V1)$ holds. 
Consider the cases F and B with $\alpha$ as in \ref{Not_global}. Let $x_\alpha(s,t), x_\alpha(u,v)\in U_{\alpha,k}$ be given. Then, by definition of $U_{\alpha,k}$, the values $\varphi_\alpha\left(x_\alpha(s,t)\right)$ and $\varphi_\alpha\left(x_\alpha(s,t)\right)$ are greater or equal than $k$. Let $y=x_\alpha(s,t)$ and $z=x_\alpha(u,v)$. Equation (\ref{Equ_no6}) implies
$$
\varphi_\alpha\left(y\cdot z\right)=\nu\left(R(y\cdot z)\right)\geq \min\left\{\varphi_\alpha(y), \varphi_\alpha(z)\right\}.
$$
Therefore $(V1)$ holds in case F (for all root groups parametrized by $\KK$) and is true in case B, where all root groups are parametrized by $\LL$. 
In case G let $y=x_\alpha(r,s,t)$ and $z= x_\alpha(w,u,v)$ in $U_{\alpha,k}$ be given. Then, by equation (\ref{Equ_no8}), we have
$$
\varphi_\alpha(y\cdot z) \geq \min\{\varphi_\alpha(y), \varphi_\alpha(z)\} .
$$
Hence $(V1)$ holds in case G.

Condition $(V2)$ is empty in rank one, which are the cases B and G. To prove the assertion in case F one has to verify property $(V2)$ for all pairs of elements whose commutator relations are in the list of \cite[16.9]{TW} first. These, the multiplication of $\KK$ and the action of $N$, as described in \ref{Prop_oct1} imply axiom $(V2)$ for arbitrary root group elements. Checking $(V2)$ for all pairs with commutator relation in \cite[16.9]{TW} forces a lot of calculations which are all of the same kind: Given roots $\alpha,\beta$ one has to determine $p_\gamma,q_\gamma$ for all $\gamma\in (\alpha,\beta)$. Since the elements describing the commutators already appear in the correct order (see \cite[16.9]{TW}) it remains to check that they are contained in $U_{\gamma,p_\gamma k+q_\gamma l}$, i.e. one has to calculate the value under the valuation $\varphi_\gamma$ and compare it with the values of the chosen elements in $U_{\alpha,k}$ and $U_{\beta,l}$. 
We will not write out all the calculations. Hopefully the given example will enable the interested reader to do the missing calculations himself.

Let $\alpha=\alpha_1$ and $\beta=\alpha_4$ and let $k\define \varphi_1(x_1(t))$ and $l\define \varphi_4(x_4(0,u))$. By definition we have
\begin{equation}\label{NUM7}
\varphi_1\left(x_1(t)\right) =\sqrt{2+\sqrt{2}}\; \nu(t)  
\end{equation}
and
\begin{equation}\label{NUM8}
\varphi_4\left(x_4(0,u)\right) = \nu(R(0,u)) = \sqrt{2+\sqrt{2}}\;\nu(u). 
\end{equation}
By \cite[16.9]{TW} we have
\begin{equation}\label{NUM9}
[x_1(t),x_4(0,u)]=x_2(tu)   
\end{equation}
With $p_{\alpha_2} = \frac{\sqrt{2}}{\sqrt{2+\sqrt{2}}}$ and $q_{\alpha_2}=\frac{\sqrt{2}}{2+\sqrt{2+\sqrt{2}}}$ one can check that $\varphi_2\left(x_2(tu)\right)=p_{\alpha_2} k + q_{\alpha_2} l$.

Assume now, that $(V2)$ holds for all pairs of roots where the commutator is contained in~\cite[16.9]{TW}. Following the arguments in (16.9) in \cite{TW} and applying the action of $N$ to this list, axiom $(V2)$ holds for all pairs of arbitrary elements in odd and all monomials in even root groups. They also hold for elements of even root groups having the form $x_i(u,v)$ where neither $u$ nor $v$ equals $0$. It is enough to prove the cases $[x_1(t),x_8(u,v)]$ and $[x_8(u,v),x_1(t)]$. All other missing commutator relations involving non-monomial elements of even root groups can be deduced from these using the action of $N$.

By definition of the multiplication in $\KK$ we have
$$
x_8(u,v)=y_8(v)^{-1}x_8\left(u+v^{\sigma+2}\right).
$$ 
Let $u'=u+v^{\sigma+2}$. 
Since $[a,cb]=[a,b][a,c]^b$ for arbitrary $a,b$ we have 
$$
\left[x_1(t), y_8(v)^{-1}x_8(u')\right]=\left[x_1(t),x_8(u')\right]\left[x_1(t),y_8(v)^{-1}\right]^{x_8(u')}.
$$
It is easy to see that if $\varphi_8\left(x_8(u,v)\right)\geq k$ then also $\varphi_8\left(x_8(u')\right)$ and $\varphi_8\left(y_8(v)^{-1}\right)$ are greater or equal than $k$. By \cite[16.9]{TW} we can write $\left[x_1(t),x_8(u')\right]$ and $\left[x_1(t),y_8(v)^{-1}\right]$ as 
$$
\left[x_1(t),x_8(u')\right]= w_2\ldots w_7 
\hspace{3ex}\text{and}\hspace{3ex}
\left[x_1(t),y_8(v)^{-1}\right]= z_2\ldots z_7 
$$  
where the $w_i$ and $z_i$ are monomials. Therefore the expression
\begin{align*}
\left[x_1(t), x_8(u,v)\right] &=\left[x_1(t),x_8(u')\right]x_8(u')^{-1}\left[x_1(t),y_8(v)^{-1}\right]x_8(u')\\
		&= w_2\ldots w_7 x_8(u')^{-1} z_2\ldots z_7 x_8(u')	
\end{align*}
can be re-sorted using the commutator relations. Hence 
$$
\left[x_1(t), x_8(u,v)\right]=w'_2 \ldots w'_7  x_8(u')^{-1}x_8(u')
$$ 
and $(V2)$ holds. The proof for $\left[x_8(u,v),x_1(t)\right]$ is analogous to the one just finished.

First prove \emph{$(V3)$} in case B. Assume that $\alpha = \beta$ are equal to the positive root, then the positive and negative root are switched by the reflection $s_\alpha=s_\beta$. Let $x_+(s,t)$ and $x\define x_+(u,v)$ be elements of $U_+$. By Lemma~\ref{Lem_oct2} and the formula for the inverse of an element $x_+(s,t)$ observe that
\begin{align*}
\varphi_-\left(x^{m_\Sigma(s,t)}\right) - \varphi_+(x) 
	&= \varphi_+\left(x^{m_\Sigma(s,t)m_\Sigma(1,0)^{-1}}\right) - \varphi_+(x) \\
	&= \varphi_+\left(x^{\left( m_\Sigma(1,0)m_\Sigma(s, t+s^{\theta+1})\right)^{-1} }\right) - \varphi_+(x)\\
	&= \varphi_+\left(x^{\left( h_+(s, t+s^{\theta+1})\right)^{-1} }\right) - \varphi_+(x).
\end{align*}
Lemma (\ref{Lem_tec21}) implies that
$$
\varphi_+\left(x_+(a,b)^{ h(c,d)^{-1} }\right) = \varphi_+\left(x_+\left(aR(c,d)^{-\theta}, b R(c,d)^{\theta-2}\right)\right)
$$ 
for all $(a,b)$ and $(c,d)$ in $\LL$. 
Hence
\begin{align*}
\varphi_-\left(x^{m_\Sigma(s,t)}\right) - \varphi_+(x) 
	&= \varphi_+\left(x^{\left( m_\Sigma(1,0)m_\Sigma(s+t^{\theta+1}, t)\right)^{-1} }\right) - \varphi_+(x).\\
	&= \varphi_+\left(x_+\left(u R\left(s+t^{\theta+1}, t\right)^{-\theta}, vR\left(s+t^{\theta+1}, t\right)^{\theta-2}\right)\right) - \varphi_+(x)\\
	&= \nu\left( R\left(u R\left(s+t^{\theta+1}, t\right)^{-\theta}, vR\left(s+t^{\theta+1}, t\right)^{\theta-2}\right) \right) - \nu\left(R(u,v)\right).
\end{align*}
Simple calculations imply that $R\left(s+t^{\theta+1}, t\right)=R(s,t)$ and therefore
$$
 \varphi_-\left(x^{m_\Sigma(s,t)}\right) - \varphi_+(x) = -2\nu\left(R(s,t)\right)=-2\varphi_+\left(x_+(s,t)\right).
$$
The calculations for $U_-$ are completely analogous to this case. Therefore axiom $(V3)$ is true in case B. 

To prove $(V3)$ in case F 
we have to verify that for each pair of roots $\alpha,\beta$ and $u\in U_\alpha^*$ there exists $t\in\R$, independent of $x\in U_\beta$, such that 
\begin{equation}\label{NUM10}
\varphi_{s_\alpha(\beta)}(x^{m_\Sigma(u)})=\varphi_\beta(x)+t.
\end{equation}
Assume (\ref{NUM10}) holds for all pairs $(\alpha,\beta)$ such that either $\alpha=\alpha_1$ and $\beta\in \{\alpha_2,\ldots, \alpha_8\}$ or $\alpha=\alpha_8$ and $\beta\in \{\alpha_1,\ldots, \alpha_7\}$. Using the action of $N$, as described in \ref{Prop_oct1} equation (\ref{NUM10}) holds for arbitrary pairs: given $(\alpha,\beta)=(\alpha_i,\alpha_j)$ there exists a unique $g\in N$ such that either $\alpha^g=\alpha_1$ and $\beta^g\in \{\alpha_2,\ldots, \alpha_8\}$ or $\alpha^g=\alpha_8$ and $\beta^g\in \{\alpha_1,\ldots, \alpha_7\}$. The assertion follows by Lemma~\ref{Lem_oct2}.

It remains to calculate all the cases where $\alpha=\alpha_1$ and $\beta\in \{\alpha_2,\ldots, \alpha_8\}$ or $\alpha=\alpha_8$ and $\beta\in \{\alpha_1,\ldots, \alpha_7\}$ and the cases $\alpha=\beta\in \{a_1,a_8\}$. 

Let $\alpha=\alpha_1, \beta=\alpha_8$. For arbitrary $x=x_8(u,v)\in U_8$ and $x_1(k)\in U_1^*$, condition (\ref{NUM10}) holds by \cite[32.13]{TW} and the following calculation
\begin{align*}
\varphi_2 & \left(x^{m_\Sigma(k)}\right) -\varphi_8(x) = \varphi_2\left(x_2\left(k^{\theta+1}u,kv\right)\right)-\varphi_8(x)\\
	&= \nu\left(R\left(k^{\theta+1}u, kv\right)\right)\const -\varphi_8(x)\\
	&= \nu\left((kv)^{\theta+2}+k^{\theta+2}uv+k^{\theta+2}u^\theta\right)\const -\varphi_8(x)\\
	&= \nu\left(k^{\theta+2}\right)\const +\nu\left(R(u,v)\right)\const -\nu\left(R(u,v)\right)\const \\
	&= \cconst\; \nu(k) \ddefine t.
\end{align*}
The parameter $t$ is independent of the choice of $x$ since $\nu(k)=\varphi_1(x_1(k))$ is independent of the choice of $x$. The remaining cases where $\alpha\neq \beta$ can easily be calculated the same way using \cite[32.13]{TW}.

Let $\alpha=\beta=\alpha_1$. Recall that $s_1(\alpha_1)=\alpha_9$ and that $\alpha_9^{m_1}=\alpha_1$. By \ref{Lem_oct2} and \ref{Lem_NUM19} we have for arbitrary $x=x_1(t), u=x_1(k) \in U_1$ the following 
\begin{align*}
\varphi_{s_\alpha(\beta)} \left(x^{m_\Sigma(k)}\right) -\varphi_1(x) &= \varphi_9\left(x^{m_\Sigma(k)}\right)-\varphi_1(x)\\
	&= \varphi_1\left(x^{h_1(k,1)}\right)-\varphi_1(x)\\
	&= \varphi_1\left(x_1(k^{-2}t)\right)-\varphi_1(x)\\
	&= -2\nu(k)= -2\varphi_1(u).
\end{align*}
This implies (\ref{NUM10}) with $t=-2\varphi_1(u)$.

Let $\alpha=\beta=\alpha_8$. Using $s_8(\alpha_8)=\alpha_0$, $\alpha_0^{m_8}=\alpha_8$, Lemmata \ref{Lem_oct2} and \ref{Lem_NUM19} we have for arbitrary $x=x_8(v,v'), u=x_8(k,k') \in U_1$ the following
\begin{align*}
\varphi_{s_\alpha(\beta)}\left(x^{m_\Sigma(k,k')}\right) -\varphi_8(x) &= \varphi_0\left(x^{m_\Sigma(k,k')}\right)-\varphi_8(x)\\
	&= \varphi_8\left(x^{h_8(k,k')}\right)-\varphi_8(x)\\
	&= \varphi_8\left(x_8\left(vR_{k,k'}^{-\sigma},v'R_{k,k'}^{\sigma-2}\right)\right)-\varphi_8(x)\\
	&= -2\nu\left(R_{k,k'}\right)\const= -2\varphi_8(u).
\end{align*}
This implies (\ref{NUM10}) with $t=-2\varphi_8(u)$.

We will finish the proof by verifying axiom $(V3)$ in case G. Let $\alpha$ and $\beta$ be equal to the positive root and 
let $x=x_+(w,u,v)$ and $y=x_+(r,s,t)$ in $U_+$ be given. Then, using Proposition~\ref{Prop_tec1}.2., we can calculate
\begin{align*}
 \varphi_-\left(x^{m_\Sigma(y)}\right)-\varphi_+(x)  
	&= \varphi_+\left(x^{m_\Sigma(y)m_\Sigma(x_+(0,0,1))^{-1}}\right) -\varphi_+(x) \\
	&= \varphi_+\left(x^{ (m_\Sigma(x_+(0,0,1))m_\Sigma(y^-1))^{-1}}\right) -\varphi_+(x)	
\end{align*}
By \ref{Lem_tec21} we have for arbitrary $a,b, \ldots, f$ that
$$
x_\alpha(a,b,c)^{h_\alpha(d,e,f)^{-1}}=x_\alpha\left(aN(d,e,f)^{\theta-2}, bN(d,e,f)^{1-\theta}, cN(d,e,f)^{-1}\right)
$$
and hence
\begin{equation}\label{Equ_no20}
 \varphi_-\left(x^{m_\Sigma(y)}\right)-\varphi_+(x)  
	= \varphi_+\left(x_+(w N\left(y^{-1}\right)^{\theta-2}, u N\left(y^{-1}\right)^{1-\theta}, v N\left(y^{-1}\right)^{-1} \right)) -\varphi_+(x).	
\end{equation}
It is easy to verify that (\ref{Equ_no20}) implies
\begin{equation}\label{Equ_no21}
\varphi_+\left(x_+\left(w N\left(y^{-1}\right)^{\theta-2}, u N\left(y^{-1}\right)^{1-\theta}, v N(y^{-1})^{-1} \right)\right)=\nu\left(N\left(y^{-1}\right)^{-2}N(w,u,v)\right).  
\end{equation}
Using equation \ref{Equ_no21} and the fact that $N(y)$ equals $N(y^{-1})$ we can conclude that the term $\varphi_-(x^{m_\Sigma(y)})-\varphi_+(x) $ is independent of $x$, in particular
\begin{align*}
\varphi_-\left(x^{m_\Sigma(y)}\right)-\varphi_+(x)  
	&= \nu\left(N\left(y^{-1}\right)^{-2}N(w,u,v)\right) - \nu\left(N(w,u,v)\right) \\
	&= -2 \nu\left(N\left(y^{-1}\right)^{-2}\right) = -2 \varphi\left(y^{-1}\right).
\end{align*}
Therefore $(V3)$ holds in all cases and the proof of \ref{Thm_RDexistence} is completed.
\end{proof}

 \subsection{Examples of \texorpdfstring{$\theta$-}\.invariant valuations}
\label{Sec_example}
\subsubsection*{Quotient field of polynomials in two variables}

Let $K$ be a field of characteristic $p$, let $\theta$ be the Tits endomorphism. A real valued valuation $\nu$ of the ring of polynomials $K[s,t]$ in two variables is given by the following formula: Let $p(s,t)=\sum_{i,j\in\N}a_{ij}s^i t^j$ with finitely many $a_{ij}\neq 0$ be given and define
$$
\nu(p)\define \min\left\{(i+\sqrt{p}\;j) : a_{ij}\neq 0\right\}.
$$
In particular $\nu(s)=1, \nu(t)=\sqrt{p}$.

Define a valuation, also denoted by $\nu$, of the quotient field $K(s,t)$ using $\nu:K[s,t]\rightarrow \R$ as follows: Given $\frac{p}{q}\in K(s,t)$ then
$$
\nu\left(\frac{p}{q}\right)=\nu(p)-\nu(q).
$$
We extend $\theta$ to an endomorphism from $K(s,t)$ to $\R$ by setting $t^\theta=s^p$ and $s^\theta=t$.
\begin{lemma}
  The map $\nu:K(s,t)\rightarrow \R$ is a $\theta$-invariant valuation. 
\end{lemma}
\begin{proof}
Given two elements $\frac{p}{q}$ and $\frac{p'}{q'}$ of $K(s,t)$ it is obvious that
$$
\nu\left(\frac{p}{q}\frac{p'}{q'}\right)= \nu\left(\frac{p}{q}\right)+\nu\left(\frac{p'}{q'}\right).
$$ 
Assume that $\nu\left(\frac{p}{q}\right)\geq \nu\left(\frac{p'}{q'}\right)$ which directly implies that $\nu(p'q)\leq\nu(pq')$. Hence
\begin{align*}
  \nu\left(\frac{p}{q}+\frac{p'}{q'}\right) 
	&= \nu(pq'+p'q) -\nu(qq') \\
	&\geq \min\{\nu(pq'),\nu(p'q) \} -\nu(qq') \\
	&=\nu\left(\frac{p'}{q'}\right) = \min\{\nu\left(\frac{p}{q}\right), \nu\left(\frac{p'}{q'}\right)\}.
\end{align*}
Therefore $\nu$ is a valuation. To prove $\theta$-invariance let $p(s,t)=\sum_{i,j\in\N}a_{ij}s^i t^j$ be given and assume that $\nu(p)=(i_0+\sqrt{p}\;j_0)$.
Observe that $p^\theta (s,t) = \sum_{i,j\in\N}a_{ij}^\theta t^i s^{p\cdot j}$  and therefore
$$
\nu\left(p^\theta\right)=\left(\sqrt{p}\;i_0+ p\cdot j_0\right)=\sqrt{p}\;\nu(p).
$$
By definition of the valuation on $K(s,t)$ we can conclude that $\nu$ is $\theta$-invariant.
\end{proof}

\subsubsection*{Formal Laurent series with real valued exponents}

The following example of a real valued $\theta$-invariant valuation was suggested by Linus Kramer.

Let $(F, \theta)$ be a field of characteristic $p$. Define $K$ to be the field of formal Laurent series $\sum_{i\in \R} a_i x^i$ with exponents $i\in \R$ and coefficients $a_i\in F$ such that there exists a finite subset $I$ of $\R$ with the property that $a_i\neq 0$ if and only if $i\in I$. Extend $\theta$  to $K$ by mapping $x$ to $x^{\sqrt{p}}$ and define a valuation $\nu$ on $K$ as follows
$$
\nu\left(\sum_{i\in I} a_i x^i\right) = \min\left\{i\in \R : a_i\neq 0\right\}=\min\left\{i\in I\right\} .
$$

\begin{lemma}
The map  $\nu$ is a $\theta$-invariant valuation.
\end{lemma}
\begin{proof}
Let $P(x)=\sum_{i\in \R} a_i x^i$ and $Q(x)=\sum_{i\in \R} b_i x^i$ be elements of $K$. Let $I_P$ and $I_Q$ denote the index set of non-zero coefficients of $P$, respectively $Q$. Assume without loss of generality that $\nu(P)\leq \nu(Q)$. The index set $I_{P \cdot Q}$ of the product of $P$ and $Q$ consists of all sums $i + j$ with $i\in I_P$ and $j\in I_Q$. Since $a_ix^i\cdot b_jx^j = a_ib_jx^{i+j}\neq 0$ if $a_i, b_j$ are non-zero, we have
$$
\nu\left(P\cdot Q \right) = \nu(P) + \nu(Q).
$$

The index set of the sum $P+Q$ is
$$
I_{P+Q}= (I_P\cup I_Q)\setminus\{i\in I_P\cap I_Q : a_i=-b_i\}.
$$
It is a finite set and $\min I_{P+Q} \geq \min I_P\cup I_Q$. Therefore
$$
\nu(P+Q)=\min I_{P+Q} \geq \min I_P\cup I_Q=\min\{\nu(P), \nu(Q)\}.
$$
Hence $\nu$ is a valuation which is $\theta$-invariant by the following observation
\begin{align*}
\nu\left(P^\theta\right) 
	&= \nu\left( \sum_{i\in I_P} a_i^\theta x^{\sqrt{p}\;i}\right) 
	 = \min\{\sqrt{p}\; i : i\in I_P\}
	 = \sqrt{p}\; \nu(P).	
\end{align*}
\end{proof}

 \newpage
\section{Convexity revisited}
\label{Sec_convexityRevisited}

In this section we prove an analog of Theorem \ref{Thm_convexity} for generalized affine buildings. It is obvious that we have to replace the representation theoretic arguments used in the proof of Proposition~\ref{Prop_existenceLSgalleries} by something combinatorial. During the preparation of this thesis Parkinson and Ram published the preprint \cite{ParkinsonRam} providing a combinatorial proof of \ref{Prop_existenceLSgalleries} on which the following is based. The main result in this section is \ref{Thm_convexityGeneral}.

\subsection{A geometric proof of \ref{Prop_existenceLSgalleries}}\label{Sec_ParkinsonRam}

Let us illustrate the main argument of \cite{ParkinsonRam}. Even though we formulated Proposition~\ref{Prop_existenceLSgalleries} in terms of galleries we will now give the equivalent statement using paths and the root operators as defined in \cite{LittelmannPaths}.
This is the language that also applies for non-discrete affine buildings as defined in \cite{Ronan} or \cite{TitsComo}, which is precisely the case with $\Lambda=\R$ in Definition \ref{Def_LambdaBuilding}.

\begin{notation}\label{Not_convRev1}
Let $(X,\App)$ be the geometric realization of a simplicial affine building. Hence $\RS$ is crystallographic. The model space $\MS$ is isomorphic to the tiled vector space $V$ underlying $\RS$. Let $B$ be a basis of $\RS$ with elements indexed by $I=\{1,2,\ldots,n\}$. Denote by $\Pi$ the set of all piecewise linear paths $\pi:[0,1]\rightarrow \MS$ such that $\pi(0)=0$. The \emph{concatenation of paths} $\pi_1$ and $\pi_2$ is denoted $\pi=\pi_1 \ast \pi_2$ and defined by
\begin{equation*}
\pi(t):=\left\lbrace 
\begin{array}{ll} 
	\pi_1(2t), & \text{ if } 0\leq t\leq 1/2 \\
	\pi_1(1) + \pi_2(2t-1), &\text{ if } 1/2\leq t\leq 1 .
\end{array}
\right.
\end{equation*}
We consider paths only up to reparametrization, i.e. paths $\pi_1,\pi_2$ are identified if there exists a continuous, piecewise linear, surjective, nondecreasing map $\phi:[0,1]\rightarrow [0,1]$ such that $\pi_1\circ\phi=\pi_2$. 
\end{notation}

Let $B$ be a basis of $\RS$. For any $\alpha\in B$ let $r_\alpha(\pi)$ be the path $t\mapsto r_\alpha(\pi(t))$.
Define a function $h_\alpha:[0,1] \rightarrow \R$ by $t\mapsto \lb \pi(t),\alpha^\vee\rb$ and let $n_\alpha$ be the critical value
\begin{equation}\label{Equ_tec1}
n_\alpha:=\min\{h_\alpha(t) : t\in [0,1]\}.
\end{equation}
If $n_\alpha \leq -1$  define $t_1$ to be the minimal value in $[0,1]$ such that $n_\alpha = h_\alpha(t_1)$ and let $t_0$ be the maximal value in $[0,t_1]$ such that $h_\alpha(t) \geq n_\alpha + 1$ for all $t\in [0,t_0]$.
Choose points $t_0=s_0 < s_1 < \ldots < s_r = t_1$
such that one of the two conditions holds
\begin{enumerate}
  \item \label{Num_1}$h_\alpha(s_{i-1})=h_\alpha(s_i)$ and $h_\alpha(t)\geq h_\alpha(s_{i-1})$ for all $t\in[s_i, s_{i-1}]$ or
  \item \label{Num_2} $h_\alpha$ is strictly decreasing on $[s_{i-1}, s_i]$ and $h_\alpha(t)\geq r_\alpha(s_{i-1})$ for $t\leq s_{i-1}$.
\end{enumerate}
Define $s_{-1}=0$ and $s_{r+1}=1$ and let $\pi_i$ be the path 
$$\pi_i(t)=\pi(s_{i-1} +t(s_i-s_{i-1})) - \pi(s_{i-1}) \text{ for all } i=0,1\ldots, r+1.$$
Then $\pi=\pi_0\ast\pi_1\ast \cdots\ast\pi_{r+1}$.

\begin{definition}\label{Def_rootOperator}
\index{root operators}
Let $\alpha$ be an element of $B$. Let $e_\alpha\pi:=0$ if $n_\alpha> -1$. Otherwise, let $\eta_i:=\pi_i$ if $h_\alpha$ behaves on $[s_{i-1}, s_i]$ as in \ref{Num_1}. and let $\eta_i:=r_\alpha(\pi_i)$ if $h_\alpha$ is on $[s_{i-1}, s_i]$ as in \ref{Num_2}. 
Then define the \emph{root operator} $e_\alpha$ associated to $\alpha$ by
$$e_\alpha\pi=\pi_0\ast \eta_1\ast\eta_2\ast\cdots\ast\eta_r\ast\pi_{r+1}.$$
\end{definition}

\begin{lemma}{\cite[Lemma 2.1.]{LittelmannPaths}}\label{Lem_rootOperator}
Let $\alpha$ be an element of $B$ and $e_\alpha$ as defined in \ref{Def_rootOperator}.
If $\pi$ is a path such that $e_\alpha\pi\neq 0$ then $(e_\alpha\pi)(1)=\pi(1)+\frac{2}{(\alpha,\alpha)}\alpha$.
\end{lemma}

\begin{notation}\label{Not_convRev2}
Let us add some notation to \ref{Not_convRev1}. Let $x$ be a special vertex in $\MS$ and assume without loss of generality that $x$ is contained in the fundamental Weyl chamber $\Cf$ determined by $B$. In \ref{Not_extendedRetr} we extended the two types of retractions $r$ and $\rho$ to galleries. Analogously it is possible to consider them as maps on paths by defining the image of a point $\pi(t)$ to be the corresponding point in the image of an alcove containing $\pi(t)$. Again denote these extensions by $\hat{r}$ and $\hat{\rho}$.
\end{notation}

 \begin{figure}[htbp]
 \begin{center}
 	\resizebox{!}{0.30\textheight}{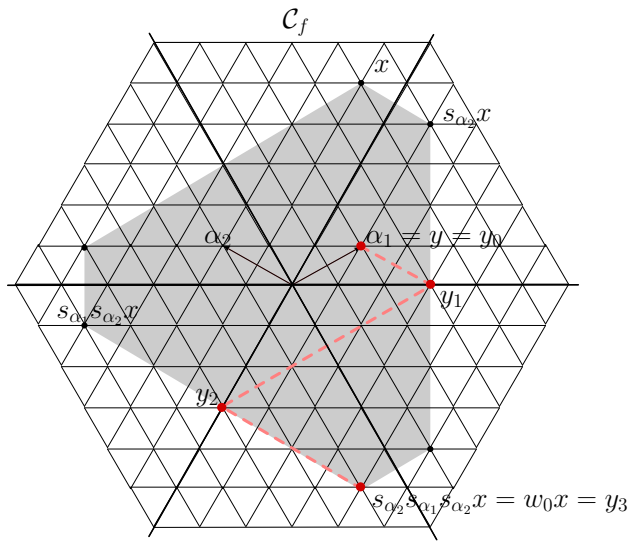}
 	\caption[adjacent]{Illustration of Lemma~\ref{Lem_ParkRam} with $x$ and $y$ as pictured, $w_0=s_{\alpha_2}s_{\alpha_1}s_{\alpha_2}$. The defined constants are $m_1=1, m_2=3$ and $m_3=2$.}
 	\label{Fig_folding}
 \end{center} 
 \end{figure}

\begin{definition}
We say that a path $\pi'$ is a \emph{positive fold} of a path $\pi$ if there exists a finite sequence of simple roots $\alpha_{i_j}\in B$ such that $\pi'$ is the image of $\pi$ under the concatenation of the associated root operators $e_{\alpha_{i_j}}$.
\end{definition}

\begin{remark}
It would be possible to define \emph{positively folded paths} similarly to positively folded galleries using the notion of a billiard path as defined by Kapovich and Millson in \cite{KapovichMillson}. In fact Kapovich and Millson remark that a consequence of their results is that the so called \emph{Hecke paths} defined in \cite[3.27]{KapovichMillson} correspond precisely to the positively folded galleries defined in \cite{GaussentLittelmann}. In particular, by Theorem 5.6 of \cite{KapovichMillson}, the LS-paths defined in \cite{LittelmannPaths} are a subclass of the Hecke paths.

Furthermore it is proven in \cite{KapovichMillson} that a path in an apartment $A$ of $X$ is a Hecke path if and only if it is the image of a geodesic segment in $X$ under a ``folding'' which is nothing else than a retraction centered at an alcove as defined in \ref{Def_simplChamberRetraction}. (Compare also Lemma 4.3 and 4.4 and Theorem 4.16 of \cite{KapovichMillson}.)

By \cite{LittelmannPaths} the set of $LS$-paths is invariant under the action of the root operators $e_\alpha$ with $\alpha\in B$, hence any image of a geodesic line (which trivially satisfies the axioms of an $LS$-path) obtained by applications of root operators is again an $LS$-path and therefore a Hecke path.
\end{remark}

Denote by $\QQ^+$ the positive cone in the co-root-lattice $\QQ(\RS^\vee)$.
Following \cite{ParkinsonRam} the first important observation is

\begin{lemma}{\cite[Lemma 3.1]{ParkinsonRam}}\label{Lem_ParkRam}
Identify the co-roots $\alpha^\vee$ with $\frac{2}{(\alpha,\alpha)}\alpha$. Let $x$ be as in \ref{Not_convRev2} and let $y\in \bigcap_{w\in \sW } w(x - \QQ^+)$. Fix a presentation $w_0=s_{i_1}\cdots s_{i_n}$ of the longest word $w_0\in\sW$ and denote by $\alpha_{i_k}$ the root corresponding to $s_{i_k}$. Define vertices $y_i$ in the convex hull $\dconv(\sW.x)$ inductively by $y_0=y$ and for all $k=1,2,\ldots,n$ by the recursive formula 
$$
y_k = y_{k-1} - m_k\alpha^\vee_{i_k} \;\text{ where }  
$$
$$ 
  m_k=\max\{m\in\Z : y_{k-1} - m \alpha^\vee_{i_k} \in \dconv(\sW .x)\cap (x+\QQ) \} .
$$

Then $y_n=w_0x$.
\end{lemma}

Let us restate the assertion of \ref{Prop_existenceLSgalleries} using paths instead of galleries.

\begin{prop}\label{Prop_existenceLSpaths}
Let $x$ be a special vertex in $\MS$ and assume without loss of generality that $x$ is contained in $\Cf$. Let $\pi: 0 \rightsquigarrow x^+$ be the unique geodesic from $0$ to $w_0x$, where $w_0$ is the longest word in the spherical Weyl group $\sW$. Let $y$ be a special vertex in $\MS$. There exists a positive fold $\pi'$ of $\pi$ with endpoint $\pi'(1)=y$ if and only if $y$ is contained in 
$
A^\QQ(x):=\{z\in\MS : z\in \dconv(\sW.x)\cap (x+\QQ)\}.
$
\end{prop}

\begin{figure}[h]
\begin{center}
 \begin{minipage}[b]{7 cm}
    \resizebox{!}{0.18\textheight}{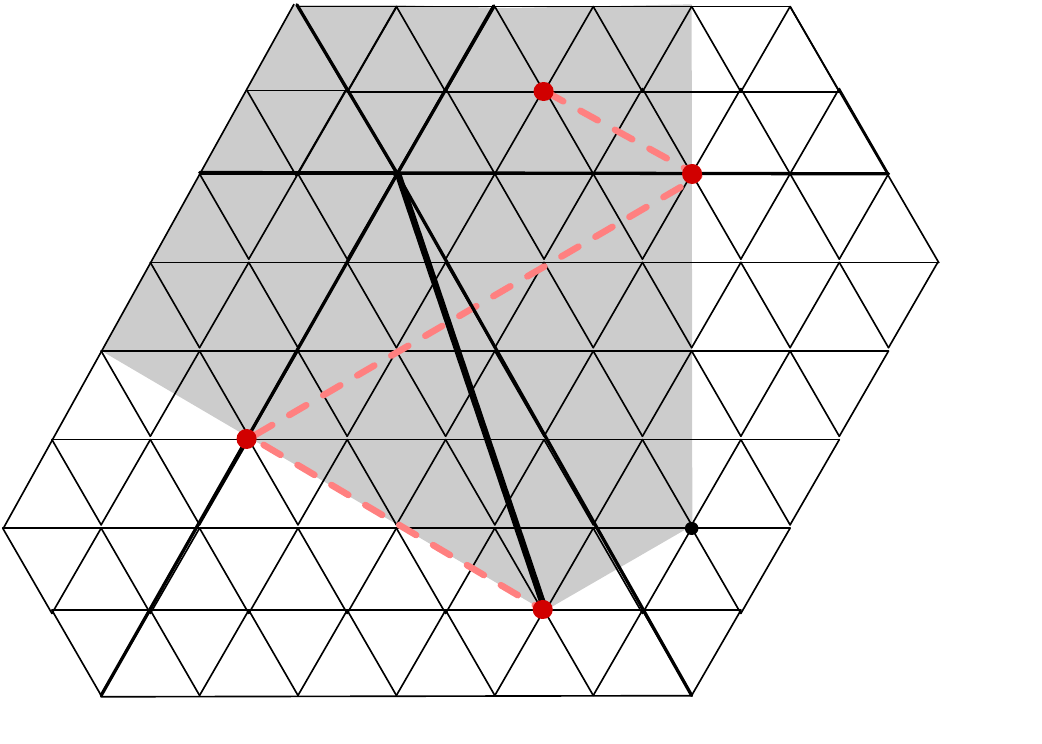}
  \end{minipage}
  \begin{minipage}[b]{7 cm}
    	\resizebox{!}{0.18\textheight}{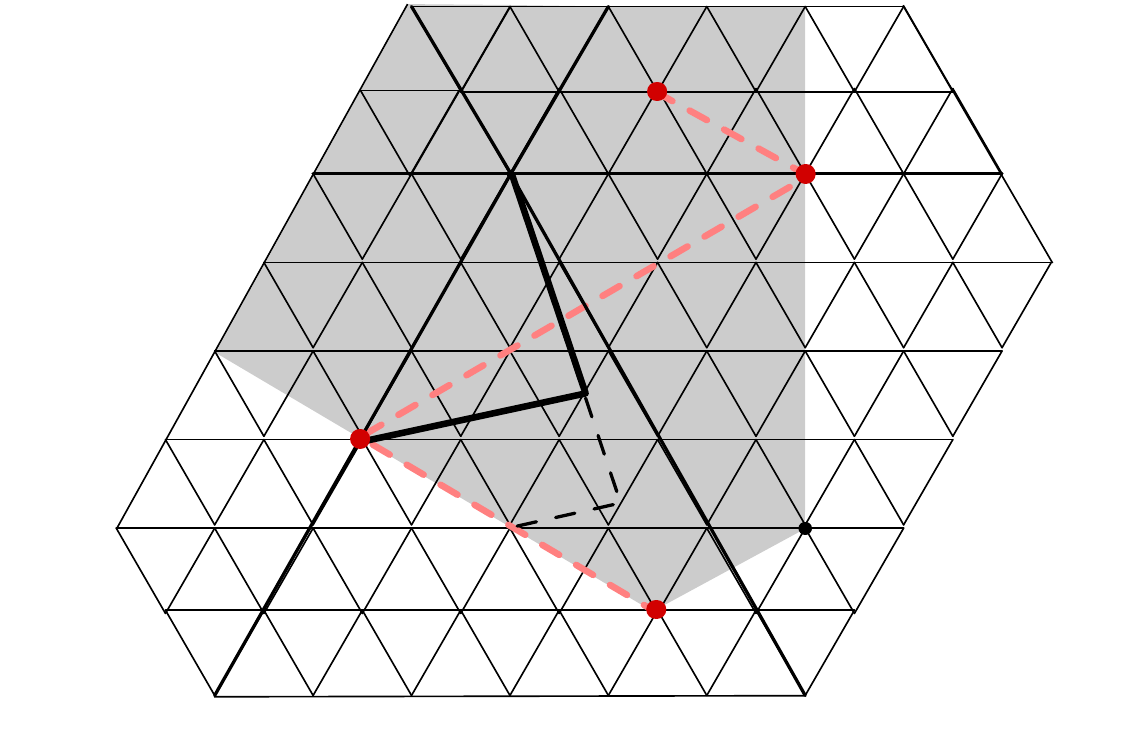}
  \end{minipage}

  \begin{minipage}[b]{7 cm}
    	\resizebox{!}{0.18\textheight}{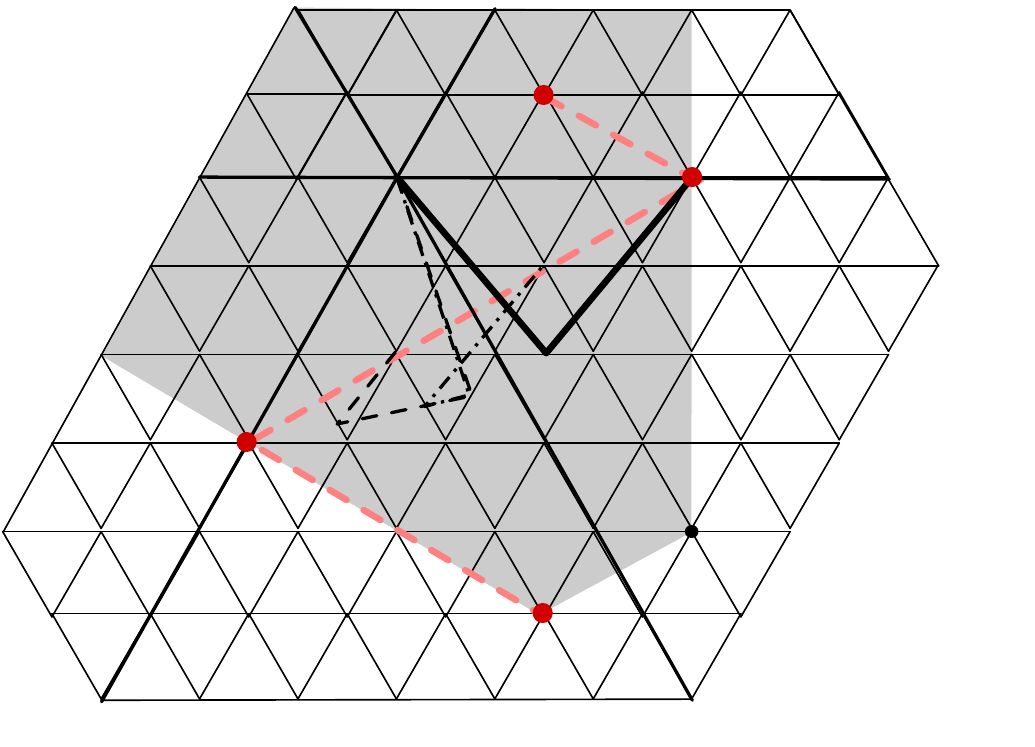}
  \end{minipage}
  \begin{minipage}[b]{7 cm}
    	\resizebox{!}{0.18\textheight}{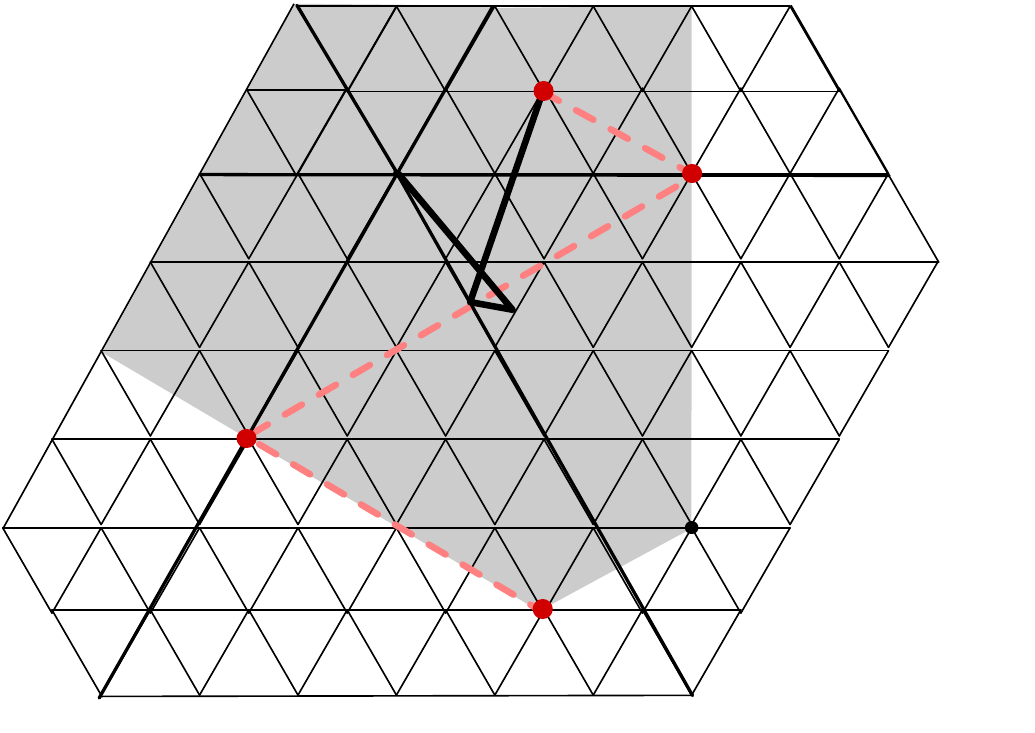}
  \end{minipage}
  \caption[adjacent]{Shifted paths according to the algorithm in the proof of Proposition \ref{Lem_rootOperator} with $w_0=s_{\alpha_2}s_{\alpha_1}s_{\alpha_2}$. The constants $m_k$ are $m_1=1, m_2=3$ and $m_3=2$.}
  \label{Fig_foldedpaths}
\end{center} 
\end{figure}

\begin{remark}\label{Rem_idea}
The main idea of the combinatorial proof of Proposition \ref{Prop_existenceLSpaths} is to reverse Lemma~\ref{Lem_ParkRam} and ``shift'' a minimal path  (respectively gallery) $\pi$ from $0$ to $w_0x$ to a folded path from $0$ to $y$. Let $\pi_n\define \pi$. By reverse induction define paths $\pi_{k-1}$ by an $m_k$-fold application of the root operator $e_{\alpha_{i_k}}$ to the previous path $\pi_{k}$ in step $k$, where $k$ goes from $n$ to $1$. According to Proposition~\ref{Lem_rootOperator} the endpoint of $\pi_k$ is translated by $\frac{2}{(\alpha_{i_k},\alpha_{i_k})} m_k\alpha_{i_k}$ in step $k$.
\end{remark}

\begin{example}
In Figure \ref{Fig_foldedpaths} we illustrate the paths one obtains by applying the algorithm of the proof of \ref{Lem_rootOperator} to the example of Figure \ref{Fig_folding}.

The vertex $x$ was chosen in $\Cf$ and an element $y\in\dconv(\sW.x)$ is given. We construct a positive fold of the minimal path connecting $0$ and $w_0x$ having endpoint $y$.

Picture (1) of Figure~\ref{Fig_foldedpaths} shows the minimal path $\pi=\pi_3$ connecting $0$ and $w_0x$. Since $m_3=2$ we apply $e_{\alpha_2}$ twice in the first step. In picture (2) the path $e_{\alpha_2} \pi$ is drawn with a dashed line, the path  $\pi_2$, with notation as in \ref{Rem_idea}, equals $e^{m_3}_{\alpha_2} \pi$ and has endpoint $y_2$. It is drawn with a solid line. Both paths are positive folds of $\pi$. Continue with a $3$-fold application of $e_{\alpha_1}$ to $\pi'$, as illustrated in picture (3). One obtains a path $\pi_1$ whose endpoint is $y_1$. Finally a single application of $e_{\alpha_2}$ gives us the path pictured in (4) which is the desired positive fold of $\pi$ with endpoint $y$.
\end{example}

\subsection{The convexity theorem}\label{Sec_convexityThm}

Let us first collect some necessary definitions. For simplicity assume that $\Lambda=\Lambda'$ in Definition  \ref{Def_LambdaBuilding} of the model space $\MS=\MS(\RS,\Lambda)$. 

\begin{definition}
A \emph{dual hyperplane} in the model space $\MS(\RS,\Lambda)$ is a set
$$\Hdual_{\alpha,k}=\{x\in V : \lb x,\fcw \rb = k\}$$
where $k\in \Lambda$ and $\fcw$ is a fundamental co-weight of $\RS$ as defined in Section~\ref{Sec_rootSystems}.
Again, dual hyperplanes determine \emph{dual half-apartments} ${\Hdual_{\alpha,k}}^\pm$ and \emph{convexity} is defined as in \ref{Def_convex}. The \emph{convex hull} of a set $C$ is denoted by $\dconv(C)$.
\end{definition}

\begin{definition}
Fix a basis $B$ of $\RS$. The \emph{positive cone} $\Cp$ in $\MS$ with respect to $B$ is the set of all positive linear combinations of roots $\alpha\in B$ with coefficients in $\Lambda_{\geq 0}$.
\end{definition}

\begin{remark}
In the simplicial case let $x$ be a special vertex in an apartment $A$. The set of special vertices in $A$ having the same type as $x$ are translates of $x$ by $\QQ(\RS^\vee)$. Therefore we have by Lemma~\ref{Lem_conv^*} that
$$
\bigcap_{w\in \sW } w(\lambda^+ - \QQ^+) = \dconv(\sW\lambda) \cap (\lambda +\QQ)
$$
where $\QQ^+$ is the positive cone in $\QQ$ and $\lambda^+$ is the unique element of $\sW\lambda$ contained in the fundamental Weyl chamber $\Cf$.

In the general case the regarded root system $\RS$ underlying the model space $\MS(\RS,\Lambda,T)$ does not need to be crystallographic. The notion of co-root lattice $\QQ(\RS^\vee)$ does therefore not make sense. The analog role is played by $T$, which is transitive on the images of $0$ under the affine Weyl group $\WT$. Assuming $\RS$ is crystallographic, defining $\Lambda=\R$ and $T=\QQ(\RS^\vee)$ these are precisely the special vertices of fixed type and $\WT$ corresponds to the affine Weyl group associated to $\RS$. 
\end{remark}

\begin{lemma}\label{Lem_conv^*Lambda}
For any special vertex $x$ in $\MS(\RS,\Lambda,T)$ let $x^+$ denote the unique element of $\sW.x$ contained in $\Cf$. Then
$$\dconv(\sW.x)\cap T = \{y\in V :  x^+-y^+\in (\Cp\cap T) \}=\bigcap_{w\in \sW } w(\lambda^+ - (\Cp\cap T)) .$$ 
\end{lemma}
\begin{proof}
Replace $\QQ$ by $T$ in the proof of Lemma~\ref{Lem_conv^*}.
\end{proof}

The two retractions in question are on one hand a retraction centered at a germ of a Weyl chamber and on the other hand a retraction centered at infinity. For generalized affine buildings they are defined as follows

\begin{definition}\label{Def_vertexRetraction}
\index{vertex retraction}
Let $(X,\App)$ be an affine building. Fix a Weyl chamber $S$ based at a vertex $x$ in $X$. Denote the germ of $S$ at $x$ by $\Delta_xS$. By Corollary~\ref{Cor_tec18} the building $X$ is as a set the union of all apartments containing $\Delta_xS$.
For an arbitrary vertex $y$ in $X$ fix a chart $g\in\App$ such that $y$ and $\Delta_xS$ are contained in  $g(\MS)$. Define 
$$ r_{A,\Delta_xS}(y) = (f\circ w\circ g^{-1} )(y)  $$
where $w\in\aW$ is such that $g\vert_{g^{-1}(f(\MS))}=(f\circ w)\vert_{g^{-1}(f(\MS))}$.
The map $r_{A,\Delta_xS}$ is called \emph{retraction onto $A$ centered at $\Delta_xS$}.
\end{definition}

\begin{definition}\label{Def_retractionInfty}
\index{retraction centered at infinity}
Let $(X,\App)$ be an affine building. Fix an equivalence class of Weyl chambers $c=\partial S\subset \binfinity X$ and an apartment $A=f(\MS)$ containing some representative $S$ of $c$. By Corollary~\ref{Cor_tec17} the building $X$ is the union of apartments containing a sub-Weyl chamber of $S$. For an arbitrary vertex $x$ in $X$ fix a chart $g\in\App$ such that $(x\cup S')\subset g(\MS)$ for some Weyl chamber $S'$ parallel to $S$ and  define 
$$ \rho_{A,c}(x) = (f\circ w\circ g^{-1} )(x)  $$
where $w\in\aW$ is such that $g\vert_{g^{-1}(f(\MS))}=(f\circ w)\vert_{g^{-1}(f(\MS))}$.
The map $\rho_{A,c}$ is called \emph{retraction onto $A$ centered at $c=\partial S$} or simply \emph{retraction centered at infinity}.
\end{definition}

\begin{prop}\label{Prop_r-rho}
Let $(X,\App)$ be an affine building. Fix an apartment $A$ of $X$. Let $S$ be a Weyl chamber contained in $A$ and let $c$ be a chamber in $\partial A$. Then
\begin{enumerate}
  \item The retractions $r_{A,\Delta_xS}$ and $\rho_{A,c}$, as defined above, are well defined.
  \item The restriction of $\rho_{A,c}$ to an apartment $A'$ containing $c$ at infinity is an isomorphism onto $A$.
  \item The restriction of $r_{A,\Delta_xS}$ to an apartment $A'$ containing $\Delta_xS$ is an isomorphism onto $A$.
\end{enumerate}
\end{prop}
\begin{proof}
The second and third assertions are clear by definition.
The fact that $\rho_{A,c}$ is well defined is proven on p. 33 of \cite{BennettDiss}. An argument along the same lines, given below, proves that $r_{A,\Delta_xS}$ is well defined.

Let $y$ be an element of $X$ contained in $f_1(\MS)\cap f_2(\MS)$, where $A_i\define f_i(\MS)$, $i=1,2$ are apartments of $X$ containing $\Delta_xS$. Denote by $w_i$ the element of $\aW$ in the definition of $r_{A,\Delta_xS}(y)$ with respect to $f_i$. It suffices to prove
\begin{equation}\label{Equ_tec28}
f\circ w_1\circ f_1^{-1}(y)=f\circ w_2\circ f_2^{-1}(y).
\end{equation}
The germ $\Delta_xS$ is contained in $A_1\cap A_2$ hence there exists by $(A2)$ an element $w_{12}\in\aW$ such that 
$$ 
f_2\circ w_{12}\;\vert_{f_1^{-1}(f_2(\MS))} = f_1 \;\vert_{f_1^{-1}(f_2(\MS))} .
$$
Since $y\in A_1\cap A_2$, we have 
\begin{equation}\label{Equ_tec29}
f\circ w_2\circ f_2^{-1} 
	= f\circ w_2\circ f_2^{-1}(f_2\circ w_{12}(f_1^{-1}(y)))
	= f\circ w_2 w_{12}(f_1^{-1}(y)).
\end{equation}
Equation (\ref{Equ_tec29}) is true for all $y\in A_1\cap A_2$, hence it is in particular true for the intersection $C$ of the Weyl chambers $S_1$ and $S_2$ contained in apartments $A_1$ and $A_2$, respectively, with equal germ $\Delta_xS_i=\Delta_xS, i=1,2$. Therefore 
$$
f\circ w_1\circ f_1^{-1}(C) = f\circ w_2\circ w_{12} \circ f_1^{-1}(C)
$$
and hence $w_2 w_{12}=w_1.$ Combining this with (\ref{Equ_tec29}) yields equation (\ref{Equ_tec28}).

\end{proof}

\begin{lemma}\label{Lem_compRetractions}
Let $(X,A)$ be an affine building.
\begin{enumerate}
\item Given a Weyl chamber $S$ in $X$ and apartments $A_i, i=1,\ldots, n$ containing a sub-Weyl chamber of $S$. Denote by $\rho_i$ the retraction $\rho_{A_i, \partial S}$. Then
$$(\rho_1\circ\rho_2\circ\ldots\circ\rho_n) = \rho_1 .$$
\item Let $\Delta_xS$ be a germ of a Weyl chamber $S$ at $x$. Let $A_i, i=1,\ldots, n$ be a set of apartments containing $\Delta_xS$ and denote by $r_i$ the retraction onto $A_i$ centered at $\Delta_xS$. Then 
$$(r_1\circ r_2\circ\ldots\circ r_n) = r_1 .$$
\end{enumerate}
\end{lemma}
\begin{proof}
For all $i$ the restriction of $\rho_i$ to an apartment containing a sub-Weyl chamber of $S$ is an isomorphism. By Corollary~\ref{Cor_tec17} the building $X$ equals as a set the union of all apartments containing a representative of $\partial S$. Therefore 1. follows. Similar arguments using Corollary~\ref{Cor_tec18} imply the second part of the lemma.
\end{proof}

\begin{prop}\label{Prop_compRetractions}
Let $(X,\App)$ be an affine building modeled on $\MS(\RS,\Lambda)$. For any retraction $\rho_{A,c}$ centered at infinity and all $x\in X$ there exists $y\in A$ such that $$\rho_{A,c}(x)=r_{A,\Delta_yS}(x),$$ where $S$ is the unique Weyl chamber based at $y$ and contained in $c$.
\end{prop}
\begin{proof}
By Proposition~\ref{Prop_tec16} there exists an apartment $B$ of $X$ containing $x$ and a Weyl chamber $S$ contained in $c\in\binfinity X$. The intersection of $A$ and $B$ contains a sub-Weyl chamber $S'$ of $S$ and the restriction of $\rho_{A,c}$ to $B$ is an isomorphism onto $A$ fixing $A\cap B$ pointwise. Let $y$ be a vertex in $S'$ and denote by $S_y$ the Weyl chamber based at $y$ contained in $S'$. The restriction of the retraction $r_{A,\Delta_yS}$ to $B$ is by \ref{Prop_r-rho} an isomorphism onto $A$ which fixes $A\cap B$ pointwise. Therefore $\rho_{A,c}(x)=r_{A,\Delta_yS}(x)$.
\end{proof}

The remainder of the present subsection is devoted to the proof of Theorem~\ref{Thm_convexityGeneral} which is split into the two propositions \ref{Prop_preimage} and \ref{Prop_image}. For convenience let us first fix notation.

\begin{notation}\label{Not_convexity}
Let $(X,\App)$ be a thick affine building in the sense of Definitions \ref{Def_LambdaBuilding} and \ref{Def_thick} which is modeled on $\MS=\MS(\RS,\Lambda)$ equipped with the full affine Weyl group $\aW$.  Let $A$ be an apartment of $X$ and identify $A$ with $\MS(\RS,\Lambda)$ via a given chart $f$ of $A$ contained in $\App$. Hence an origin $0$ and a fundamental Weyl chamber $\Cf$ are fixed as well. In the following denote the retraction onto $A$ centered at $\Delta_0\Cf$ by $r$ and the retraction onto $A$ centered at $\partial (\Cfm)$ by $\rho$.
\end{notation}

The following lemma is important for the proof of Proposition~\ref{Prop_preimage}.

\begin{lemma}\label{Lem_key}
Let notation be as in \ref{Not_convexity} and let $x$ be a vertex in $A$. Let $w_0=s_{i_1}\cdots s_{i_n}$ be a reduced presentation of the longest word in $\sW$. Denote by $\alpha_{i_k}$ the root corresponding to $s_{i_k}$. Given a vertex $y\in \dconv(\sW.x)$ inductively define vertices $y_i$ in $\dconv(\sW.x)$ by putting $y_0=y$ and defining 
$$
y_k = y_{k-1} - \lambda_k\frac{2}{(\alpha_{i_k},\alpha_{i_k})}\alpha_{i_k}
$$ 
for all $k=1,2,\ldots,n$ with $\lambda_k\in\Lambda$ maximal such that $y_{k-1} - \lambda_k\frac{2}{(\alpha_{i_k},\alpha_{i_k})} \alpha_{i_k}$ is contained in $\dconv(\sW .x)$.
Then $y_n=w_0x^+$.
\end{lemma}
\begin{proof}
Assume without loss of generality that $x\in\Cf$. Define vertices $x_0,x_1,\ldots,x_n$ in $\dconv(\sW.x)$ 
by $x_0=x$ and $x_k= r_{i_k}x_{k-1} = x_{k-1} - \lb x_{k-1},\alpha^\vee_{i_k}\rb\alpha_{i_k}$ for all $1\leq k \leq n$. 

Define a partial order on $A$ by setting $y\prec x$ if and only if $y-x\in\Cp$. 
We will show that $y_k\prec x_k$ for all $k=0,1,\ldots,n$. Then $y_n\prec x_n=w_0x$ and $y_n\in\dconv(\sW.x)$ and thus the result of the Lemma follows. We have $y_0=y\prec x=x_0$ and by induction hypothesis $x_{k-1} - y_{k-1}\in\Cp$. Therefore
$$
x_k - y_k 
= x_{k-1} - y_{k-1} +(\lambda'_k - \lb  x_{k-1},\alpha_{i_k}^\vee\rb) \alpha_{i_k} 
= z + (\lambda'_k + c - \lb  x_{k-1}, \alpha_{i_k}^\vee\rb) \alpha_{i_k} 
$$ 
where $c\in\Lambda_{\geq 0}$ and $z$ is such that $x_{k-1} - y_{k-1} = z + c\alpha_{i_k}$ and that $z\in \Span_{\Lambda_{\geq 0}}(\{\alpha\in B\setminus\{\alpha_{i_k}\}\})$.
We will prove that $(\lambda'_k + c - \lb x_{k-1}, \alpha_{i_k}^\vee\rb)\geq 0 $.
With $\lb  x_{k-1}, \alpha^\vee_{i_k}\rb\geq 0$ and $\lb\alpha_j,\alpha^\vee_i\rb\leq 0$ if $i\neq j$  we can conclude
$$
0\leq \lb x_{k-1},\alpha_{i_k}^\vee\rb 
= \lb y_{k-1}+c\alpha_{i_k} + z,\alpha_{i_k}^\vee\rb 
\leq \lb y_{k-1} + c\alpha_{i_k}, \alpha_{i_k}^\vee\rb.
$$
We have
$$
y_{k-1} - (\lb x_{k-1},\alpha^\vee_{i_k}\rb -c) \alpha_{i_k} 
= (y_{k-1} + c\alpha_{i_k}) - \lb x_{k-1}, \alpha_{i_k}^\vee\rb\alpha_{i_k} \in\dconv(\sW.x)
$$
and hence by definition of $\lambda_k'$ we can conclude that $\lb x_{k-1}, \alpha_{i_k}^\vee\rb -c \leq \lambda'_k$.
\end{proof}

\begin{prop}\label{Prop_preimage}
With notation as in \ref{Not_convexity} let $x$ be an element of $A$. 
Given a vertex $y\in\dconv(\sW.x)$ there exists a preimage of $y$ under $\rho$ which is contained in $r^{-1}(\sW.x)$.
\end{prop}
\begin{proof}
Assume without loss of generality that $x\in\Cf$.
We will inductively construct a preimage of $y$. Let the points $y_k$ and values $\lambda_k$ for $k=1,\ldots,n$ be as defined in the assertion of Lemma~\ref{Lem_key}.
Abbreviate $\lambda'_k:= \frac{2\lambda_k }{(\alpha_{i_k}, \alpha_{i_k})}$. The first induction step is as follows:
let $H_1$ be the unique hyperplane in the parallel class of $H_{\alpha_{i_1}, 0}$ containing $z_1:=y_0 - \frac{1}{2}\lambda'_1\alpha_{i_1}$. 
The associated reflection $r_1$ fixes $z_1$ and maps $y=y_0$ onto $y_1$ and vice versa.
Since $X$ is thick there exists an apartment $A_1$ in $X$ containing $y_0$ such that $A\cap A^1 = H_1^+$. We claim that $0\in H_1^+$.

The set $\dconv(\sW.x)$ is by definition $\sW$-invariant. Therefore $r_\alpha(y)$ is contained in $\dconv(\sW.x)$ for all $\alpha\in\sW$ and all $y\in\dconv(\sW.x)$. Hence $r_{\alpha_{i_1}}(y) = y- \lb y,\alpha_{i_1}^\vee \rb \alpha_{i_1}$ is contained in $\dconv(\sW.x)$ and $\lambda'_{i_1}\geq \lb y, \alpha_{i_1}^\vee \rb$ and we can conclude that $0\in H_1^+$. But $\Cf$ is contained in $H_1^+$ as well, since $\alpha_{i_1}\in B$.
Therefore $A_1$ contains $x$. We can now consider the convex hull of the orbit of $x$ under $\sW$ in the apartment $A_1$, which we denote by $\dconv_{A_1}(\sW.x)$ in order to be able to distinguish between the convex sets in the different apartments. Notice that $r$ restricted to $A_1$ is an isomorphism onto $A$. Denote the unique preimage of $y_1$ in $A_1$ by $y_1^1$ and for all $k$ the unique preimage of $y_k$ in $A-1$ by $y_k^1$.
Hence we have ``transferred'' the situation to $A_1$. 

The image of $y^1_n$ under $r$ is, by construction and Lemma~\ref{Lem_key}, the point $w_0.x$. Now $(A\setminus A_1 )\cup (A_1\setminus A)$ is again an apartment of $X$ containing a sub-Weyl chamber of $\Cfm$. By construction $\rho(y^1_1)=y$.

Repeating the first induction step with $A_k$ and $y^k_k$ in place of $A$ and $y$ for all $k=1,\ldots,n$ we get a sequence of apartment $A_k$ for $k=1,\ldots,n$ with the property that they are all isomorphically mapped onto $A$ by $r$. Hence Lemma~\ref{Lem_key} implies that $y_n^n$ is a preimage of $w_0x$. Analogously to the first step we conclude that $\rho_{A_{k-1},\partial(\Cfm)}(y^k_k)=y^{k-1}_{k-1}$. With Lemma~\ref{Lem_compRetractions} finally conclude that $\rho(y^n_n)=y$.
\end{proof}

In order to prove Proposition \ref{Prop_image} we need to introduce a technical condition first.

\begin{property}{\bf Finite covering condition}\label{Prop_finiteCovering}
Let $(X,\App)$ be an affine building modeled over $\MS(\RS,\Lambda)$. Let $x,y$ be points in $X$ and let $A$ be an apartment containing $x$ and $y$. The segment of $x$ and $y$ in $A$ is defined by 
$$
\seg_{A}(x,y)\define \{p\in A : d(x,y)=d(x,p)+d(p,y)\}.
$$
Fix a chamber $c$ in the spherical building $\binfinity X$ at infinity of $X$. If the following statement is true for $X$ we say that $X$ satisfies the \emph{finite covering condition}.
\begin{itemize}
  \item[$\mathrm{(FC)}$] The segment $\seg_{A}(x,y)$ of $x$ and $y$ is contained in a finite union of apartments containing a Weyl chamber $S$ such that $c=\partial S$.
\end{itemize}
\end{property}

\begin{remark}\label{Rem_FC}
Similar arguments as in the proof of Lemma 7.4.21 in \cite{BruhatTits} imply that all generalized affine buildings with $\Lambda=\R$ satisfy (FC).
Axiom (T2), the $Y$-condition, implies that $\Lambda$-trees satisfy (FC). We conjecture that (FC) holds for all generalized affine buildings.
\end{remark}

\begin{property}{\bf Observation}\label{Prop_observation}
Let notation be as in \ref{Not_convexity}. Assume that the building $(X,\App)$ satisfies condition (FC).
Let $y$ be an element of $r^{-1}(\sW.x)$ and $\overline A$ an apartment containing $0$ and $y$. By (FC) there exists a finite collection of apartments $A_0=A$, $A_i$, $i=1,\ldots, N$, consecutively enumerated, such that 
\begin{itemize}
  \item for all $i$ the apartment $A_i$ contains a sub-Weyl chamber of $\Cfm$
  \item their union contains $\seg_{\overline{A}}(0,y)$ and 
  \item $y$ is contained in $A_N$.
\end{itemize}
Hence we can find a finite sequence of points $x_i$, $i=1,\ldots,N$ with $x_0=0$ and $x_N=y$ such that 
\begin{itemize}
  \item $d(0,y)=\sum_{i=0}^{N-1} d(x_i,x_{i+1})$
  \item $x_i$ and $x_{i+1}$ are contained in $A_i$ and 
  \item for $N>i>0$ the Weyl chambers $S_{i-}$ and $S_{i+}$ based at $x_i$ containing $x_{i-1}$ and $x_{i+1}$, respectively, are such that $\xi_i\define \Delta_{x_i}S_{i-}$ and $\eta_i\define\Delta_{x_i}S_{i+}$ are opposite chambers in $\Delta_{x_i}X$. 
\end{itemize}
Let $\eta_0$ and $\xi_N$ be defined analogously.

Without loss of generality we may further assume that for all $i$ the point $x_{i-1}$ is not contained in $A_i$. Otherwise, if for some $i_0$ the point $x_{i_0}\in A_{i_0-1}$, we can define a shorter sequence of point by setting $x'_i\define x_i$ for all $i< i_0$ and $x'_i\define x_{i+1}$ for all $i\geq i_0$, i.e. we omit $x_i$ in the sequence of points. Similarly define apartments $A_i'$ by omitting $A_{i-1}$.

Notice that $\rho$ restricted to $A_i$ is an isomorphism onto $A$ for all $i$. Hence the distance $d(x_i, x_{i+1})=d(\rho(x_i), \rho(x_{i+1}))$ for all $i\neq N$.

Let $z$ be a point contained in the interior of $\Cfm\cap \bigcap_{i=0}^N A_i$. Then $r_z\define r_{A,\Delta_z\Cf}$ equals $\rho$ on $\seg_{\overline{A}}(0,y)$ and the restriction of $r_z$ to $A_i$ is an isomorphism onto $A$ for all $i=0,\ldots,N$.
\end{property}

\begin{corollary}\label{Cor_rho-distancediminishing}
The retractions $r_{A,S}$ and $\rho_{A,c}$ are distance non-increasing. 
\end{corollary}
\begin{proof}
Above we made the observation, that $r_z$ retracted to $A_i$ is an isomorphism onto $A$ for all $i$ and hence $d(x_i, x_{i+1})=d(r_z(x_i), r_z(x_{i+1}))$ for all $i\neq N$. Therefore the image of $\bigcup_{i=0}^{N-1}\seg_{A_i}(x_i, x_{i+1})$ under $r_z$ is connected and $d(r(x),r(y))\leq d(x,y)$ using the triangle inequality.
By Proposition \ref{Prop_compRetractions} the assertion holds for $\rho$ as well.
\end{proof}

\begin{lemma}\label{Lem_germs}
With notation as in \ref{Prop_observation} define $\xi'_i\define r_z(\xi_i)$ and $\eta'_i\define r_z(\eta_i)$.
We prove: 
let $x'_i$ denote $r_z(x_i)$ then $\xi'_i$ and $\eta'_i$ are chambers in $\Delta_{x'_i}A$ and there exists $w\in\sW$ such that $\xi'_i$ is the image of $w_0\eta'_i$ by $w=s_{i_1}\cdots s_{i_l}$.
\end{lemma}
\begin{proof}
By construction the Weyl chamber germs $\xi_i$ and $\eta_i$ are opposite at $x_i$ and contained in the apartments $A_{i-1}$ and $A_{i}$, respectively. Let $f_{i-1}, f_i$ be charts of $A_{i-1}$ and $A_i$ and assume without loss of generality that there exist a point $p_i\in\MS$ such that $f_{i-1}(p_i)=f_i(p_i)= x_i$. By axiom $(A2)$ there exists $w\in\aW$ such that 
$$
f_{i-1}\vert_{f_{i-1}^{-1}(f_i(\MS))} = (f_{i}\circ w) \vert_{f_{i-1}^{-1}(f_i(\MS))}.
$$
In fact by the assumption $f_{i-1}(p_i)=f_i(p_i)= x_i$ the translation part of $w$ is trivial and $w$ is contained in $\sW$.
The opposite germ of $f_{i}^{-1}(\eta_i)$ is mapped onto $f_{i-1}^{-1}(\xi_i)$ by $w$.

Denote the restriction of the projection $r_z$ to $A_{i-1}$, respectively $A_i$, by $\iota_{i-1}$ and $\iota_i$. Then $\iota_j: A_j\rightarrow A$ is an isomorphism for $j=i-1,i$. Therefore
$$
f_{i-1}^{-1}\circ \iota_{i-1}^{-1}\circ\iota_i\circ f_i : \MS\rightarrow \MS
$$
is an isomorphism, which coincides with $f_{i-1}^{-1} \circ (f_{i}\circ w)$ in a neighborhood of $x_i$.

Fix a chart $h$ of $A$ such that $h(p_i)=x'_i$. Then, since $\iota_i, \iota_{i-1}$ are isomorphisms, we conclude
$f_{i-1}^{-1}(\xi_i) = h^{-1}(\iota_{i-1}(\xi_i))$ and $f_i^{-1}(\eta_i) = h^{-1}(\iota_i(\eta_i))$.
Therefore $\xi'_i=\iota_{i-1}(\xi_i)$ is the image of $w_0 \eta'_i=\iota_{i}(\eta_i)$ by $w$, where $w_0$ is the longest word in $\sW$.

Restricting the retractions $r_{A_i, \Delta_z\Cf}$ to an apartment $A_j$ is an isomorphism, hence by Lemma~\ref{Lem_compRetractions} applying several of these retractions only to apartments of that type is commutative.
\end{proof}

\begin{lemma}\label{Lem_positiveFold}
Let notation be as in \ref{Not_convexity}. Let $x\neq 0$ be an element of $A$ and $H$ a hyperplane separating $0$ and $\Cfm$. If $y$ is a vertex in $\dconv(\sW.x)$ contained in the same half-apartment determined by $H$ which contains a sub-Weyl chamber of $\Cfm$, then the reflected image of $y$ at $H$ is contained in $\dconv(\sW.x)$.
\end{lemma}
\begin{proof}
Denote by $H'$ the hyperplane parallel to $H$ containing $0$ and let $\alpha\in\RS$ be such that $H'$ is fixed by $r_{\alpha}$. The convex hull of $\sW.x$ is by definition $\sW$-invariant. Therefore the image of $y$ under the reflection at $H'$ is contained in $\dconv(\sW.x)$ and equals, by Proposition~\ref{Prop_aboveHyperplane}, $y-2y^\alpha\alpha$, since $y=m_\alpha+y^\alpha\alpha$ with $m_{\alpha}\in H'$. Similarly we can find $m\in H$ and $\lambda < y^\alpha$ such that $y=m+\lambda\alpha$. The image of $y$ under the reflection at $H$ equals $y-2\lambda\alpha$ which is obviously contained in $\dconv(\sW.x)$ as well.
\end{proof}

\begin{prop}\label{Prop_image}
With notation as in \ref{Not_convexity}, assume that (FC) holds. Let $x$ be an element of $A$. Given a vertex $y$ in $r^{-1}(\sW.x)$ the image $\rho(y)$ is contained in $\dconv(\sW.x)$, where we identify $\sW$ with the stabilizer of $0$ in $\aW$.
\end{prop}
\begin{proof}
By induction on $N$ and Lemma \ref{Lem_compRetractions} it is enough to prove the assertion in case that $N=1$.

We re-use the notation of Lemma~\ref{Lem_germs} and its proof.
Recall that $A_0$ contains $x_0=0$, $x_1$ and a sub-Weyl chamber $C_0$ of $\Cfm$.

Let $(c_1=\xi_1, c_2, \ldots, c_k=\eta_1)$ be a minimal gallery in $\Delta_{x_1}X$ from $\xi_1$ to $\eta_1$. If $x_1$ is not contained in a bounding hyperplane of $A_{0}\cap A_1$ we may replace $x_1$ by another point satisfying this assumption.
Hence without loss of generality we can assume that $\eta_1\subset (A_1\setminus A_{0})$ and $\xi_1\subset (A_{0}\setminus A_1)$. 
Then there exists an index $j\in\{2,\ldots, k\}$ such that $c_{j-1}$ is contained in $A_{0}\setminus A_1$ and $c_j$ in $A_1\setminus A_{0}$. 

Denote by $H_1$, respectively $H_{0}$, the hyperplane spanned by the panel $p_j\define c_{j-1}\cap c_j$ in $A_1$, respectively $A_{0}$. Notice that $x_1\in H_1 \cap H_{0}$.

Assume a) that $H_1$ separates $x_{2}$ and $C_2\subset \Cfm$ in $A_2$ and that $0$ and $\Cfm$ are separated by $H_{0}$ in $A$.
The image of $(c_j,\ldots, c_k=\eta_j)$ under $\rho$ is the unique gallery $(c'_j,\ldots, c'_k=\eta'_j)$ of the same type which is contained in $A$ and starts in $c_{j-1}$. 
Hence $\seg_{A_1}(x_1, x_{2})$ is mapped onto the segment 
$\seg_{A}(x_1, \rho(x_{2}))$ which has initial direction $\eta'_1$.
Let $x'_{2}$ denote $\rho(x_{2})$.
By assumption the hyperplane $H_{0}$ separates $0$ and $\Cfm$. Apply Lemma \ref{Lem_germs} and obtain that  there exists $w\in\sW$ such that $\eta'_1=\rho(\eta_1)$ is the image of $w_0\xi_1$. (We actually prove the opposite in \ref{Lem_germs}, but the argument is symmetric.)
Hence $x'_{2}$ is the reflected image of $r(x_{2})$ by a finite number of reflections along hyperplanes containing $x_1$ and separating $C_0$ and $0$.
Therefore $x'_{2}$ is obtained from $r(x_{2})$ by a positive fold in $A$. By Lemma \ref{Lem_positiveFold} $x'_2$ is contained in $\dconv(\sW.x)$.

In the second case b) $H_1$ separates $x_{2}$ and $\Cfm$ but $0$ and $\Cfm$ are not separated by $H_{0}$. Then $\rho$ maps $(c_j,\ldots, c_k=\eta_1)$ onto $(\rho(c_j),\ldots, \rho(\eta_j)$ which is a gallery of the same type contained in $A$ and $\rho(c_j)$ is the unique chamber in $\Delta_{x_1}A$ sharing the panel $c_j\cap c_{j-1}$ with $c_{j-1}$. Therefore $\xi_1$ and $\eta'_1\define\rho(\eta_1)$ are opposite in $x_1$ and $\rho(x_2)=r(x_2)$.

The case c) that both $x_{2}$ and $0$ are not separated from $\Cfm$ by $H_1$, respectively $H_{0}$, can not happen: Let $S_{1-}$ and $S_{1+}$ be the Weyl chambers based at $x_1$ having germs $\xi_1,\eta_1$ and are contained in $A$ and $A_1$, respectively. By Corollary \ref{Cor_CO} the Weyl chambers $S_{1-}$ and $S_{1+}$ are contained in a unique apartment $B$. The span $H_B$ of the panel $p_j=c_{j-1}\cap c_j$ in $B$ separates the segments $\seg_B(0,x_1)$ and $\seg_B(x_1,x_2)$ and hence $0$ and $x_2$ which can therefore not be contained in the same half-apartment determined by $H_B$.

Finally assume d) that $H_1$ does not separate $x_{2}$ and $\Cfm$ but that $0$ and $\Cfm$ are separated by $H_{0}$. The germ $\gamma$ of $\Cfm$ at $x_1$ is then contained in the same half-apartment of $\Delta_{x_1}$ as $\eta_1$. By the assumption that $0$ is separated from $\Cfm$ by $H_1$ we have that there exists a minimal  gallery in $\Delta_{x_1} A$ from $\xi_1$ to $\gamma$ containing either $p_j$ or its opposite panel $q_j$ in $\Delta_{x_1}A$. If $\xi_1$ and $\gamma$ are not opposite this gallery is unique. But this contradicts the choice of $j$.
\end{proof}

\begin{thm}\label{Thm_convexityGeneral}
Let notation be as in \ref{Not_convexity} and assume in addition that (FC) holds. Given a vertex $x$ in $A$ we can conclude
$$
\rho(r^{-1}(\sW.x)) =\dconv(\sW.x).
$$
\end{thm}
\begin{proof}
Combining Proposition~\ref{Prop_preimage} and \ref{Prop_image} the assertion follows.
\end{proof}

\begin{remark}
After handing in her thesis the author was able to prove what was conjectured in Remark \ref{Rem_FC}. The fact that condition (FC) is true for all generalized affine buildings is given in Appendix \ref{Sec_FC}.
\end{remark}

\begin{remark}
Note that Theorem \ref{Thm_convexity} is \emph{not} a special case of \ref{Thm_convexityGeneral}. This would follow from the result conjectured in \ref{Conj_convexityGeneral}.
\end{remark}

\subsection{Loose ends}\label{Sec_looseEnds}

Obviously the same questions as above can be asked if we restrict ourselves to affine buildings that are thick with respect to an affine Weyl group $\WT$ which is a proper subgroup of the full affine Weyl group $\aW$, meaning that the building is only assumed to branch at special hyperplanes in the sense of Definition \ref{Def_specialHyperplane}.

The arising problem is the following.
In Figure~\ref{Fig_folding} we gave an example of a sequence of points $y_i$ appearing according to Lemma~\ref{Lem_ParkRam} in an apartment of an $\widetilde{A}_2$ affine building. Each $y_i$ with $i\neq 0$ is contained in the boundary of the convex hull of the orbit (the gray shaded region). 

This is not the case in general as you can see in Figure~\ref{Fig_folding2} where we illustrate an example in an apartment of type $\widetilde{G}_2$. The shaded region is the convex hull of the orbit of $x$. The vertex $y$ is a special point contained in the convex hull of $\sW.x$ and has the same type as $x$. The sequence of vertices $y_i$ defined in Lemma~\ref{Lem_ParkRam} is connected with the dotted line. Note that $y_1$ and $y_2$ are in the interior of the convex hull.

\begin{figure}[h]
 \begin{center}
	\resizebox{!}{0.37\textheight}{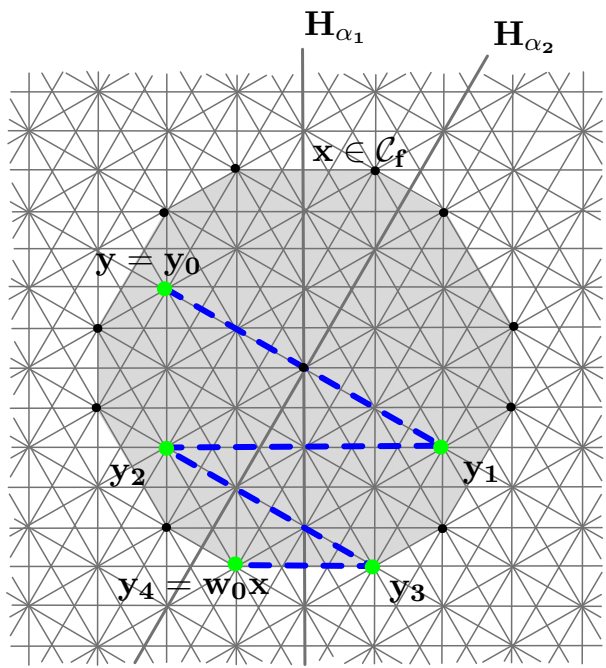}
 	\caption[adjacent]{Illustration of Lemma~\ref{Lem_ParkRam} with $w_0=s_{\alpha_1}s_{\alpha_2}s_{\alpha_1}s_{\alpha_2}s_{\alpha_1}s_{\alpha_2}$.}
 	\label{Fig_folding2}
 \end{center} 
 \end{figure}

If $X$ is a simplicial affine building this observation is irrelevant since the minimum taken in \ref{Lem_ParkRam} to define the sequence of the $y_i$ always exists. But dealing with a generalized affine building and special vertices defined with respect to  some proper subgroup $\WT$ of the full affine Weyl group $\aW$ it is not true in general that such a minimum exists. There is no obvious way to adjust the construction, but we hope nevertheless that making appropriate changes will enable us to prove what we conjectured in \ref{Conj_convexityGeneral}.

\begin{notation}\label{Not_convexity2}
Let $(X,\App)$ be an affine building modeled over $\MS=\MS(\RS,\Lambda, T)$. Let $A$ be an apartment of $X$ and identify $A$ with $\MS$ via a given chart $f\in\App$. Hence an origin $0$ and a fundamental chamber $\Cf$ are fixed. In the following denote the retraction onto $A$ centered at $\Delta_x\Cf$ by $r$ and the retraction onto $A$ centered at $\partial (\Cfm)$ by $\rho$.
\end{notation}

\begin{conj}\label{Conj_convexityGeneral}
With notation as in \ref{Not_convexity2} assume that $X$ is thick with respect to $\WT$ as defined in \ref{Def_thick}. For any special vertex $x$ in $A$ we have
$$
\rho(r^{-1}(\sW.x)) =\dconv(\sW.x)\cap (x+T).
$$
\end{conj}

This coincides with the classical case if we assume $\RS$ to be crystallographic, put $\Lambda=\R$ and let $T$ equal the co-root lattice $\QQ(\RS^\vee)$ of $\RS$.

We can prove the conjecture if $X$ is one-dimensional, i.e. $X$ is a $\Lambda$-tree with sap, the defining root system $\RS$ is of type $A_1$ and $\MS\cong \Lambda$. 
\begin{thm}\label{Thm_Conj1}
With notation as in \ref{Not_convexity2} assume that $X$ is thick with respect to $\WT$, as defined in \ref{Def_thick}. Assume further that $X$ is of dimension one. Let $x$ be a special vertex in $A$. Then
$$
\rho(r^{-1}(\sW.x)) =\dconv(\sW.x)\cap (x+T).
$$
\end{thm}
\begin{proof}
The set $\sW.x \subset A$ has two elements $x^+$ and $x^-$ contained in $\Cf$ and $\Cfm$, respectively. Both points, $x^+$ and $x^-$ are at distance $d(0,x)$ from $0$. The preimage $r^{-1}(\sW.x)$ is the set of all vertices $y\in X$ such that $d(0,x)=d(0,y)$. We denote the positive root in $\RS$ by $\alpha$. 

Let us first prove that $\rho(r^{-1}(\sW.x)) \subset\dconv(\sW.x)\cap \{x+T\}$. Fix an element $y$ of $\rho(r^{-1}(\sW.x))$ which is not contained in $A$.
Is $A'$ an apartment containing $\partial(\Cfm)$ and $y$, the restriction of $\rho$ to $A'$ is an isometry onto $A$. Let $a$ denote the end of $A'$ different from $\partial(\Cfm)$ and define $z$ to be the branch point $\kappa(\partial(\Cfm) , \partial\Cf, a)$ of $A$ and $A'$. There are two cases:  Either $z$ is contained in $\Cf\setminus\{0\}$ or in $\Cfm$. In the first case $r$ and $\rho$ coincide on $A'$ and $y$ is mapped onto $x^+$. If $z\in\Cfm$ then $y\in r^{-1}(x^-)$, the distance of $x^-$ and $z$ equals $d(y,z)$ and $\rho(y)=x^- + 2d(x^-,z)\ddefine y'$, which is obviously contained in $\dconv(\sW.x)$. The point $y'$ is the same as the image of $x^-$ reflected at (the hyperplane) $z$. Since there are no branchings other than at special hyperplanes, the branchpoint $z$ of $A$ and $A'$ is a fixed point (fixed hyperplane) of a reflection $r=t\circ r_{\alpha}$ in $\WT$. Therefore $y'\in\dconv(\sW.x)\cap \{x+T\}$.

To prove the converse let $y\in\dconv(\sW.x)\cap \{x+T\}$ be given. For arbitrary $v\in\MS$ let $t_v$ denote the unique translation of $\MS$ mapping $0$ to $v$. The unique element of $T$ mapping $x^-$ to $y$ is, with this notation, the map $t_{y-x^-}$. The reflection at $x^-$ is given by $t_{2x^-}\circ r_\alpha$. Apply the element $t_{y+x^-} \circ r_\alpha$ of $\WT$ to $x^-$ and observe that the image is $y$. Further easy calculations imply that $z\define x^- + \frac{1}{2} (y-x^-)$ is fixed by $t_{y+x^-} \circ r_\alpha$.
Therefore $z$ is a special hyperplane and, since $X$ is thick, there exists an apartment $A'$ intersecting $A$ precisely in the ray $\overrightarrow{z \partial\Cf}$. Obviously $A'$ contains $0, x^+$ and $z$, and the restriction of $r$ to $A'$ is an isomorphism onto $A$. Denote by $y'$ the preimage $r^{-1}(x^-)$ in $A'$. By construction we have $d(y',z)=d(x^-,z)$. The apartment $A''\define (A\setminus A' )\cup (A'\setminus A)$ contains $x^-,z$ and $y'$. Observe that $\rho$ restricted to $A''$ is an isomorphism mapping $y'$ onto $y$. This proves $\rho(r^{-1}(\sW.x)) \supset\dconv(\sW.x)\cap \{x+T\}$ and we are done.
\end{proof}

 \newpage

\begin{appendix}
\section{Proof of Theorem \ref{Thm_trees}}\label{Sec_prooftrees}

We will now give a detailed proof of Theorem~\ref{Thm_trees} providing the existence of a tree associated to a projective valuation. The proof is based on \cite{BennettDiss} and it is done in four steps. 

\subsubsection*{Step 1 - two technical Lemmata}

\begin{lemma}{\bf (3-point Lemma) }\label{3ptlemma}
Given a $\Lambda$-valued projective valuation $\val$ on a set $E$ and four pairwise distinct elements $a, a_1, a_2, a_3 $ of $E$. Then exactly one of the following cases occurs:
\begin{enumerate}
\item $\val(a_1, a; a_2, a_3) > 0$ \label{num12}
\item $\val(a_2, a; a_3, a_1) > 0$ \label{num13}
\item $\val(a_3, a; a_1, a_2) > 0$ \label{num14}
\item $\val(a_i, a; a_j, a_k) = 0$ for all $\{i,j,k\}=\{1,2,3\}$. \label{num15}
\end{enumerate}
\end{lemma}

The proof of this Lemma, which is not difficult, was not contained in \cite{BennettDiss}.
\begin{proof}
Assume \ref{num12}. Then $(PV2)$ implies 
$$ \val(a_1, a_3; a_2, a)=k>0\;\; \text{ and } \;\;\val(a_1, a_2; a, a_3)=0 .$$
Using $(PV1)$ we can calculate
\begin{align*}
0 < k = & \val(a_1, a_3; a_2, a) = \val(a_2, a; a_1, a_3)= -\val(a_2, a; a_3, a_1) \\
0 = &\val(a_1, a_2; a, a_3) = - \val(a_3, a; a_1, a_2) 
\end{align*}
and hence $\neg$~\ref{num13} and $\neg$~\ref{num14}. 
By the first line of the last calculation $0 < k = \val(a_2, a; a_1, a_3)$ meaning in particular that $\val(a_2,a; a_1,a_3)\neq 0$. Hence \ref{num15} is not true. 

Similar arguments prove that \ref{num13}. (or \ref{num14}.) imply the negation of the other three cases. The fact that \ref{num15}. implies the negation of the others is obvious. 
\end{proof}

\begin{remark}
If the valuation $\val$ is the canonical valuation of a tree $T$ cases ~\ref{num12} -~\ref{num15} of Lemma~\ref{3ptlemma} may be illustrated as in figure~\ref{figure6} (the given enumeration coincides with the enumeration of Lemma~\ref{3ptlemma} and Figure~\ref{figure6}). Write $\overrightarrow{ab}$ for the directed line determined by the ends $a$ and $b$, then the four cases can be stated as follows: 
The first case $(1)$ is that the directed line $\overrightarrow{a_2a_3}$ first passes $\kappa(a_1, a, a_2)$ then $\kappa(a_1, a, a_3)$.
In case $(2)$ the directed line $\overrightarrow{a_3a_1}$ passes $\kappa(a_2, a, a_3)$ before $\kappa(a_2, a, a_1)$.
In case $(3)$ the point $\kappa(a_3, a, a_1)$ is first passed by the line $\overrightarrow{a_1a_2}$ which afterwards passes $\kappa(a_3, a, a_2)$. And in the last case $(4)$ all $\kappa(a_i,a,a_j)$ are equal for all choices of $i\neq j$. 
\end{remark}

\begin{figure}[htbp]
\begin{center}
	\resizebox{!}{0.3\textheight}{\input{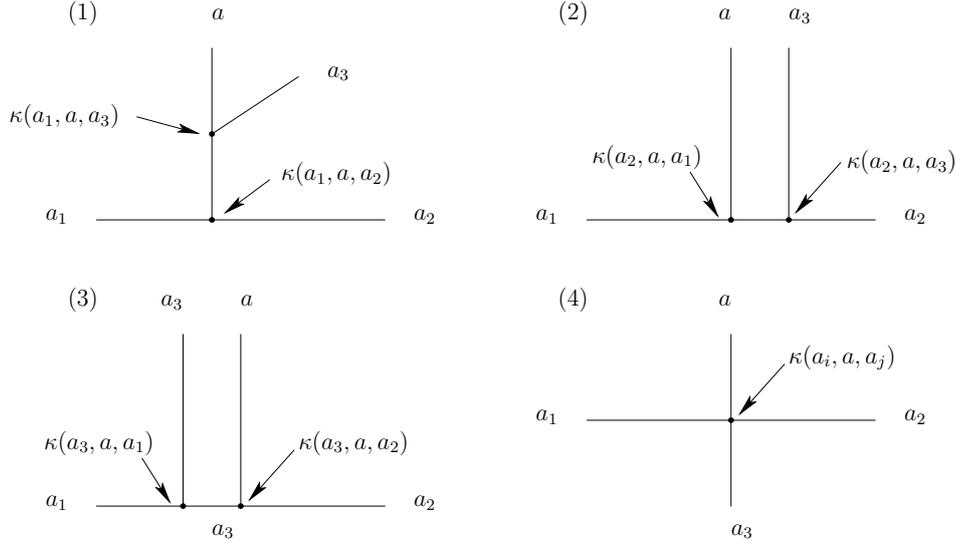}}
	\caption[tree configurations]{Possible geometric configurations in a tree.}
	\label{figure6}
\end{center} 
\end{figure}

\begin{lemma}{ \cite[Lemma 4.5]{BennettDiss}}\label{tec6}
For $a\in E$ let $(i_a, j_a, k_a)$ be an even permutation\footnote{All permutations can be interpreted as symmetries of an equilateral triangle with vertices labeled clockwise with $1,2$ and $3$. There are three reflections, i.e. the odd permutations and three rotations, i.e. the even permutations. Hence the even permutations are exactly $(2,3,1), (3,1,2)$ and $(1,2,3)$.} 
of $(1,2,3)$ such that $\val(a_{i_a}, a; a_{j_a}, a_{k_a})$ is maximal.
Let $a\in E\setminus\{a_1, a_2, a_3\}$, $b\in E\setminus\{a_{i_a}, a_{j_a}\}$ and assume $\val(a_{i_a}, a; a_{j_a}, b) > 0$. 
Then $i_b=i_a, j_b=j_a$ and $k_b=k_a$. 
\end{lemma}
\begin{proof}
According to the assumptions Lemma \ref{3ptlemma} implies that either $\val(a_{i_a}, a; a_{j_a}, a_{k_a}) >0$ (cases \ref{num12}-\ref{num14}) or that $\val(a_{i_a}, a; a_{j_a}, a_{k_a}) =0$ (case \ref{num15}). To make the proof easier to read we assume w.l.o.g. that $(i_a, j_a, k_a)=(1,2,3)$.

Axiom $(PV3)$ implies
\begin{equation}\label{num16}
\val(a_1, b; a_2, a_3) = \val(a_1, b; a_2, a) + \val(a_1, b;a, a_3)
\end{equation}
By assumption and $(PV2)$
\begin{equation}\label{num17}
0 < \val(a_1, a; a_2, b)=\val(a_1, b; a_2, a)
\end{equation}
By regarding two different cases, we prove the assertion:

1. Assume that $\val(a_1, b; a, a_3) \geq 0$:
Then, according to (\ref{num16}) and (\ref{num17}), one can conclude $\val(a_1, b; a_2,a_3)>0$. Hence $\val(a_1, b; a_2, a_3)$ is maximal by Lemma~\ref{3ptlemma} and $i_b=i_a, j_b=j_a,k_b=k_a$.

2. Assume now $\val(a_1, b; a, a_3) < 0$:
Axiom $(PV1)$ implies
\begin{align}
\val(a_1, b; a, a_3) &=-\val(a_1, b; a_3, a)=-\val(a_3, a; a_1, b) > 0\label{num18} 
\end{align}
and, by $(PV2)$, we have 
\begin{equation}\label{num19}
-\val(a_1, a; a_3, b) >0.
\end{equation}
Repeated use of $(PV1)$ to $(PV3)$ implies
\begin{align*} \val(a_1, b; a_2, a_3)  &= \val(a_1, b; a_2, a) + \val(a_1, b; a, a_3) 
			 = \val(a_1, a; a_2, b) + \val(a_1, b; a, a_3) \\ 
			&= \val(a_1, a; a_2, b) - \val(a_1, b; a_3, a)  
			 = \val(a_1, a; a_2, b) - \val(a_1, a; a_3, b) \\ 
			&= \val(a_1, a; a_2, b) + \val(a_1, a; b, a_3) 
			 = \val(a_1, a; a_2, a_3).
\end{align*}

The terms $\val(a_1, b; a_2, a)$ and $\val(a_1, a; b, a_3)$ are positive by equations (\ref{num17}) and (\ref{num19}). If $\val(a_1, a; a_2, a_3) > 0$  the assertion follows and $i_b=i_a, j_b=j_a, k_b=k_a$.

Hence we assume that $\val(a_1, a; a_2, a_3) = 0$, which implies that we are in case \ref{num15} of Lemma~\ref{3ptlemma}. 
One has to prove $\val(a_i, b; a_j, a_k)=0$ for all $\{i,j,k\}=\{1,2,3\}$.
Since 
$$\val(a_i, b; a_j, a_k)=\val(a_i, a; a_j, a_k)+\val(a, b; a_j, a_k)$$
it is enough to prove that $\val(a,b;a_j, a_k)=0$ for all $\{j,k\}\subset\{1,2,3\}$. 

The rest of the proof is done in three steps a) to c). Assume 
\begin{equation}\label{num20}
\exists\; j,k \text{ such that } l\define \val(a,b;a_j, a_k)>0. 
\end{equation}

a) The assumption (\ref{num20}) together with $(PV2)$ implies 
$$0 < l= \val(a, a_k; a_j,b) 	= -\val(a_k, a; a_j, b)
			= -\val(a_k, a; a_j, a_i)-\val(a_k, a; a_i, b).$$
Since $\val(a_k, a; a_j, a_i)=0$ we have
$$ \val(a, a_k; a_j, b)= - \val(a_k, a; a_i, b)= \val(a_k, a; b, a_i) > 0$$
and hence, using $(PV1)$, we conclude
\begin{equation}\label{num21}
l=\val(a,b;a_i, a_k) > 0. 
\end{equation}

b) Axiom $(PV1)$ together with (\ref{num20}) implies
$\val(b,a;a_k, a_j) >0$. Similar calculations as in a) lead to
\begin{equation}\label{num22}
l=\val(b,a;a_i, a_j) > 0 .
\end{equation}

c) We now combine a) and b) to finish 2. Using $(PV3)$ we have
$$l=\val(a,b; a_j, a_k)=\val(a,b; a_j, a_i)+\val(a,b; a_i, a_k).$$
According to (\ref{num21}) the term $\val(a,b; a_i, a_k)$ equals $l$, implying $\val(a,b; a_j, a_i)=0$ which is a contradiction to (\ref{num22}).
Hence assumption (\ref{num20}) is wrong and for all $i,j$ the valuation $\val(a,b; a_j, a_k)$ equals $0$. This finishes the proof.
\end{proof}

\begin{remark}
The assumptions of Lemma~\ref{tec6} are slightly different than the ones in \cite[Lemma 4.5]{BennettDiss}. There $a,b$ were assumed to be arbitrary elements of $E$. Making this assumption is actually not fully correct, since $\val$ is just defined on ordered quadruples of pairwise distinct elements of $E$, i.e. elements of $\Evier$. They suggested to extend $\val$ to arbitrary quadruples of elements. We restrict the hypothesis of Lemma~\ref{tec6} and suggest another solution in Step 2. 
Note that our proof is different from the one in \cite{BennettDiss}. 
\end{remark}

\subsubsection*{Step 2 - Definition of a rooted tree datum}

A rooted tree datum encodes more than a projective valuation $\val$ does. Besides the branching of the tree an origin (or root) is fixed. Hence we want to create more information than we actually have. The following construction therefore depends on the choice of the origin (which is equivalent to the choice of three ends $a_1, a_2, a_3$) and is unique up to that choice.

Let $E$ be a set and $\val$ a projective valuation on $E$.
\begin{definition}\label{Def_rtd}
Fix pairwise distinct elements $a_1, a_2, a_3$ of $E$. Let $(i_a, j_a, k_a)$ be the even permutation of $(1,2,3)$ such that $\val(a_{i_a}, a; a_{j_a},a_{k_a})$ is maximal. Define a function $\wedge : E\times E \rightarrow \Lambda$ by  
\begin{enumerate}
\item \label{num23} $a\wedge b = \max\{0,\val(a_{i_a}, a; a_{j_a}, b)\}$ for all $a\in  E\setminus\{a_{i_a},a_{j_a},a_{k_a}\}$ and $b\in E\setminus \{a_{i_a},a_{j_a}\}$.
\item \label{num24} $a\wedge a = \infty$ for all $a\in E$. 
\item \label{num25} $a_i\wedge a_j = 0$ for all $ i,j \in \{1,2,3\}$ such that $i\neq j$.
\item \label{num26} $a\wedge a_l = 0 \ddefine a_l\wedge a$ if $l= i_a$ or $l=j_a$.
\end{enumerate}
\end{definition}

\begin{remark}
Cases \ref{num24}. to \ref{num26}. of Definition~\ref{Def_rtd} do not appear in Bennett's Dissertation. Yet they are necessary to determine the origin $o_T$ of the tree in the Alperin-Bass construction. It is not possible to cover these cases with $\val$ since a projective valuation is just defined on quadruples in $\Evier$. 

Koudela suggested to extend $\val$ to arbitrary quadruples to eliminate this inaccuracy of Bennett's definition which is in a certain sense equivalent to our Definition~\ref{Def_rtd}.

Our motivation for Definition~\ref{Def_rtd} came from Figure~\ref{figure8} and the fact that $a\wedge b$ is supposed to describe the distance of $o_T$ to the branching point of the rays $S_a, S_b$ starting at $o_T$ directing to $a, b$, respectively. 
\end{remark}

\begin{figure}[htbp]
\begin{center}
	{\input{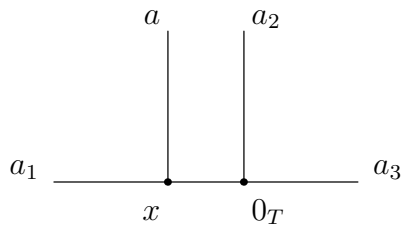}}
	\caption[rooted tree data]{Definition of a rooted tree datum.}
	\label{figure8}
\end{center}
\end{figure}

\begin{lemma}{ \cite[Lemma 4.6]{BennettDiss}}\label{tec7}
The pair $(E,\wedge)$ is a rooted tree datum.
\end{lemma}
The following proof is based on \cite{BennettDiss}.
\begin{proof}
Axiom $(RT0)$ is obvious by definition.
For the proof of $\mathrm{(RT1)}$ consider two cases. 
First if $a \wedge b > 0$ then, by Lemma~\ref{tec6}, we have that $i_b=i_a, j_b=j_a$ and $k_b=k_a$. And hence
$$
0<a\wedge b =\val(a_{i_a}, a; a_{j_a}, b)=\val(a_{i_b}, a; a_{j_b}, b).
$$ 
Using $\mathrm{(PV2)}$ we can conclude
$$
a\wedge b=\val(a_{i_b}, a; a_{j_b}, b)=\val(a_{i_b}, b; a_{j_b}, a)=b\wedge a.
$$
Second assume $a \wedge b = 0$. Here we are in case~\ref{num15} of Lemma~\ref{3ptlemma}. If $b\wedge a >0$ the same argument as in case (1) implies $b\wedge a= a\wedge b >0$. Hence $b\wedge a =0 = a\wedge b$. 

We prove $\mathrm{(RT2)}$: Let $a,b,c \in E$ and consider seven cases.
First assume that $a\wedge b, a\wedge c$ and $b\wedge c$ are given by \ref{num23} of Definition~\ref{Def_rtd} meaning that $a,b,c$ are not contained in $\{a_1,a_2,a_3\}$.
\begin{itemize}
\item[a)] if $a\wedge b=0$ or $b\wedge c=0$ then $\mathrm{(RT2)}$ is obvious.
\item[b)] assume $a\wedge b >0$ and $b\wedge c>0$. Lemma~\ref{tec6} implies $i_b=i_a=i_c, j_b=j_a=j_c$ and $k_b=k_a=k_c$. Let without loss of generality $i_a=1$ and $j_a=2$. The fact that $a\wedge b >0$ implies $a\wedge b=\val(a_1, a; a_2, b)$. By $\mathrm{(PV3)}$ we therefore have
\begin{equation}\label{num27}
	a\wedge c \geq \val(a_1, a; a_2, c) = a\wedge b + \val(a_1, a; b,c).	
\end{equation}
The following is true since $c\wedge b >0$ 
\begin{equation}\label{num28}
	a\wedge c=c\wedge a \geq \val(a_1, c; a_2, a)=c\wedge b+ \val(a_1, c; b,a).
\end{equation}
Therefore $c\wedge b=\val(a_1, c; a_2, b)$.
\end{itemize}

Secondly assume that $b=c$ then $a\wedge c= a\wedge b$ and $b\wedge c = b\wedge b = \infty$, hence $(RT2)$ is obvious. 
In case three assume that $a=c$ then $a\wedge c=\infty$ and the assertion is obvious. 
If $a=b$, which is case four, then $a\wedge b=\infty$ and $b \wedge c= a\wedge c$. 
Fifth, if $\{a,c\}=\{a_i, a_j\}$ where $i\neq j$ then $a\wedge c=0$ by definition. We also can conclude that either $i$ or $j$ is contained in $\{i_b, j_b\}$ and hence $\min\{a\wedge b, b\wedge c\}=0$, which implies $(RT2)$. 
The second last case is, that $c\in \{a_{i_a}, a_{j_a}\}$. But then by definition $a\wedge c=0$ and $\min\{a\wedge b, b\wedge c\}=0$.
Finally, if $b\in \{a_{i_a}, a_{j_a}\}$, then $a\wedge b =0$ and $c\wedge b =0$. Since $a\wedge c\geq 0$ the assertion follows. 
\end{proof}

\subsubsection*{Step 3 - Constructing a tree with sap}

Let $T= T(E,\wedge)$ be the Alperin-Bass tree associated to $(E, \wedge)$. We now define an apartment system for $T$ by first defining rays, then lines. 

\begin{definition}
A \emph{ray} in $T$ is a set 
$\{\langle e,\lambda \rangle : e\in E, \lambda\in \Lambda, \lambda\geq \lambda_0 \}$
where $\lambda_0$ is a fixed element of $\Lambda$.

A subset 
$\overleftrightarrow{ab}\define \{\langle a,\lambda\rangle : \lambda \geq a \wedge b\} 
			\cup \{\langle b,\lambda\rangle : \lambda \geq a \wedge b\}$
of $T$ is called \emph{line} and $\App$ denotes the set of all lines in $T$.
\end{definition}

Notice that ends of rays (and lines) are, by definition, elements of $E$. The set $\App$ is an apartment system for the Alperin-Bass tree.

\subsubsection*{Step 4 - The projective valuation $\val$ equals the canonical valuation}

Finally one has to prove
\begin{prop}\label{Prop_valuations}
If $T= T(E,\wedge)$ is the Alperin-Bass tree associated to a rooted tree datum $(E,\val)$ and if  $\val_T$ denotes its canonnical valuation defined in~\ref{Def_canval}, then $\val_T=\val.$
\end{prop}
For the proof of Proposition~\ref{Prop_valuations} we need two additional Lemmata.

\begin{figure}[htbp]
\begin{center}
	{\input{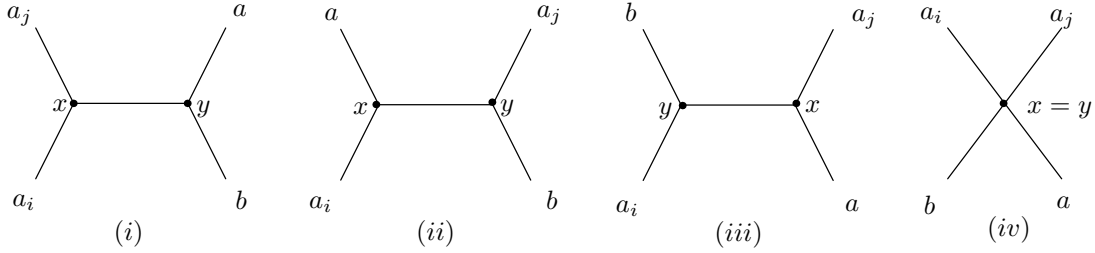}}
	\caption[constellations in the tree]{The present figure illustrates all possible constellations in the Alperin-Bass tree, and is used to distinguish cases in Lemma~\ref{Lem_tec8}. }
	\label{Fig_figure9}
\end{center} 
\end{figure}

The first one of the following lemmata differs from the one in \cite{BennettDiss} in which one of the occuring cases for the value of $\val_T$ was missing.

\begin{lemma}\label{Lem_tec8}
Let $a_1, a_2, a_3 \in E$ be fixed and pairwise distinct. Choose $i\neq j$ and $a\neq b\in E\setminus\{a_i, a_j\}$. Let $(i)-(iv)$ be the possible constellations in the Alperin-Bass tree, as described in Figure~\ref{Fig_figure9}, then 
\begin{align}\label{num29}
\val_T(a_i, a; a_j, b) = 
\left\{\begin{array}{ll}
	a\wedge b - a_j\wedge a - a_i\wedge b  & \text{ in cases }(i), (iii), (iv)\\
	a\wedge b - a_j\wedge b - a_i\wedge a  & \text{ in cases }(i), (ii), (iv)
      \end{array}
\right. 
\end{align}
with  $x=\kappa(a_i, a, a_j)$ and $y=\kappa(a_i,a,b)$.
\end{lemma}
\begin{proof}
It is easy to convince oneself that the constellations described in Figure~\ref{Fig_figure9} are the only possible ones in an Alperin-Bas tree. We prove the above lemma doing a case by case analysis. The cases are enumerated as in Figure~\ref{Fig_figure9}.

We assume w.l.o.g. that $i=1$ and $j=2$.
\begin{bf}Case (i):\end{bf} Hence $x=\kappa(a_1, a, a_2)$ and $y=\kappa(a_1, a,b)$. There are just two possibilities for the position of the origin $o_T=\kappa(a_1, a_2, a_3)$ which lies on $\overleftrightarrow{a_1a_2}$. Either {\bf a)} $o_T$ in contained in the ray $\overrightarrow{xa_1}$ or, {\bf b)} in $\overrightarrow{xa_2}$. Compare figure~\ref{Fig_figure10}.

\begin{figure}[htbp]
\begin{center}
	\resizebox{!}{0.3\textwidth}{\input{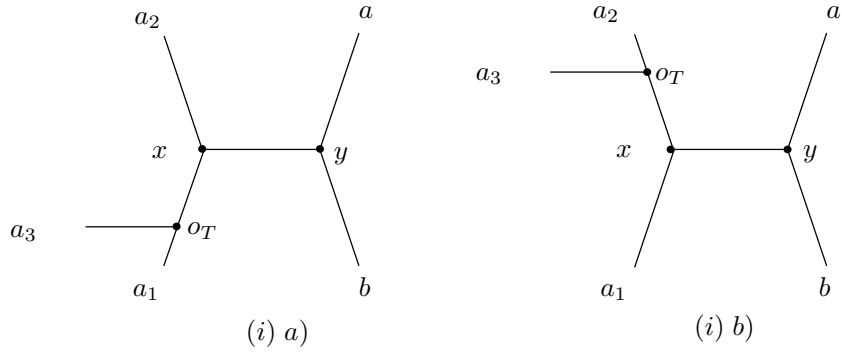}}
	\caption[case of proof]{Possible positions of $o_T$ in case (i) figure~\ref{Fig_figure9}.}
	\label{Fig_figure10}
\end{center} 
\end{figure}

The origin of the tree, as constructed, is $o_T=\kappa(a_1,a_2,a_3)$. The points $x$ and $y$ are, as any point in the Alperin-Bass-tree $T$ of the following form 
\begin{align*}
\left.\begin{array}{l}
x=(a, a\wedge a_2)\\
y=(a, a\wedge b)
\end{array} \right\}  & \text{ if } o_T\in\overrightarrow{xa_1} \\
\left.\begin{array}{l}
x=(a, a\wedge a_1)\\
y=(a, a\wedge b)
\end{array} \right\}  & \text{ if } o_T\in\overrightarrow{xa_2} \\
\end{align*}
The construction also implies, that 
$$\val_T(a_1,a; a_2,b)=\left\{\\
	\begin{array}{ll}
		 d(x,y) & \;\text{ if }y\in\overrightarrow{xa}\\
		-d(x,y) & \;\text{ if }y\notin\overrightarrow{xa}
	\end{array}\right..
$$

Is $o_T\in\overrightarrow{xa_1}$, which is case $(i)\; a)$ of Figure~\ref{Fig_figure10}, we have that $y\in \overrightarrow{xa}$. The definition of $\val_T$ implies
\begin{align*}
 \val_T(a_1,a; a_2,b)	&= d(x,y)= d((a, a\wedge a_2), (b,a\wedge b)) \\
			&= \vert a\wedge a_2 - a\wedge b\vert \\
			&= a\wedge b - a\wedge a_2
\end{align*}
where the second last equation holds since $d(0_T,y) > d(o_T, x)$.

If $o_T\in\overrightarrow{xa_2}$, which is case $(i)\; b)$ of Figure~\ref{Fig_figure10}, we have $y\in \overrightarrow{xa}$. The definition of $\val_T$ hence implies
\begin{align*}
 \val_T(a_1,a; a_2,b)	&= d(x,y)= d((a, a\wedge a_1), (b,a\wedge b)) \\
			&= \vert a\wedge a_1 - a\wedge b\vert \\
			&= a\wedge b - a\wedge a_1
\end{align*}
where the equalities hold by $d(0_T,y) > d(0_T, x)$.\\
Furthermore we have
\begin{align*}
\begin{array}{ll}
a_2\wedge a = a_2\wedge b \;\text{ and }\; b\wedge a_1=a\wedge a_1 =0 & \text{ if } o_T\in\overrightarrow{xa_1} \\
a_1\wedge a = a_1\wedge b \;\text{ and }\; b\wedge a_2=a\wedge a_2 =0 & \text{ if } o_T\in\overrightarrow{xa_2}.
\end{array}
\end{align*}
And hence the following equation holds
$$\val_T(a_1,a; a_2,b) 	=a\wedge b - a_1\wedge a -a_2\wedge b
			=a\wedge b - a_1\wedge b -a_2\wedge a.
$$
This finishes case (i).

\begin{bf}Case (ii):\end{bf} 
In this situation $x=\kappa(a_1,a,a_2)$ and $y=\kappa(a_1,b,a_2)$. As in case $(i)$ we do a case by case analysis. The possible situations here are the following: Either {\bf a)} $o_T\in [x,y]$ or, {\bf b)} $o_T\in \overrightarrow{ya_2}$ or, as a third case, {\bf c)} $o_T\in \overrightarrow{xa_1}$. Compare also figure~\ref{Fig_figure11}. The construction of the tree leads to the following descriptions of the points $x$ and $y$, which are written as in the Alperin-Bass construction, according to the three cases:
\begin{align*}
\left.\begin{array}{r}
x=(a, a\wedge a_1)\\
y=(a, b\wedge a_2)
\end{array} \right\}  & \text{ if } o_T\in [x,y] \\
\left.\begin{array}{r}
x=(a, a\wedge a_1)\\
y=(a, b\wedge a_1)
\end{array} \right\}  & \text{ if } o_T\in \overrightarrow{ya_2} \\
\left.\begin{array}{r}
x=(a, a\wedge a_2)\\
y=(a, b\wedge a_2)
\end{array} \right\}  & \text{ if } o_T\in \overrightarrow{xa_1}. \\
\end{align*}
Since $y\notin\overrightarrow{xa}$ the definition of $\val_T$ implies that $\val_T(a_1,a; a_2,b)= -d(x,y)$.

\begin{figure}[htbp]
\begin{center}
	\resizebox{!}{0.3\textwidth}{\input{graphics/caseii.pdftex_t}}
	\caption[case of proof]{Possible positions of $0_T$ in case (ii) of figure~\ref{Fig_figure9}.}
	\label{Fig_figure11}
\end{center} 
\end{figure}

If $o_T\in[x,y]$, then $a\wedge a_1\geq a\wedge b=0$ and $b\wedge a_2\geq a\wedge b=0$. We calculate:
\begin{align*}
\val_T(a_1,a;a_2,b) 	&= -d(x,y) = -d((a, a\wedge a_1),(b,b\wedge a_2))\\
			&= - (\vert a\wedge a_1 - a\wedge b\vert + \vert b\wedge a_2 - a\wedge b) \\
			&=a\wedge b - a\wedge a_1 - b\wedge a_2.
\end{align*}

If the ray  $\overrightarrow{ya_2}$ contains $o_T$, then $a\wedge a_1 > a\wedge b$, $b\wedge a_1 = a\wedge b$ and $b\wedge a_2=0$ and hence
\begin{align*}
\val_T(a_1,a;a_2,b) 	&= -d(x,y) = -d((a, a\wedge a_1),(b,b\wedge a))\\
			&= - \vert a\wedge a_1 - a\wedge b\vert \\
			&=a\wedge b - a\wedge a_1 - b\wedge a_2.
\end{align*}

In case $o_T\in\overrightarrow{xa_1}$ we have $b\wedge a_2 > a\wedge b$, and $a\wedge a_1 = 0$ and one can conclude:
\begin{align*}
\val_T(a_1,a;a_2,b) 	&= -d(x,y) = -d((a, a\wedge b),(b,b\wedge a_2))\\
			&= - \vert a\wedge b - b\wedge a_2\vert \\
			&=a\wedge b - a\wedge a_1 - b\wedge a_2.
\end{align*}
Therefore, using the above calculations, 
$$
\val_T(a_1,a; a_2,b) 	=a\wedge b - a_1\wedge a -a_2\wedge b. 
$$
Note, that $=a\wedge b - a_1\wedge b -a_2\wedge a \neq \val_T(a_1,a; a_2,b)$.
This finishes case (ii).

Similar observations in {\bf case (iii) and (iv) } imply 
\begin{align*}
\val_T(a_1,a; a_2,b) 	&=a\wedge b - a_1\wedge b -a_2\wedge a \\
			&\neq a\wedge b - a_1\wedge a -a_2\wedge b  \hspace{4ex}\text{ in case (iii)}
\end{align*}

and 
\begin{align*}
\val_T(a_1,a; a_2,b) 	&=a\wedge b - a_1\wedge a -a_2\wedge b \\
			&=a\wedge b - a_1\wedge b -a_2\wedge a   \hspace{4ex}\text{ in case (iv)}.
\end{align*}
Combining these results, the assertion follows.
\end{proof}

In the next lemma an explicit formula for the value of $\val_T$ is given.

\begin{lemma}\label{tec20}
In addition to the assumptions in the previous lemma let 
$$\val_T(a_i,a; a_j,b)=a\wedge b - a_j\wedge a - a_i\wedge b.$$ 
Then, with $\{i,j,k\}=\{1,2,3\}$, we have 
\begin{align*}
\val_T(a_i,a;a_j,b) = &\max\{0,\val(a_{i_a}, a; a_{j_a},b)\} - \max\{0, \val(a_k,a; a_i,a_j)\} \\
			& -\max\{0,\val(a_j,b; a_k,a_i)\}
\end{align*}
where $(i_a,j_a,k_a)$ is an even permutation of $(1,2,3)$ such that the term  $\val(a_{i_a},a;a_{j_a},a_{k_a})$ is maximal. 
A similar statement is true if 
$\val_T(a_i,a; a_j,b)=a\wedge b - a_i\wedge a - a_j\wedge b.$
\end{lemma}
\begin{proof}
Without loss of generality we may assume that $i=1, j=2$.
By definition of $\wedge $ the term $a\wedge b$ equals $\max\{0,\val(a_{i_a},a; a_{j_a},b)\}$. It remains to calculate $a_2\wedge a$ and $a_1\wedge b$.

{\bf a)} Prove $a_1\wedge b= \max\{0,\wedge(a_2,b; a_3,a_1)\}$: \\
If $i_b$ or $j_b=1$ then \ref{num26} of Definition~\ref{Def_rtd} implies $a_1\wedge b=0$. The pair $(i_b, j_b)$ then equals $(1,2)$ or $(3,1)$.

If $(i_b, j_b)=(1,2)$ then, by Lemma~\ref{3ptlemma}, 
$\val(a_1,b;a_2,a_3)  >0$ or $\val(a_i,b; a_j,a_k) =0$ for all choices of $i,j,k $. In the second case the assertion is obvious. In the first case  $(PV2)$ implies 	
$$ 0 < \val(a_1,a_3; a_2,b)=\val(a_2,b; a_1,a_3) = -\val(a_2,b; a_3,a_1) .$$
Hence $\val(a_2,b;a_3,a_1) < 0 $ and we have
$$\max\{0,\val(a_2,b;a_3,a_1)\} = 0 = a_1\wedge b .$$

If $(i_b, j_b)=(3,1)$ then $\val(a_3,b;a_1,a_2)  =0$  or $\val(a_3,b; a_1,a_2)  >0$ for all choices of $i,j,k$. Lemma~\ref{3ptlemma} implies $\val(a_2, b; a_3, a_1)\leq 0$ and hence
$$\max\{0,\val(a_2, b; a_3, a_1) \}=0=a_1\wedge b.$$ 

In the case that $i_b\neq 1$ and $j_b\neq 1$ we have that $(i_b, j_b)=(2,3)$ and $\val(a_2, b; a_3, a_1)  =0$ or that $\val(a_2, b; a_3, a_1) >0 $. In the first case the assertion is obvious, in the second we can conclude
$$\max\{0,\val(a_2,b; a_3,a_1)\} = \val(a_2,b; a_3,a_1) = b\wedge a_1.$$
This finishes part a). 

{\bf b)} Prove $a_2\wedge a= \max\{0,\wedge(a_3,a; a_1,a_2)\}$: \\
If $i_a$ or $j_a=2$ then, by definition, $a\wedge a_2=0$. The pair $(i_a, j_a)$ equals $(2,3)$ or $(1,2)$. 

If $(i_a, j_a)=(1,2)$ then $\val(a_1, a; a_2, a_3) =0$ or $\val(a_1, a; a_2, a_3)>0 .$
Axioms $(PV1), (PV2)$ and Lemma~\ref{3ptlemma} imply $\val(a_3, a; a_1, a_2) < 0$ and hence
$$
\max\{0, \val(a_1,a; a_2, a_3)\} = 0 = a_2\wedge a .
$$

If $(i_a, j_a)=(2,3)$ then $\val(a_2, a; a_3, a_1)=0$ or $\val(a_2, a; a_3, a_1)>0$
In the first case the assertion follows by Lemma~\ref{3ptlemma} and in the second axiom $(PV2)$ implies
\begin{align*}
	0 &< \val(a_2, a; a_3, a_1) =\val(a,a_2;a_1, a_3) \\
		&=\val(a, a_3; a_1, a_2) = - \val(a_3, a; a_1, a_2).
\end{align*}
Hence $\max\{0,\val(a_2, a; a_3, a_1)\}=0=a\wedge a_2.$

If $i_a \neq 2$ and $j_a\neq 2$, then $(i_a, j_a)=(3,1)$ and 
$$
a\wedge a_2 = \{0, \val(a_{i_a}, a; a_{j_a}, a_2)\}=\val(a_3,a; a_1,a_2)
$$
which is maximal. This finishes the proof. 
\end{proof}

Now we are able to prove the main proposition of step 4.

\begin{proof}[Proof of Proposition~\ref{Prop_valuations}]
Let $a_1, a_2, a_3$ be three ends of the Alperin-Bass tree such that $\kappa(a_1,a_2, a_3)=o_T$. 
Given pairwise distinct $a,b,c,d \notin\{ a_1,a_2,a_3\}$ axiom $(PV3)$ implies
\begin{align*}
\val(a,b;c,d) 	&= \val(a_i,b;c,d) + \val(a, a_i; c,d) \\
		&= \val(a_i,b;c,a_j) + \val(a_i; b; a_j, d) + \val(a, a_i;a_j,d) + \val(a, a_i; c,a_j) \\
		&= -\val(a_i,b;a_j,c) + \val(a_i; b; a_j, d) - \val(a_i, a;a_j,d) + \val(a_i, a; a_j, c).
\end{align*}
It is therefore sufficient to prove
$$\val(a_i, a; a_j, b)=\val(a_i, a; a_j, b)$$ for any distinct $a,b \notin\{ a_1,a_2,a_3\}$ and any $i\neq j$.

According to Lemma~\ref{3ptlemma} there are four cases. We will give the proof in the case where $\val(a_3, a; a_1, a_2)> 0$. The proofs of all the other cases use similar arguments and are left to the reader. Here axioms $(PV2)$ and $(PV1)$ imply 
$$0<\val(a_3, a; a_1, a_2)= -\val(a_1, a; a_2, a_3).$$
We are able to apply Lemma~\ref{tec20}. Together with the fact that $\val(a_3, a; a_1, a_2)> 0$ we have
\begin{align}\label{num32}
\val_T(a_1,a; a_2,b)	
	=& \max\{0, \val(a_3,a; a_1,b)\} - \val(a_3,a; a_1,a_2) -\max\{0,\val(a_2,b; a_3,a_1)\}.
\end{align}
By $(PV3)$ 
\begin{equation}\label{num33}
\val(a_1,a;a_2,b) = \val(a_1,a;a_2,a_3) + \val(a_1,a; a_3,b)
\end{equation}
and $\val(a_1,a; a_2, a_3)<0$.
A case by case analysis will complete the proof. 

{\bf Case (i):} assume $\val(a_3,a;a_1,b)>0.$\\
Axiom $\mathrm{(PV2)}$ implies
\begin{equation}\label{num34}
\val(a_1,a;a_3,b)=\val(a_3,a;a_1,b) > 0.
\end{equation}
By (\ref{num32}), (\ref{num33}) and (\ref{num34}) we have
\begin{align*}
\val_T(a_1,a;a_2,b)	
	&= \max\{0,\val(a_3,a;a_1,b)\}-\max\{0,\val(a_2,b;a_3,a_1)\}+\val(a_3,a;a_1,a_2)\\
	&= \val(a_1,a;a_2,b)\}-\max\{0,\val(a_2,b;a_3,a_1)\}.
\end{align*}
Hence it suffices to prove
$$ b \wedge a_1=\max\{0,\val(a_2,b;a_3,a_1)\}=0.$$
Lemma~\ref{tec6} implies $i_a=i_b=3, j_a=j_a=1$. Therefore $b\wedge a_1=0$ by definition and $\val_T(a_1,a;a_2,b)=\val(a_1,a; a_2,b)$.

{\bf Case (ii):} assume $\val(a_3,a;a_1,b)<0.$\\
By $\mathrm{(PV1)}$ we have $\val(a_3,a;b,a_1)>0$. The following two equations hold therefore by $\mathrm{(PV2)}$:
\begin{align*}
0 &= \val(a_3,b;a,a_1)=\val(a_3,b;a_1,a) \text{ and }\\
0 &< \val(a_3,a_1;b,a)=-\val(a_3,a_1,a,b).
\end{align*}
Equation (\ref{num32}) together with the assumption of case (ii) implies 
$$\val_T(a_1,a;a_2,b)= -\max\{0,\val(a_2,b;a_3,a_1)\} + \val(a_1,a; a_2,a_3)$$
and
$$ \val_T(a_1,a;a_2,b)=\val(a_1,a;a_2,b)-\val(a_1,a;a_3,b) -\max\{0,\val(a_2,b;a_3,a_1) \}.$$
But $\val(a_1,a;a_3,b)=0$ and therefore
$$ \val_T(a_1,a;a_2,b)=\val(a_1,a;a_2,b) -\max\{0,\val(a_2,b;a_3,a_1) \}.$$
It remains to prove $\val(a_2,b;a_3,a_1)\leq 0$ for the assertion to hold in case (ii).
But since $i_b=i_a=3$ and $j_b=j_a=1$ we have $\val(a_2,b; a_3,a_1)=\val(a_2,a; a_3,a_1)\geq 0$ and are done.

{\bf Case (iii):} $\val(a_3,a;a_1,b)=0.$\\
We prove several inequalities which, combined with (\ref{num33}), will allow us to finish the proof. 
First, if $\val(a_3,a_1;a,b)<0$, axioms $\mathrm{(PV1)}$ and $\mathrm{(PV2)}$ imply $\val(a_3,a;a_1,b)<0$
 which contradicts the assumption of case (iii). Hence 
\begin{equation}\label{num35}
\val(a_3,a_1;a,b)\geq 0.
\end{equation}
Second, assume $\val(a_1,a;a_3,b)>0$, then, by $\mathrm{(PV2)}$ again $\val(a_3,a;a_1,b)>0$ contradicting the assumption of case (iii). Therefore
\begin{equation}\label{num36}
\val(a_1,a;a_3,b)\leq 0.
\end{equation}
We prove
\begin{equation}\label{num37}
\val(a_3,a_1;a,b)=-\val(a_1,a;a_3,b).
\end{equation}
Because of (\ref{num36}) and $\mathrm{(PV1)}$ we have $\val(a_1,a;b,a_3)\geq 0$.
If $k\define\val(a_1,a;b,a_3)>0$, then $\val(a_1,b;a,a_3)=0$ and $$\val(a_3,a_1;a,b)=k=\val(a_1,a;b,a_3)=-\val(a_1,a;a_3,b)$$ and (\ref{num37}) holds in this case. 
Let  $\val(a_1,a;b,a_3)=0$. Then $\val(a_1,a;a_3,b)=0$ as well. Assume $\val(a_3,a_1;a,b)>0$, which in general is greater or equal to $0$ by (\ref{num35}). By $\mathrm{(PV1)}$ we have then $\val(a_1,a_3;b,a)>0$. But also $\val(a_1,a;b,a_3)>0$, by $(PV2)$, which contradicts the assumption that $\val(a_1,a;b,a_3)=0$. Hence $\val(a_3,a_2;a,b)=0$ and we are finished with the proof of (\ref{num37}).

By (\ref{num37}) one can conclude
\begin{equation}\label{num38}
\val(a_3,a_1;a_2,b) = -\val(a_1,a;a_3,b)+\val(a_2,a; a3,a_1).
\end{equation}
Since $(2,3,1)$ is an even permutation of $(1,2,3)$ and since $i_a=3$ and $j_a=1$ we have $\val(a_2,a;a_3,a_1)\leq 0$. If $\val(a_2,a;a_3,a_1) < 0$ then $\val(a_2,a;a_1,a_3)>0$ and, by $(PV2)$, $\val(a_1,a;a_2,a_3)>0$. But the latter contradicts Lemma~\ref{3ptlemma}, since $i_a=3$ and $j_a=1$. Therefore $\val(a_2,a;a_3,a_1)$ equals $0$ and hence the following holds using equation (\ref{num38})
\begin{equation}\label{num39}
\val(a_3,a_1;a_2,b)=-\val(a_1,a;a_3,b).
\end{equation}

Now we are ready to prove case (iii). Recall (\ref{num32})
$$
\val_T(a_1,a;a_2,b) = -\max\{0,\val(a_2,b;a_3,a_1)\} + \val(a_1,a;a_2,b)-\val(a_1,a;a_3,b).
$$
Using, in this order, (\ref{num37}),(\ref{num39}) and (\ref{num36}), allows us to calculate
\begin{align*}
\val_T(a_1,a;a_2,b) 
	&= -\max\{0,\val(a_2,b;a_3,a_1)\} + \val(a_1,a;a_2,b)+\val(a_3,a_1;a,b)\\
	&= -\max\{0,-\val(a_1,a;a_3,b)\} + \val(a_1,a;a_2,b)  +\val(a_3,a_1;a,b)\\
	&= \val(a_1,a;a_2,b)  +\val(a_1,a; a_3,b) +\val(a_3,a_1;a,b).
\end{align*}
Since $\val(a_1,a; a_3,b) +\val(a_3,a_1;a,b)=0$ case (iii) is finished.
\end{proof}


\newpage
\section{Dictionary}\label{Sec_dictionary}

On the following few pages, we provide a list of expressions used in this thesis and the corresponding names of objects in \cite{BruhatTits, Bennett, Parreau, TitsComo, KleinerLeeb, KapovichMillson} and \cite{AffineW}. Each column contains one object. In the first line the name used in the present thesis is given.

Notice, however, that the correspondences are only one to one if buildings are considered which are covered by both definitions. For example did Parreau prove in \cite{Parreau} that the \emph{Euclidean buildings} of Kleiner and Leeb \cite{KleinerLeeb} are a subclass of affine $\R$-buildings (or \emph{syst\`eme d'appartements}) in the sense of \cite{TitsComo}, namely precisely the ones equipped with the complete system of apartments.
In general one has the following inclusions:

$$
\left\{\!
\begin{array}{c}
\text{generalized}\\
\text{affine buildings}\\
\text{\cite{Bennett}}
\end{array}
\!\right\} 
\supsetneq
\left\{\!
\begin{array}{c}
\text{affine}\\
\R-\text{buildings}\\
\text{\cite{TitsComo, BruhatTits}}\\
\text{\cite{Parreau}}
\end{array}
\!\right\} 
\supsetneq
\left\{\!
\begin{array}{c}
\text{Euclidean}\\
\text{buildings}\\
\text{\cite{KleinerLeeb}}
\end{array}
\!\right\} 
\supsetneq
\left\{\!
\begin{array}{c}
\text{simplicial}\\
\text{affine buildings}\\
\text{\cite{AffineW, KapovichMillson}}
\end{array}
\!\right\} 
$$

Sometimes it was not possible to find a precise counterpart of our definitions; in some cases simply because they did not appear in the considered reference, sometimes because another approach to buildings was used. For this reason we made comments on some of the entries of the table.

Doubtless the presented table is in no way complete. Making a choice is always a subjective thing and implies that I have omitted references others would wish to find here. I am aware of this fact and take the chance to apologize for not including standard references on buildings as for example \cite{AB, Brown, Garrett} or \cite{Ronan}.

\vspace{5ex}
\begin{remark}\label{Remark} Come back to the following remarks whenever there is a reference given in the table.
 \begin{enumerate}
	\item \label{B1} The mentioned faces of a Weyl chamber, called sector panels in \cite{Bennett}, are the Weyl simplices of dimension one. Other Weyl simplices are not named.
	 \item \label{BT1}  In \cite{BruhatTits} spherical Weyl groups are not defined explicitly but the authors remark that for each special vertex $x$ the affine Weyl group can be written as the semi-direct product $W=W_xV$, where $V$ is the translation part of $W$ and $W_x$ the stabilizer of $x$ in $W$.
	\item \label{BT3} Associated with a given ``quartier'' $D$ the authors of \cite{BruhatTits} define a basis $B(D)$. Then $D$ is precisely what is called ``the fundamental Weyl chamber associated to $B(D)$'' in the present thesis.
	\item \label{BT4} Bruhat and Tits use the notion of a retraction based at an alcove $C$. In the simplicial case an alcove together with a distinguished vertex $v$ of $C$ is the same as a germ of a Weyl chamber at $v$ in our sense.
	\item \label{KL2} Kleiner and Leeb define $\Delta_{mod}$ to be the quotient of the Tits boundary $\delta_{Tits}E$ by the spherical Weyl group. This definition does not precisely correspond to the fundamental Weyl chamber, as used in the present thesis.
	\item \label{W1} In \cite{AffineW} buildings are defined as certain chamber systems. The Coxeter chamber system $\Sigma_\Pi$ has a representation as a set of alcoves in a Euclidean vector space $V$. This representation corresponds precisely to our model apartment $\MS$.
	\item \label{W2} A wall of a given root $\alpha$ in \cite{AffineW} is actually defined to be  the collection of panels (which is a certain set of chambers) whose intersections with $\alpha$ have cardinality one.
	\item \label{W3} Since Weiss uses the chamber system approach to affine buildings, which is in a certain sense dual to the simplicial approach, a slight difference to our definition occurs. A sector $S$ is a collection of chambers with terminus $R$. Here $R$ is a residue containing the chamber at the tip of $S$. The type of the basepoint $x$ of $S$ (in our language) determines the type $I\setminus\{o\}$ of $R$.
	\item \label{W4} In \cite{AffineW} convexity is defined with respect to galleries: a set of chambers is convex if for all $c,d$ in the set all galleries connecting $c$ and $d$ are contained in the set as well. This corresponds to $\WT$-convexity in our sense if the intersection of half-apartments is not contained in a wall.
 \end{enumerate}
. 
\end{remark}

\begin{sidewaystable}[!p]
\centering

\begin{tabularx}{\textheight}{| >{\raggedright\arraybackslash} X || >{\raggedright\arraybackslash} X | >{\raggedright\arraybackslash} X | >{\raggedright\arraybackslash} X | >{\raggedright\arraybackslash} X | >{\raggedright\arraybackslash} p{4.5cm} |}
\hline
 present thesis 
	& spherical Weyl group $\sW$
	& affine Weyl group $\WT$, \newline full affine Weyl group $\aW$
	& model space \newline $\MS(\RS,\Lambda,\WT)=\MS$
	& generalized affine building $(X,\App)$ \newline modeled on $\MS$
	& $\App$ atlas; $f\in\App$ chart; apartments are images of charts
	\\ 
\hline
  C. Bennett \newline \cite{Bennett}
	& spherical Weyl group $\overline W$ 
	& affine Weyl group $W= T\overline W$,\newline  full affine Weyl group $W'=\Lambda^n\overline{W}$
	& model space $\Sigma$  
	& $\Lambda$-affine building $(\Delta,\mathcal{F})$ 
	& $\mathcal{F}\ni f$ chart; apartments are images of charts
	\\
 \hline 
  J. Bruhat, J. Tits \newline \cite{BruhatTits}
	& $W_x$ \newline \emph{see \ref{Remark}.\ref{BT1}}
	& $W$ groupe de Weyl affine  
	& $\mathbb{A}$ espace affine 
	& $\mathcal{I}$ immeuble affine 
	& $\phi:\mathbb{A} \rightarrow \mathcal{I}$ application structurale; apartments are images of charts
	\\
\hline 
  M. Kapovich,\newline J. Millson \newline \cite{KapovichMillson}
	& $W_{\mathrm{sph}}$ finite Weyl group
	& $W_{\mathrm{aff}}$,  or simply $W$, discrete Euclidean Coxeter group
	& $(A,W)$ Euclidean Coxeter complex; $A$ model apartment
	& $X$ Bruhat-Tits building (associated with a group)
	& atlas; $\varphi : A\hookrightarrow X$ chart; apartments are images of charts
	\\	 
 \hline 
  B. Kleiner,\newline B. Leeb \newline \cite{KleinerLeeb}
	& spherical Weyl group $\overline W$ 
	& affine Weyl group $W_{\mathrm{aff}}$ 
	& $(E,W_{\mathrm{aff}})$ the Euclidean Coxeter complex
	& Euclidean building $X$ modeled on $(E,W_{\mathrm{aff}})$
	& $\mathcal{A}$ atlas; $ \iota: E\rightarrow X$ chart; apartments are images of charts
	\\
\hline 
  A. Parreau \newline \cite{Parreau}
	& groupe des reflexions fini $\overline W$ 
	&  $W= T W_a$ where $W_a=Stab_W(a)$\newline $\widetilde{W}$ (the full affine Weyl group) 
	& la structur model\'ee $\mathbb{A}$ 
	& $(\Delta,\mathcal{A})$ immeuble affine model\'e sur $(\mathbb{A},W)$ 
	& $\mathcal{A}\ni f$ appartements marqu\'ees; appartement = image de $f$
	\\
 \hline 
  J. Tits \newline \cite{TitsComo}
	& {\vspace{-8pt}$\overline W$ }
	& $W$ 
	& $\mathbb{A}$ espace affine
	& $(\mathcal{I}, \mathcal{F})$
	& $\mathcal{F}$ famille d'injections $f:\mathbb{A}\rightarrow \mathcal{I}$; appartement =  images de $f\in\mathcal{F}$ 
	\\
 \hline 
  R. Weiss \newline \cite{AffineW}
	& $W_\Pi$, $\Pi$ a spherical diagram
	& $W_\Pi$, $\Pi$ an affine diagram
	& $\Sigma_\Pi$ Coxeter chamber system of type $\Pi$\newline \emph{see~\ref{Remark}.\ref{W1}} 
	& affine building $\Delta$ of type $\Pi$
	& apartments are images of a special isomorphism $\phi:\Sigma_\Pi \rightarrow \Delta$
 	\\ 
 \hline 
\end{tabularx}

\end{sidewaystable} 
\clearpage

\begin{sidewaystable}[!p]
\centering

\begin{tabularx}{\textheight}{| >{\raggedright\arraybackslash} X || >{\raggedright\arraybackslash} X | >{\raggedright\arraybackslash} X | >{\raggedright\arraybackslash} X | >{\raggedright\arraybackslash} X | >{\raggedright\arraybackslash} p{0.17 \textheight} |}
\hline
 present thesis
	& $H_{\alpha,k}$ wall or hyperplane
	& half-apartment $H_{\alpha,k}^\pm$
	& $\Cf$ fundamental Weyl chamber (with respect to a basis $B$ of $\RS$)
	& Weyl chamber $S$ based at a point $x$
	& Weyl simplices of $S$  or (if codimension is one) panel of $S$
	\\ 
 \hline
  C. Bennett \newline \cite{Bennett}
	&  $M_r$ wall 
	& half-apartment 
	& fundamental sector $\hat{S}$
	& (closed) sector $S$ based at $x$
	& sector panel $P_i$ of type $i$\newline \emph{see \ref{Remark}.\ref{B1}}
	\\ 
 \hline 
   J. Bruhat, J. Tits \newline \cite{BruhatTits} 
	& mur $L_{a,k}$,\newline l'ensemble des murs $\mathcal{H}$
	& racine affine $\alpha$ dans $\mathbb{A}$\newline 
		demi-appartement = image de racine affine dans  $\mathcal{I}$ 
	& \emph{see \ref{Remark}.\ref{BT3} } 
	& quartier $D$
	& ---
	\\ 
 \hline 
   M. Kapovich,\newline J. Millson \newline \cite{KapovichMillson}\newline{\ }
	& affine wall $H_{\alpha,t}$
	& half-apartment 
	& positive Weyl chamber
	& Weyl chamber $\Delta$ of $W_{\mathrm{sph}}$ with tip $v$
	& Weyl simplices
	\\ 
 \hline 
   B. Kleiner, B. Leeb \newline \cite{KleinerLeeb}
	& wall 
	& half-apartment
	& $\Delta_{\mathrm{mod}}$ anisotropy polyhedron \newline \emph{see \ref{Remark}.\ref{KL2} } 
	& Weyl chamber with tip at a point $p$
	& ---
	\\ 
 \hline 
    A. Parreau \newline \cite{Parreau} 
	& mur affines
	& demi-appartement 
	& {\vspace{-8pt}$\overline{C}$ chambre de Weyl ferm\'ee fondamentale}
	& $C$ chambre de Weyl affine ferm\'ees
	& facettes affines ferm\'ee 
	\\ 
 \hline 
  J. Tits \newline \cite{TitsComo}
	& mur 
	& racine (dans $\mathbb{A}$)\newline demi-appartement (dans l'immeuble) 
	& --- 
	& quartier de sommet $x$
	& facettes des quartiers et cloisons des quartiers\newline(si co\-di\-men\-sion est 1)
	\\ 
 \hline 
   R. Weiss \newline \cite{AffineW}
	& wall of a root \newline \emph{see \ref{Remark}.\ref{W2}}
	& root $\alpha$
	& \emph{see \ref{Remark}.\ref{BT3} } 
	& sector $S$ with terminus $R$ \newline \emph{see \ref{Remark}.\ref{W3}}
	& face of a sector
	\\ 
 \hline
\end{tabularx}

\end{sidewaystable} 
\clearpage

\begin{sidewaystable}[!p]
\centering

\begin{tabularx}{\textheight}{| >{\raggedright\arraybackslash} X || >{\raggedright\arraybackslash} X | >{\raggedright\arraybackslash} p{4.2cm} | >{\raggedright\arraybackslash} X | >{\raggedright\arraybackslash} X | >{\raggedright\arraybackslash} p{4.2cm} |}
\hline
 present thesis \newline{\ }
	& retraction $r$ of axiom $(A5)$
	& retraction $\rho_{A,c}$ onto $A$ centered at $c=\partial S$, or retraction centered at infinity
	& retraction $r_{A,\Delta_xS}$ onto $A$ centered at $\Delta_xS$
	& $\Delta_xS$ germ of $S$ at $X$
	& $\Delta_xX$ residue of $X$ at $x$
	.\\ 
 \hline
  C. Bennett \newline \cite{Bennett}\newline{\ }
	& retraction $\rho$ of axiom $(A5)$
	& sector retraction $\rho_{A,S}$ onto $A$ based at $S$
	& ---
	& ---
	& ---
	\\ 
 \hline 
   J. Bruhat, J. Tits \newline \cite{BruhatTits} \newline{\ }
	& ---
	& $\rho_{A, \mathcal{C}}$ r\'etraction de $\mathcal{I}$ sur $A$ relativement au quartier $\mathcal{C}$ 
	& $\rho$ r\'etraction de $\mathcal{I}$ sur $A$ de centre $C$ \newline \emph{see \ref{Remark}.\ref{BT4}}
	& $\gamma(x,D)$ germe de quartier $D$
	& ---
	\\ 
 \hline 
   B. Kleiner, B. Leeb \newline \cite{KleinerLeeb} \newline{\ }
	& ---
	& ---
	& ---
	& ---
	& $\Sigma_xX$ space of directions at $x$
	\\
 \hline 
   M. Kapovich, J. Millson \newline \cite{KapovichMillson} \newline{\ }
	& ---
	& ---
	& $Fold_{a, A}$ folding of $X$ onto $A$\newline ($a$ alcove in $A$) 
	& ---
	& $\Sigma_xX$ space of directions at $X$; as a polysimplicial complex the link of $X$ at $x$
	\\ 
 \hline 
    A. Parreau \newline \cite{Parreau} \newline{\ }
	& r\'etraction $r$ de l'axiom $(A5')$
	& r\'etraction $r$ centr\'ee en $C(\infty)$
	& $r$ centr\'ee en $\Sigma_xC$
	& $\Sigma_xC$ germe en $x$ de $C$
	& $\Sigma_xX$ germes en $x$
	\\ 
 \hline 
  J. Tits \newline \cite{TitsComo} \newline{\ }
	& r\'etraction $\rho$ de l'axiom $(A5)$
	& ---
	&  ---
	& ---
	& ---
	\\ 
 \hline 
 R. Weiss \newline \cite{AffineW} \newline{\ }
	& ---
	& sector retraction $\rho_{A,S}$
	& retraction map $r_{A,x}$ 
	& $d$ apex of $S$
	& gem; or residue of type $I\setminus\{o\}$ 
	\\ 
 \hline
\end{tabularx}

\end{sidewaystable} 
\clearpage

\begin{sidewaystable}[!p]
\centering

\begin{tabularx}{\textheight}{| >{\raggedright\arraybackslash} X || >{\raggedright\arraybackslash} X | >{\raggedright\arraybackslash} X | >{\raggedright\arraybackslash} X | p{0.28\textheight}}
 \cline{1-4}
 present thesis\newline{\ }\newline{\ }
	& parallel walls / hyperplanes
	& $\aW$-convex set
	& $\seg_A(x,y)$ segment of $x$ and $y$ in $A$
	&
	\\ 
 \cline{1-4}
  C. Bennett \newline \cite{Bennett} \newline{\ }
	& parallel walls
	& closed convex set
	& $c(\{x,y\})$ convex hull of $x$ and $y$
	&
	\\ 
 \cline{1-4}
   J. Bruhat, J. Tits \newline \cite{BruhatTits}  \newline{\ }
	& murs equipollentes 
	& ---
	& segment ferm\'e $[xy]$
	&
	\\ 
 \cline{1-4}
   B. Kleiner,\newline B. Leeb \newline \cite{KleinerLeeb} \newline{\ }
	& asymptotic lines 
	& Weyl polyhedron
	& geodesic segment $\overline{xy}$
	&
	\\
 \cline{1-4}
   M. Kapovich, \newline J. Millson \newline \cite{KapovichMillson} \newline{\ }
	& ---
	& ---
	& geodesic segment $\overline{xy}$	
	&
	\\ 
 \cline{1-4}
    A. Parreau \newline \cite{Parreau} \newline{\ }
	& asymptotes
	& convexe ferm\'e
	& segment $[x,y]$
	&
	\\ 
 \cline{1-4}
  J. Tits \newline \cite{TitsComo} \newline{\ }\newline{\ }
	& (facettes de quartiers) sont parall\`eles
	& convexe ferm\'e
	& ---
	&
	\\ 
 \cline{1-4}
 R. Weiss \newline \cite{AffineW} \newline{\ }
	& parallel walls
	& convex \newline \emph{see \ref{Remark}.\ref{W4}}
	& ---
	&
	\\ 
 \cline{1-4}
\end{tabularx}

\end{sidewaystable} 
\clearpage
\newpage
\section{Segments and condition (FC)}\label{Sec_FC}

In section \ref{Sec_convexityRevisited} we made use of the \emph{finite cover condition} (FC) to prove that certain retractions are distance non-increasing. By the time this thesis was handed in, the proof that this condition holds for all generalized affine buildings in the sense of definition \ref{Def_LambdaBuilding} was not yet completed.
In the meantime I was able to verify (FC). Continue reading and you will find a proof on the subsequent pages.

We first prove that segments are contained in apartments.  Let in the following $(X,\App)$ be a generalized affine building and recall the definition of the \emph{segment} of points $x$ and $y$ in $X$:
$$\seg(x,y)\define\{z\in X : d(x,y)=d(x,z)+d(z,y) \}.$$

\begin{lemma}\label{Lem_tec31}
Let $x,y$ be points in $X$. Let $A$ be an apartment containing $x$ and $y$. For all retractions $r\define r_{A,S}$ based at a germ of a Weyl chamber $S$ in $A$ and all $z\in \seg(x,y)$ the following holds
\begin{enumerate}
  \item $r(z)\in \seg(x,y)$ and 
  \item $d(x,z)=d(x,r(z))$ as well as $d(y,z)=d(y,r(z))$.
\end{enumerate}
\end{lemma}
\begin{proof}
Since $r$ is distance non-increasing and $d$ satisfies the triangle inequality we have
$$
d(x,y)\leq d(x,r(z)) + d(y,r(z)) \leq d(x,z)+d(y,z)=d(x,y)
$$
and $r(z)$ is contained in $\seg(x,y)$.

Therefore 
\begin{equation}\label{Equ_1}
d(x,r(z))+d(r(z),y) = d(x,z)+d(z,y).
\end{equation}
The assumption $d(x,r(z)) < d(x,z)$ contradicts equation \ref{Equ_1}, hence $d(x,z)=d(x,r(z))$ and for symmetric reasons $d(y,z)=d(y,r(z))$.
\end{proof}

\begin{lemma}\label{Lem_oppSectors}
Given an apartment $A$ and two Weyl chambers $S$ and $T$ in $A$ facing in opposite directions. If $T\cap S$ is nonempty then 
$T\cap S$ is contained in the segment of a pair of points $s,t$ with $s\in S\setminus T$ and $t\in T\setminus S$.
\end{lemma}
\begin{proof}
Since $S$ and $T$ are opposite, the intersection is the $\aW$-convex hull of the base points $s$ and $t$ of $S$ and $T$. By Lemma \ref{Lem_segment} this equals the segment of $s$ and $t$. \end{proof}

\begin{prop}\label{Prop_SegInApp}
Segments are contained in apartments.
\end{prop}
\begin{proof}
Fix two points $x,y\in X$. Let $A$ be an apartment containing $x$ and $y$ and assume there exists $z\in\seg(x,y)\setminus A$.
Let $p$ be a point in $A$ such that there exists a Weyl chamber $S$ at $p$ containing $z$ whose intersection with $A$ equals $p$, i.e. the germ of $S$ is not contained in $A$. There are two different Weyl chambers $S_1$ and $S_2$ in $A$ with $\sigma_i\define \Delta_pS_i$ is opposite $\Delta_pS$ for $i=1,2$. Denote by $r_i$ the retraction $r_{A,S_i}$ onto $A$ based at the germ of $S_i$ and denote by $z_i$ the image of $z$ under $r_i$. Obviously $z_1\neq z_2$.

There exists a $w\in\sW$ such that $S_2=wS_1$ and $w.z_1 = z_2$. We therefore can find a sequence of reflections $r^i\define r_{\alpha_i, k_i}$, $i=1,\ldots, l$, with $\alpha_i\in B$ and $k_i\in\Lambda$ such that 
$$z_2= r^l(r^{l-1}( \ldots r^1(z_1) \ldots )).$$

But recall, that by Lemma \ref{Lem_segment} the segment of $x$ and $y$ in $A$ is the same as the convex hull of $x$ and $y$. Therefore $\seg(x,y)$ is contained in a Weyl chamber based at $x$.  By Lemma \ref{Lem_tec31} both $z_1$ and $z_2$ are contained in $\seg(x,y)$, therefore $w$ has to be the identity. But this contradicts the fact that $S_1\neq S_2$.
\end{proof}

In particular Proposition \ref{Prop_SegInApp} shows that $\seg(x,y)$ is the same as $\seg(x,y)$ for all apartments $A$ containing $x$ and $y$.
We continue proving a local version (FC') of the cover condition and will conclude the chapter with a proof of (FC).

\begin{lemma}
Given an apartment $A$ and a point $z$ in $X$. Then $A$ is contained in the (finite) union of all Weyl chambers based at $z$ with equivalence class contained in $\partial A$.
\end{lemma}
\begin{proof}
Is $z$ contained in $A$ this is obvious. Hence assume that $z$ is not contained in $A$. For all $p\in A$ there exists, by definition, an apartment $A'$ containing $z$ and $p$. Let $S_+$ be a Weyl chamber based at $p$ containing $z$ and denote by $\sigma_+$ its germ at $p$. There exists a unique Weylchamber $S_-$ contained in $A$, based at $p$, such that its germ $\sigma_-$ is opposite $\sigma_+$ in the residue $\Delta_pX$. By Corollary~\ref{Cor_CO} the Weyl chambers $S_-$ and $S_+$ are contained in common apartment $A''$. Let $T$ be the unique representative of $\partial S_-$ which is based at $z$. Since $z\in S_+$ and $\sigma_+$ and $\sigma_-$ are opposite in $\Delta_pX$, the point $p$ is contained in $T$. Therefore $T$ contains $p$ and $\partial T$ is contained in $A$. To finish the prove observe that there are only finitely many chambers in $\partial A$, and that by Corollary~\ref{Cor_WeylChamber} for each of them there exists a unique representing Weyl chamber based at $z$.
\end{proof}

\begin{lemma}\label{Lem_FC'}
Let $(X,\App)$ be an generalized affine building, let $x,y$ be points in $X$ contained in a common apartment $A$. For any $z\in X$ the following is true:
\begin{itemize}[label={(FC')}, leftmargin=*]
  \item[$\mathrm{(FC')}$] The segment $\seg(x,y)$ of $x$ and $y$ in $A$ is contained in a finite union of Weyl chambers based at $z$.
\end{itemize}
Furthermore, is $\mu$ a germ of a Weyl chamber based at $z$, then $\seg(x,y)$ is contained in a finite union of apartments containing $\mu$.
\end{lemma}
\begin{proof}
Condition (FC') is a direct consequence of the previous lemma.

Let $T$ be a Weyl chamber based at $z$, containing a given point $p\in \seg(x,y)$ with the property that $\partial T\in\partial A$. Apply Proposition~\ref{Prop_tec16} to $c\define\partial T$ and any Weyl chamber $S$ having germ  $\mu$ to obtain that $\seg(x,y)$ is contained in a finite union of apartments containing $\mu$.
\end{proof}

\begin{prop}\label{Prop_retraction}
For any apartment $A$ and any Weyl chamber $S$ in $X$ with germ in $A$ the retraction $r_{A,S}$, as defined in \ref{Def_vertexRetraction}, is distance non-increasing.
\end{prop}
\begin{proof}
Given two points $x$ and $y$ in $X$. By Proposition \ref{Lem_FC'} there exists a finite collection of apartments $A_0,\ldots, A_n$ enumerated and chosen in such a way that 
\begin{itemize}
  \item each $A_i$ contains the germ of $S$
  \item their union contains the segment of $x$ and $y$ in an apartment $A\ni x,y$ and 
  \item $A_i\cap A_{i+1} \neq \emptyset$ for all $i=0,\ldots, n-1$. 
\end{itemize}
Observe that one can find a finite sequence of points $x_i$, $i=0,\ldots,n$ with $x_0=0$ and $x_n=y$ such that 
$$d(x,y)=\sum_{i=0}^{n-1} d(x_i,x_{i+1})$$ and the points $x_i$ and $x_{i+1}$ are contained in $A_i$.

Notice that for all $i$ the restriction of $r_{A,S}$ to $A_i$ is an isomorphism onto $A$. Hence $d(x_i, x_{i+1})=d(\rho(x_i), \rho(x_{i+1}))$ for all $i\neq N$. By the triangle inequality for $d$ we have that $d(r(x),r(y))\leq d(x,y)$.
\end{proof}

Backtracking the proof of \ref{Prop_retraction} it is easy to see, that besides axioms $(A1)-(A4)$ and $(A6)$ we only needed the fact that the distance function on $X$, induced by the distance function on the model space which was defined in Section \ref{Sec_modelSpace}, satisfies the triangle inequality. From this we can define a distance non-increasing retraction satisfying the condition of $(A5)$. Therefore we have the following

\begin{corollary}
Let $(X,\App)$ be a space satisfying all conditions of Definition~\ref{Def_LambdaBuilding} except axiom $(A5)$. Then the following are equivalent:
\begin{itemize}[label={(A5*)}, leftmargin=*]
\item[(A5)] For any apartment $A$ and all $x\in A$ there exists a \emph{retraction} $r_{A,x}:X\to A$ such that $r_{A,x}$ does not increase distances and $r^{-1}_{A,x}(x)=\{x\}$.
\item[(A5')] The distance function $d$ on $X$, which is induced by the distance function on the model space, satisfies the triangle inequality.
\end{itemize}
\end{corollary}

\begin{lemma}\label{Lem_tec40}
Let $S$ be a Weyl chamber in an apartment $A$ of a generalized affine building $(X,\App)$. Let $Y\define \seg(x,y)$ for two points $x,y$ in $X$. Then there exists a sub-Weyl chamber $S'$ of $S$ based at a point $z\in A$ such that $r\define r_{A,\Delta_zS}$ satisfies the following:
$$d(r(p),w) = d(p,w) \text{ for all } w\in S' \text{ and all } p\in Y.$$
Furthermore a Weyl chamber based at $z$ containing a point $p$ of $Y$ is opposite $S'$.
\end{lemma}
\begin{proof}
Let $x'$ be some point in $A$ and let $\lambda$ be such that $Y$ is contained in $B_\lambda(x')$. Define $Z\define B_\lambda(x')\cap A$  and let $S''$ be a Weyl chamber in $A$ opposite $S$ such that $Z$ is contained in $S''$. Denote by $z$ its basepoint and let $S'$ be the Weyl chamber at $z$ parallel to $S$.
Lemma \ref{Lem_oppSectors} and the fact that $r(Y)\subset Z\subset S''$ imply 
$$d(q,r(p))=d(q,z)+d(z,r(p)) = d(q,z) + d(z,p)$$
for all $q\in S'$ and all $p\in Y$. Therefore $z$ is contained in the segment of $p$ and $q$.
By Lemma \ref{Prop_SegInApp} $p,q$ and $z$ are contained in a common apartment.
\end{proof}

\begin{prop}\label{Prop_FC}
Let $x$ and $y$ be points in $X$ both contained in an apartment $A$. For any $c\in \binfinity X$ the following is true:
\begin{itemize}[label={(FC')}, leftmargin=*]
  \item[$\mathrm{(FC)}$] The segment $\seg(x,y)$ of $x$ and $y$ is contained in a finite union of apartments containing $c$ at infinity.
\end{itemize}
\end{prop}
\begin{proof}
Lemma \ref{Lem_tec40} combined with Proposition \ref{Prop_SegInApp} implies that there exists a point $z$ such that for each $p\in \seg(x,y)$ there is a Weyl chamber $S$ based at $z$ containing $p$ which is in $\Delta_zX$ opposite the unique representative of $c$ based at $z$. By Corollary \ref{Cor_CO}
these two are contained in a common apartment. Together with $(FC')$ the assertion follows.
\end{proof}

As already argued in Corollary \ref{Cor_rho-distancediminishing} condition (FC) implies that the retractions centered at infinity, defined in \ref{Def_retractionInfty}, are distance non-increasing.

\newpage
\end{appendix}

\phantomsection
\renewcommand{\refname}{Bibliography}
\addcontentsline{toc}{section}{\refname}
\bibliography{files/literaturliste}

\begin{thebibliography}{HKW08}

\bibitem[AB87]{AlperinBass}
R.~Alperin and H.~Bass.
\newblock Length functions of group actions on {$\Lambda$}-trees.
\newblock In {\em Combinatorial group theory and topology (Alta, Utah, 1984)},
  volume 111 of {\em Ann. of Math. Stud.}, pages 265--378. Princeton Univ.
  Press, Princeton, NJ, 1987.

\bibitem[AB08]{AB}
P.~Abramenko and K.~S. Brown.
\newblock {\em Buildings, {T}heory and applications}, volume 248 of {\em
  Graduate Texts in Mathematics}.
\newblock Springer, New York, 2008.

\bibitem[Ben90]{BennettDiss}
C.~D. Bennett.
\newblock Affine {$\Lambda$}-buildings.
\newblock {\em Dissertation, Chicago Illinois}, pages 1--106, 1990.

\bibitem[Ben94]{Bennett}
C.~D. Bennett.
\newblock Affine {$\Lambda$}-buildings. {I}.
\newblock {\em Proc. London Math. Soc. (3)}, 68(3):541--576, 1994.

\bibitem[BK08]{BerensteinKapovich}
A.~Berenstein and M.~Kapovich.
\newblock Affine buildings for dihedral groups.
\newblock {\em arXiv:0809.0300v1 [math.MG]}, 2008.

\bibitem[Bou02]{Bourbaki4-6}
N.~Bourbaki.
\newblock {\em {L}ie groups and {L}ie algebras. {C}hapters 4--6}.
\newblock Elements of Mathematics (Berlin). Springer Verlag, Berlin, 2002.
\newblock Translated from the 1968 French original by Andrew Pressley.

\bibitem[Bou05]{Bourbaki7-9}
N.~Bourbaki.
\newblock {\em {L}ie groups and {L}ie algebras. {C}hapters 7--9}.
\newblock Elements of Mathematics (Berlin). Springer-Verlag, Berlin, 2005.
\newblock Translated from the 1975 and 1982 French originals by Andrew
  Pressley.

\bibitem[Bro89]{Brown}
K.~S. Brown.
\newblock {\em Buildings}.
\newblock Springer-Verlag, New York, 1989.

\bibitem[BT72]{BruhatTits}
F.~Bruhat and J.~Tits.
\newblock Groupes r\'eductifs sur un corps local.
\newblock {\em Inst. Hautes \'Etudes Sci. Publ. Math.}, (41):5--251, 1972.

\bibitem[BT84]{BruhatTits2}
F.~Bruhat and J.~Tits.
\newblock Groupes r\'eductifs sur un corps local. {II}. {S}ch\'emas en groupes.
  {E}xistence d'une donn\'ee radicielle valu\'ee.
\newblock {\em Inst. Hautes \'Etudes Sci. Publ. Math.}, (60):197--376, 1984.

\bibitem[Chi01]{Chiswell}
I.~Chiswell.
\newblock {\em Introduction to {$\Lambda$}-trees}.
\newblock World Scientific Publishing Co. Inc., River Edge, NJ, 2001.

\bibitem[CK03]{ChiswellKoudela}
I.~M. Chiswell and E.~Koudela.
\newblock Projective valuations, {$D$}-relations, and ends of
  {$\Lambda$}-trees.
\newblock {\em Comm. Algebra}, 31(9):4547--4569, 2003.

\bibitem[CL01]{LC}
R.~Charney and A.~Lytchak.
\newblock Metric characterizations of spherical and {E}uclidean buildings.
\newblock {\em Geom. Topol.}, 5:521--550 (electronic), 2001.

\bibitem[Dav98]{Davis}
M.~W. Davis.
\newblock Buildings are {${\rm CAT}(0)$}.
\newblock In {\em Geometry and cohomology in group theory ({D}urham, 1994)},
  volume 252 of {\em London Math. Soc. Lecture Note Ser.}, pages 108--123.
  Cambridge Univ. Press, Cambridge, 1998.

\bibitem[Gar97]{Garrett}
P.~Garrett.
\newblock {\em Buildings and classical groups}.
\newblock Chapman \& Hall, London, 1997.

\bibitem[GL05]{GaussentLittelmann}
S.~Gaussent and P.~Littelmann.
\newblock L{S} galleries, the path model and {MV} cycles.
\newblock {\em Duke Math. J.}, 127(1):35--88, 2005.

\bibitem[Hit08]{Convexity}
P.~Hitzelberger.
\newblock Kostant convexity for affine buildings. {P}reprint, submitted.
\newblock {\em arXiv:math/0701094v2}, 2008.

\bibitem[HKW08]{Octagons}
P.~Hitzelberger, L.~Kramer, and R.~M. Weiss.
\newblock Non-discrete {E}uclidean {B}uildings for the {R}ee and {S}uzuki
  groups.
\newblock {\em Preprint, submitted. arXiv:0810.2725v1}, 2008.

\bibitem[Hum72]{Humphreys}
J.~E. Humphreys.
\newblock {\em Introduction to {L}ie algebras and representation theory}.
\newblock Springer-Verlag, New York, 1972.
\newblock Graduate Texts in Mathematics, Vol. 9.

\bibitem[KL97]{KleinerLeeb}
B.~Kleiner and B.~Leeb.
\newblock Rigidity of quasi-isometries for symmetric spaces and {E}uclidean
  buildings.
\newblock {\em Inst. Hautes \'Etudes Sci. Publ. Math.}, (86):115--197, 1997.

\bibitem[KM08]{KapovichMillson}
M.~Kapovich and J.~Millson.
\newblock A path model for geodesics in {E}uclidean buildings and its
  applications to representation theory.
\newblock {\em Groups Geom. Dyn.}, 2(3):405--480, 2008.

\bibitem[Kos73]{Kostant}
B.~Kostant.
\newblock On convexity, the {W}eyl group and the {I}wasawa decomposition.
\newblock {\em Ann. Sci. \'Ecole Norm. Sup. (4)}, 6:413--455 (1974), 1973.

\bibitem[Kou97]{Koudela}
E.~Koudela.
\newblock Group actions on {$\Lambda$}-trees.
\newblock {\em Dissertation, Queen Mary and Westfield College}, 1997.

\bibitem[KT02]{KramerTent}
L.~Kramer and K.~Tent.
\newblock Affine {$\Lambda$}-buildings, ultrapowers of {L}ie groups and
  {R}iemannian symmetric spaces: an algebraic proof of the {M}argulis
  conjecture.
\newblock {\em arXiv:math.DG/0209122}, 2002.

\bibitem[KT04]{KTcones}
Linus Kramer and Katrin Tent.
\newblock Asymptotic cones and ultrapowers of {L}ie groups.
\newblock {\em Bull. Symbolic Logic}, 10(2):175--185, 2004.

\bibitem[Lee00]{Leeb}
B.~Leeb.
\newblock {\em A characterization of irreducible symmetric spaces and
  {E}uclidean buildings of higher rank by their asymptotic geometry}.
\newblock Bonner Mathematische Schriften [Bonn Mathematical Publications], 326.
  Universit\"at Bonn Mathematisches Institut, Bonn, 2000.

\bibitem[Lit95]{LittelmannPaths}
P.~Littelmann.
\newblock Paths and root operators in representation theory.
\newblock {\em Ann. of Math. (2)}, 142(3):499--525, 1995.

\bibitem[Mor92]{Morgan}
J.~W. Morgan.
\newblock {$\Lambda$}-trees and their applications.
\newblock {\em Bull. Amer. Math. Soc. (N.S.)}, 26(1):87--112, 1992.

\bibitem[MS84]{MorganShalen}
J.~W. Morgan and P.~B. Shalen.
\newblock Valuations, trees, and degenerations of hyperbolic structures. {I}.
\newblock {\em Ann. of Math. (2)}, 120(3):401--476, 1984.

\bibitem[MV00]{MirkovicVilonen}
I.~Mirkovi{\'c} and K.~Vilonen.
\newblock Perverse sheaves on affine {G}rassmannians and {L}anglands duality.
\newblock {\em Math. Res. Lett.}, 7(1):13--24, 2000.

\bibitem[Par00]{Parreau}
A.~Parreau.
\newblock Immeubles affines: construction par les normes et \'etude des
  isom\'etries.
\newblock In {\em Crystallographic groups and their generalizations (Kortrijk,
  1999)}, volume 262 of {\em Contemp. Math.}, pages 263--302. Amer. Math. Soc.,
  Providence, RI, 2000.

\bibitem[PR08]{ParkinsonRam}
J.~Parkinson and A.~Ram.
\newblock Alcove walks, buildings, symmetric functions and representations.
\newblock {\em arXiv:0807.3602v1}, 2008.

\bibitem[Rap00]{Rapoport}
M.~Rapoport.
\newblock A positivity property of the {S}atake isomorphism.
\newblock {\em Manuscripta Math.}, 101(2):153--166, 2000.

\bibitem[Ron89]{Ronan}
M.~Ronan.
\newblock {\em Lectures on buildings}, volume~7 of {\em Perspectives in
  Mathematics}.
\newblock Academic Press Inc., Boston, MA, 1989.

\bibitem[Rou08]{Rousseau}
G.~Rousseau.
\newblock Euclidean buildings.
\newblock In {\em Géométries à courbure négative ou nulle, groupes discrets
  et rigidité}, volume numéro 18,
  (http://hal.archives-ouvertes.fr/hal-00094363) of {\em Séminaires et
  congrès (Société Mathématique de France)}, pages 1--40. 2008.

\bibitem[Sch23]{Schur}
I.~Schur.
\newblock {\em Sitzungsber. Berl. Math. Ges.}, 22:9--20, 1923.

\bibitem[Sch06]{Schwer}
C.~Schwer.
\newblock Galleries, {H}all-{L}ittlewood polynomials, and structure constants
  of the spherical {H}ecke algebra.
\newblock {\em Int. Math. Res. Not.}, pages Art. ID 75395, 31, 2006.

\bibitem[Sil75]{Silberger}
A.~J. Silberger.
\newblock Convexity for a simply connected {$p$}-adic group.
\newblock {\em Bull. Amer. Math. Soc.}, 81(5):910--912, 1975.

\bibitem[Tit83]{TitsOctagons}
J.~Tits.
\newblock Moufang octagons and the {R}ee groups of type {$\sp{2}F\sb{4}$}.
\newblock {\em Amer. J. Math.}, 105(2):539--594, 1983.

\bibitem[Tit86]{TitsComo}
J.~Tits.
\newblock Immeubles de type affine.
\newblock In {\em Buildings and the geometry of diagrams (Como, 1984)}, volume
  1181 of {\em Lecture Notes in Math.}, pages 159--190. Springer, Berlin, 1986.

\bibitem[Tit95]{TitsSuzukiRee}
J.~Tits.
\newblock Les groupes simples de {S}uzuki et de {R}ee.
\newblock In {\em S\'eminaire {B}ourbaki, {V}ol.\ 6}, pages Exp.\ No.\ 210,
  65--82. Soc. Math. France, Paris, 1995.

\bibitem[TW02]{TW}
J.~Tits and R.~M. Weiss.
\newblock {\em Moufang polygons}.
\newblock Springer Monographs in Mathematics. Springer-Verlag, Berlin, 2002.

\bibitem[Wei03]{Weiss}
R.~M. Weiss.
\newblock {\em The structure of spherical buildings}.
\newblock Princeton University Press, Princeton, NJ, 2003.

\bibitem[Wei08]{AffineW}
R.~M. Weiss.
\newblock {\em The structure of affine buildings}.
\newblock Annals of mathematics studies. Princeton University Press, Princeton,
  NJ, 2008.

\end{thebibliography}
\bibliographystyle{alpha}
\newpage

\end{document}